\documentclass{article}

\usepackage[final]{neurips_2025}

\usepackage[utf8]{inputenc}
\usepackage[T1]{fontenc}

\usepackage{amsmath,amssymb,amsfonts,amsthm,bbm}
\usepackage{mathtools}
\usepackage{microtype}
\usepackage{xcolor}

\usepackage{url}
\usepackage{booktabs}
\usepackage{nicefrac}
\usepackage{enumerate,enumitem}
\usepackage{epsfig,subfigure}
\usepackage{algorithm,algorithmic}
\usepackage{framed}
\usepackage{setspace}
\usepackage[title]{appendix}
\usepackage{tikz}
\usetikzlibrary{positioning,calc,arrows.meta}
\usepackage[most]{tcolorbox}
\usepackage{pifont}
\usepackage{multirow}
\usepackage{tabulary}
\usepackage{colortbl}
\usepackage{tablefootnote}
\usepackage{threeparttable}
\allowdisplaybreaks[4]

\usepackage{hyperref}

\usepackage{aliascnt}

\theoremstyle{plain}
\newtheorem{theorem}{Theorem}[section]

\newaliascnt{lemma}{theorem}
\newtheorem{lemma}[lemma]{Lemma}
\aliascntresetthe{lemma}

\newaliascnt{corollary}{theorem}
\newtheorem{corollary}[corollary]{Corollary}
\aliascntresetthe{corollary}

\newaliascnt{proposition}{theorem}

\aliascntresetthe{proposition}

\theoremstyle{definition}
\newaliascnt{definition}{theorem}
\newtheorem{definition}[definition]{Definition}
\aliascntresetthe{definition}

\newaliascnt{assumption}{theorem}
\newtheorem{assumption}[assumption]{Assumption}
\aliascntresetthe{assumption}

\theoremstyle{remark}
\newaliascnt{remark}{theorem}
\newtheorem{remark}[remark]{Remark}
\aliascntresetthe{remark}

\usepackage[capitalize,noabbrev]{cleveref}
\crefname{equation}{Eq.}{Eqs.}
\Crefname{equation}{Eq.}{Eqs.}

\crefname{theorem}{Theorem}{Theorems}
\Crefname{theorem}{Theorem}{Theorems}

\crefname{lemma}{Lemma}{Lemmas}
\Crefname{lemma}{Lemma}{Lemmas}

\crefname{corollary}{Corollary}{Corollaries}
\Crefname{corollary}{Corollary}{Corollaries}

\crefname{proposition}{Proposition}{Propositions}
\Crefname{proposition}{Proposition}{Propositions}

\crefname{definition}{Definition}{Definitions}
\Crefname{definition}{Definition}{Definitions}

\crefname{assumption}{Assumption}{Assumptions}
\Crefname{assumption}{Assumption}{Assumptions}

\crefname{remark}{Remark}{Remarks}
\Crefname{remark}{Remark}{Remarks}

\def\ca#1{\mathcal{#1}}

\def\w#1{\widetilde{#1}}

\title{Problem-Parameter-Free \\ Decentralized Bilevel Optimization}

\author{
  Zhiwei Zhai\quad Wenjing Yan\quad Ying-Jun Angela Zhang \\
  Department of Information Engineering\\
  The Chinese University of Hong Kong\\
  \texttt{\{zz024, wjyan, yjzhang\}@ie.cuhk.edu.hk} \\
}

\begin{document}

\maketitle

\begin{abstract}
Decentralized bilevel optimization has garnered significant attention due to its critical role in solving large-scale machine learning problems. However, existing methods often rely on prior knowledge of problem parameters—such as smoothness, convexity, or communication network topologies—to determine appropriate stepsizes. In practice, these problem parameters are typically unavailable, leading to substantial manual effort for hyperparameter tuning. In this paper, we propose \textbf{AdaSDBO}, a fully problem-parameter-free algorithm for decentralized bilevel optimization with a single-loop structure. AdaSDBO leverages adaptive stepsizes based on cumulative gradient norms to update all variables simultaneously, dynamically adjusting its progress and eliminating the need for problem-specific hyperparameter tuning. Through rigorous theoretical analysis, we establish that AdaSDBO achieves a convergence rate of $\w{\mathcal{O}}\left(\frac{1}{T}\right)$, matching the performance of well-tuned state-of-the-art methods up to polylogarithmic factors. Extensive numerical experiments demonstrate that AdaSDBO delivers competitive performance compared to existing decentralized bilevel optimization methods while exhibiting remarkable robustness across diverse stepsize configurations.
\end{abstract}

\section{Introduction}
\vspace{-0.15cm}

Bilevel optimization is a powerful framework widely applied in machine learning, artificial intelligence, and operations research \citep{camacho2024metaheuristics, caselli2024bilevel}. In bilevel optimization, the objective is to optimize a function that is itself dependent on an optimization problem, creating a hierarchical structure of decision-making. This framework models numerous real-world problems where decisions at one level influence outcomes at another, including reinforcement learning \citep{hong2023two, thoma2024contextual, shen2025principled}, meta-learning \citep{bertinetto2018meta, rajeswaran2019meta, ji2020convergence}, adversarial learning \citep{mkadry2017towards}, hyperparameter optimization \citep{pedregosa2016hyperparameter, franceschi2018bilevel}, and imitation learning \citep{arora2020provable}. The flexibility of bilevel optimization makes it an essential tool for modeling complex systems and tackling a wide range of challenges in modern machine learning and optimization.

As datasets continue to grow and machine learning models become more complex, bilevel optimization increasingly necessitates decentralized computation paradigms \citep{kong2024decentralized}. Decentralized approaches distribute computation across multiple agents that communicate only with their neighbors, thereby significantly reducing communication overhead and enhancing scalability for large-scale problems. These frameworks are particularly valuable in scenarios where centralizing data is infeasible due to privacy concerns or infrastructure limitations \citep{zhang2019metapred, kayaalp2022dif}. Applications of decentralized bilevel optimization are prevalent in various domains, including resource allocation \citep{ji2023network}, collaborative decision-making \citep{hashemi2024cobo}, and distributed machine learning \citep{jiao2022asynchronous}, where agents collaboratively solve a global bilevel problem while addressing local constraints.

Given its importance, numerous studies have explored the challenges of decentralized bilevel optimization, focusing on algorithm design \citep{lu2022stochastic}, convergence analysis \citep{wang2024fully}, and practical applications \citep{lu2022decentralized,liu2022interact}. Among existing methods, double-loop frameworks have been extensively studied for their effectiveness in achieving convergence across various settings \citep{chen2024decentralized, chen2023decentralized}. However, these approaches are computationally expensive due to their nested structure, which requires repeatedly solving lower-level problems during each upper-level iteration. This results in significant computational and communication overhead in decentralized settings. To address these limitations, single-loop methods have emerged as a computationally efficient alternative \citep{zhu2024sparkle, dong2023single}. By integrating updates for both levels into a unified process, single-loop frameworks reduce overall complexity and are better suited for real-world decentralized bilevel optimization tasks.

Despite these advancements, existing decentralized bilevel optimization methods face a critical challenge: problem-specific hyperparameter tuning (e.g., stepsizes). In particular, the selection of hyperparameters in these algorithms often relies on problem-specific information, such as smoothness and strong convexity constants, the spectral gap of the graph adjacency matrix, or other topological characteristics. However, obtaining such information is typically infeasible due to physical or privacy constraints and computational limitations, especially in large-scale machine learning applications involving massive datasets. The nested structure of upper- and lower-level objectives in decentralized bilevel problems further exacerbates this challenge. As a result, extensive hyperparameter tuning remains necessary in existing methods, significantly limiting their practicality in real-world scenarios. This raises a fundamental question: 
\vspace{-0.2cm}
\begin{tcolorbox}[colback=gray!5,
    colframe=black!70, 
    frame style={line width=10pt}, arc=2mm  ]
\textbf{Can we design a single-loop decentralized bilevel optimization algorithm that eliminates reliance on problem-specific parameters while achieving comparable performance to well-tuned counterparts?
}
\end{tcolorbox}
\vspace{-0.3cm}

\subsection{Main Contributions}
\vspace{-0.2cm}
In this paper, we provide an affirmative answer to the above question by proposing an {\bf Ada}ptive {\bf S}ingle-loop {\bf D}ecentralized {\bf B}ilevel {\bf O}ptimization Algorithm (AdaSDBO). AdaSDBO leverages accumulated gradient norms to dynamically adjust stepsizes per iteration, thereby eliminating the need for hyperparameter tuning. 
We conduct a comprehensive convergence analysis with nonconvex-strongly-convex problem settings, showing that AdaSDBO achieves performance comparable to existing well-tuned approaches. 
Our main contributions are summarized as follows:
\begin{itemize}[topsep=0pt, itemsep=2pt, parsep=0pt, leftmargin=0.25cm]
    \item We propose AdaSDBO, the first parameter-free method for decentralized bilevel optimization with a single-loop structure. AdaSDBO employs adaptive stepsizes based on accumulated (hyper)gradient norms to update all variables simultaneously. However, due to the coupling of bilevel objectives, adaptive stepsizes in a single-loop framework must carefully orchestrate the progress of primal, dual, and auxiliary variables. Additionally, network heterogeneity in decentralized settings introduces inconsistencies in local-gradient-based adaptive stepsizes. To address these challenges, our method incorporates two key mechanisms: 1) \textbf{hierarchical stepsize design}, which respects the interdependence of different variables while preserving the autonomy of adaptive stepsizes; 2) \textbf{stepsize tracking scheme}, which synchronizes gradient-norm accumulators, effectively managing stepsize discrepancies among agents.
    \item We provide a comprehensive theoretical analysis, demonstrating that our algorithm eliminates the need for problem-specific hyperparameter tuning while achieving a convergence rate of \(\w{\mathcal{O}}\left(\frac{1}{T}\right)\), matching well-tuned counterparts \citep{ji2022will, dong2023single} up to polylogarithmic factors. Our analysis is inspired by the two-stage framework \citep{xie2020linear, ward2020adagrad}, but uniquely addresses the intricate coupling between optimization variables and adaptive stepsizes in single-loop bilevel optimization. Furthermore, we conduct a more rigorous analysis to control the interaction between hierarchical optimization errors and network-induced discrepancies, while preserving the problem-parameter-free property.
    \item We conduct experiments on several machine learning problems, showing that our method performs comparably with existing well-tuned approaches on both synthetic and real-world datasets. Moreover, our method exhibits remarkable robustness across a wide range of initial stepsizes, validating the effectiveness of our adaptive stepsizes design.
\end{itemize}  

\begin{table*}[t]\small
\centering
\renewcommand{\arraystretch}{1.4}
\setlength{\tabcolsep}{4pt}
\setlength{\heavyrulewidth}{1.5pt}
\setlength{\lightrulewidth}{1pt}
\setlength{\cmidrulewidth}{0.6pt}
\caption{Comparison between different bilevel optimization algorithms.\\ \(T\) denotes the number of (upper-level) iterations;  \(\epsilon\) is the target stationarity such that \(\sum_{t=0}^{T-1}\|\nabla\Phi(\bar{x}_t)\|^2/T \leq \epsilon\); \(\rho_W\) measures the connectivity of the underlying graph; \(\mu\) and \(L\) are the strongly convex and Lipschitz constants, respectively;
\(\beta\) represents the momentum parameter.}
\begin{threeparttable}
\begin{tabular}{@{}lccccc@{}}
\toprule
\textbf{Algorithm} & \textbf{Loopless} & \textbf{Convergence Rate$^\diamond$} & \textbf{Gradient Complexity$^\dagger$} &  \textbf{Parameters$^\uparrow$} \\ \midrule
DBO \citep{chen2024decentralized}    & \ding{55} & $\mathcal{O}\left(\frac{1}{{T}}\right)$ & $\mathcal{O}\left(\frac{1}{\epsilon^2} \log \left( \frac{1}{\epsilon} \right)\right)$  & $\mu,L,\rho_W$ \\
MDBO \citep{gao2023convergence} & \ding{55} & $\mathcal{O}\big(\frac{1}{\sqrt{T}}\big)$ & $\mathcal{O}\left(\frac{1}{\epsilon^2} \log \left( \frac{1}{\epsilon} \right)\right)$ & $\mu,L,\rho_W,\beta$ \\
FSLA \citep{li2022fully}  & \ding{51} & $\mathcal{O}\big(\frac{1}{\sqrt{T}}\big)$ & $\mathcal{O}\left(\frac{1}{\epsilon^2} \right)$ &  $\mu,L,\beta$ \\
AID \citep{ji2022will} & \ding{51}& $\mathcal{O}\big(\frac{1}{{T}}\big)$ & $\mathcal{O}\left(\frac{1}{\epsilon} \right)$ & $\mu,L,\epsilon$ \\
SLDBO \citep{dong2023single}   & \ding{51} & $\mathcal{O}\big(\frac{1}{{T}}\big)$ & $\mathcal{O}\left(\frac{1}{\epsilon} \right)$ & $\mu,L,\rho_W$ \\ \rowcolor[gray]{0.9}
\textbf{AdaSDBO (This paper)} & \ding{51} & $\mathcal{O}\big(\frac{\log^4(T)}{{T}}\big)$ & $\mathcal{O}\left(\frac{1}{\epsilon} \log^4\left(\frac{1}{\epsilon}\right)\right)$ & None \\ \bottomrule
\end{tabular}
\vspace{0cm}
\label{tab:comparison}
 \begin{tablenotes}  
   \footnotesize
   \item[$\diamond$] The convergence rate when $T\rightarrow\infty$. 
   \item[$\dagger$] The number of gradient/Jacobian/Hessian evaluations per agent to achieve $\epsilon$-accuracy when $\epsilon\rightarrow 0$.
   \item[$\uparrow$] Stepsize-related problem-specific parameters.
  \end{tablenotes}
\end{threeparttable}
\vspace{-0.5cm}
\label{table1}
\end{table*}

\vspace{-0.3cm}
\subsection{Related Works}
\vspace{-0.2cm}
\textbf{Decentralized Bilevel Optimization.}
Recent advancements in decentralized bilevel optimization have focused on addressing the challenges of large-scale data and leveraging the computational benefits of parallel environments. 
\citet{chen2024decentralized} proposed DBO, a general framework that incorporates convergence analysis while accounting for data heterogeneity across agents. Similarly, MA-DSBO \citep{chen2023decentralized} and MDBO \citep{gao2023convergence} employed a double-loop framework with momentum techniques \citep{liu2020improved}.  
More recently, single-loop frameworks have emerged as efficient alternatives to double-loop methods. These approaches \citep{chen2024optimal, dagreou2022framework, kong2024decentralized, zhang2023communication} enable approximate solutions to decentralized bilevel problems within a single iteration, significantly improving computational efficiency by reducing redundant computations. Such methods have made decentralized bilevel optimization more practical and scalable for large-scale applications. \citet{dong2023single} further introduced SLDBO, a low-complexity single-loop decentralized bilevel algorithm that leverages gradient tracking technology.  
Despite these advancements, existing methods rely on fixed or uniformly decaying stepsizes. Further, they require prior knowledge of problem parameters for stepsize selection. This dependency imposes additional challenges, particularly in decentralized settings where such information is often unavailable or difficult to estimate. 
Further details on bilevel optimization, adaptive methods, and their applications are provided in \cref{secrelated work}.

\vspace{-0.35cm}
\subsection{Comparisons with Prior Approaches}
\vspace{-0.15cm}
We compare AdaSDBO with representative bilevel optimization methods, as summarized in \cref{table1}. Notably, AdaSDBO adopts a loopless framework while achieving a convergence rate that matches the state-of-the-art results, up to a polylogarithmic factor of $\log^4(T)$. Since logarithmic factors grow significantly slower than polynomial terms, this factor is negligible relative to $T$, a common consideration in optimization research \citep{yang2022nest,li2024problem}. By carefully controlling network-induced errors, AdaSDBO matches both the convergence rate and gradient complexity of its centralized counterpart AID~\citep{ji2022will}, while outperforming the centralized method FSLA~\citep{li2022fully} in both metrics.
Compared to decentralized approaches such as DBO \citep{chen2024optimal}, MDBO \citep{gao2023convergence}, and SLDBO \citep{dong2023single}, AdaSDBO achieves the best-known convergence rate of $\mathcal{O}\big(\frac{1}{T}\big)$ and gradient complexity of $\mathcal{O}\big(\frac{1}{\epsilon}\big)$, while surpassing double-loop methods in gradient complexity—underscoring the efficiency of its single-loop framework. Most importantly, AdaSDBO is a completely tuning-free algorithm independent of problem parameters, which is in sharp contrast to other methods that require extensive hyperparameter tuning. This advantage significantly simplifies algorithm deployment,
facilitating the implementation of bilevel optimization in diverse environments.

\vspace{-0.3cm}
\section{Algorithm Development}
\vspace{-0.2cm}
\subsection{Problem Model}
\vspace{-0.2cm}
In this paper, we consider a networked system consisting of $n$ nodes (agents) that collectively solve the following nonconvex-strongly-convex bilevel optimization problem:
\vspace{-0.23cm}
\begin{equation}\label{eqcontent1}
\begin{aligned}
    \min_{x \in \mathbb{R}^p} &\Phi(x) = f(x, y^*(x)) := \frac{1}{n} \sum_{i=1}^n f_i(x, y^*(x)), \\
    \text{s.t.} ~~& y^*(x) = \arg\min_{y \in \mathbb{R}^q} l(x, y) :=  \frac{1}{n} \sum_{i=1}^n l_i(x, y),
\end{aligned}
\end{equation}
where $f(\cdot)$ represents the upper-level objective function, which is minimized with respect to $x$, subject to the constraint that $y^*(x)$ is a minimizer of the lower-level function $l(\cdot)$. Each agent \(i\) has a possibly nonconvex objective function \(f_i(\cdot) : \mathbb{R}^p \times \mathbb{R}^q \to \mathbb{R}\) and a strongly convex objective function \(l_i(\cdot): \mathbb{R}^p \times \mathbb{R}^q \to \mathbb{R}\) with respect to $y$. The agents are interconnected through a communication network modeled as a graph \(\mathcal{G} = (\mathcal{N}, \mathcal{E})\), where \(\mathcal{N} = \{1, 2, \dots, n\}\) is the set of nodes (agents), and \(\mathcal{E} \subseteq \mathcal{N} \times \mathcal{N}\) is the set of edges representing communication links. An edge \((i, j) \in \mathcal{E}\) indicates a communication link between nodes \(j\) and node \(i\). We represent the weight matrix of the communication network $\mathcal{G}$ as $W = (w_{ij}) \in \mathbb{R}^{n \times n}$, where $w_{ij} = 0$ if \((i, j) \notin \mathcal{E}\). The following assumption is made on $W$. 
\begin{assumption}
The matrix $W$ is doubly stochastic, i.e., $W \mathbf{1}\! =\! \mathbf{1}$,  $\mathbf{1}^\top W \!=\! \mathbf{1}^\top$, and $\rho_W\!:= \|W\!-\!\mathbf{J}\|_2^2 \!< \!1$, where \( \mathbf{J} = \mathbf{1} \mathbf{1}^\top/n \) is the averaging matrix with $n$ dimension.\label{Assumption3}
\end{assumption}
\vspace{-0.2cm}
A key challenge in solving the Problem \eqref{eqcontent1} is the computation of the hypergradient $\nabla \Phi(x)$, which is expressed as:
{\setlength{\abovedisplayskip}{3pt}
\setlength{\belowdisplayskip}{5pt}
\begin{equation} \label{eqcontent2}
    \begin{aligned}
    \nabla \Phi(x) :=\nabla_x f(x, y^*(x)) - \nabla_{x}\nabla_{y} l(x, y^*(x)) [\nabla_{y}\nabla_{y} l(x, y^*(x))]^{-1} \nabla_y f(x, y^*(x)),
\end{aligned}
\end{equation}}derived using the implicit function theorem \citep{ghadimi2018approximation}. First, the lower-level minimizer $y^*(x)$ is typically not directly accessible, requiring iterative algorithms to approximate it with an estimate $\hat{y}$. Second, the expression in \cref{eqcontent2} involves the inversion of the Hessian matrix $\nabla_{y}\nabla_{y} l(x, y^*(x))$, which is computationally expensive, with the complexity of $\ca{O}(q^3)$. Additionally, in decentralized settings, the difficulty is further exacerbated because each agent has access only to its local problem and lacks information about the global lower-level objective $l(\cdot)$.  
To overcome this challenge, we introduce the following linear system:
{\setlength{\abovedisplayskip}{3pt}
\setlength{\belowdisplayskip}{3pt}
\begin{equation}\label{eqcontent3}
\min_{v} r(x, \hat{y}, v) := \frac{1}{2} v^\top \nabla_y\nabla_y l(x, \hat{y}) v \!-\! v^\top \nabla_y f(x, \hat{y}),
\end{equation}}which seeks to approximate the inversion of the Hessian matrix by $ v^*(x)= [\nabla_{y}\nabla_{y} l(x, y^*(x))]^{-1} \nabla_y f(x, y^*(x))$ when $\hat{y}$ approaching $y^*(x)$, where $v\in\mathcal{V}$. 
An iterative algorithm is then employed to compute an approximate solution $\hat{v}$ to the Problem \eqref{eqcontent3}. Using the approximations $\hat{y}$ and $\hat{v}$, the $x$ is subsequently updated based on a hypergradient estimate as:
{\setlength{\abovedisplayskip}{7pt}
\setlength{\belowdisplayskip}{2.5pt}
\begin{equation} \label{eqcontent_grad_approx}
\bar{\nabla} f(x, \hat{y}, \hat{v}) := \nabla_x f(x, \hat{y}) - \nabla_x \nabla_y l(x, \hat{y}) \hat{v}.
\end{equation}}

Building on this framework, solving the decentralized bilevel optimization problem \eqref{eqcontent1} can be reformulated into three interconnected subproblems:
\begin{subequations} \label{eqcontent4}
\vspace{-0.5cm}
\begin{center}
\begin{tabular}{@{}c@{\quad}c@{}}
  \begin{minipage}{0.45\textwidth}
    \begin{equation}
    x^* = \arg\min_{x \in \mathbb{R}^p} \frac{1}{n} \sum_{i=1}^n f_i(x, y^*(x),\tag{5a}\label{4a}
    \end{equation}
  \end{minipage}
  &
  \begin{minipage}{0.45\textwidth}
    \begin{equation}
    y^*(x) = \arg\min_{y \in \mathbb{R}^q} \frac{1}{n} \sum_{i=1}^n l_i(x, y), \tag{5b}\label{4b}
    \end{equation}
  \end{minipage}
\end{tabular}
\end{center}
\vspace{-0.4cm}
\begin{equation}
v^*(x) = \arg\min_{v \in \mathbb{R}^q} \frac{1}{n} \sum_{i=1}^n r_i(x, y^*(x), v). \quad\quad\!\!\!\text{(5c)}\label{4c}\notag
\end{equation}
\vspace{-0.4cm}
\end{subequations}

In \eqref{eqcontent4}, the subproblems (5a) and (5c) depend on the optimal solution of (5b). This naturally motivates many existing works \citep{chen2023decentralized,chen2024decentralized} to adopt a nested structure, where (5b) is solved up to a certain accuracy before sequentially computing (5c) and (5a). However, these nested loops result in substantial computational overhead and implementation difficulties. To address this, a single-loop framework \citep{dong2023single,zhu2024sparkle} has been proposed to solve the three subproblems in parallel, thereby improving computational efficiency in bilevel optimization. Nevertheless, the single-loop structure introduces significant challenges in algorithm design and convergence analysis.

Furthermore, existing decentralized bilevel optimization solutions predominantly rely on problem-specific parameters for algorithm tuning, such as smoothness and strong convexity constants, the spectral gap of the graph adjacency matrix, or other topological characteristics. Obtaining such information is often impractical due to physical or privacy constraints, as well as computational limitations—particularly in large-scale machine learning applications involving massive datasets.

In light of these challenges, this paper aims to develop a single-loop algorithm for solving Problem \eqref{eqcontent4} while achieving comparable convergence guarantees without the need for hyperparameter tuning.

\subsection{Algorithm Development}

In this subsection, we propose an { Ada}ptive { S}ingle-loop { D}ecentralized {B}ilevel { O}ptimization Algorithm (AdaSDBO) based on iterative gradient updates, as presented in \cref{alg1}. Let \(x_{i,t} \in \mathbb{R}^p\), \(y_{i,t} \in \mathbb{R}^q\), and \(v_{i,t} \in \mathbb{R}^q\) denote the iterates at agent $i$ for variables $x$, $y$, and $v$, respectively. According to Problem \eqref{eqcontent4}, the local gradients at each agent $i$ can be expressed as:
\begin{equation*}
\bar{\nabla}f_i(x_{i,t}, y_{i,t}, v_{i,t}) = \nabla_{x} f_i(x_{i,t}, y_{i,t}) - \nabla_{x}\nabla_{y} l_i(x_{i,t}, y_{i,t}) v_{i,t},
\end{equation*}
\begin{equation*}
    \nabla_{v} r_i(x_{i,t}, y_{i,t}, v_{i,t}) = \nabla_{y}\nabla_{y} l_i(x_{i,t}, y_{i,t}) v_{i,t} -\nabla_{y} f_i(x_{i,t}, y_{i,t}).
\end{equation*}
For brevity, let $g^x_{i,t}:=\bar{\nabla}f_i(x_{i,t}, y_{i,t}, v_{i,t})$, $g^y_{i,t} := \nabla_yl_i(x_{i,t}, y_{i,t})$, and $g^v_{i,t}:=\nabla_v r_i(x_{i,t}, y_{i,t}, v_{i,t})$.

\textbf{Adaptive Stepsizes Design.}
To eliminate the dependency on problem-specific parameters, we design the adaptive stepsizes strategy based on accumulated gradient norms. Specifically, we introduce an accumulator $m^y_{i,t+1}$ as $[m^y_{i,t+1}]^2 = [m^y_{i,t}]^2 + \|g^y_{i,t}\|^2.$ Using this accumulator, the dual variable $y_{i,t}$ is updated by $y_{i,t+1} = y_{i,t} - \frac{\gamma_y}{m^y_{i,t+1}} g^y_{i,t}$,
where $\gamma_y >0$ is a control coefficient that is independent of the problem parameters. Similarly, the accumulators $m^x_{i,t}$ and $m^v_{i,t}$ are defined by the updates $[m^x_{i,t+1}]^2 = [m^x_{i,t}]^2 + \|g^x_{i,t}\|^2$ and $[m^v_{i,t+1}]^2 = [m^v_{i,t}]^2 + \|g^v_{i,t}\|^2$, respectively.
However, the adaptive update rule for dual variable \(y_{i,t}\) cannot be directly extended to primal variable \(x_{i,t}\) and auxiliary variable \(v_{i,t}\) due to their intricate interdependencies with other variables. The primary challenges are:  
\begin{itemize}[topsep=0pt, itemsep=2pt, parsep=0pt, leftmargin=0.25cm]
    \item \textbf{Auxiliary-Level Update}: The update of the auxiliary variable \(v_{i,t}\) requires the optimal solution \(y^*\) of the lower-level subproblem. Since single-loop algorithms perform only one-step updates of the variable \(y_{i,t}\), the suboptimality gap \(\|y_{i,t+1} - y^*\|^2\) must be carefully managed to prevent error accumulation. Thus, \(v_{i,t}\) must progress no faster than \(y_{i,t}\) to maintain approximation accuracy.  

    \item \textbf{Upper-Level Update}: The primal variable \(x_{i,t}\) depends on both the optimal variables \(y^*\) and \(v^*\), introducing additional complexity. Errors from both the lower-level and auxiliary-level updates must be considered, necessitating more conservative updates of the variable \(x_{i,t}\) to align with the slower dynamics of these levels.  
\end{itemize}

To overcome these challenges, we propose the following hierarchical stepsizes:
\begin{itemize}[topsep=0pt, itemsep=2pt, parsep=0pt, leftmargin=0.25cm]
 \item \textbf{Auxiliary variable update} (\(v_{i,t}\)):  
   We employ a stepsize inversely proportional to \({\max(m^v_{i,t+1}, m^y_{i,t+1})}\), resulting in the update rule 
$v_{i,t+1} = v_{i,t} - \frac{\gamma_v}{\max(m^v_{i,t+1}, m^y_{i,t+1})} g^v_{i,t}$.
    \item \textbf{Primal variable update} (\(x_{i,t}\)):  
    We use a more conservative stepsize, inversely proportional to \(m^x_{i,t+1} \max(m^v_{i,t+1}, m^y_{i,t+1})\), yields the update \(x_{i,t+1} = x_{i,t} - \frac{\gamma_x}{m^x_{i,t+1} \max(m^v_{i,t+1}, m^y_{i,t+1})} g^x_{i,t}\).
\end{itemize}

\vspace{-0.1cm}

Here, \(\gamma_v > 0\) and \(\gamma_x > 0\) are control coefficients.  

This hierarchical stepsize design plays a critical role in balancing the progress speeds at different levels, ensuring stability and convergence in single-loop optimization. Additionally, the adaptive stepsizes based on accumulated gradient norms dynamically adjust to the local optimization geometry, enhancing the efficiency and accuracy of the algorithm without requiring problem-specific parameters.

\begin{algorithm}[tb]
   \caption{Adaptive Single-Loop Decentralized Bilevel Optimization: Procedures at Each Agent $i\in[n]$}
   \label{alg1}
\begin{algorithmic}[1]
    \STATE {\bfseries Initialization:} $x_{i,0}$, $y_{i,0}$, $v_{i,0}$, $m^x_{i,0} = m^y_{i,0} = m^v_{i,0} > 0$, $\gamma_x=\gamma_y=\gamma_v > 0$.
    \FOR{$t = 0, 1, \cdots, T-1$}
        \STATE Compute the gradients: \\
        $g^y_{i,t} = \nabla_y l_i(x_{i,t}, y_{i,t})$,\\
        $g^v_{i,t} = \nabla_y\nabla_y l_i(x_{i,t}, y_{i,t}) v_{i,t}-\nabla_y f_i(x_{i,t}, y_{i,t})$,\\
        $g^x_{i,t} = \nabla_x f_i(x_{i,t}, y_{i,t}) - \nabla_x\nabla_y l_i(x_{i,t}, y_{i,t}) v_{i,t}$.
        \STATE Accumulate the gradient norms:\\
        $[m^x_{i,t+1}]^2 = [m^x_{i,t}]^2 + \|g^x_{i,t}\|^2$,
        $[m^y_{i,t+1}]^2 = [m^y_{i,t}]^2 + \|g^y_{i,t}\|^2$,
        $[m^v_{i,t+1}]^2 = [m^v_{i,t}]^2 + \|g^v_{i,t}\|^2$.
        \STATE Update the primal, dual, and auxiliary variables by: \\
        $y_{i,t+1} = y_{i,t} - \frac{\gamma_y}{m^y_{i,t+1}} g^y_{i,t}$,\\
        $v_{i,t+1} = v_{i,t} - \frac{\gamma_v }{\max(m^v_{i,t+1}, m^y_{i,t+1})} g^v_{i,t}$,\\
        $x_{i,t+1} = x_{i,t} - \frac{\gamma_x }{m^x_{i,t+1}\max(m^v_{i,t+1}, m^y_{i,t+1})} g^x_{i,t}$.
        \STATE Information exchange with neighbors:\\
        $\{x, y, v\}_{i,t+1} \gets \sum_{j} w_{i,j} \{x, y, v\}_{j,t+1}$,\\
        $\{m^x, m^y, m^v\!\}_{i,t+1} \gets \sum_{j} w_{i,j} \{m^x, m^y, m^v\!\}_{j,t+1}$.
        \STATE Projection of auxiliary variable on the set $\mathcal{V}$: $v_{i,t+1} \leftarrow \mathcal{P}_{\mathcal{V}}(v_{i,t+1})$.
   \ENDFOR
\end{algorithmic}
\end{algorithm}

For notational convenience, we define the stepsize variables $q_{i,t+1} := m^x_{i,t+1}\max \big\{m^v_{i,t+1}, m^y_{i,t+1}\big\}$, $u_{i,t+1} := m^y_{i,t+1}$, and $z_{i,t+1} := \max \big\{m^v_{i,t+1}, m^y_{i,t+1}\big\}$. Correspondingly, we define the following diagonal stepsize matrices as: 
{\setlength{\abovedisplayskip}{2pt}
\setlength{\belowdisplayskip}{1pt}
\begin{equation*}\label{eq8}
    Q_{t+1} = \text{diag} \big\{q_{i,t+1}\big\}_{i=1}^n, \quad 
U_{t+1} = \text{diag} \big\{u_{i,t+1}\big\}_{i=1}^n,\quad Z_{t+1}= \text{diag} \big\{z_{i,t+1}\big\}_{i=1}^n.\nonumber
\end{equation*}}

Additionally, define the concatenated variable matrices as: $\mathbf{x}_t \!\!:=\!\! [\dots, x_{i,t},\dots]^\top \!\!\!\in\!\! \mathbb{R}^{n \times p}$, $\mathbf{y}_t \!\!:=\!\! [\dots, y_{i,t},\dots]^\top \!\!\!\in\!\! \mathbb{R}^{n \!\times q}$, and $\mathbf{v}_t\! \!:=\!\! [\dots, v_{i,t},\dots]^\top \!\!\!\in\!\! \mathbb{R}^{n \!\times q}$. We also concatenate the gradient vectors as:
{\setlength{\abovedisplayskip}{3pt}
\setlength{\belowdisplayskip}{2pt}
\begin{equation*}
    \bar{\nabla} {F}(\mathbf{x}_t, \mathbf{y}_t, \mathbf{v}_t) :=
    \begin{bmatrix}
         \cdots, \bar{\nabla}f_i(x_{i,t}, y_{i,t}, v_{i,t}),\cdots
    \end{bmatrix}^\top,\quad
    \nabla_y {L}(\mathbf{x}_t, \mathbf{y}_t) :=
    \begin{bmatrix}
        \cdots,  \nabla_yl_i(x_{i,t}, y_{i,t}),\cdots
    \end{bmatrix}^\top,
\end{equation*}
\begin{equation*}
    \!\nabla_v {R}(\mathbf{x}_t, \mathbf{y}_t, \mathbf{v}_t) =
    \begin{bmatrix}
        \cdots,\nabla_v r_i(x_{i,t}, y_{i,t}, v_{i,t}),\cdots
    \end{bmatrix}^\top.
\end{equation*}}

Based on these definitions, the update rules for the primal, dual, and auxiliary variables are given by:
\begin{equation*}
\mathbf{x}_{t+1} = W \Big(\mathbf{x}_t - \gamma_x Q_{t+1}^{-1} \bar{\nabla} F(\mathbf{x}_t, \mathbf{y}_t,\mathbf{v}_t)\Big), \quad
\mathbf{y}_{t+1} = W \big(\mathbf{y}_t - \gamma_y U_{t+1}^{-1} \nabla_y L(\mathbf{x}_t, \mathbf{y}_t)\big),
\end{equation*}
\begin{equation*}
\mathbf{v}_{t+1} = \mathcal{P}_{\mathcal{V}}\Big(W \big(\mathbf{v}_t - \gamma_v Z_{t+1}^{-1} \nabla_v R(\mathbf{x}_t, \mathbf{y}_t,\mathbf{v}_t)\big)\Big),
\end{equation*}
where $\mathcal{P}_{\mathcal{V}}(\cdot)$ denotes the projection operation onto the set $\mathcal{V}$.

\textbf{Addressing Stepsize Inconsistencies.} In decentralized bilevel settings, agents compute their adaptive stepsizes independently based on local private objective functions, and the coupling of multiple optimization variables within adaptive stepsizes amplifies their network-wide inconsistencies. These discrepancies can hinder convergence if not properly controlled. To formalize this, let \(\bar{x}_t := \frac{1}{n} \sum_{i=1}^n x_{i,t}\) denote the average of all primal variables at the \(t\)-th iteration, and \(\bar{q}_t := \frac{1}{n} \sum_{i=1}^n q_{i,t}\) represent the average of their respective stepsizes. We then define the stepsize discrepancy vector as \(\tilde{\mathbf{q}}_t^{-1} := \big[ \cdots, q_{i,t}^{-1} - \bar{q}_t^{-1}, \cdots \big]^\top\). Using this definition, the update rule for \(\bar{x}_t\) can be expressed as:  
\begin{equation}\label{eqcontentnew6}
\bar{x}_{t+1} = \bar{x}_t - \gamma_x \bigg(\underbrace{\frac{\bar{q}_{t+1}^{-1} \mathbf{1}^\top}{n}}_{\text{(a)}} 
+ \underbrace{\frac{\big(\tilde{\mathbf{q}}_{t+1}^{-1}\big)^\top}{n}}_{\text{(b)}} \bigg) \bar{\nabla} F(\mathbf{x}_t, \mathbf{y}_t, \mathbf{v}_t).
\end{equation}  
In \cref{eqcontentnew6}, term (a) resembles centralized gradient descent, while term (b) introduces an undesired perturbation caused by stepsize inconsistencies among agents. This perturbation can disrupt network consensus due to varying update rates and may lead to uncontrollable error growth, as term (b) represents accumulated gradient-norm discrepancies over iterations. Such stepsize inconsistencies similarly affect the updates of the dual and auxiliary variables.  

To address this challenge, we incorporate a stepsize tracking mechanism in \cref{alg1}. At each iteration, the gradient-norm accumulators $m^x_{i,t}$, $m^y_{i,t}$, and $m^v_{i,t}$ are tracked by:  
\begin{align*}
[m^b_{i,t+1}]^2 = \sum_{j=1}^n w_{ij}[m^b_{j,t+1}]^2 = \sum_{j=1}^n w_{ij} \left([m^b_{j,t}]^2 + \|g^b_{j,t}\|^2 \right)\!,
\end{align*}
where $b \in \{x, y, v\}$ and $i \in [n]$. Let $\mathbf{k}^b_t := \left[\cdots, [m^b_{i,t}]^2, \cdots\right]^{\top}$ and $\mathbf{h}^b_t := \left[\cdots, \|g^b_{i,t}\|^2, \cdots\right]^{\top}$ denote the concatenated gradient accumulators and corresponding norms, respectively. The above equation can then be compactly expressed as
$\mathbf{k}^b_{t+1} = W \big(\mathbf{k}^b_t + \mathbf{h}^b_t\big).$
This tracking mechanism enforces consensus on the gradient-norm accumulators before computing the adaptive stepsizes $(q_{i,t}, u_{i,t}, z_{i,t})$. By synchronizing these accumulators, it effectively bounds stepsize discrepancies among agents, preventing error accumulation while preserving the adaptive nature of the updates.

It is worth noting that the additional communication overhead of our method is modest—only scalar values (the stepsize accumulators) are exchanged—especially compared with transmitting the primal, dual, and auxiliary variables commonly communicated in decentralized bilevel methods. Some approaches \citep{gao2023convergence, chen2023decentralized, zhu2024sparkle} also employ a gradient tracking mechanism, which involves exchanging extra tracker states with the same dimensionality as the optimization variables. In contrast, our method adds only lightweight scalars for transmission, while offering a robust, problem-parameter-free solution for decentralized bilevel optimization.

\vspace{-0.1cm}
\section{Theoretical Analysis}
\vspace{-0.1cm}
\subsection{Technical Challenges}
\vspace{-0.1cm}

The analysis of problem-parameter-free decentralized bilevel optimization with a single-loop structure involves several fundamental challenges:
\begin{itemize}[topsep=0pt, itemsep=2pt, parsep=0pt, leftmargin=0.25cm]
    \item \textbf{Interdependent Variable Updates:} The coupling of bilevel objectives creates intricate interdependencies among the variables \((x, y, v)\), making the convergence analysis significantly more challenging compared to single-level optimization.
    \item \textbf{Coupled Stepsize Dynamics:} The adaptive stepsizes exhibit highly intertwined dynamics, forming a multi-stage system in which the progress of each variable directly affects the others, requiring meticulous coordination to manage these interactions effectively.
    \item \textbf{Accumulated Stepsize Inconsistencies:} Inconsistencies in adaptive stepsizes across agents disrupt network consensus, while their cumulative effect over iterations further exacerbates the challenge.
    \item \textbf{Interplay Between Optimization and Consensus Errors:} The interaction between hierarchical optimization errors and network-induced discrepancies necessitates rigorous theoretical bounds to guarantee convergence while maintaining the problem-parameter-free property.
\end{itemize}

\subsection{Assumptions and Definitions}\label{sec3.2}

In this subsection, we present the standard assumptions and definitions used in our analysis.

\begin{assumption}
For any $i\in[n]$, the objective functions $f_i(x, y)$ and $l_i(x, y)$ are twice continuously differentiable, and $l_i(x, y)$ is $\mu$-strongly 
convex with respect to $y$.\label{Assumption1}
\end{assumption}

\begin{assumption}
For any $i\in[n]$, the function $f_i(x, y)$ is $L_{f,0}$-Lipschitz continuous; the gradients $\nabla f_i(x, y)$ and 
$\nabla l_i(x, y)$ are $L_{f,1}$- and $L_{l,1}$-Lipschitz continuous, respectively; the second-order gradients 
$\nabla_x \nabla_y l_i(x, y)$ and $\nabla_y\nabla_y l_i(x, y)$ are $L_{l,2}$-Lipschitz continuous.\label{Assumption2}
\end{assumption}
\vspace{-0.19cm}

The above assumptions are commonly adopted in prior works, including \citep{zhu2024sparkle, chen2024decentralized, chen2023decentralized,dong2023single, ji2022will}.

\begin{remark}
    \cref{Assumption2} indicates that there exist constants $C_{f_x}$, $C_{f_y}$,  $C_{l_{xy}}$, and $C_{l_{yy}}$ such that $\|\nabla_x f_i(x, y)\| \leq C_{f_x}$, $\|\nabla_y f_i(x, y)\| \leq C_{f_y}$, $\|\nabla_x \nabla_y l_i(x, y)\| \leq C_{l_{xy}}$, and $\|\nabla_y \nabla_y l_i(x, y)\| \leq C_{l_{yy}}$.\label{remark1}
\end{remark}

Define \( \bar{u}_t \!\!:=\!\! \frac{1}{n} \sum_{i=1}^n u_{i,t} \), \( \bar{z}_t \!\!:=\!\! \frac{1}{n} \sum_{i=1}^n z_{i,t} \), and recall that \(\bar{q}_t \!\!:=\!\! \frac{1}{n} \sum_{i=1}^n q_{i,t}\). We then introduce the following metrics to quantify the level of stepsize inconsistency among agents: 
\begin{align}
\zeta_q^2 := \sup_{i \in [n], t > 0} 
\frac{\big(q_{i,t}^{-1} - \bar{q}_t^{-1}\big)^2}{\big(\bar{q}_t^{-1}\big)^2},\quad
\zeta_u^2 := \sup_{i \in [n], t > 0} 
\frac{\big(u_{i,t}^{-1} - \bar{u}_t^{-1}\big)^2}{\big(\bar{u}_t^{-1}\big)^2},\quad
\zeta_z^2 := \sup_{i \in [n], t > 0} 
\frac{\big(z_{i,t}^{-1} - \bar{z}_t^{-1}\big)^2}{\big(\bar{z}_t^{-1}\big)^2},\nonumber
\end{align}
\vspace{-0.3cm}
\begin{align*}
\sigma_q^2 := \inf_{i \in [n], t > 0} 
\frac{\big(q_{i,t}^{-1} - \bar{q}_t^{-1}\big)^2}{\big(\bar{q}_t^{-1}\big)^2},\quad
\sigma_u^2 := \inf_{i \in [n], t > 0} 
\frac{\big(u_{i,t}^{-1} - \bar{u}_t^{-1}\big)^2}{\big(\bar{u}_t^{-1}\big)^2},\quad
\sigma_z^2 := \inf_{i \in [n], t > 0} 
\frac{\big(z_{i,t}^{-1} - \bar{z}_t^{-1}\big)^2}{\big(\bar{z}_t^{-1}\big)^2}.
\end{align*}
These metrics are guaranteed to remain bounded under \cref{Assumption2} and \cref{remark1}. 
\begin{definition}
An output \(\bar{x}\) of an algorithm is the \(\epsilon\)-accurate stationary point of the objective function \(\Phi(x)\) if it satisfies \(\|\nabla \Phi(\bar{x})\|^2 \leq \epsilon\), where \(\epsilon \in (0, 1)\).
\end{definition}

\subsection{Theoretical Results}\label{sec3.3}

In this subsection, we present the main theoretical results for the proposed \cref{alg1}, with the proof sketch provided in \cref{proof}. As outlined in \cref{secc2}, the descent behavior of the objective function \(\Phi(\cdot)\) is governed by three key factors: approximation errors, consensus errors, and stepsize inconsistencies. These components are bounded in the following lemmas.

To facilitate the analysis, we first define \(\bar{x}_t := \frac{1}{n} \sum_{i=1}^n x_{i,t}\), \(\bar{y}_t := \frac{1}{n} \sum_{i=1}^n y_{i,t}\), and \(\bar{v}_t := \frac{1}{n} \sum_{i=1}^n v_{i,t}\). Then, let \(\bar{m}^x_t := \frac{1}{n} \sum_{i=1}^n m^x_{i,t}\) and \(\bar{m}^y_t := \frac{1}{n} \sum_{i=1}^n m^y_{i,t}\) represent the average of the gradient accumulators. Additionally, denote $f(\bar{x}_t,\bar{y}_t,\bar{v}_t)\!:=\!\frac{1}{n}\sum_{i=1}^n f_i(\bar{x}_t,\bar{y}_t,\bar{v}_t)$,
$l(\bar{x}_t, \bar{y}_t)\!:=\! \frac{1}{n}\sum_{i=1}^nl_i(\bar{x}_t, \bar{y}_t)$, 
$r(\bar{x}_t, \bar{y}_t, \bar{v}_t)\!:=\! \frac{1}{n}\sum_{i=1}^nr_i(\bar{x}_t, \bar{y}_t, \bar{v}_t)$ as the corresponding aggregated functions when the variables $(\mathbf{x},\mathbf{y},\mathbf{v})$ achieve consensus to $(\bar{x},\bar{y},\bar{v})$.

From the descent lemma in \cref{secc2}, we obtain that the approximation error $\|\nabla \Phi(\bar{x}_t) -\bar{\nabla} f(\bar{x}_t, \bar{y}_t, \bar{v}_t)\|^2$ is attributed in terms of $\mathcal{O}(\|\nabla_y l(\bar{x}_t, \bar{y}_t)\|^2)$ and $\mathcal{O}(\|\nabla_v r(\bar{x}_t, \bar{y}_t, \bar{v}_t)\|^2)$. Hence, we establish the following lemma to provide bounds for these terms associated with approximation errors during the optimization process.

\begin{lemma}[\textbf{Approximation Errors}] Under Assumptions \ref{Assumption1} and \ref{Assumption2}, for any integer $k_0 \in [0, t)$, we have $
\sum_{k=k_0}^t \frac{\|\nabla_y l(\bar{x}_k, \bar{y}_k)\|^2}{\bar{m}^y_{k+1}} \leq a_5 \log(t + 1) + b_5$ and
$\sum_{k=k_0}^t \frac{\|\nabla_v r(\bar{x}_k, \bar{y}_k, \bar{v}_k)\|^2}{\bar{z}_{k+1}} \leq a_6 \log(t + 1) + b_6$,
where the constants $a_5$, $b_5$, $a_6$, and $b_6$ are defined in \cref{eqnewlemma13} of \cref{secc80}.
\end{lemma}
\begin{lemma}[\textbf{Accumulated Gradients}]\label{proposition2}  
Under Assumptions \ref{Assumption1} and \ref{Assumption2}, the gradient accumulators satisfy the bounds
$\bar{m}^x_t \leq \mathcal{O}\left(\log(t)\right)$ and $\bar{z}_t \leq \mathcal{O}\left(\log(t)\right).$
\end{lemma}
\cref{proposition2} shows that the accumulated gradient norms for all variables \((x, y, v)\) grow at most logarithmically with respect to the iteration index \(t\).

\begin{lemma}[\textbf{Consensus Errors}]\label{Lemma5} Suppose that Assumptions \ref{Assumption3}, \ref{Assumption1}, and \ref{Assumption2} hold. Let \(\Delta_t := \|\mathbf{x}_t - \mathbf{1} \bar{x}_t\|^2 + \|\mathbf{y}_t - \mathbf{1} \bar{y}_t\|^2 + \|\mathbf{v}_t - \mathbf{1} \bar{v}_t\|^2\) represent the consensus errors for all variables at the \(t\)-th iteration. Then, the time-averaged consensus error satisfies $\frac{1}{T}\sum_{t=0}^{T-1} \Delta_t
\leq \mathcal{O}\left({\log(T)}/{T}\right)$.
\end{lemma}

Next, we provide the bounds for the terms associated with the stepsize inconsistencies in our analysis.
\begin{lemma}[\textbf{Stepsize Inconsistencies}] Suppose Assumptions \ref{Assumption3}, \ref{Assumption1}, and \ref{Assumption2} hold. Define discrepancy vectors as \(\tilde{\mathbf{q}}_t^{-1} \!\!:=\!\! \big[ \!\cdots, q_{i,t}^{-1}\!\! -\!\! \bar{q}_t^{-1}, \cdots \!\big]^\top\), \(\tilde{\mathbf{u}}_t^{-1} \!\!:=\!\! \big[ \!\cdots, u_{i,t}^{-1}\!\! -\!\! \bar{u}_t^{-1}, \cdots \!\big]^\top\), and \(\tilde{\mathbf{z}}_t^{-1} \!\!:=\!\! \big[ \!\cdots, z_{i,t}^{-1}\!\! -\!\! \bar{z}_t^{-1}, \cdots \!\big]^\top\). Then, under Algorithm~\ref{alg1}, we have that
$\frac{1}{T}\sum_{t=0}^{T-1} \left\| {\big(\tilde{\mathbf{q}}_{t+1}^{-1}\big)^{\top}} \bar{\nabla} F(\mathbf{x}_t, \mathbf{y}_t, \mathbf{v}_t)/{n\bar{q}_{t+1}^{-1}} \right\|^2$, $\frac{1}{T}\sum_{t=0}^{T-1} \left\| {\big(\tilde{\mathbf{u}}_{t+1}^{-1}\big)^\top} {\nabla}_y L(\mathbf{x}_t, \mathbf{y}_t)/{n\bar{u}_{t+1}^{-1}} \right\|^2$, and $\frac{1}{T}\sum_{t=0}^{T-1} \left\| {\big(\tilde{\mathbf{z}}_{t+1}^{-1}\big)^\top} \nabla_v R(\mathbf{x}_t, \mathbf{y}_t, \mathbf{v}_t) /n\bar{z}_{t+1}^{-1}\right\|^2$ are each upper-bounded by $\mathcal{O}\left({\log(T)}/{T}\right)$. 
\end{lemma}
\vspace{-0.2cm}

Combining the above results, we establish the convergence of \cref{alg1} as follows.
\begin{theorem}\label{theorem 1}
Under Assumptions \ref{Assumption3}, \ref{Assumption1}, and \ref{Assumption2}, for any positive constants $\gamma_x$, $\gamma_y$, $\gamma_v$, $m^x_{i,0}$,  $m^y_{i,0}$, and $m^v_{i,0}$, the iterates generated by \cref{alg1} satisfy:
\vspace{-0.19cm}
\begin{equation*}
\begin{aligned}
\frac{1}{T} \sum_{t=0}^{T-1} \|\nabla \Phi(\bar{x}_t)\|^2
&{\leq} \frac{2}{T}\!\Bigg[\!\!\left(\!{4}\!\left(\!\frac{\Phi(\bar{x}_0)\!-\!\Phi^*}{\gamma_x}\!\right)\!\!+\!a_7\log(T)\!\!+\!b_7\!\right)^{\!2}\!\!\!{(a_1\!\log(T) \!+\! b_1)^2}  \\
& \quad+ \!C_{m^x} \!\!\left(\!{4}\!\left(\!\frac{\Phi(\bar{x}_0)\!-\!\Phi^*}{\gamma_x}\!\right)\!\!+\!a_7\!\log(T)\!\!+\!b_7\!\right)\!(a_1\!\log(T)\!\!+\! b_1)\Bigg] \!=\! \mathcal{O}\left(\frac{\log^4(T)}{T}\right),
\end{aligned}
\end{equation*}
where $\Phi^* := \inf_{x \in \mathbb{R}^p} \Phi(x) > -\infty$, and the constants $C_{m^x}$, $a_1$, $b_1$, $a_7$, and $b_7$ are defined in Eqs.~\eqref{eq52}, \eqref{eq84}, and \eqref{eq128} in the Appendix.
\end{theorem}

\begin{remark}
\Cref{theorem 1} implies that for any positive coefficients \((\gamma_x, \gamma_y, \gamma_v)\) and positive initial stepsizes \((m^x_{i,0}, m^y_{i,0}, m^v_{i,0})\), \cref{alg1} guarantees convergence with a rate of \(\mathcal{O}\left(\frac{\log^4(T)}{T}\right)\). This convergence rate matches the state-of-the-art results, up to a polylogarithmic factor of \(\log^4(T)\), which is regarded negligible relative to \(T\) in the optimization literature \citep{yang2022nest,li2024problem}.  
\end{remark}

\begin{corollary}\label{collary1}
From \cref{theorem 1}, to achieve an \(\epsilon\)-accurate stationary point, \cref{alg1} requires \(T = \mathcal{O}\left(\frac{1}{\epsilon} \log^4\left(\frac{1}{\epsilon}\right)\right)\) iterations, resulting in a gradient complexity of \(\mathrm{Gc}(\epsilon) = \mathcal{O}\left(\frac{1}{\epsilon} \log^4\left(\frac{1}{\epsilon}\right)\right)\).
\end{corollary}
\vspace{-0.19cm}

In sharp contrast, the theoretical convergence of existing decentralized bilevel methods heavily relies on the correct selection of hyperparameters based on problem-specific constants, such as \(\mu\), \(L\), and \(\rho_W\). This reliance restricts their applicability, as these parameters are unknown or difficult to determine. In comparison, our approach significantly simplifies the implementation of decentralized bilevel optimization by operating in a completely problem-parameter-free manner.

\vspace{-0.2cm}
\section{Numerical Experiments}
\vspace{-0.2cm}
In this section, we evaluate the performance of \cref{alg1} on the hyperparameter optimization problem, as illustrated in \cref{hypersec}.
Our algorithm is compared with several decentralized bilevel optimization methods, including SLDBO \citep{dong2023single}, MA-DSBO \citep{chen2023decentralized}, MDBO \citep{gao2023convergence}, and DBO \citep{chen2024decentralized}. Experiments are conducted on both synthetic and real-world datasets, with detailed configurations and additional results provided in \cref{experimentsec}.

\begin{figure*}[t]
\centering
\begin{minipage}[b]{0.2465\textwidth}
    \includegraphics[width=\textwidth]{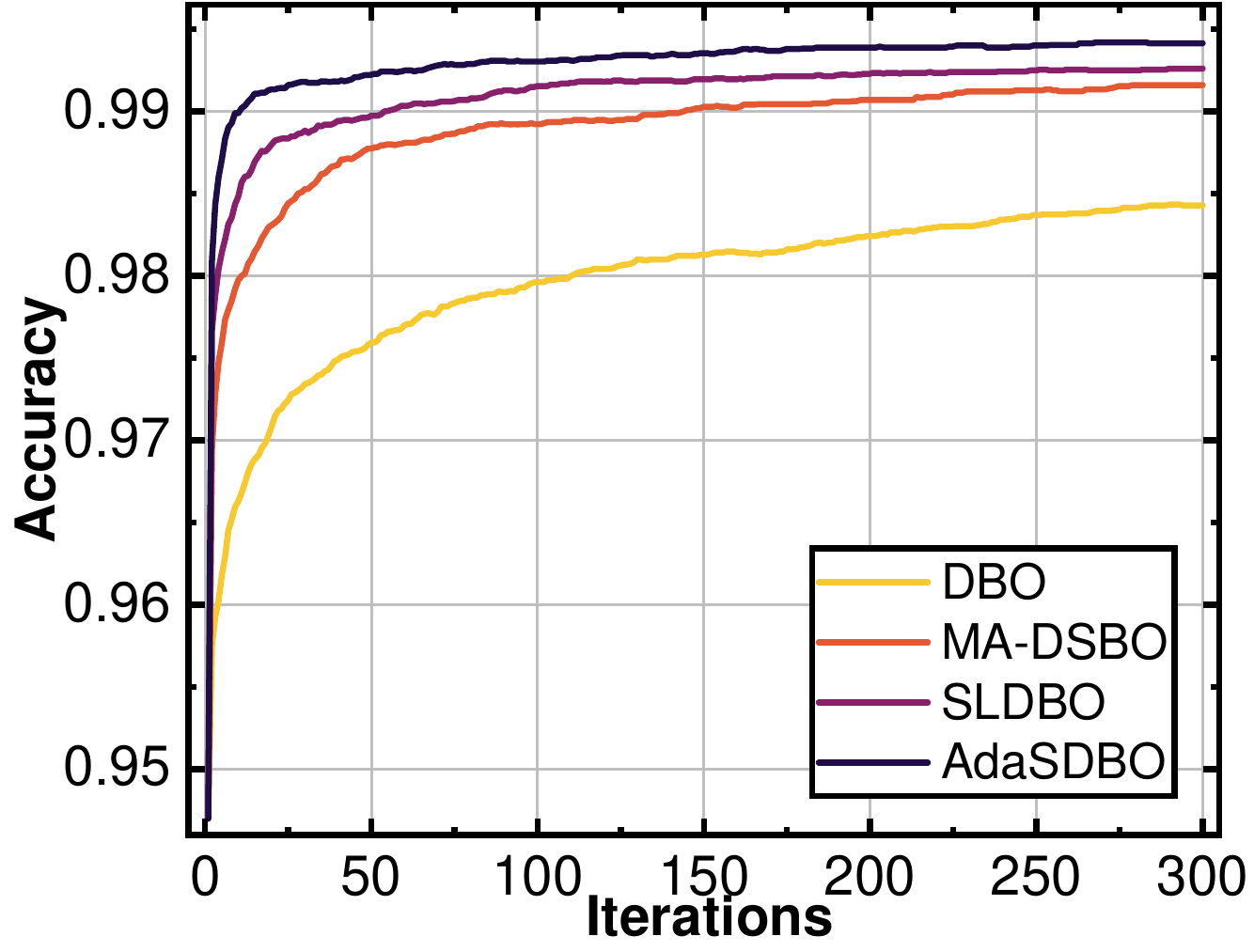}
    \par\vspace{-0.1cm}
    \makebox[\textwidth]{{\hspace{0.46cm}(a)\! Synthetic\! ($p\!=\!50$)}}
\end{minipage}
\begin{minipage}[b]{0.2465\textwidth}
    \includegraphics[width=\textwidth]{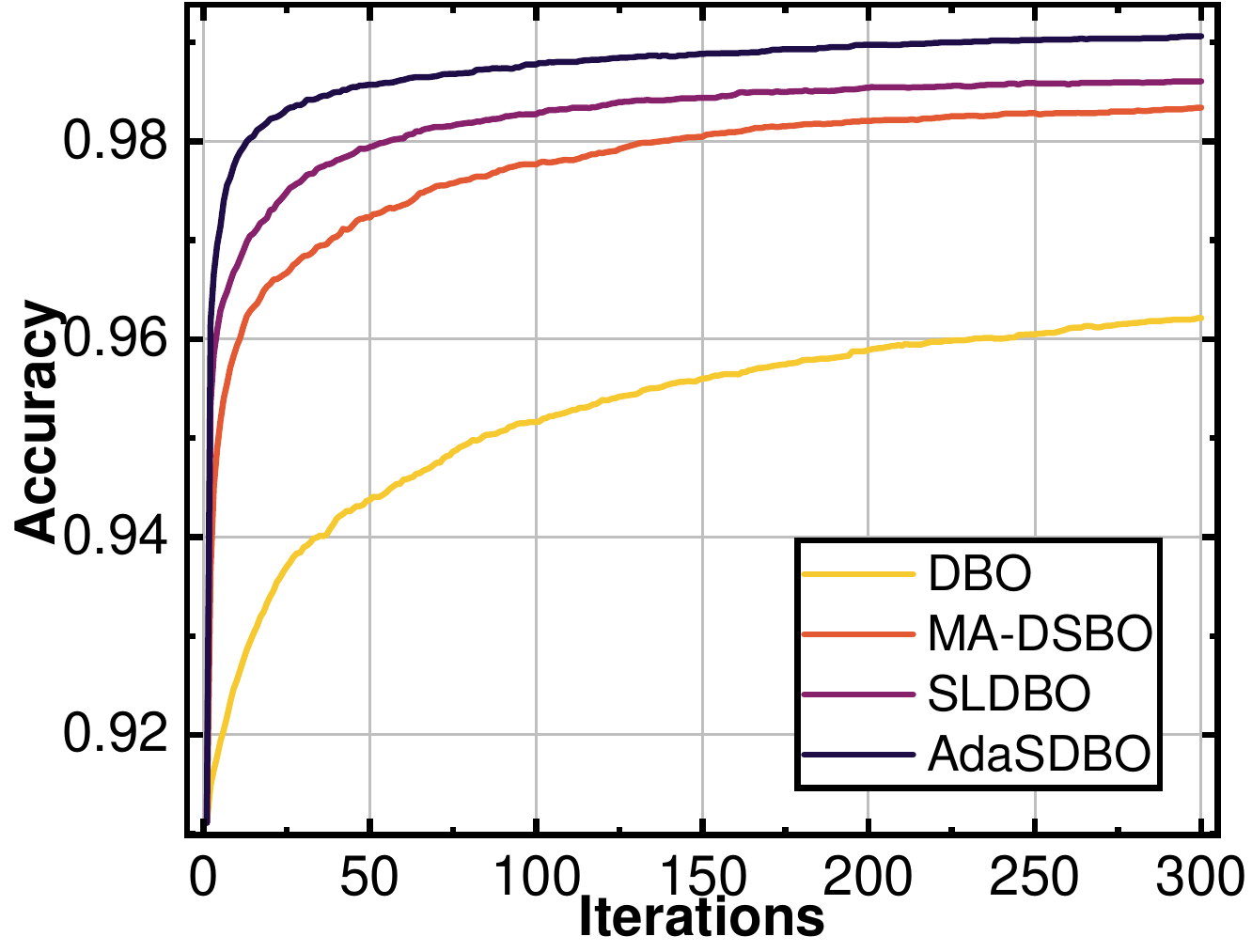}
    \par\vspace{-0.1cm}
    \makebox[\textwidth]{{\hspace{0.439cm}(b)\! Synthetic\! ($p\!=\!200$)}}
\end{minipage}
\begin{minipage}[b]{0.244\textwidth}
    \includegraphics[width=\textwidth]{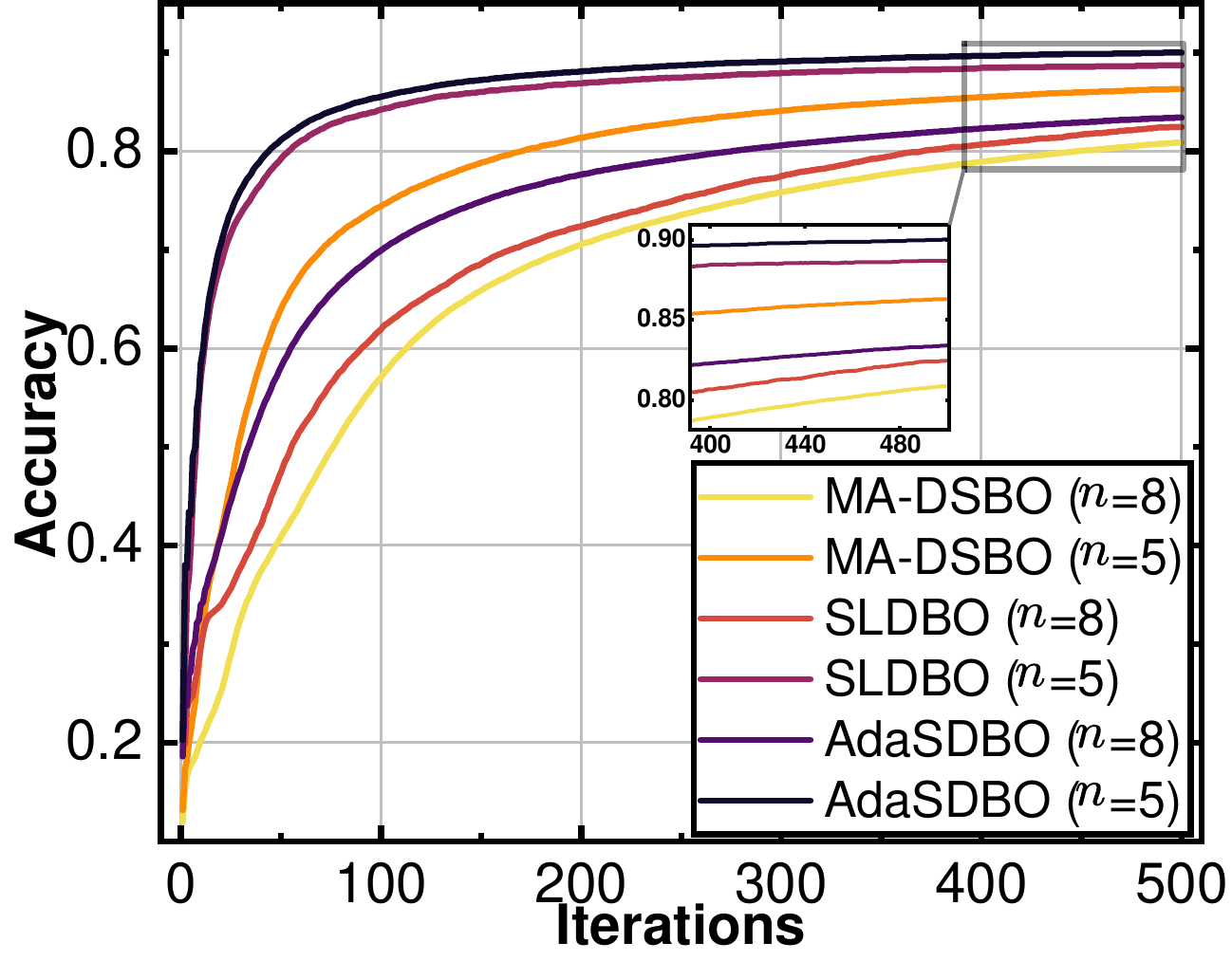}
    \par\vspace{-0.1cm}
    \makebox[\textwidth]{{\hspace{0.39cm}(c)\! MNIST \!(varying $n$)}}
\end{minipage}
\begin{minipage}[b]{0.243\textwidth}
    \includegraphics[width=\textwidth]{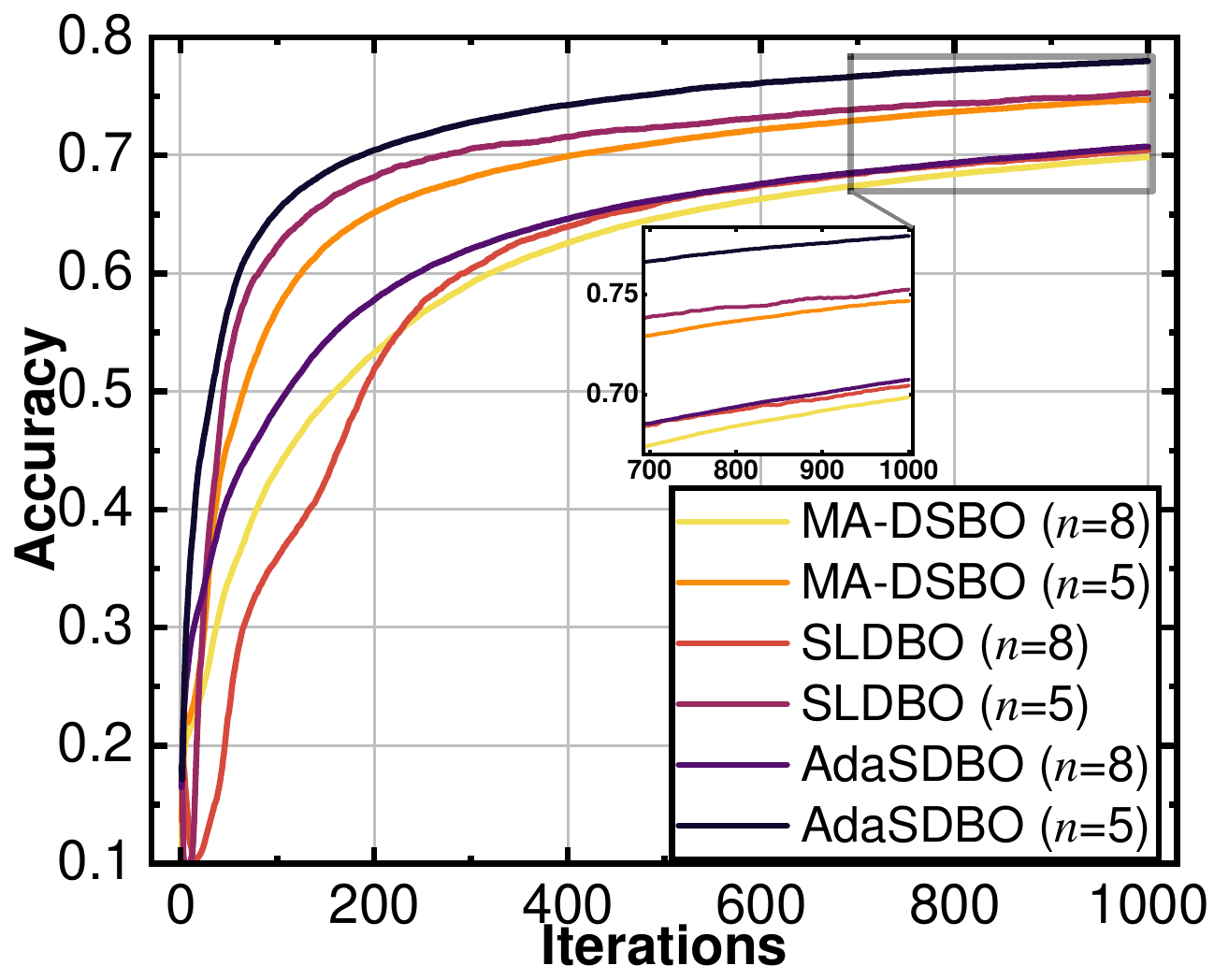}
    \par\vspace{-0.1cm}
    \makebox[\textwidth]{{\hspace{0.397cm}(d)\! FMNIST \!(varying $n$)}}
\end{minipage}
\vspace{-0.5cm}
\caption{Test Accuracy on different datasets.}
\vspace{-0.1cm}
\label{fig:accuracy4algorithm}
\end{figure*}
\begin{figure*}[t]
\centering
\hspace{-0.3cm}
\begin{minipage}[b]{0.255\textwidth}
    \includegraphics[width=1.1\textwidth, height=0.73\textwidth]{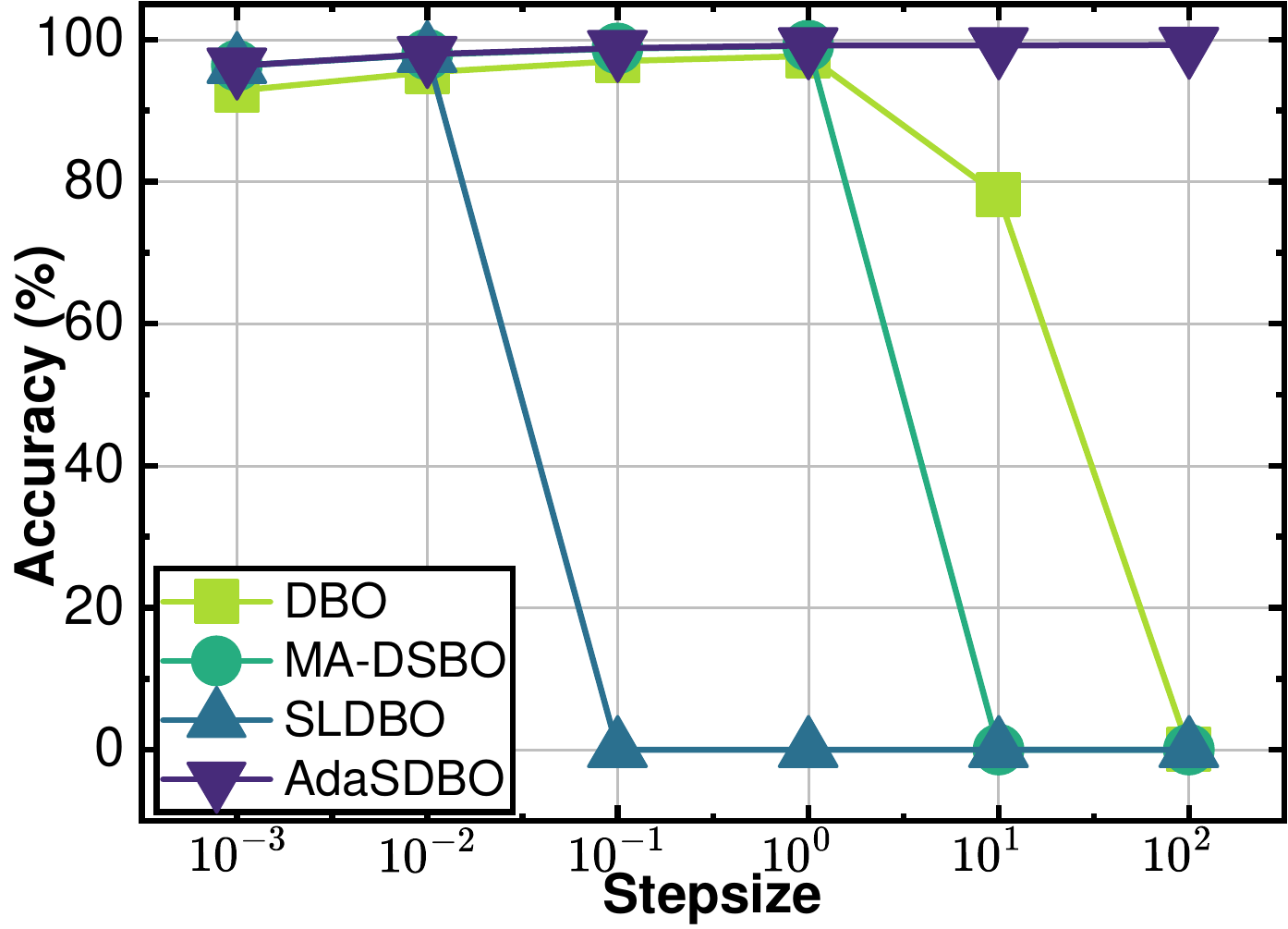} 
    \par\vspace{-0.1cm} 
    \makebox[\textwidth]{{\hspace{0.5cm}(a) Synthetic}}
\end{minipage}
\hspace{0.06\textwidth} 
\begin{minipage}[b]{0.255\textwidth}
    \includegraphics[width=1.1\textwidth, height=0.73\textwidth]{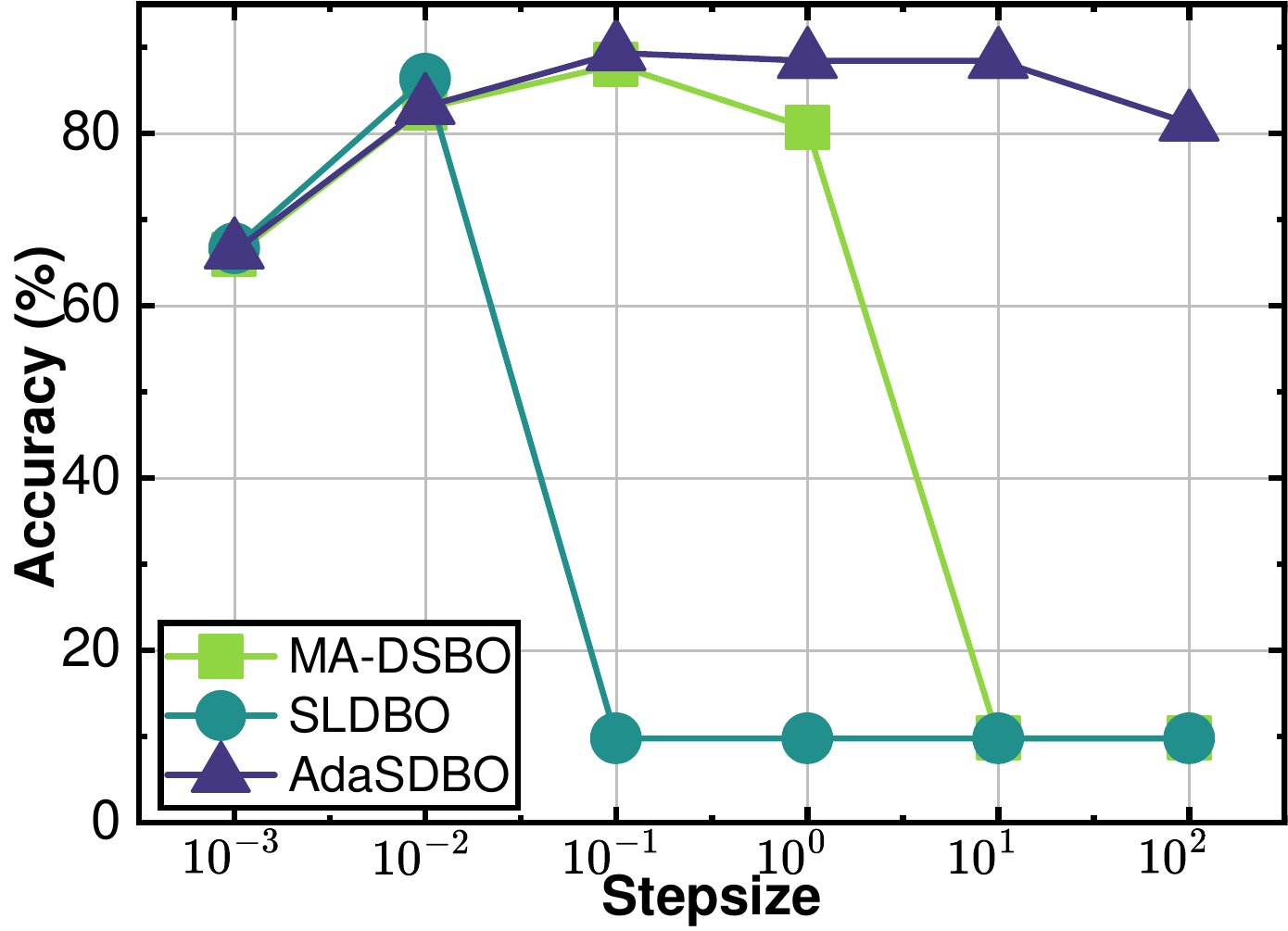}
    \par\vspace{-0.1cm}
    \makebox[\textwidth]{{\hspace{0.5cm}(b) MNIST}}
\end{minipage}
\hspace{0.06\textwidth}
\begin{minipage}[b]{0.255\textwidth}
    \includegraphics[width=1.1\textwidth, height=0.73\textwidth]{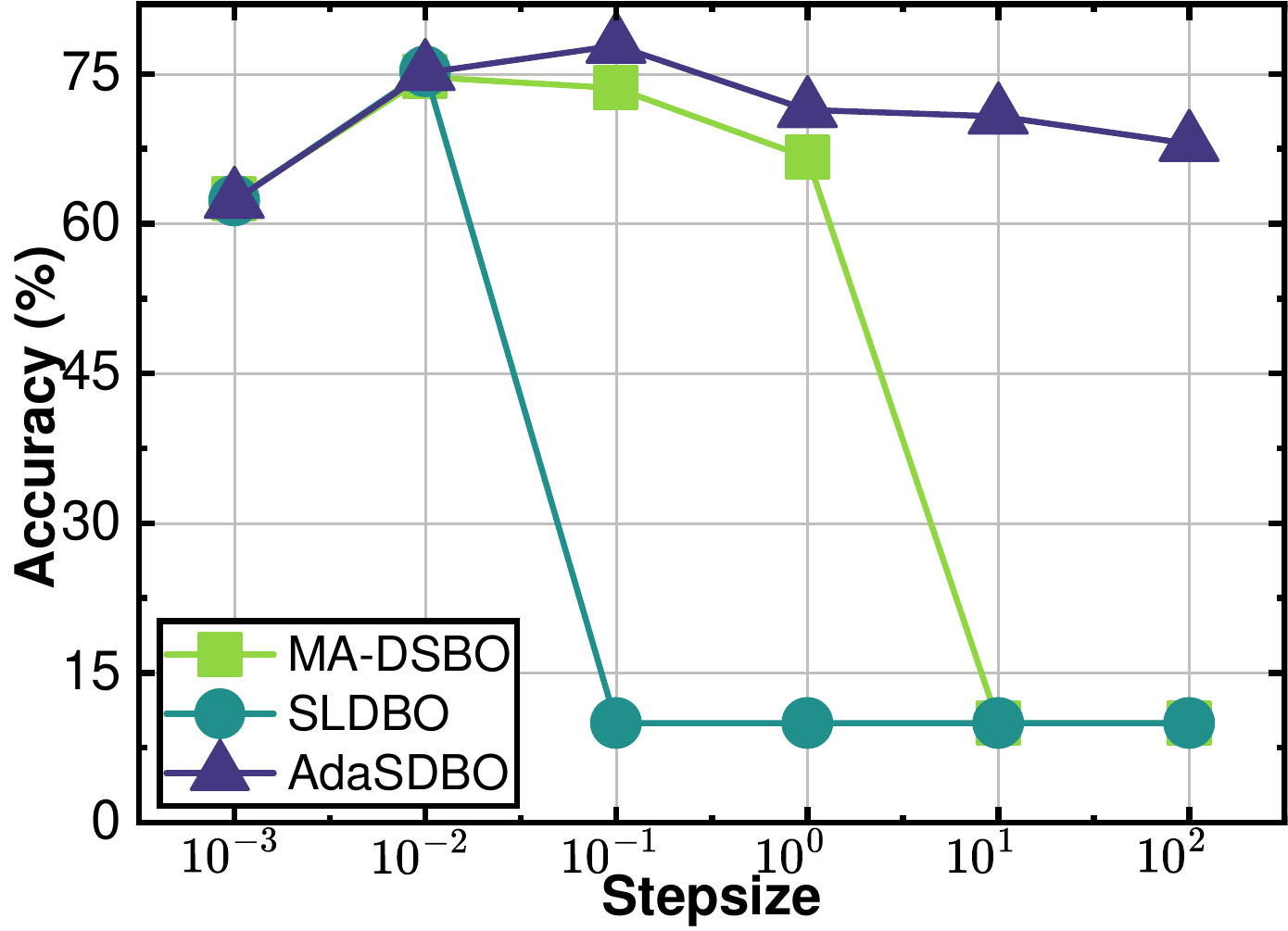} 
    \par\vspace{-0.1cm}
    \makebox[\textwidth]{{\hspace{0.5cm}(c) FMNIST}}
\end{minipage}
\vspace{-0.15cm}
\caption{Test accuracy versus stepsize on different datasets.}
\label{fig:stepsize}
\vspace{-0.5cm}
\end{figure*}

\vspace{-0.3cm}
\subsection{Synthetic Data Experiments}
\vspace{-0.1cm}

For synthetic data, the data distribution at node \(i\) follows \(\mathcal{N}(0, i^2\cdot r^2)\), where the parameter $r$ controls the level of data heterogeneity. To evaluate the advantages of AdaSDBO, we first compare different methods under low data heterogeneity conditions with $r = 1$ (Experiments under higher heterogeneity conditions are presented in \cref{fig:syntheticr5} of \cref{addiresult}). Both algorithms compute full gradients and utilize a training dataset and a testing dataset, each containing 20,000 samples.
As illustrated in \cref{fig:accuracy4algorithm} (a) and \cref{fig:accuracy4algorithm} (b), AdaSDBO achieves faster convergence than baseline methods across different data dimensions ({i.e.,} \(p=50\) and \(p=200\)). Notably, our method consistently outperforms the double-loop frameworks DBO and MA-DSBO. These results validate both the superiority of the adaptive stepsizes design and the single-loop structure of AdaSDBO.

\vspace{-0.2cm}
\subsection{Real-World Data Experiments}
\vspace{-0.1cm}

We evaluate our method on the hyperparameter optimization task using the MNIST~\citep{lecun1998gradient} and FMNIST~\citep{xiao2017fashion} datasets. In \cref{fig:accuracy4algorithm}(c) and \cref{fig:accuracy4algorithm}(d), we vary the number of agents to assess scalability. The results illustrate that AdaSDBO consistently maintains a robust convergence rate across different network sizes, highlighting its scalability and robustness to variations in the number of agents. Furthermore, AdaSDBO achieves a competitive convergence rate compared to state-of-the-art methods, further corroborating its effectiveness. Additional scalability evaluations under broader network configurations are presented in \cref{figsca1} and \cref{figsca2} of \cref{addiresult}.

\cref{fig:stepsize} compares the test accuracy of various algorithms versus stepsizes on the synthetic, MNIST, and FMNIST datasets. All algorithms were evaluated over 1,000  rounds to ensure a fair comparison. To comprehensively assess robustness, the stepsizes were varied over a wide range ({i.e.,} from \(10^{-3}\) to \(10^2\)). It can be observed that the AdaSDBO algorithm demonstrates remarkable resilience to stepsize selection, maintaining stable performance over a substantially broader range of stepsizes compared to baseline methods. In contrast, the baseline algorithms exhibit relatively narrower regions of stable performance, underscoring the enhanced stepsize robustness of our proposed parameter-free method.

\vspace{-0.2cm}
\subsection{Decentralized Meta-Learning}
\vspace{-0.1cm}
We evaluate our method on decentralized meta-learning using the CIFAR-10 dataset~\citep{mcmahan2017communication}, where multiple tasks are constructed following the protocol in~\citep{finn2017model}. This approach minimizes the test loss with respect to shared parameters as the upper-level loss, while the training loss is managed by task-specific parameters at the lower level. The detailed configuration of this experiment can be found in \cref{dmetasec}. To highlight the effectiveness of our approach, we compare against the state-of-the-art SLDBO method~\citep{dong2023single}. As shown in \cref{fig:meta} of \cref{addiresult}, our method achieves notably better training accuracy. This improvement stems from its problem-parameter-free design, which allows the algorithm to automatically adapt stepsizes, consistently reaching optimal convergence rates without manual tuning.

\vspace{-0.2cm}
\section{Conclusions and Limitations}\label{limits}
\vspace{-0.1cm}
In this paper, we proposed AdaSDBO, a parameter-free algorithm for decentralized bilevel optimization with a single-loop framework, supported by a rigorous finite-time convergence analysis. AdaSDBO adaptively adjusted stepsizes without relying on prior knowledge of problem parameters, achieving a convergence rate comparable to well-tuned counterparts. Extensive experiments showed that AdaSDBO delivered strong generalization performance and eliminated the need for tedious hyperparameter tuning, showcasing its potential for large-scale machine learning applications. Nevertheless, our analysis targets deterministic settings with full-gradient information and assumes a strongly convex lower-level problem. Extending the results to stochastic regimes and to generally convex lower-level objectives remains an open direction. We aim to address these limitations in future work to broaden the scope of applicability of AdaSDBO.

\bibliography{ref}

\begin{thebibliography}{74}
\providecommand{\natexlab}[1]{#1}
\providecommand{\url}[1]{\texttt{#1}}
\expandafter\ifx\csname urlstyle\endcsname\relax
  \providecommand{\doi}[1]{doi: #1}\else
  \providecommand{\doi}{doi: \begingroup \urlstyle{rm}\Url}\fi

\bibitem[Antonakopoulos et~al.(2025)Antonakopoulos, Sabach, Viano, Hong, and Cevher]{antonakopoulos2024adaptive}
Kimon Antonakopoulos, Shoham Sabach, Luca Viano, Mingyi Hong, and Volkan Cevher.
\newblock Adaptive bilevel optimization.
\newblock \emph{ACM/IMS Journal of Data Science}, 2\penalty0 (2):\penalty0 1--29, 2025.

\bibitem[Arora et~al.(2020)Arora, Du, Kakade, Luo, and Saunshi]{arora2020provable}
Sanjeev Arora, Simon Du, Sham Kakade, Yuping Luo, and Nikunj Saunshi.
\newblock Provable representation learning for imitation learning via bi-level optimization.
\newblock In \emph{International Conference on Machine Learning}, pages 367--376. PMLR, 2020.

\bibitem[Bertinetto et~al.(2018)Bertinetto, Henriques, Torr, and Vedaldi]{bertinetto2018meta}
Luca Bertinetto, Joao~F Henriques, Philip~HS Torr, and Andrea Vedaldi.
\newblock Meta-learning with differentiable closed-form solvers.
\newblock \emph{arXiv preprint arXiv:1805.08136}, 2018.

\bibitem[Bracken and McGill(1973)]{bracken1973mathematical}
Jerome Bracken and James~T McGill.
\newblock Mathematical programs with optimization problems in the constraints.
\newblock \emph{Operations Research}, 21\penalty0 (1):\penalty0 37--44, 1973.

\bibitem[Bubeck et~al.(2015)]{bubeck2015convex}
S{\'e}bastien Bubeck et~al.
\newblock Convex optimization: Algorithms and complexity.
\newblock \emph{Foundations and Trends{\textregistered} in Machine Learning}, 8\penalty0 (3-4):\penalty0 231--357, 2015.

\bibitem[Camacho-Vallejo et~al.(2024)Camacho-Vallejo, Corpus, and Villegas]{camacho2024metaheuristics}
Jos{\'e}-Fernando Camacho-Vallejo, Carlos Corpus, and Juan~G Villegas.
\newblock Metaheuristics for bilevel optimization: A comprehensive review.
\newblock \emph{Computers \& Operations Research}, 161:\penalty0 106410, 2024.

\bibitem[Caselli et~al.(2024)Caselli, Iori, and Ljubi{\'c}]{caselli2024bilevel}
Giulia Caselli, Manuel Iori, and Ivana Ljubi{\'c}.
\newblock Bilevel optimization with sustainability perspective: a survey on applications.
\newblock \emph{arXiv preprint arXiv:2406.07184}, 2024.

\bibitem[Chen et~al.(2023)Chen, Huang, Ma, and Balasubramanian]{chen2023decentralized}
Xuxing Chen, Minhui Huang, Shiqian Ma, and Krishna Balasubramanian.
\newblock Decentralized stochastic bilevel optimization with improved per-iteration complexity.
\newblock In \emph{International Conference on Machine Learning}, pages 4641--4671. PMLR, 2023.

\bibitem[Chen et~al.(2024{\natexlab{a}})Chen, Huang, and Ma]{chen2024decentralized}
Xuxing Chen, Minhui Huang, and Shiqian Ma.
\newblock Decentralized bilevel optimization.
\newblock \emph{Optimization Letters}, pages 1--65, 2024{\natexlab{a}}.

\bibitem[Chen et~al.(2024{\natexlab{b}})Chen, Xiao, and Balasubramanian]{chen2024optimal}
Xuxing Chen, Tesi Xiao, and Krishnakumar Balasubramanian.
\newblock Optimal algorithms for stochastic bilevel optimization under relaxed smoothness conditions.
\newblock \emph{Journal of Machine Learning Research}, 25\penalty0 (151):\penalty0 1--51, 2024{\natexlab{b}}.

\bibitem[Dagr{\'e}ou et~al.(2022)Dagr{\'e}ou, Ablin, Vaiter, and Moreau]{dagreou2022framework}
Mathieu Dagr{\'e}ou, Pierre Ablin, Samuel Vaiter, and Thomas Moreau.
\newblock A framework for bilevel optimization that enables stochastic and global variance reduction algorithms.
\newblock \emph{Advances in Neural Information Processing Systems}, 35:\penalty0 26698--26710, 2022.

\bibitem[Domke(2012)]{domke2012generic}
Justin Domke.
\newblock Generic methods for optimization-based modeling.
\newblock In \emph{Artificial Intelligence and Statistics}, pages 318--326. PMLR, 2012.

\bibitem[Dong et~al.(2023)Dong, Ma, Yang, and Yin]{dong2023single}
Youran Dong, Shiqian Ma, Junfeng Yang, and Chao Yin.
\newblock A single-loop algorithm for decentralized bilevel optimization.
\newblock \emph{arXiv preprint arXiv:2311.08945}, 2023.

\bibitem[Duchi et~al.(2011)Duchi, Hazan, and Singer]{duchi2011adaptive}
John Duchi, Elad Hazan, and Yoram Singer.
\newblock Adaptive subgradient methods for online learning and stochastic optimization.
\newblock \emph{Journal of Machine Learning Research}, 12\penalty0 (7), 2011.

\bibitem[Finn et~al.(2017)Finn, Abbeel, and Levine]{finn2017model}
Chelsea Finn, Pieter Abbeel, and Sergey Levine.
\newblock Model-agnostic meta-learning for fast adaptation of deep networks.
\newblock In \emph{International Conference on Machine Learning}, pages 1126--1135. PMLR, 2017.

\bibitem[Franceschi et~al.(2017)Franceschi, Donini, Frasconi, and Pontil]{franceschi2017forward}
Luca Franceschi, Michele Donini, Paolo Frasconi, and Massimiliano Pontil.
\newblock Forward and reverse gradient-based hyperparameter optimization.
\newblock In \emph{International Conference on Machine Learning}, pages 1165--1173. PMLR, 2017.

\bibitem[Franceschi et~al.(2018)Franceschi, Frasconi, Salzo, Grazzi, and Pontil]{franceschi2018bilevel}
Luca Franceschi, Paolo Frasconi, Saverio Salzo, Riccardo Grazzi, and Massimiliano Pontil.
\newblock Bilevel programming for hyperparameter optimization and meta-learning.
\newblock In \emph{International Conference on Machine Learning}, pages 1568--1577. PMLR, 2018.

\bibitem[Gao et~al.(2023)Gao, Gu, and Thai]{gao2023convergence}
Hongchang Gao, Bin Gu, and My~T Thai.
\newblock On the convergence of distributed stochastic bilevel optimization algorithms over a network.
\newblock In \emph{International Conference on Artificial Intelligence and Statistics}, pages 9238--9281. PMLR, 2023.

\bibitem[Ghadimi and Wang(2018)]{ghadimi2018approximation}
Saeed Ghadimi and Mengdi Wang.
\newblock Approximation methods for bilevel programming.
\newblock \emph{arXiv preprint arXiv:1802.02246}, 2018.

\bibitem[Grazzi et~al.(2020)Grazzi, Franceschi, Pontil, and Salzo]{grazzi2020iteration}
Riccardo Grazzi, Luca Franceschi, Massimiliano Pontil, and Saverio Salzo.
\newblock On the iteration complexity of hypergradient computation.
\newblock In \emph{International Conference on Machine Learning}, pages 3748--3758. PMLR, 2020.

\bibitem[Hansen et~al.(1992)Hansen, Jaumard, and Savard]{hansen1992new}
Pierre Hansen, Brigitte Jaumard, and Gilles Savard.
\newblock New branch-and-bound rules for linear bilevel programming.
\newblock \emph{SIAM Journal on Scientific and Statistical Computing}, 13\penalty0 (5):\penalty0 1194--1217, 1992.

\bibitem[Hashemi et~al.(2024)Hashemi, He, and Jaggi]{hashemi2024cobo}
Diba Hashemi, Lie He, and Martin Jaggi.
\newblock Cobo: Collaborative learning via bilevel optimization.
\newblock \emph{arXiv preprint arXiv:2409.05539}, 2024.

\bibitem[Hong et~al.(2023)Hong, Wai, Wang, and Yang]{hong2023two}
Mingyi Hong, Hoi-To Wai, Zhaoran Wang, and Zhuoran Yang.
\newblock A two-timescale stochastic algorithm framework for bilevel optimization: Complexity analysis and application to actor-critic.
\newblock \emph{SIAM Journal on Optimization}, 33\penalty0 (1):\penalty0 147--180, 2023.

\bibitem[Huang et~al.(2024{\natexlab{a}})Huang, Wang, Li, and Chen]{huang2024adaptive}
Feihu Huang, Xinrui Wang, Junyi Li, and Songcan Chen.
\newblock Adaptive federated minimax optimization with lower complexities.
\newblock In \emph{International Conference on Artificial Intelligence and Statistics}, pages 4663--4671. PMLR, 2024{\natexlab{a}}.

\bibitem[Huang et~al.(2024{\natexlab{b}})Huang, Li, Shen, He, and Xu]{huang2024achieving}
Yan Huang, Xiang Li, Yipeng Shen, Niao He, and Jinming Xu.
\newblock Achieving near-optimal convergence for distributed minimax optimization with adaptive stepsizes.
\newblock \emph{arXiv preprint arXiv:2406.02939}, 2024{\natexlab{b}}.

\bibitem[Ji and Ying(2023)]{ji2023network}
Kaiyi Ji and Lei Ying.
\newblock Network utility maximization with unknown utility functions: A distributed, data-driven bilevel optimization approach.
\newblock In \emph{Proceedings of the Twenty-fourth International Symposium on Theory, Algorithmic Foundations, and Protocol Design for Mobile Networks and Mobile Computing}, pages 131--140, 2023.

\bibitem[Ji et~al.(2020)Ji, Lee, Liang, and Poor]{ji2020convergence}
Kaiyi Ji, Jason~D Lee, Yingbin Liang, and H~Vincent Poor.
\newblock Convergence of meta-learning with task-specific adaptation over partial parameters.
\newblock \emph{Advances in Neural Information Processing Systems}, 33:\penalty0 11490--11500, 2020.

\bibitem[Ji et~al.(2021)Ji, Yang, and Liang]{ji2021bilevel}
Kaiyi Ji, Junjie Yang, and Yingbin Liang.
\newblock Bilevel optimization: Convergence analysis and enhanced design.
\newblock In \emph{International Conference on Machine Learning}, pages 4882--4892. PMLR, 2021.

\bibitem[Ji et~al.(2022)Ji, Liu, Liang, and Ying]{ji2022will}
Kaiyi Ji, Mingrui Liu, Yingbin Liang, and Lei Ying.
\newblock Will bilevel optimizers benefit from loops.
\newblock \emph{Advances in Neural Information Processing Systems}, 35:\penalty0 3011--3023, 2022.

\bibitem[Jiao et~al.(2022)Jiao, Yang, Wu, Song, and Jian]{jiao2022asynchronous}
Yang Jiao, Kai Yang, Tiancheng Wu, Dongjin Song, and Chengtao Jian.
\newblock Asynchronous distributed bilevel optimization.
\newblock \emph{arXiv preprint arXiv:2212.10048}, 2022.

\bibitem[Kayaalp et~al.(2022)Kayaalp, Vlaski, and Sayed]{kayaalp2022dif}
Mert Kayaalp, Stefan Vlaski, and Ali~H Sayed.
\newblock Dif-maml: Decentralized multi-agent meta-learning.
\newblock \emph{IEEE Open Journal of Signal Processing}, 3:\penalty0 71--93, 2022.

\bibitem[Kong et~al.(2024)Kong, Zhu, Lu, Huang, and Yuan]{kong2024decentralized}
Boao Kong, Shuchen Zhu, Songtao Lu, Xinmeng Huang, and Kun Yuan.
\newblock Decentralized bilevel optimization over graphs: Loopless algorithmic update and transient iteration complexity.
\newblock \emph{arXiv preprint arXiv:2402.03167}, 2024.

\bibitem[LeCun et~al.(1998)LeCun, Bottou, Bengio, and Haffner]{lecun1998gradient}
Yann LeCun, L{\'e}on Bottou, Yoshua Bengio, and Patrick Haffner.
\newblock Gradient-based learning applied to document recognition.
\newblock \emph{Proceedings of the IEEE}, 86\penalty0 (11):\penalty0 2278--2324, 1998.

\bibitem[Levy(2017)]{levy2017online}
Kfir Levy.
\newblock Online to offline conversions, universality and adaptive minibatch sizes.
\newblock \emph{Advances in Neural Information Processing Systems}, 30, 2017.

\bibitem[Levy et~al.(2018)Levy, Yurtsever, and Cevher]{levy2018online}
Kfir~Y Levy, Alp Yurtsever, and Volkan Cevher.
\newblock Online adaptive methods, universality and acceleration.
\newblock \emph{Advances in Neural Information Processing Systems}, 31, 2018.

\bibitem[Li et~al.(2024)Li, Chen, Ma, and Hong]{li2024problem}
Jiaxiang Li, Xuxing Chen, Shiqian Ma, and Mingyi Hong.
\newblock Problem-parameter-free decentralized nonconvex stochastic optimization.
\newblock \emph{arXiv preprint arXiv:2402.08821}, 2024.

\bibitem[Li et~al.(2022{\natexlab{a}})Li, Gu, and Huang]{li2022fully}
Junyi Li, Bin Gu, and Heng Huang.
\newblock A fully single loop algorithm for bilevel optimization without hessian inverse.
\newblock In \emph{Proceedings of the AAAI Conference on Artificial Intelligence}, volume~36, pages 7426--7434, 2022{\natexlab{a}}.

\bibitem[Li et~al.(2022{\natexlab{b}})Li, Yang, and He]{li2022tiada}
Xiang Li, Junchi Yang, and Niao He.
\newblock Tiada: A time-scale adaptive algorithm for nonconvex minimax optimization.
\newblock \emph{arXiv preprint arXiv:2210.17478}, 2022{\natexlab{b}}.

\bibitem[Liao et~al.(2018)Liao, Xiong, Fetaya, Zhang, Yoon, Pitkow, Urtasun, and Zemel]{liao2018reviving}
Renjie Liao, Yuwen Xiong, Ethan Fetaya, Lisa Zhang, KiJung Yoon, Xaq Pitkow, Raquel Urtasun, and Richard Zemel.
\newblock Reviving and improving recurrent back-propagation.
\newblock In \emph{International Conference on Machine Learning}, pages 3082--3091. PMLR, 2018.

\bibitem[Liu et~al.(2020)Liu, Gao, and Yin]{liu2020improved}
Yanli Liu, Yuan Gao, and Wotao Yin.
\newblock An improved analysis of stochastic gradient descent with momentum.
\newblock \emph{Advances in Neural Information Processing Systems}, 33:\penalty0 18261--18271, 2020.

\bibitem[Liu et~al.(2022)Liu, Zhang, Khanduri, Lu, and Liu]{liu2022interact}
Zhuqing Liu, Xin Zhang, Prashant Khanduri, Songtao Lu, and Jia Liu.
\newblock Interact: Achieving low sample and communication complexities in decentralized bilevel learning over networks.
\newblock In \emph{Proceedings of the Twenty-Third International Symposium on Theory, Algorithmic Foundations, and Protocol Design for Mobile Networks and Mobile Computing}, pages 61--70, 2022.

\bibitem[Lu et~al.(2022{\natexlab{a}})Lu, Cui, Squillante, Kingsbury, and Horesh]{lu2022decentralized}
Songtao Lu, Xiaodong Cui, Mark~S Squillante, Brian Kingsbury, and Lior Horesh.
\newblock Decentralized bilevel optimization for personalized client learning.
\newblock In \emph{ICASSP 2022-2022 IEEE International Conference on Acoustics, Speech and Signal Processing (ICASSP)}, pages 5543--5547. IEEE, 2022{\natexlab{a}}.

\bibitem[Lu et~al.(2022{\natexlab{b}})Lu, Zeng, Cui, Squillante, Horesh, Kingsbury, Liu, and Hong]{lu2022stochastic}
Songtao Lu, Siliang Zeng, Xiaodong Cui, Mark Squillante, Lior Horesh, Brian Kingsbury, Jia Liu, and Mingyi Hong.
\newblock A stochastic linearized augmented lagrangian method for decentralized bilevel optimization.
\newblock \emph{Advances in Neural Information Processing Systems}, 35:\penalty0 30638--30650, 2022{\natexlab{b}}.

\bibitem[Luo et~al.(2019)Luo, Xiong, Liu, and Sun]{luo2019adaptive}
Liangchen Luo, Yuanhao Xiong, Yan Liu, and Xu~Sun.
\newblock Adaptive gradient methods with dynamic bound of learning rate.
\newblock \emph{arXiv preprint arXiv:1902.09843}, 2019.

\bibitem[Maclaurin et~al.(2015)Maclaurin, Duvenaud, and Adams]{maclaurin2015gradient}
Dougal Maclaurin, David Duvenaud, and Ryan Adams.
\newblock Gradient-based hyperparameter optimization through reversible learning.
\newblock In \emph{International Conference on Machine Learning}, pages 2113--2122. PMLR, 2015.

\bibitem[Madry et~al.(2017)Madry, Makelov, Schmidt, Tsipras, and Vladu]{mkadry2017towards}
Aleksander Madry, Aleksandar Makelov, Ludwig Schmidt, Dimitris Tsipras, and Adrian Vladu.
\newblock Towards deep learning models resistant to adversarial attacks.
\newblock \emph{arXiv preprint arXiv:1706.06083}, 2017.

\bibitem[Malitsky and Mishchenko(2019)]{malitsky2019adaptive}
Yura Malitsky and Konstantin Mishchenko.
\newblock Adaptive gradient descent without descent.
\newblock \emph{arXiv preprint arXiv:1910.09529}, 2019.

\bibitem[Marumo and Takeda(2024)]{marumo2024parameter}
Naoki Marumo and Akiko Takeda.
\newblock Parameter-free accelerated gradient descent for nonconvex minimization.
\newblock \emph{SIAM Journal on Optimization}, 34\penalty0 (2):\penalty0 2093--2120, 2024.

\bibitem[McMahan et~al.(2017)McMahan, Moore, Ramage, Hampson, and y~Arcas]{mcmahan2017communication}
Brendan McMahan, Eider Moore, Daniel Ramage, Seth Hampson, and Blaise~Aguera y~Arcas.
\newblock Communication-efficient learning of deep networks from decentralized data.
\newblock In \emph{Artificial Intelligence and Statistics}, pages 1273--1282. PMLR, 2017.

\bibitem[McMahan and Streeter(2010)]{mcmahan2010adaptive}
H~Brendan McMahan and Matthew Streeter.
\newblock Adaptive bound optimization for online convex optimization.
\newblock \emph{arXiv preprint arXiv:1002.4908}, 2010.

\bibitem[Nedic et~al.(2010)Nedic, Ozdaglar, and Parrilo]{nedic2010constrained}
Angelia Nedic, Asuman Ozdaglar, and Pablo~A Parrilo.
\newblock Constrained consensus and optimization in multi-agent networks.
\newblock \emph{IEEE Transactions on Automatic Control}, 55\penalty0 (4):\penalty0 922--938, 2010.

\bibitem[Paszke et~al.(2019)Paszke, Gross, Massa, Lerer, Bradbury, Chanan, Killeen, Lin, Gimelshein, Antiga, et~al.]{paszke2019pytorch}
Adam Paszke, Sam Gross, Francisco Massa, Adam Lerer, James Bradbury, Gregory Chanan, Trevor Killeen, Zeming Lin, Natalia Gimelshein, Luca Antiga, et~al.
\newblock Pytorch: An imperative style, high-performance deep learning library.
\newblock \emph{Advances in Neural Information Processing Systems}, 32, 2019.

\bibitem[Pedregosa(2016)]{pedregosa2016hyperparameter}
Fabian Pedregosa.
\newblock Hyperparameter optimization with approximate gradient.
\newblock In \emph{International Conference on Machine Learning}, pages 737--746. PMLR, 2016.

\bibitem[Rajeswaran et~al.(2019)Rajeswaran, Finn, Kakade, and Levine]{rajeswaran2019meta}
Aravind Rajeswaran, Chelsea Finn, Sham~M Kakade, and Sergey Levine.
\newblock Meta-learning with implicit gradients.
\newblock \emph{Advances in Neural Information Processing Systems}, 32, 2019.

\bibitem[Shaban et~al.(2019)Shaban, Cheng, Hatch, and Boots]{shaban2019truncated}
Amirreza Shaban, Ching-An Cheng, Nathan Hatch, and Byron Boots.
\newblock Truncated back-propagation for bilevel optimization.
\newblock In \emph{The 22nd International Conference on Artificial Intelligence and Statistics}, pages 1723--1732. PMLR, 2019.

\bibitem[Shen et~al.(2025)Shen, Yang, and Chen]{shen2025principled}
Han Shen, Zhuoran Yang, and Tianyi Chen.
\newblock Principled penalty-based methods for bilevel reinforcement learning and rlhf.
\newblock \emph{Journal of Machine Learning Research}, 26\penalty0 (114):\penalty0 1--49, 2025.

\bibitem[Shi et~al.(2005)Shi, Lu, and Zhang]{shi2005extended}
Chenggen Shi, Jie Lu, and Guangquan Zhang.
\newblock An extended kuhn--tucker approach for linear bilevel programming.
\newblock \emph{Applied Mathematics and Computation}, 162\penalty0 (1):\penalty0 51--63, 2005.

\bibitem[Thoma et~al.(2024)Thoma, P{\'a}sztor, Krause, Ramponi, and Hu]{thoma2024contextual}
Vinzenz Thoma, Barna P{\'a}sztor, Andreas Krause, Giorgia Ramponi, and Yifan Hu.
\newblock Contextual bilevel reinforcement learning for incentive alignment.
\newblock \emph{Advances in Neural Information Processing Systems}, 37:\penalty0 127369--127435, 2024.

\bibitem[Tieleman and Hinton(2017)]{tieleman2017divide}
Tijmen Tieleman and G~Hinton.
\newblock Divide the gradient by a running average of its recent magnitude. coursera: Neural networks for machine learning.
\newblock \emph{Technical report}, 2017.

\bibitem[Wang et~al.(2024)Wang, Chen, Ma, and Zhang]{wang2024fully}
Xiaoyu Wang, Xuxing Chen, Shiqian Ma, and Tong Zhang.
\newblock Fully first-order methods for decentralized bilevel optimization.
\newblock \emph{arXiv preprint arXiv:2410.19319}, 2024.

\bibitem[Ward et~al.(2020)Ward, Wu, and Bottou]{ward2020adagrad}
Rachel Ward, Xiaoxia Wu, and Leon Bottou.
\newblock Adagrad stepsizes: Sharp convergence over nonconvex landscapes.
\newblock \emph{Journal of Machine Learning Research}, 21\penalty0 (219):\penalty0 1--30, 2020.

\bibitem[Xiao et~al.(2017)Xiao, Rasul, and Vollgraf]{xiao2017fashion}
Han Xiao, Kashif Rasul, and Roland Vollgraf.
\newblock Fashion-mnist: a novel image dataset for benchmarking machine learning algorithms.
\newblock \emph{arXiv preprint arXiv:1708.07747}, 2017.

\bibitem[Xie et~al.(2024)Xie, Zhou, Li, Lin, and Yan]{xie2024adan}
Xingyu Xie, Pan Zhou, Huan Li, Zhouchen Lin, and Shuicheng Yan.
\newblock Adan: Adaptive nesterov momentum algorithm for faster optimizing deep models.
\newblock \emph{IEEE Transactions on Pattern Analysis and Machine Intelligence}, 2024.

\bibitem[Xie et~al.(2020)Xie, Wu, and Ward]{xie2020linear}
Yuege Xie, Xiaoxia Wu, and Rachel Ward.
\newblock Linear convergence of adaptive stochastic gradient descent.
\newblock In \emph{International Conference on Artificial Intelligence and Statistics}, pages 1475--1485. PMLR, 2020.

\bibitem[Yan et~al.(2025)Yan, Zhang, Wang, and Cao]{yan2025problem}
Wenjing Yan, Kai Zhang, Xiaolu Wang, and Xuanyu Cao.
\newblock Problem-parameter-free federated learning.
\newblock In \emph{The Thirteenth International Conference on Learning Representations}, 2025.

\bibitem[Yang et~al.(2022)Yang, Li, and He]{yang2022nest}
Junchi Yang, Xiang Li, and Niao He.
\newblock Nest your adaptive algorithm for parameter-agnostic nonconvex minimax optimization.
\newblock \emph{Advances in Neural Information Processing Systems}, 35:\penalty0 11202--11216, 2022.

\bibitem[Yang et~al.(2023)Yang, Li, Fatkhullin, and He]{yang2023two}
Junchi Yang, Xiang Li, Ilyas Fatkhullin, and Niao He.
\newblock Two sides of one coin: the limits of untuned sgd and the power of adaptive methods.
\newblock \emph{Advances in Neural Information Processing Systems}, 36:\penalty0 74257--74288, 2023.

\bibitem[Yang et~al.(2021)Yang, Ji, and Liang]{yang2021provably}
Junjie Yang, Kaiyi Ji, and Yingbin Liang.
\newblock Provably faster algorithms for bilevel optimization.
\newblock \emph{Advances in Neural Information Processing Systems}, 34:\penalty0 13670--13682, 2021.

\bibitem[Yang et~al.(2024)Yang, Xiao, and Ji]{yang2024simfbo}
Yifan Yang, Peiyao Xiao, and Kaiyi Ji.
\newblock Simfbo: Towards simple, flexible and communication-efficient federated bilevel learning.
\newblock \emph{Advances in Neural Information Processing Systems}, 36, 2024.

\bibitem[Yang et~al.(2025)Yang, Ban, Huang, Ma, and Ji]{yang2024tuning}
Yifan Yang, Hao Ban, Minhui Huang, Shiqian Ma, and Kaiyi Ji.
\newblock Tuning-free bilevel optimization: New algorithms and convergence analysis.
\newblock In \emph{The Thirteenth International Conference on Learning Representations}, 2025.

\bibitem[Zeiler(2012)]{zeiler2012adadelta}
Matthew~D Zeiler.
\newblock Adadelta: an adaptive learning rate method.
\newblock \emph{arXiv preprint arXiv:1212.5701}, 2012.

\bibitem[Zhang et~al.(2019)Zhang, Tang, Dodge, Zhou, and Wang]{zhang2019metapred}
Xi~Sheryl Zhang, Fengyi Tang, Hiroko~H Dodge, Jiayu Zhou, and Fei Wang.
\newblock Metapred: Meta-learning for clinical risk prediction with limited patient electronic health records.
\newblock In \emph{Proceedings of the 25th ACM SIGKDD International Conference on Knowledge Discovery \& Data Mining}, pages 2487--2495, 2019.

\bibitem[Zhang et~al.(2023)Zhang, Thai, Wu, and Gao]{zhang2023communication}
Yihan Zhang, My~T Thai, Jie Wu, and Hongchang Gao.
\newblock On the communication complexity of decentralized bilevel optimization.
\newblock \emph{arXiv preprint arXiv:2311.11342}, 2023.

\bibitem[Zhu et~al.(2024)Zhu, Kong, Lu, Huang, and Yuan]{zhu2024sparkle}
Shuchen Zhu, Boao Kong, Songtao Lu, Xinmeng Huang, and Kun Yuan.
\newblock Sparkle: a unified single-loop primal-dual framework for decentralized bilevel optimization.
\newblock \emph{Advances in Neural Information Processing Systems}, 37:\penalty0 62912--62987, 2024.

\end{thebibliography}
\bibliographystyle{plainnat}

\newpage
\appendix

\section{Additional Discussion on Related Works}\label{secrelated work}
\textbf{Bilevel Optimization.}
Bilevel optimization originated with \citep{bracken1973mathematical} and has seen significant advances in both theory and algorithms. Early methodological approaches \citep{hansen1992new,shi2005extended} predominantly addressed these problems through constrained optimization formulations, treating the inner-level problem as a parametric constraint.
Contemporary research has increasingly focused on efficient gradient-based methods, which fall into three main categories: 1) approximate implicit differentiation (AID) methods \citep{domke2012generic,liao2018reviving,ji2021bilevel,dagreou2022framework}, which use the implicit function theorem to approximate hypergradients; 2) iterative differentiation (ITD) methods \citep{maclaurin2015gradient,franceschi2018bilevel,grazzi2020iteration}, which leverage automatic differentiation; 3) Neumann series-based methods \citep{ji2021bilevel, yang2021provably}, which approximate the inverse Hessian using truncated series. However, implicit differentiation in bilevel optimization requires accurate inner problem solutions for each outer variable update, leading to high computational costs in large-scale problems. To address this, the researchers proposed solving the inner problem with a fixed number of steps and computing gradients with the "backpropagation through time" technique \citep{shaban2019truncated,franceschi2017forward}. Nevertheless, this approach remains computationally expensive for modern machine learning models with hundreds of millions of parameters. Recently, there has been a surge of interest in using implicit differentiation to derive single-loop algorithms. \citet{ghadimi2018approximation} introduced an accelerated AID method with the Neumann series. \citet{yang2021provably} proposed a warm-start strategy to reduce the number of inner steps required at each iteration. Additionally, \citet{li2022fully} introduced FSLA, a fully single-loop algorithm for bilevel optimization that eliminates the need for Hessian inversion.

\textbf{Adaptive Methods.}
The introduction of AdaGrad \citep{mcmahan2010adaptive,duchi2011adaptive} marked a milestone in adaptive gradient-based methods. Originally designed for online convex optimization, AdaGrad quickly evolved into a foundation for deep learning algorithms, spawning numerous variants such as Adadelta \citep{zeiler2012adadelta}, RMSprop \citep{tieleman2017divide}, and Adam \citep{luo2019adaptive,xie2024adan}. In particular, AdaGrad variants with normalized gradients, including AdaNGD \citep{levy2017online}, AcceleGrad \citep{levy2018online}, and AdaGrad-Norm \citep{xie2020linear}, introduced adaptive stepsizes that eliminate the need for problem-specific parameters, establishing themselves as effective parameter-free methods.
More recent refinements, such as the Lipschitzness parameter approximation \citep{malitsky2019adaptive} and the restart mechanisms \citep{marumo2024parameter}, have further enhanced both performance and robustness. Additionally, \citet{yang2023two} provided foundational insights into mainstream adaptive methods, laying the groundwork for their applications in distributed optimization \citep{li2024problem,yan2025problem}.

\textbf{Adaptive Minimax and Bilevel Methods.}
Currently, some research has begun addressing adaptive stepsize design specifically within the minimax optimization context \citep{li2022tiada, huang2024adaptive, huang2024achieving}. However, minimax problems inherently possess simpler structural properties compared to bilevel optimization problems. The nested structure inherent in bilevel optimization introduces additional complexities due to the coupling of variables between the upper and lower levels, significantly complicating the design of adaptive stepsizes.
In centralized bilevel optimization, a few adaptive methods have been proposed. For instance, \citet{antonakopoulos2024adaptive} proposed a double-loop adaptive algorithm utilizing mirror descent, which still relies on the unknown strong convexity parameter of the lower-level function.
Similarly, \citet{yang2024tuning} introduced a centralized adaptive method based on AdaGrad-Norm, achieving convergence rates comparable to well-tuned methods. However, tuning hyperparameters in practice is often prohibitively expensive and becomes considerably more difficult in decentralized scenarios, thus making the ability to automatically adjust the update dynamics particularly crucial \citep{li2024problem}. Nevertheless, achieving such adaptability in decentralized bilevel optimization remains profoundly challenging. Specifically, decentralized bilevel optimization necessitates carefully orchestrated updates across primal, dual, and auxiliary variables to manage the nested structure effectively. Additionally, the hierarchical structure of decentralized bilevel optimization introduces multiple coupled adaptive stepsizes, significantly amplifying heterogeneity across agents. Without meticulous coordination, the resulting variability can degrade convergence performance. Moreover, the decentralized setting requires simultaneously managing network-induced communication errors and hierarchical bilevel approximation errors. To achieve convergence rates matching those of optimally tuned methods, adaptive decentralized bilevel algorithms require an even more precise theoretical analysis, as these error sources interact intricately.
To the best of our knowledge, it remains an open and challenging question on how to leverage adaptive methods to design a completely problem-parameter-free algorithm for decentralized bilevel optimization.

\section{Proof Sketch}\label{proof}

\begin{figure}[h!]
\centering
\begin{tikzpicture}[
  node distance=1cm and 0.7cm,
  every node/.style={
    draw, 
    rectangle, 
    rounded corners=4.1mm,
    minimum width=3cm,
    minimum height=1.2cm,
    align=center,
    line width=0.6mm
  },
  every path/.style={
    line width=0.7mm,
    -{Latex[length=2mm, width=2mm]}
  }
]

\node (A6) at (0,0) {\cref{Lemma7}};
\node[right=of A6] (A7) {\cref{Lemma8}};
\node[right=of A7] (A10) {\cref{Lemma11}};
\node[right=of A10] (A11) {\cref{Lemma5}};

\node[below=of A6] (A8) {\cref{Lemma9}};
\node[right=of A8] (A9) {\cref{Lemma10}};
\node[right=of A9] (A12) {\cref{lemmanew12}};
\node[right=of A12] (A13) {\cref{Lemma12}};

\node[below=of A9, xshift=2.1cm] (Theorem) {\cref{theorem 1}};

\node[below=of Theorem, xshift=-2.1cm] (A5) {\cref{Lemma4}};
\node[below=of Theorem, xshift=1.7cm] (A14) {\cref{Lemma13}};

\draw[->] (A5) -- (Theorem);
\draw[->] (A5) -- (A14);

\draw[->] (A6) -- (A8);
\draw[->] (A6) -- (A9);

\draw[->] (A7) -- (A9);
\draw[->] (A7) -- (A10);

\draw[->] (A8) -- (A9);

\draw[->] (A9) -- (Theorem);
\draw[->] (A9) -- (A10);

\draw[->] (A10) -- (A11);
\draw[->] (A10) -- (A12);
\draw[->] (A10) -- (A13);

\draw[->] (A11) -- (A12);

\draw[->] (A12) -- (A14);
\draw[->] (A12) -- (Theorem);

\draw[->] (A13) -- (A14);

\draw[->] (A14) -- (Theorem);
\end{tikzpicture}
\caption{Structure of the proof}
\label{fig:proof_structure}
\end{figure}
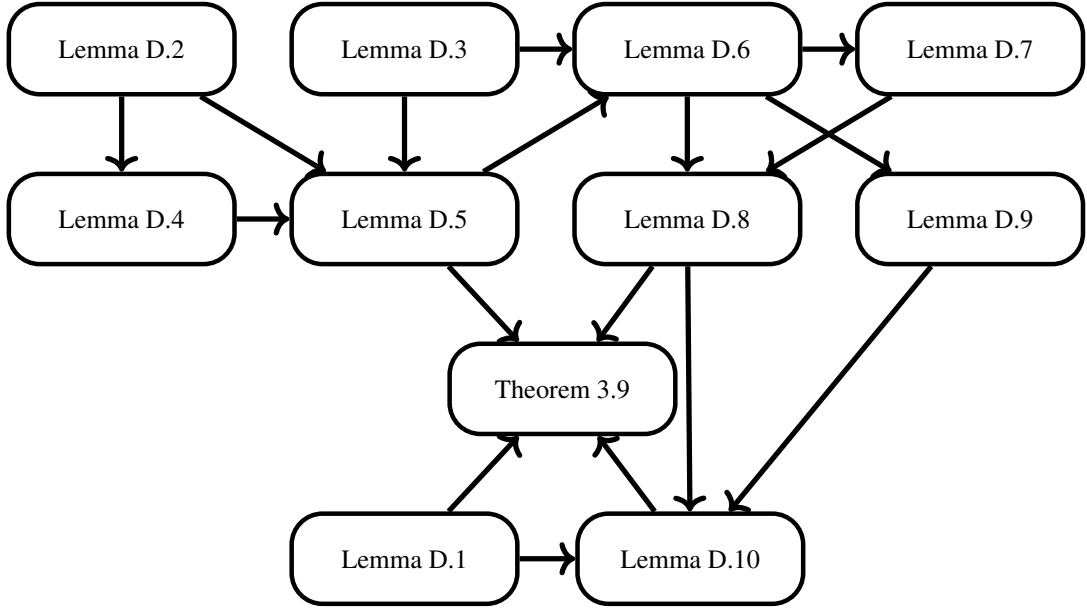

\cref{fig:proof_structure} illustrates the structure of the proof. Next, we present the proof sketch for \cref{theorem 1}.

\textbf{Proof Sketch of \cref{theorem 1}:}

\textbf{Step 1:} We start by introducing the two-stage framework outlined in \cref{Lemma6} to examine the progression of the gradient accumulators \(\bar{m}^x_t\), \(\bar{m}^y_t\), and \(\bar{m}^v_t\). This framework divides the iterations into two cases: when the gradient accumulators are below or exceed a predefined threshold. Using this structure, we derive two descent lemmas in \cref{Lemma4} for the objective function, corresponding to these two stages of \(\bar{m}^x_t\).

\textbf{Step 2:} Next, we derive tail bounds for two key components in the descent \cref{Lemma4}:
$\sum_{k=k_2}^t \frac{\|\nabla_y L(\mathbf{x}_k, \mathbf{y}_k)\|^2}{\bar{u}_{k+1}}$ (in \cref{Lemma7}) and $\sum_{k=k_3}^t \frac{\|\nabla_v R(\mathbf{x}_k, \mathbf{y}_k, \mathbf{v}_k)\|^2}{\bar{z}_{k+1}}$ (in \cref{Lemma8}),
where \(k_2\) and \(k_3\) represent the cutoff points corresponding to the second stage in the two-stage framework of \cref{Lemma6}.

\textbf{Step 3:} Using the bounds derived in Step 2, we establish upper bounds for \(\bar{m}^y_{t+1}\) and \(\bar{z}_{t+1}\) in \cref{Lemma9} and \cref{Lemma10}, respectively. Based on these results, we derive general bounds for $\sum_{k=k_0}^t \frac{\|\bar{\nabla} F(\mathbf{x}_k,\mathbf{y}_k,\mathbf{v}_k)\|^2}{[\bar{m}^x_{k+1}]^2}$, $\sum_{k=k_0}^t \frac{\|\nabla_y L(\mathbf{x}_k, \mathbf{y}_k)\|^2}{\bar{m}^y_{k+1}}$, and $\sum_{k=k_0}^t \frac{\|\nabla_v R(\mathbf{x}_k, \mathbf{y}_k, \mathbf{v}_k)\|^2}{\bar{z}_{k+1}}$ with \(k_0 \in [0, t)\), as shown in \cref{Lemma11}. 

\textbf{Step 4:} We then analyze the consensus errors for the primal, dual, and auxiliary variables in \cref{Lemma5}. By combining these results with \cref{Lemma11}, we derive the upper bounds: $\sum_{k=k_0}^t \frac{\|\bar{\nabla} f(\bar{x}_k, \bar{y}_k, \bar{v}_k)\|^2}{[\bar{m}^x_{k+1}]^2}$, $\sum_{k=k_0}^t \frac{\|\nabla_y l(\bar{x}_k, \bar{y}_k)\|^2}{\bar{m}^y_{k+1}}$ and $\sum_{k=k_0}^t \frac{\|\nabla_v r(\bar{x}_k, \bar{y}_k, \bar{v}_k)\|^2}{\bar{z}_{k+1}}$ with respect to the consensus variables $\bar{x}_k$, $\bar{y}_k$, and $\bar{v}_k$ in \cref{lemmanew12}, where \(k_0 \in [0, t)\).

\textbf{Step 5:} Additionally, in \cref{Lemma12}, we derive an upper bound for the term associated with the stepsize-inconsistency errors in \cref{Lemma4}. By substituting the results from \cref{lemmanew12} and \cref{Lemma12} into the Descent \cref{Lemma4}, we obtain the bound for \(\bar{m}^x_t\), as presented in \cref{Lemma13}.

\textbf{Step 6:} Finally, by combining the results from Steps 3, 4, and 5, we establish the convergence of \cref{alg1} based on the descent analysis in \cref{Lemma4}.

\section{Proofs of Supporting Lemmas}
\begin{lemma}[Lemma 3.2 in \citep{ward2020adagrad}]\label{Lemma1}
For any non-negative $a_1, \dots, a_T,$ and $a_1 \geq 1$, we have:
\begin{align}
\sum_{l=1}^T \frac{a_l}{\sum_{i=1}^l a_i} \leq \log \left(\sum_{l=1}^T a_l\right) + 1.\label{eq21}
\end{align}
\end{lemma}

  \begin{lemma}\label{Lemma2} Under \cref{Assumption1} and \cref{Assumption2}, we have the following basic properties:
\begin{enumerate}
    \item[1).] \(\Phi(\bar{x})\) is \(L_{\Phi}\)-smooth with respect to \(\bar{x}\), where 
    $L_{\Phi} := \left(L_{f,1} + \frac{L_{l,2} C_{f_y}}{\mu} \right)\left(1 + \frac{C_{l_{xy}}}{\mu}\right)^2$;
    \item[2).] \({y}^*(\bar{x})\) is \(L_y\)-Lipschitz continuous with respect to \(\bar{x}\), where $L_y := \frac{C_{l_{xy}}}{\mu}$;
    \item[3).] The gradient estimator \(\bar{\nabla} f(\bar{x}, \bar{y}, \bar{v})\) is \((L_{l,2} \|\bar{v}\| + L_{f,1})\)-Lipschitz continuous with respect to \((\bar{x}, \bar{y})\) and \(L_{l,1}\)-Lipschitz continuous with respect to \(\bar{v}\);
    \item[4).] \(\bar{\nabla} f(\bar{x}, \bar{y}, \bar{v})\) can be bounded as $\|\bar{\nabla} f(\bar{x}, \bar{y}, \bar{v})\| \leq C_{l_{xy}} \|\bar{v}\| + C_{f_x}.$
\end{enumerate}
\end{lemma}

\begin{proof}
The proofs of 1) and 2) can refer to \citep{ghadimi2018approximation}. For 3), under \cref{Assumption2}, we have:
\begin{align}
   &\|\bar{\nabla} f(\bar{x}_1, \bar{y}_1, \bar{v}) - \bar{\nabla} f(\bar{x}_2, \bar{y}_2, \bar{v})\| \nonumber\\
&\leq \|\nabla_{{x}}\nabla_{{y}} l(\bar{x}_1, \bar{y}_1) - \nabla_{{x}}\nabla_{{y}} l(\bar{x}_2, \bar{y}_2)\| \cdot \|\bar{v}\|
+ \|\nabla_{{x}} f(\bar{x}_1, \bar{y}_1) - \nabla_{{x}} f(\bar{x}_2, \bar{y}_2)\| \nonumber\\
&\leq (L_{l,2} \|\bar{v}\| + L_{f,1})(\|\bar{x}_1 - \bar{x}_2\| + \|\bar{y}_1 - \bar{y}_2\|),\label{eq22}
\end{align}
and
\begin{align}
\|\bar{\nabla} f(\bar{x}, \bar{y}, \bar{v}_1) - \bar{\nabla} f(\bar{x}, \bar{y}, \bar{v}_2)\| 
\leq \|\nabla_{{x}} \nabla_{{y}} l(\bar{x}, \bar{y})\| \cdot \|\bar{v}_1 - \bar{v}_2\| 
\leq L_{l,1} \|\bar{v}_1 - \bar{v}_2\|.\label{eq23}
\end{align}
By \cref{Assumption2}, we can easily prove 4) as:
\begin{align}
\|\bar{\nabla} f(\bar{x}, \bar{y}, \bar{v})\| \leq \|\nabla_{{x}} \nabla_{{y}} l(\bar{x}, \bar{y})\| \cdot \|\bar{v}\| + \|\nabla_{{x}} f(\bar{x}, \bar{y})\| 
\leq C_{l_{xy}} \|\bar{v}\| + C_{f_x}.\label{eq24}
\end{align}
Then the proof is complete.
\end{proof} 

\begin{lemma}\label{Lemma3} Under \cref{Assumption1} and \cref{Assumption2}, we have basic properties of the linear system function \(r\) in \cref{eqcontent3} as follows:

\begin{enumerate}
    \item[1).] \(r(\bar{x}, \bar{y}, \bar{v})\) is \(\mu\)-strongly convex and \(C_{l{yy}}\)-smooth with respect to \(\bar{v}\);
    \item[2).] \(\nabla_{{v}} r(\bar{x}, \bar{y}, \bar{v})\) is \((L_{l,2} \|\bar{v}\| + L_{f,1})\)-Lipschitz continuous with respect to \((\bar{x}, \bar{y})\);
    \item[3).] \(\nabla_{{v}} r(\bar{x}, \bar{y}, \bar{v})\) can be bounded as $\|\nabla_{{v}} r(\bar{x}, \bar{y}, \bar{v})\| \leq C_{l_{yy}} \|\bar{v}\| + C_{f_y};$
    \item[4).] \({v}^*(\bar{x})\) in \cref{eqcontent3} can be bounded as $\|{v}^*(\bar{x})\| \leq \frac{C_{fy}}{\mu}$, and \({v}^*(\bar{x}, \bar{y}) := \arg\min_{\bar{v}} r(\bar{x}, \bar{y}, \bar{v})\) can also be bounded as $\|{v}^*(\bar{x}, \bar{y})\| \leq \frac{C_{f_y}}{\mu}$;
    \item[5).] \(\bar{\nabla} f(\bar{x}, \bar{y}, \bar{v})\) is $\bar{L}_f$-Lipschitz continuous with respect to $(\bar{x}, \bar{y}, \bar{v})$, where $\bar{L}_f=\max\left\{\frac{C_{fy}L_{l,2}}{\mu} + L_{f,1}, L_{l,1}\right\}$;
    \item[6).] \(\nabla_{{v}} r(\bar{x}, \bar{y}, \bar{v})\) is $\bar{L}_r$-Lipschitz continuous with respect to $(\bar{x}, \bar{y}, \bar{v})$, where $\bar{L}_r=\max\left\{\frac{C_{fy}L_{l,2}}{\mu} + L_{f,1}, C_{l_{yy}}\right\}$;
    \item[7).] \({v}^*(\bar{x})\) is \(L_v\)-Lipschitz continuous with respect to \(\bar{x}\) and \({v}^*(\bar{x}, \bar{y})\) is \(\bar{L}_v\)-Lipschitz continuous with respect to \(\bar{y}\), where $L_v := \left(\frac{L_{f,1}}{\mu} + \frac{C_{f_y} L_{l,2}}{\mu^2} \right)(1 + L_y)$ and $\bar{L}_v := \frac{L_{f,1}}{\mu} + \frac{C_{f_y} L_{l,2}}{\mu^2}$.
\end{enumerate}
\end{lemma}

\begin{proof}
First of all, since \(\nabla_{{v}}\nabla_{{v}}r(\bar{x}, \bar{y}, \bar{v}) = \nabla_{{y}}\nabla_{{y}}l(\bar{x}, \bar{y})\), we know \(\mu I \preceq \nabla_{{y}}\nabla_{{y}}l(\bar{x}, \bar{y})\). Thus, according to \cref{Assumption1} and \cref{Assumption2}, we have:
\begin{align}
    \|\nabla_{{v}}\nabla_{{v}}r(\bar{x}, \bar{y}, \bar{v}_1) - \nabla_{{v}}\nabla_{{v}}r(\bar{x}, \bar{y}, \bar{v}_2)\| \leq \|\nabla_{{y}}\nabla_{{y}}l(\bar{x}, \bar{y})\| \|\bar{v}_1 - \bar{v}_2\| \leq C_{l_{yy}} \|\bar{v}_1 - \bar{v}_2\|.\label{eq25}
\end{align}
Then 1) is proved. 

Next, by using Lipschitz continuity in \cref{Assumption2}, we have:
\begin{align}
    &\|\nabla_{{v}}r(\bar{x}_1, \bar{y}_1, \bar{v}) - \nabla_{{v}}r(\bar{x}_2, \bar{y}_2, \bar{v})\|\nonumber\\ &\leq \|\nabla_{{y}}\nabla_{{y}}l(\bar{x}_1, \bar{y}_1) - \nabla_{{y}}\nabla_{{y}}l(\bar{x}_2, \bar{y}_2)\| \|\bar{v}\| + \|\nabla_{{y}}f(\bar{x}_1, \bar{y}_1) - \nabla_{{y}}f(\bar{x}_2, \bar{y}_2)\|\nonumber \\
    &\leq (L_{l,2}\|\bar{v}\| + L_{f,1})(\|\bar{x}_1 - \bar{x}_2\| + \|\bar{y}_1 - \bar{y}_2\|).\label{eq26}
\end{align}
Then 2) is proved. 

By \cref{Assumption2}, we can easily prove 3) as:
\begin{align}
\|\nabla_{{v}}r(\bar{x}, \bar{y}, \bar{v})\| \leq \|\nabla_{{y}}\nabla_{{y}}l(\bar{x}, \bar{y})\| \|\bar{v}\| + \|\nabla_{{y}}f(\bar{x}, \bar{y})\| \leq C_{l_{yy}} \|\bar{v}\| + C_{f_y}.\label{eq27}
\end{align}

Next, for \({v}^*(\bar{x}, \bar{y})\), we have:
\begin{align}
\nabla_{{v}}r(\bar{x}, \bar{y}, \hat{v}^*(\bar{x}, \bar{y})) = \nabla_{{y}}\nabla_{{y}}l(\bar{x}, \bar{y})\hat{v}^*(\bar{x}, \bar{y}) - \nabla_{{y}}f(\bar{x}, \bar{y}) = 0,\label{eq28}
\end{align}
which indicates that
\begin{align}
\|\hat{v}^*(\bar{x}, \bar{y})\| = \|\left(\nabla_{{y}}\nabla_{{y}}l(\bar{x}, \bar{y})\right)^{-1}\nabla_{{y}}f(\bar{x}, \bar{y})\| \leq \|\left(\nabla_{{y}}\nabla_{{y}}l(\bar{x}, \bar{y})\right)^{-1}\| \|\nabla_{{y}}f(\bar{x}, \bar{y})\| \leq \frac{C_{f_y}}{\mu}.\label{eq29}
\end{align}
Since \({v}^*(\bar{x})\) is a special case where \({v}^*(\bar{x}) = \hat{v}^*(\bar{x}, {y}^*(\bar{x}))\), 4) is proved. 

By \cref{Lemma2}, we have:
\begin{align}
&\|\bar{\nabla} f(\bar{x}_1, \bar{y}_1, \bar{v}_1) - \bar{\nabla} f(\bar{x}_2, \bar{y}_2, \bar{v}_2)\| \nonumber\\
&\leq \|\bar{\nabla} f(\bar{x}_1, \bar{y}_1, \bar{v}_1) - \bar{\nabla} f(\bar{x}_2, \bar{y}_2, \bar{v}_1)\| + \|\bar{\nabla} f(\bar{x}_2, \bar{y}_2, \bar{v}_1) - \bar{\nabla} f(\bar{x}_2, \bar{y}_2, \bar{v}_2)\| \nonumber \\
&\leq (L_{l,2} \|\bar{v}_1\| + L_{f,1})(\|\bar{x}_1 - \bar{x}_2\| + \|\bar{y}_1 - \bar{y}_2\|) + L_{l,1} \|\bar{v}_1 - \bar{v}_2\|\nonumber\\
&\leq\bar{L}_f(\|\bar{x}_1 - \bar{x}_2\| + \|\bar{y}_1 - \bar{y}_2\|+\|\bar{v}_1 - \bar{v}_2\|),\label{eqlf}
\end{align}
where $\bar{L}_f=\max\{\frac{C_{fy}L_{l,2}}{\mu} + L_{f,1}, L_{l,1}\}$. Thus, 5) is proved.

For the proof of 6), we have:
\begin{align}
    &\|\nabla_v r(\bar{x}_1, \bar{y}_1, \bar{v}_1) - \nabla_v r(\bar{x}_2, \bar{y}_2, \bar{v}_2)\| \notag \\
    &\!= \!\|\nabla_y\nabla_y l(\bar{x}_1, \bar{y}_1)\bar{v}_1 - \nabla_y f(\bar{x}_1, \bar{y}_1) 
    - (\nabla_y\nabla_y l(\bar{x}_2, \bar{y}_2)\bar{v}_2 - \nabla_y f(\bar{x}_2, \bar{y}_2))\| \notag \\
    &\!= \!\|(\nabla_y\nabla_y l(\bar{x}_1, \bar{y}_1) \!-\!\! \nabla_y\nabla_y l(\bar{x}_2, \bar{y}_2))\bar{v}_1 
    \!+\! \nabla_y\nabla_y l(\bar{x}_2, \bar{y}_2)(\bar{v}_1 \!-\! \bar{v}_2) 
    \!-\! (\nabla_y f(\bar{x}_1, \bar{y}_1) \!-\! \nabla_y f(\bar{x}_2, \bar{y}_2))\|\notag\\
    &\!\leq \!\|\nabla_y\nabla_y l(\bar{x}_1, \bar{y}_1) - \nabla_y\nabla_y l(\bar{x}_2, \bar{y}_2)\|  \|\bar{v}_1\|  + \|\nabla_y\nabla_y l(\bar{x}_2, \bar{y}_2)\| \|\bar{v}_1 - \bar{v}_2\| \notag\\
    &\quad+ \|\nabla_y f(\bar{x}_1, \bar{y}_1) - \nabla_y f(\bar{x}_2, \bar{y}_2)\|\notag\\
     &\overset{(a)}{\leq}\! L_{l,2} (\|\bar{x}_1 - \bar{x}_2\| + \|\bar{y}_1 - \bar{y}_2\|) \|\bar{v}_1\| + C_{l_{yy}}  \|\bar{v}_1 - \bar{v}_2\| + L_{f,1} (\|\bar{x}_1 - \bar{x}_2\| + \|\bar{y}_1 - \bar{y}_2\|)\notag\\
     &\!\overset{(b)}{\leq}\! \left( \frac{C_{f_y}L_{l,2}}{\mu} + L_{f,1}\right) (\|\bar{x}_1 - \bar{x}_2\| + \|\bar{y}_1 - \bar{y}_2\|) + C_{l_{yy}} \|\bar{v}_1 - \bar{v}_2\|\notag\\
     &\!\leq\! \bar{L}_r(\|\bar{x}_1 - \bar{x}_2\| + \|\bar{y}_1 - \bar{y}_2\|+\|\bar{v}_1 - \bar{v}_2\|),\label{eqvnew}
\end{align}
where (a) uses \cref{Assumption2}, (b) uses \cref{eq29} and $\bar{L}_r=\max\{\frac{C_{fy}L_{l,2}}{\mu} + L_{f,1}, C_{l_{yy}}\}$. Then the proof of 6) is finished.

The proof of the first part of 7) can refer to Lemma 4 in \citep{yang2024simfbo}; for the second part, we have:
\begin{align}
   & \big\|\hat{v}^*(\bar{x}, \bar{y}_1) - \hat{v}^*(\bar{x}, \bar{y}_2)\big\|\nonumber\\
   &= \big\|\left[\nabla_{{y}}\nabla_{{y}}l(\bar{x}, \bar{y}_1)\right]^{-1}\nabla_{{y}}f(\bar{x}, \bar{y}_1) - \left[\nabla_{{y}}\nabla_{{y}}l(\bar{x}, \bar{y}_2)\right]^{-1}\nabla_{{y}}f(\bar{x}, \bar{y}_2)\big\|\nonumber\\
    &\leq \big\|\left[\nabla_{{y}}\nabla_{{y}}l(\bar{x}, \bar{y}_1)\right]^{-1}\left(\nabla_{{y}}f(\bar{x}, \bar{y}_1) - \nabla_{{y}}f(\bar{x}, \bar{y}_2)\right)\big\|\nonumber \\
    &\quad + \big\|([\nabla_{{y}}\nabla_{{y}}l(\bar{x}, \bar{y}_1)]^{-1} - \left[\nabla_{{y}}\nabla_{{y}}l(\bar{x}, \bar{y}_2)\right]^{-1})\nabla_{{y}}f(\bar{x}, \bar{y}_2)\big\| \nonumber\\
    &\leq \frac{L_{f,1}}{\mu} \big\|\bar{y}_1 - \bar{y}_2\big\| + C_{f_y}\big\|\left[\nabla_{{y}}\nabla_{{y}}l(\bar{x}, \bar{y}_1)\right]^{-1}\left(\nabla_{{y}}\nabla_{{y}}l(\bar{x}, \bar{y}_2) - \nabla_{{y}}\nabla_{{y}}l(\bar{x}, \bar{y}_1)\right)\left[\nabla_{{y}}\nabla_{{y}}l(\bar{x}, \bar{y}_2)\right]^{-1}\!\!\big\| \nonumber\\
    &\leq \left(\frac{L_{f,1}}{\mu} + \frac{C_{f_y} L_{l,2}}{\mu^2}\right)\big\|\bar{y}_1 - \bar{y}_2\big\|.\label{eq30}
\end{align}
Thus, the second part of 7) is proved, and the proof of \cref{Lemma3} is complete.
\end{proof}

\begin{lemma}[\textbf{The Upper Bound of $\sum_{k=0}^{t}\|\mathbf{x}_k-\mathbf{1}\bar{x}_k\|^2$}]\label{Lemmanewx} Suppose \cref{Assumption3}, \cref{Assumption1}, and \cref{Assumption2} hold. Then, the consensus error for the primal variable satisfies:
\begin{align}
 \sum_{k=0}^{t} \|\mathbf{x}_k-\mathbf{1}\bar{x}_k\|^2 
&\leq \frac{2\Delta_x^0}{1 - \rho_W}  
+ \frac{8 \gamma_x^2 \rho_W (1 + \zeta_q^2)}{(1 - \rho_W)^2} \sum_{k=0}^{t} 
 \bar{q}_{k+1}^{-2} \|\bar{\nabla}F(\mathbf{x}_k, \mathbf{y}_k, \mathbf{v}_k)\|^2, \label{eqnewx39}
\end{align}
where \(\Delta_0^x=\|\mathbf{x}_0-\mathbf{1}\bar{x}_0\|^2\) is the initial consensus error for the primary variable $x$, which can be set to 0 with proper initialization.
\end{lemma}

\begin{proof}
By the updating rule of the primal variable, we have:
\begin{align}
&\|\mathbf{x}_{t+1} - \mathbf{1} \bar{x}_{t+1} \|^2\nonumber\\
&= \left\| W \left( \mathbf{x}_t - \gamma_x Q_{t+1}^{-1} \bar{\nabla} F(\mathbf{x}_t, \mathbf{y}_t, \mathbf{v}_t) \right) 
- \mathbf{J} \left( \mathbf{x}_t - \gamma_x Q_{t+1}^{-1} \bar{\nabla} F(\mathbf{x}_t, \mathbf{y}_t, \mathbf{v}_t) \right) \right\|^2 \nonumber \\
&\overset{(a)}{\leq} (1 + \lambda) \rho_W  \|\mathbf{x}_t - \mathbf{1} \bar{x}_t\|^2
+2 \left( 1 + \frac{1}{\lambda} \right)\rho_W  \gamma_x^2 \bar{q}_{t+1}^{-2} \|\bar{\nabla} F(\mathbf{x}_t, \mathbf{y}_t, \mathbf{v}_t) \|^2 \nonumber \\
&\quad + 2\left( 1 + \frac{1}{\lambda} \right)\rho_W  \gamma_x^2  
 \left\| \left(Q_{t+1}^{-1} - \bar{q}_{t+1}^{-1} \mathbf{I} \right) \bar{\nabla} F(\mathbf{x}_t, \mathbf{y}_t, \mathbf{v}_t) \right\|^2\nonumber\\
&\overset{(b)}{\leq} \frac{1 + \rho_W}{2}  \|\mathbf{x}_t - \mathbf{1} \bar{x}_t\|^2 
+ \frac{2 \gamma_x^2 (1 + \rho_W) \rho_W}{1 - \rho_W} \bar{q}_{t+1}^{-2} \|\bar{\nabla} F(\mathbf{x}_t, \mathbf{y}_t, \mathbf{v}_t) \|^2 \nonumber \\
&\quad + \frac{2 \gamma_x^2 (1 + \rho_W) \rho_W}{1 - \rho_W} 
 \left\| \left(Q_{t+1}^{-1} - \bar{q}_{t+1}^{-1} \mathbf{I} \right) \bar{\nabla} F(\mathbf{x}_t, \mathbf{y}_t, \mathbf{v}_t) \right\|^2, \label{eq40}
\end{align}
where (a) uses Young's inequality and we take $\lambda=\frac{1-\rho_W}{2\rho_W}$ in (b).

By the definition of \(\zeta_q^2\) in \cref{sec3.2}, we have:
\begin{align}
\left\| \left( Q_{t+1}^{-1} - \bar{q}_{t+1}^{-1} \mathbf{I} \right) \bar{\nabla} F(\mathbf{x}_t, \mathbf{y}_t, \mathbf{v}_t) \right\|^2  
\leq \zeta_q^2 \bar{q}_{t+1}^{-2} \|\bar{\nabla} F(\mathbf{x}_t, \mathbf{y}_t, \mathbf{v}_t)\|^2.\label{eq41}
\end{align}

Thus, summing over \(k = 0, \ldots, t-1\), we have:
\begin{align}
&\sum_{k=0}^{t-1} \|\mathbf{x}_{k+1} - \mathbf{1} \bar{x}_{k+1} \|^2 \nonumber\\
&\leq \sum_{k=0}^{t-1}\left(\frac{1 + \rho_W}{2} \right)^k \|\mathbf{x}_0 - \mathbf{1} \bar{x}_0\|^2 
+ \frac{4 \gamma_x^2 \rho_W(1 + \zeta_q^2)}{1 - \rho_W}\sum_{k=0}^{t-1} \left(\frac{1 + \rho_W}{2} \right)^k\bar{q}_{k+1}^{-2} \|\bar{\nabla} F(\mathbf{x}_k, \mathbf{y}_k, \mathbf{v}_k) \|^2\nonumber\\
&\leq \frac{2}{1 - \rho_W} \|\mathbf{x}_0 - \mathbf{1} \bar{x}_0\|^2 
+ \frac{8 \gamma_x^2 \rho_W (1 + \zeta_q^2)}{(1 - \rho_W)^2} \sum_{k=0}^{t-1} 
 \bar{q}_{k+1}^{-2} \|\bar{\nabla}F(\mathbf{x}_k, \mathbf{y}_k, \mathbf{v}_k)\|^2.\label{eq42}
\end{align}
Thus, we can get the result in \cref{eqnewx39}.
\end{proof}

\begin{lemma}\label{Lemma6}
Let variables \(x\), \(y\), and \(v\) be updated over \(T_1\), \(T_2\), and \(T_3\) iterations, respectively. Suppose that the sequences \(\{\bar{m}^x_{t}\}\), \(\{\bar{m}^y_{t}\}\), and \(\{\bar{m}^v_{t}\}\) are generated by \cref{alg1}. For any constants \(C_{m^x} \geq \bar{m}^x_{0}\), \(C_{m^y} \geq \bar{m}^y_{0}\), and \(C_{m^v} \geq \bar{m}^v_{0}\), the following statements hold:
\begin{enumerate}
    \item[(1)] Either \(\bar{m}^x_{t} \leq C_{m^x}\) for any \(t \leq T_1\), or \(\exists k_1 \leq T_1\) such that \(\bar{m}^x_{k_1} \leq C_{m^x}\), \(\bar{m}^x_{k_1+1} > C_{m^x}\).
    \item[(2)] Either \(\bar{m}^y_{t} \leq C_{m^y}\) for any \(t \leq T_2\), or \(\exists k_2 \leq T_2\) such that \(\bar{m}^y_{k_2} \leq C_{m^y}\), \(\bar{m}^y_{k_2+1} > C_{m^y}\).
    \item[(3)] Either \(\bar{m}^v_{t} \leq C_{m^v}\) for any \(t \leq T_3\), or \(\exists k_3 \leq T_3\) such that \(\bar{m}^v_{k_3} \leq C_{m^v}\), \(\bar{m}^v_{k_3+1} > C_{m^v}\).
\end{enumerate}
\end{lemma}

\begin{proof}
The proof is analogous to that of Lemma 4.1 in \citep{ward2020adagrad}. We will demonstrate the argument for part (1) concerning \(\bar{m}^x_{t}\); the proofs for parts (2) and (3) follow similarly.

Assume that \(\bar{m}^x_{T_1} > C_{m^x}\). Since \(C_{m^x} \geq \bar{m}^x_{0}\) and the sequence \(\{\bar{m}^x_{t}\}\) is monotonically increasing, there must exist an iteration \(k_1 \leq T_1\) where \(\bar{m}^x_{k_1} \leq C_{m^x}\) and \(\bar{m}^x_{k_1+1} > C_{m^x}\).
If no such \(k_1\) exists, then it must be that \(\bar{m}^x_{t} \leq C_{m^x}\) for all \(t \leq T_1\).
This completes the proof for part (1). 
\end{proof}

\section{Proofs of \cref{theorem 1} and \cref{collary1}}
\subsection{Notation}

We stack the gradient accumulators as: 
\begin{equation}
    \mathbf{m}^x_t := [m^x_{1,t}, m^x_{2,t}, \dots, m^x_{n,t}]^\top \in \mathbb{R}^{n \times p},\quad
    \mathbf{m}^y_t:= [m^y_{1,t}, m^y_{2,t}, \dots, m^y_{n,t}]^\top \in \mathbb{R}^{n \times q},\nonumber
    \end{equation}
    \begin{equation}
        \mathbf{m}^v_t := [m^v_{1,t}, m^v_{2,t}, \dots, m^v_{n,t}]^\top \in \mathbb{R}^{n \times q}.
    \end{equation}

Similarly, the gradient vectors are stacked as:
\begin{equation}
\begin{aligned}
    \!\bar{\nabla} {F}(\mathbf{x}_t, \mathbf{y}_t, \mathbf{v}_t) &:=
    \begin{bmatrix}
         \bar{\nabla}f_1(x_{1,t}, y_{1,t}, v_{1,t}),\bar{\nabla}f_2(x_{2,t}, y_{2,t}, v_{2,t}),\cdots, \bar{\nabla}f_n(x_{n,t}, y_{n,t}, v_{n,t})
    \end{bmatrix}^\top\!\!\!,\\
    \!\nabla_y {L}(\mathbf{x}_t, \mathbf{y}_t) &:=
    \begin{bmatrix}
        \nabla_yl_1(x_{1,t}, y_{1,t}),\nabla_yl_2(x_{2,t}, y_{2,t}),\cdots,  \nabla_yl_n(x_{n,t}, y_{n,t})
    \end{bmatrix}^\top\!\!\!,\\
    \!\nabla_v {R}(\mathbf{x}_t, \mathbf{y}_t, \mathbf{v}_t) &:=
    \begin{bmatrix}
        \nabla_v r_1(x_{1,t}, y_{1,t}, v_{1,t}),\nabla_v r_2(x_{2,t}, y_{2,t}, v_{2,t}),\cdots,\nabla_v r_n(x_{n,t}, y_{n,t}, v_{n,t})
    \end{bmatrix}^\top\!\!\!,\\
    \!{\nabla}_y {F}(\mathbf{x}_t, \mathbf{y}_t) &:=
    \begin{bmatrix}
         {\nabla}_yf_1(x_{1,t}, y_{1,t}),{\nabla}_yf_2(x_{2,t}, y_{2,t}),\cdots, {\nabla}_yf_n(x_{n,t}, y_{n,t})
    \end{bmatrix}^\top\!\!\!.
\end{aligned}
\end{equation}

For notational simplicity, we also use $\mathbf{g}_t^x$, $\mathbf{g}_t^y$, and $\mathbf{g}_t^v$ to denote $\bar{\nabla} F(\mathbf{x}_t, \mathbf{y}_t, \mathbf{v}_t)$, $\nabla_y L(\mathbf{x}_t, \mathbf{y}_t)$, and $\nabla_v R(\mathbf{x}_t, \mathbf{y}_t, \mathbf{v}_t)$, respectively.

The stepsize discrepancy vectors can be defined as:
\begin{equation}\label{eqquz}
\begin{aligned}
&\tilde{\mathbf{q}}_t^{-1} := \big[  q_{1,t}^{-1} - \bar{q}_t^{-1}, q_{2,t}^{-1} - \bar{q}_t^{-1},\cdots, q_{n,t}^{-1} - \bar{q}_t^{-1}\big]^\top,\\
&\tilde{\mathbf{u}}_t^{-1} := \big[u_{1,t}^{-1} - \bar{u}_t^{-1},u_{2,t}^{-1} - \bar{u}_t^{-1}, \cdots, u_{n,t}^{-1} - \bar{u}_t^{-1} \big]^\top,\\
&\tilde{\mathbf{z}}_t^{-1} := \big[ z_{1,t}^{-1} - \bar{z}_t^{-1},z_{2,t}^{-1} - \bar{z}_t^{-1},\cdots, z_{n,t}^{-1} - \bar{z}_t^{-1} \big]^\top.
\end{aligned}
\end{equation}

Additionally, for ease of presentation, we define the following notations:
\begin{equation}\label{eqnewlemmaproof1}
\begin{aligned}
\!\!\!\!\!\begin{cases}
\bar{x}_t := \frac{1}{n} \sum_{i=1}^n x_{i,t},\quad \bar{y}_t := \frac{1}{n} \sum_{i=1}^n y_{i,t}, \quad \bar{v}_t := \frac{1}{n} \sum_{i=1}^n v_{i,t},\\
\bar{q}_t := \frac{1}{n} \sum_{i=1}^n q_{i,t},\quad \bar{u}_t := \frac{1}{n} \sum_{i=1}^n u_{i,t}, \quad \bar{z}_t := \frac{1}{n} \sum_{i=1}^n z_{i,t}.
\end{cases}
\end{aligned}
\end{equation}

We define the following metrics to represent the level of stepsize inconsistency for the primal, dual, and auxiliary variables:
\begin{align}
\zeta_q^2 := \!\!\!\!\sup_{i \in [n], t > 0} \!\!\!\!
\frac{\big(q_{i,t}^{-1} - \bar{q}_t^{-1}\big)^2}{\big(\bar{q}_t^{-1}\big)^2},\quad
\zeta_u^2 := \!\!\!\!\sup_{i \in [n], t > 0} \!\!\!\!
\frac{\big(u_{i,t}^{-1} - \bar{u}_t^{-1}\big)^2}{\big(\bar{u}_t^{-1}\big)^2},\quad
\zeta_z^2 := \!\!\!\!\sup_{i \in [n], t > 0} \!\!\!\!
\frac{\big(z_{i,t}^{-1} - \bar{z}_t^{-1}\big)^2}{\big(\bar{z}_t^{-1}\big)^2},\notag\\
\sigma_q^2 := \!\!\!\!\inf_{i \in [n], t > 0} \!\!\!\!
\frac{\big(q_{i,t}^{-1} - \bar{q}_t^{-1}\big)^2}{\big(\bar{q}_t^{-1}\big)^2},\quad
\sigma_u^2 := \!\!\!\!\inf_{i \in [n], t > 0} \!\!\!\!
\frac{\big(u_{i,t}^{-1} - \bar{u}_t^{-1}\big)^2}{\big(\bar{u}_t^{-1}\big)^2},\quad
\sigma_z^2 := \!\!\!\!\inf_{i \in [n], t > 0} \!\!\!\!
\frac{\big(z_{i,t}^{-1} - \bar{z}_t^{-1}\big)^2}{\big(\bar{z}_t^{-1}\big)^2}.\label{eq20}
\end{align}

Below, we define several preset constants for notational convenience at their first use. We first define some Lipschitzness parameters for \(\Phi(\bar{x})\) as:
\begin{equation}
  L_\Phi := \left( L_{f,1} + \frac{L_{l,2} C_{f_y}}{\mu} \right) \left( 1 + \frac{C_{l_{xy}}}{\mu} \right)^2,  \label{eq50}
\end{equation}
\begin{equation}
    \bar{L} := \max \left\{ 2 \left( \frac{C_{f_y}^2 L_{l,2}^2}{\mu^2} + L_{f,1}^2 \right)^{\frac{1}{2}}, \, \sqrt{2} C_{l_{xy}} \right\}.\label{eq51}
\end{equation}

Next, we define the following constants as thresholds for parameters \(\bar{m}_{t}^x\), \(\bar{m}_{t}^y\), and \(\bar{m}_{t}^v\) as:
\begin{equation}
    C_{m^x} := \max \left\{ \frac{8n\gamma_x L_\Phi \left(1 +\zeta_q^2\right)}{\bar{z}_0}, \bar{m}^x_0\right\},\label{eq52}
\end{equation}
\begin{equation}
   C_{m^y} := \max \left\{ \frac{2\gamma_y(1+\zeta_u^2)(\mu + L_{l,1})}{\sqrt{1+\sigma_u^2}}, \frac{\gamma_y \mu L_{l,1} }{\mu + L_{l,1}}, \bar{m}^y_0, 64 a_0^2, 1 \right\}, \label{eq53}
\end{equation}
\begin{equation}
   C_{m^v} := \max \left\{ \frac{2\gamma_v(1+\zeta_z^2)(\mu + C_{l_{yy}})}{\sqrt{1+\sigma_z^2}}, \frac{2\mu C_{l_{yy}}\gamma_v}{\mu + C_{l_{yy}}}, \bar{m}^v_0, 64 a_0^2, 1, C_{l_{yy}} \right\}, \label{eq54}
\end{equation}
\begin{equation}
    C_{z} := C_{m^y} + C_{m^v}.\label{eq55}
\end{equation}

The constant \(a_0\) is defined as:
\begin{align}
a_0 &:= \left( \frac{1}{\mu^2}\left(8+\frac{4(\mu + C_{l_{yy}})^2}{\mu C_{l_{yy}}} \right)\left(\frac{nL_{l,2}C_{f_y}}{\mu}+L_{f,1}\right)^2+1\right)  \notag \\
&\quad \cdot\left(\frac{32 \gamma_x^2 \rho_W(1 + \zeta_q^2)}{n(1 - \rho_W)^2} \left(\frac{(\mu+L_{l,1})^3}{\gamma_y\mu^2}+\frac{2L_{l,1}(\mu+L_{l,1})}{\gamma_y\mu}\right)+\frac{8\gamma_x^2  L_y^2\left( 1 + \zeta_q^2 \right) (\mu + L_{l,1})^2}{n\gamma_y^2\mu L_{l,1}\sqrt{1+\sigma_u^2}}\right)\nonumber\\
&\quad +\frac{16(\mu + C_{l_{yy}})}{n\gamma_v\mu^2\sqrt{1+\sigma_u^2}}\left(\frac{nL_{l,2}C_{f_y}}{\mu}\!+\!L_{f,1}\right)^2\nonumber\\
&\quad\cdot\left(\frac{4 \gamma_x^2 \rho_W (1 + \zeta_q^2)}{\bar{z}_0^2(1 - \rho_W)^2} \left(\frac{(\mu+L_{l,1})^2\sqrt{1+\sigma_u^2}}{n\mu^2 }+\frac{2L_{l,1}\sqrt{1+\sigma_u^2}}{n\mu }\right)\right.\nonumber\\
&\quad +\!\left.\frac{\gamma_x^2L_y^2(1+\zeta_q^2)(\mu + L_{l,1})}{n\bar{z}_0\gamma_y \mu L_{l,1}}\!\right)
\!+ \!\frac{64 \gamma_x^2 \rho_W \bar{L}_r^2(\mu + C_{l_{yy}})(1 + \zeta_q^2)}{n\mu C_{l_{yy}}\gamma_v(1 - \rho_W)^2}\! +\! \frac{2\gamma_x^2L_v^2(1+\zeta_q^2)(\mu + C_{l_{yy}})^2}{n\mu C_{l_{yy}}\bar{m}_0^v\gamma_v^2\sqrt{1+\sigma_z^2}}.\label{eq56}
\end{align}

\subsection{Descent in Objective Function}\label{secc2}
\begin{lemma}[\textbf{Descent Lemma}]\label{Lemma4} Suppose \cref{Assumption3}, \cref{Assumption1}, and \cref{Assumption2} hold. Then, no matter $k_1$ in \cref{Lemma6} exists or not, we always have:
\begin{align}
&\Phi(\bar{x}_{t+1})\nonumber\\
&\leq\Phi(\bar{x}_t)-\frac{\gamma_x \bar{q}_{t+1}^{-1}}{8}\|\nabla \Phi(\bar{x}_t)\|^2+\frac{\gamma_x\bar{q}_{t+1}^{-1}(2\gamma_x \bar{L}_{f}^2L_\Phi \bar{q}_{t+1}^{-1} \left(1 +\zeta_q^2\right)+\bar{L}_{f}^2)}{n} \Delta_t\nonumber\\
&\quad -\left(\frac{\gamma_x}{2}-{2n\gamma_x^2 L_\Phi \bar{q}_{t+1}^{-1} \left(1 +\zeta_q^2\right)}\right)\frac{\|\nabla_x f(\bar{x}_t, \bar{y}_t, \bar{v}_t)\|^2}{\bar{q}_{t+1}}
+ \frac{\bar{L}^2\gamma_x }{\mu^2} \frac{\|\nabla_v r(\bar{x}_t, \bar{y}_t, \bar{v}_t)\|^2}{\bar{q}_{t+1}} \nonumber \\
&\quad +\left( \frac{\gamma_x\bar{L}^2}{2\mu^2} \!\!+\!\! \frac{\gamma_x\bar{L}^2}{\mu^4} \left( \frac{L_{l,2} C_{f_y}}{\mu} \!\!+ \!\!L_{f,1} \right)^2 \right) \frac{\|\nabla_y l(\bar{x}_t, \bar{y}_t)\|^2}{\bar{q}_{t+1}}\!+\! 2\gamma_x \bar{q}_{t+1}^{-1}  \left\| \frac{\big(\tilde{\mathbf{q}}_{t+1}^{-1}\big)^\top}{n\bar{q}_{t+1}^{-1}} \bar{\nabla} F(\mathbf{x}_t, \mathbf{y}_t, \mathbf{v}_t) \right\|^2\!.\label{eqnew1}
\end{align}
Additionally, if $k_1$ in \cref{Lemma6} exists, we have:
\begin{align}
&\Phi(\bar{x}_{t+1})\nonumber\\
&\leq\Phi(\bar{x}_t)-\frac{\gamma_x \bar{q}_{t+1}^{-1}}{8}\|\nabla \Phi(\bar{x}_t)\|^2+\frac{\gamma_x\bar{q}_{t+1}^{-1}(2\gamma_x \bar{L}_{f}^2L_\Phi \bar{q}_{t+1}^{-1} \left(1 +\zeta_q^2\right)+\bar{L}_{f}^2)}{n} \Delta_t\nonumber\\
&\quad -\frac{\gamma_x }{4}\frac{\|\nabla_x f(\bar{x}_t, \bar{y}_t, \bar{v}_t)\|^2}{\bar{q}_{t+1}}
+ \frac{\bar{L}^2\gamma_x }{\mu^2} \frac{\|\nabla_v r(\bar{x}_t, \bar{y}_t, \bar{v}_t)\|^2}{\bar{q}_{t+1}}+ 2\gamma_x \bar{q}_{t+1}^{-1}  \left\| \frac{\big(\tilde{\mathbf{q}}_{t+1}^{-1}\big)^\top}{n\bar{q}_{t+1}^{-1}} \bar{\nabla} F(\mathbf{x}_t, \mathbf{y}_t, \mathbf{v}_t) \right\|^2\nonumber\\
&\quad+\left(\! \frac{\gamma_x\bar{L}^2}{2\mu^2} + \frac{\gamma_x\bar{L}^2}{\mu^4} \!\!\left(\! \frac{L_{l,2} C_{f_y}}{\mu} + L_{f,1} \!\right)^2 \right)\!\! \frac{\|\nabla_y l(\bar{x}_t, \bar{y}_t)\|^2}{\bar{q}_{t+1}},\label{eqnew2}
\end{align}
where $\bar{L} := \max \left\{ 2 \left( \frac{C_{f_y}^2 L_{l,2}^2}{\mu^2} + L_{f,1}^2 \right)^{\frac{1}{2}}, \, \sqrt{2} C_{l_{xy}} \right\}$.
\end{lemma}

\begin{proof}
By the smoothness of $\Phi$, we have:
\begin{equation}
    \Phi(\bar{x}_{t+1}) - \Phi(\bar{x}_t) \leq \langle \nabla \Phi(\bar{x}_t), \bar{x}_{t+1} - \bar{x}_t \rangle 
+ \frac{L_\Phi}{2} \|\bar{x}_{t+1} - \bar{x}_t\|^2.\label{eq32}
\end{equation}
Noticing that the scalars $\bar{q}_t, \bar{u}_t$, and $\bar{z}_t$ are random variables, we then have:
\begin{align}
\frac{\Phi(\bar{x}_{t+1}) - \Phi(\bar{x}_t)}{\gamma_x \bar{q}_{t+1}^{-1}}
&\leq -\Bigg\langle \nabla \Phi(\bar{x}_t), \frac{\mathbf{1}^\top}{n} \bar{\nabla} F(\mathbf{x}_t, \mathbf{y}_t, \mathbf{v}_t) \Bigg\rangle - \Bigg\langle \nabla \Phi(\bar{x}_t), \frac{\big(\tilde{\mathbf{q}}_{t+1}^{-1}\big)^\top}{n\bar{q}_{t+1}^{-1}} \bar{\nabla} F(\mathbf{x}_t, \mathbf{y}_t, \mathbf{v}_t) \Bigg\rangle \nonumber \\
&\quad + \frac{\gamma_x L_\Phi}{2\bar{q}_{t+1}^{-1}}  \bigg\|\bigg(\frac{\bar{q}_{t+1}^{-1}\mathbf{1}^\top}{n} 
+ \frac{\big(\tilde{\mathbf{q}}_{t+1}^{-1}\big)^\top}{n} \bigg)\bar{\nabla} F(\mathbf{x}_t, \mathbf{y}_t, \mathbf{v}_t)\bigg\|^2.\label{eq33}
\end{align}
Then, we bound the inner-product terms on the right-hand side (RHS). Firstly,
\begin{align}
&-\Bigg\langle \nabla \Phi(\bar{x}_t), \frac{\mathbf{1}^\top}{n} \bar{\nabla} F(\mathbf{x}_t, \mathbf{y}_t, \mathbf{v}_t) \Bigg\rangle \nonumber \\
&= -\Bigg\langle \nabla \Phi(\bar{x}_t), \frac{\mathbf{1}^\top}{n} \bar{\nabla} F(\mathbf{x}_t, \mathbf{y}_t, \mathbf{v}_t) - \frac{\mathbf{1}^\top}{n}\bar{\nabla} F(\mathbf{1}\bar{x}_t, \mathbf{1}\bar{y}_t, \mathbf{1}\bar{v}_t) + \frac{\mathbf{1}^\top}{n}\bar{\nabla} F(\mathbf{1}\bar{x}_t, \mathbf{1}\bar{y}_t, \mathbf{1}\bar{v}_t) \Bigg\rangle \nonumber \\
&\leq \frac{1}{4}\|\nabla \Phi(\bar{x}_t)\|^2
+ \bigg\|\frac{\mathbf{1}^\top}{n} \bar{\nabla} F(\mathbf{x}_t, \mathbf{y}_t, \mathbf{v}_t) - \frac{\mathbf{1}^\top}{n}\bar{\nabla} F(\mathbf{1}\bar{x}_t, \mathbf{1}\bar{y}_t, \mathbf{1}\bar{v}_t)\bigg\|^2 \nonumber \\
&\quad + \frac{1}{2}\bigg(\|\nabla \Phi(\bar{x}_t) -\bar{\nabla} f(\bar{x}_t, \bar{y}_t, \bar{v}_t)\|^2
- \|\nabla \Phi(\bar{x}_t)\|^2 - \|\bar{\nabla} f(\bar{x}_t, \bar{y}_t, \bar{v}_t)\|^2\bigg) \nonumber \\
&\leq -\frac{1}{4}\|\nabla \Phi(\bar{x}_t)\|^2
+ \bigg\|\frac{\mathbf{1}^\top}{n} \bar{\nabla} F(\mathbf{x}_t, \mathbf{y}_t, \mathbf{v}_t) - \frac{\mathbf{1}^\top}{n}\bar{\nabla} F(\mathbf{1}\bar{x}_t, \mathbf{1}\bar{y}_t, \mathbf{1}\bar{v}_t)\bigg\|^2 \nonumber\\
&+ \frac{1}{2}\|\nabla \Phi(\bar{x}_t) -\bar{\nabla} f(\bar{x}_t, \bar{y}_t, \bar{v}_t)\|^2- \frac{1}{2}\|\nabla_x f(\bar{x}_t, \bar{y}_t,\bar{v}_t)\|^2.\label{eq34}
\end{align}
Additionally, the gradient approximation error satisfies:
\begin{align}
&\|\nabla \Phi(\bar{x}_t) -\bar{\nabla} f(\bar{x}_t, \bar{y}_t, \bar{v}_t)\|^2 \nonumber\\
&= \|\bar{\nabla} f(\bar{x}_t, y^*(\bar{x}_t), v^*(\bar{x}_t)) - \bar{\nabla} f(\bar{x}_t, \bar{y}_t, \bar{v}_t)\|^2 \nonumber\\
&\leq 2 \|\bar{\nabla} f(\bar{x}_t, y^*(\bar{x}_t), v^*(\bar{x}_t)) - \bar{\nabla} f(\bar{x}_t, \bar{y}_t, v^*(\bar{x}_t)\|^2 + 2 \|\bar{\nabla} f(\bar{x}_t, \bar{y}_t, v^*(\bar{x}_t) - \bar{\nabla} f(\bar{x}_t, \bar{y}_t, \bar{v}_t)\|^2 \nonumber\\
&\leq 4 \|\nabla_x \nabla_y l(\bar{x}_t, y^*(\bar{x}_t)) v^*(\bar{x}_t) - \nabla_x \nabla_y l(\bar{x}_t, \bar{y}_t) v^*(\bar{x}_t)\|^2 \nonumber\\
&\quad + 4 \|\nabla_x f(\bar{x}_t, y^*(\bar{x}_t)) - \nabla_x f(\bar{x}_t, \bar{y}_t)\|^2 
+ 2 \|\nabla_x \nabla_y l(\bar{x}_t, \bar{y}_t) (v^*(\bar{x}_t) - \bar{v}_t)\|^2 \nonumber\\
&\overset{(a)}{\leq} 4 \left( \frac{C_{f_y}^2 L_{l,2}^2}{\mu^2} + L_{f,1}^2 \right) \|\bar{y}_t - y^*(\bar{x}_t)\|^2 
+ 2 C_{l_{xy}}^2 \|\bar{v}_t - v^*(\bar{x}_t)\|^2 \nonumber\\
&\leq \bar{L}^2 \left( \|\bar{y}_t - y^*(\bar{x}_t)\|^2 + \|\bar{v}_t - v^*(\bar{x}_t)\|^2 \right)\nonumber\\
&\overset{(b)}{\leq} \frac{\bar{L}^2}{\mu^2} \left\| \nabla_y l(\bar{x}_t, \bar{y}_t) - \nabla_y l(\bar{x}_t, y^*(\bar{x}_t)) \right\|^2 
+ \frac{\bar{L}^2}{\mu^2} \left\| \nabla_v r(\bar{x}_t, \bar{y}_t, \bar{v}_t) - \nabla_v r(\bar{x}_t, \bar{y}_t, v^*(\bar{x}_t)) \right\|^2 \nonumber \\
&\overset{(c)}{\leq} \frac{\bar{L}^2}{\mu^2} \|\nabla_y l(\bar{x}_t, \bar{y}_t)\|^2 
+ \frac{2\bar{L}^2}{\mu^2} \|\nabla_v r(\bar{x}_t, \bar{y}_t, \bar{v}_t)\|^2 \notag\\
&\quad+ \frac{2\bar{L}^2}{\mu^2} \left\| \nabla_v r(\bar{x}_t, \bar{y}_t, v^*(\bar{x}_t)) - \nabla_v r(\bar{x}_t, y^*(\bar{x}_t), v^*(\bar{x}_t)) \right\|^2 \nonumber \\
&\overset{(d)}{\leq} \frac{\bar{L}^2}{\mu^2} \|\nabla_y l(\bar{x}_t, \bar{y}_t)\|^2 
+ \frac{2\bar{L}^2}{\mu^2} \|\nabla_v r(\bar{x}_t, \bar{y}_t, \bar{v}_t)\|^2 
+ \frac{2\bar{L}^2}{\mu^2} \left( \frac{L_{l,2} C_{f_y}}{\mu} + L_{f,1} \right)^2 \|\bar{y}_t - y^*(\bar{x}_t)\|^2 \nonumber \\
&\overset{(e)}{\leq} \left( \frac{\bar{L}^2}{\mu^2} + \frac{2\bar{L}^2}{\mu^4} \left( \frac{L_{l,2} C_{f_y}}{\mu} + L_{f,1} \right)^2 \right) \|\nabla_y l(\bar{x}_t, \bar{y}_t)\|^2 
+ \frac{2\bar{L}^2}{\mu^2} \|\nabla_v r(\bar{x}_t, \bar{y}_t, \bar{v}_t)\|^2,\label{eq35}
\end{align}
where (a) is using 4) in \cref{Lemma3} and $\bar{L} := \max \left\{ 2 \left( \frac{C_{f_y}^2 L_{l,2}^2}{\mu^2} + L_{f,1}^2 \right)^{\frac{1}{2}}, \, \sqrt{2} C_{l_{xy}} \right\}$; (b) and (e) use the strong convexity; (c) and (e) result from \(\nabla_y l(\bar{x}, y^*(\bar{x})) = 0\) and \(\nabla_v r(\bar{x}, y^*(\bar{x}), v^*(\bar{x})) = 0\); (d) uses \cref{Lemma3}. Then substituting \cref{eq35} to \cref{eq34}, we have:
\begin{align}
&-\Bigg\langle \nabla \Phi(\bar{x}_t), \frac{\mathbf{1}^\top}{n} \bar{\nabla} F(\mathbf{x}_t, \mathbf{y}_t, \mathbf{v}_t) \Bigg\rangle \nonumber \\
&\overset{(a)}{\leq}-\frac{1}{4}\|\nabla \Phi(\bar{x}_t)\|^2+\frac{\bar{L}_{f}^2}{n}\Delta_t+ \left( \frac{\bar{L}^2}{2\mu^2} + \frac{\bar{L}^2}{\mu^4} \left( \frac{L_{l,2} C_{f_y}}{\mu} + L_{f,1} \right)^2 \right) \|\nabla_y l(\bar{x}_t, \bar{y}_t)\|^2 
\nonumber\\
&\quad+ \frac{\bar{L}^2}{\mu^2} \|\nabla_v r(\bar{x}_t, \bar{y}_t, \bar{v}_t)\|^2-\frac{1}{2}\|\nabla_x f(\bar{x}_t, \bar{y}_t, \bar{v}_t)\|^2,\label{eq36}
\end{align}
where (a) uses \cref{Lemma3} and $\Delta_t := \|\mathbf{x}_t - \mathbf{1} \bar{x}_t\|^2 + \|\mathbf{y}_t - \mathbf{1} \bar{y}_t\|^2+ \|\mathbf{v}_t - \mathbf{1} \bar{v}_t\|^2$ is the consensus error for primal, dual, and auxiliary variables.

For the second inner-product in \cref{eq33}, using Young’s inequality, we get:
\begin{align}
&- \Bigg\langle \nabla \Phi(\bar{x}_t), \frac{\big(\tilde{\mathbf{q}}_{t+1}^{-1}\big)^\top}{n\bar{q}_{t+1}^{-1}} \bar{\nabla} F(\mathbf{x}_t, \mathbf{y}_t, \mathbf{v}_t) \Bigg\rangle  \leq \frac{1}{8}\|\nabla \Phi(\bar{x}_t)\|^2 
+ 2  \left\| \frac{\big(\tilde{\mathbf{q}}_{t+1}^{-1}\big)^\top}{n\bar{q}_{t+1}^{-1}} \bar{\nabla} F(\mathbf{x}_t, \mathbf{y}_t, \mathbf{v}_t) \right\|^2. \label{eq37}
\end{align}

For the last term on the RHS of \cref{eq33}, recalling the definition of stepsize inconsistency in \cref{eq20}, we have:
\begin{align}
&\frac{\gamma_x L_\Phi}{2\bar{q}_{t+1}^{-1}}  \bigg\|\bigg(\frac{\bar{q}_{t+1}^{-1}\mathbf{1}^\top}{n} 
+ \frac{\big(\tilde{\mathbf{q}}_{t+1}^{-1}\big)^\top}{n} \bigg)\bar{\nabla} F(\mathbf{x}_t, \mathbf{y}_t, \mathbf{v}_t)\bigg\|^2 \nonumber\\
&\leq \frac{\gamma_x L_\Phi \bar{q}_{t+1}^{-1} \left(1 +\zeta_q^2\right)}{n} \|\bar{\nabla} F(\mathbf{x}_t, \mathbf{y}_t, \mathbf{v}_t)\|^2\nonumber\\
&\overset{(a)}{\leq} {2n\gamma_x L_\Phi \bar{q}_{t+1}^{-1} \left(1 +\zeta_q^2\right)} \|\bar{\nabla} f(\bar{x}_t, \bar{y}_t, \bar{v}_t)\|^2+\frac{2\gamma_x \bar{L}_{f}^2 L_\Phi\bar{q}_{t+1}^{-1} \left(1 +\zeta_q^2\right)}{n}\Delta_t,\label{eq38} 
\end{align}
where (a) uses $ \|\bar{\nabla} F(\mathbf{x}_t, \mathbf{y}_t, \mathbf{v}_t)\|^2\leq 2\|\bar{\nabla} F(\mathbf{1}\bar{x}_t, \mathbf{1}\bar{y}_t, \mathbf{1}\bar{v}_t)\|^2 + 2\bar{L}_{f}^2 (\|\mathbf{x}_t - \mathbf{1}\bar{x}_t\|^2 + \|\mathbf{y}_t - \mathbf{1}\bar{y}_t\|^2+ \|\mathbf{v}_t - \mathbf{1}\bar{v}_t\|^2)$, $\|\bar{\nabla} F(\mathbf{1}\bar{x}_t, \mathbf{1}\bar{y}_t, \mathbf{1}\bar{v}_t)\|^2 \leq\|\bar{\nabla} F(\mathbf{1}\bar{x}_t, \mathbf{1}\bar{y}_t, \mathbf{1}\bar{v}_t)\|_{\rm F}^2=\|n\bar{\nabla} f(\bar{x}_t,\bar{y}_t,\bar{v}_t)\|^2$, and \cref{Lemma3}. By plugging \cref{eq36}, \cref{eq37}, and \cref{eq38} into \cref{eq33}, we obtain \cref{eqnew1}. Moreover, if $k_1$ in \cref{Lemma6} exists, then for $t\geq k_1$, we have $\bar{m}^x>C_{m^x}>\frac{8n\gamma_x L_\Phi \left(1 +\zeta_q^2\right)}{\bar{z}_0}$. Therefore, from \cref{eqnew1} we can get \cref{eqnew2}. 
\end{proof}

\subsection{The Upper Bound of $\sum_{k=k_2}^t \frac{ \|\nabla_y L(\mathbf{x}_k, \mathbf{y}_k)\|^2}{\bar{u}_{k+1}}$}
\begin{lemma}\label{Lemma7} Under \cref{Assumption1} and \cref{Assumption2}, for \cref{alg1}, suppose the total iteration rounds is \(T\). If \(k_2\) in \cref{Lemma6} exists within \(T\) iterations, for all integers \(t \in [k_2, T]\), we have:
\begin{align}
    &\sum_{k=k_2}^t \frac{ \|\nabla_y L(\mathbf{x}_k, \mathbf{y}_k)\|^2}{\bar{u}_{k+1}}\nonumber\\
    &\leq\frac{ 2nC_{m^y}^2(\mu + L_{l,1})}{\mu^2\gamma_y\sqrt{1+\sigma_u^2}} +\left(\frac{8\Delta_0^x}{1-\rho_W}+\frac{32 \gamma_x^2 \rho_WC_{m^x}^2 (1 + \zeta_q^2)}{\bar{q}_0^2(1 - \rho_W)^2} \right)\left(\frac{(\mu+L_{l,1})^3}{\gamma_y\mu^2}+\frac{2L_{l,1}(\mu+L_{l,1})}{\gamma_y\mu}\right)\nonumber\\
&\quad+ \frac{32 \gamma_x^2 \rho_W(1 + \zeta_q^2)}{\bar{z}_0(1 - \rho_W)^2} \left(\frac{(\mu+L_{l,1})^3}{\gamma_y\mu^2}+\frac{2L_{l,1}(\mu+L_{l,1})}{\gamma_y\mu}\right)+\frac{8\gamma_x^2  L_y^2\left( 1 + \zeta_q^2 \right) (\mu + L_{l,1})^2}{\gamma_y^2\mu L_{l,1}\bar{z}_0\sqrt{1+\sigma_u^2}} \nonumber\\
&\quad+ \left[\frac{32 \gamma_x^2 \rho_W(1 + \zeta_q^2)}{(1 - \rho_W)^2} \left(\frac{(\mu+L_{l,1})^3}{\gamma_y\mu^2}+\frac{2L_{l,1}(\mu+L_{l,1})}{\gamma_y\mu}\right)\right.\nonumber\\
&\quad\left.+\frac{8\gamma_x^2  L_y^2\left( 1 + \zeta_q^2 \right) (\mu + L_{l,1})^2}{\gamma_y^2\mu L_{l,1}\sqrt{1+\sigma_u^2}}\right]\!\! \sum_{k=\min\{k_1,k_2\}}^t \frac{\|\bar{\nabla} F(\mathbf{x}_k,\mathbf{y}_k,\mathbf{v}_k)\|^2}{[\bar{m}^x_{k+1}]^2\max \big\{\bar{m}^v_{k+1}, \bar{m}^y_{k+1}\big\}}.\label{eq58}
\end{align}
\end{lemma}

\begin{proof}
For \(k_2 \leq t < T\), we have \(\bar{m}^y_{k_2} \leq C_{m^y}\) and \(\bar{m}^y_{t+1} > C_{m^y}\). For any positive scalar \(\bar{\lambda}_{t+1}\), using Young’s inequality, we have:
\begin{equation}
   \frac{1}{n}\|\mathbf{y}_{t+1} - \mathbf{1}y^*(\bar{x}_{t+1})\|^2 
\leq \frac{(1 + \bar{\lambda}_{t+1})}{n} \|\mathbf{y}_{t+1} - \mathbf{1}y^*(\bar{x}_t)\|^2 
+ \left( 1 + \frac{1}{\bar{\lambda}_{t+1}} \right) \|y^*(\bar{x}_t) - y^*(\bar{x}_{t+1})\|^2. \label{eq62}
\end{equation}

For the first term on the RHS of \cref{eq62}, we have:
\begin{align}
&\frac{1}{n}\|\mathbf{y}_{t+1} - \mathbf{1}y^*(\bar{x}_t)\|^2 \nonumber\\
&= \frac{1}{n}\left\|\mathbf{y}_{t} - \gamma_y U^{-1}_{t+1} \nabla_y L(\mathbf{x}_t, \mathbf{y}_t) - \mathbf{1}y^*(\bar{x}_t) \right\|^2 \nonumber \\
&= \frac{1}{n}\|\mathbf{y}_{t} - \mathbf{1}y^*(\bar{x}_t))\|^2 
+ \frac{\gamma_y^2}{n} \|U^{-1}_{t+1} \nabla_y L(\mathbf{x}_t, \mathbf{y}_t)\|^2 
- \frac{1}{n}\sum_{i=1}^n 2\bar{u}_{t+1}^{-1} \gamma_y\langle \nabla_y l_i(x_{i,t}, y_{i,t}), y_{i,t}-y^*(\bar{x}_t) \rangle \nonumber \\
&\quad-\frac{1}{n}\sum_{i=1}^n 2\bar{u}_{t+1}^{-1} \gamma_y\Bigg\langle \left(\frac{u_{i,t+1}^{-1}-\bar{u}_{t+1}^{-1}}{\bar{u}_{t+1}^{-1}}\right)\nabla_y l_i(x_{i,t}, y_{i,t}), y_{i,t}-y^*(\bar{x}_t)\Bigg\rangle \nonumber\\
&\leq \frac{1}{n}\|\mathbf{y}_{t} - \mathbf{1}y^*(\bar{x}_t))\|^2 
+ \frac{\gamma_y^2}{n} \|U^{-1}_{t+1} \nabla_y L(\mathbf{x}_t, \mathbf{y}_t)\|^2 \nonumber\\
&\quad- \frac{1}{n}\sum_{i=1}^n 2\bar{u}_{t+1}^{-1} \gamma_y\langle \nabla_y l_i(x_{i,t}, y_{i,t})- \nabla_y l_i(x_{i,t}, y^*(\bar{x}_t)), y_{i,t}-y^*(\bar{x}_t) \rangle\nonumber\\
&\quad - \frac{1}{n}\sum_{i=1}^n 2\bar{u}_{t+1}^{-1} \gamma_y\langle \nabla_y l_i(x_{i,t}, y^*(\bar{x}_t))-\nabla_y l_i(\bar{x}_{t}, y^*(\bar{x}_t)), y_{i,t}-y^*(\bar{x}_t) \rangle \nonumber \\
&\quad-\frac{1}{n}\sum_{i=1}^n 2\bar{u}_{t+1}^{-1} \gamma_y\Bigg\langle \left(\frac{u_{i,t+1}^{-1}-\bar{u}_{t+1}^{-1}}{\bar{u}_{t+1}^{-1}}\right)(\nabla_y l_i(x_{i,t}, y_{i,t})- \nabla_y l_i(x_{i,t}, y^*(\bar{x}_t))), y_{i,t}-y^*(\bar{x}_t)\Bigg\rangle \nonumber\\
&\quad -\frac{1}{n}\sum_{i=1}^n 2\bar{u}_{t+1}^{-1} \gamma_y\Bigg\langle \left(\frac{u_{i,t+1}^{-1}-\bar{u}_{t+1}^{-1}}{\bar{u}_{t+1}^{-1}}\right)\left(\nabla_y l_i(x_{i,t}, y^*(\bar{x}_t))-\nabla_y l_i(\bar{x}_{t}, y^*(\bar{x}_t))\right), y_{i,t}-y^*(\bar{x}_t)\Bigg\rangle \nonumber\\
&\overset{(a)}{\leq} \left( \frac{1}{n} - \frac{\gamma_y\bar{u}^{-1}_{t+1} \mu L_{l,1}}{n(\mu + L_{l,1})} \right) \|\mathbf{y}_{t} - \mathbf{1}y^*(\bar{x}_t))\|^2+ \frac{\gamma_y^2\bar{u}^{-2}_{t+1}(1+\zeta_u^2)}{n}\|\nabla_y L(\mathbf{x}_t, \mathbf{y}_t)\|^2 \nonumber\\
&\quad- \frac{2\gamma_y\bar{u}_{t+1}^{-1}\sqrt{1+\sigma_u^2}}{n(\mu + L_{l,1})} \|\nabla_y L(\mathbf{x}_t, \mathbf{y}_t)\pm\nabla_yL(\mathbf{1}\bar{x}_t, \mathbf{1}{y}^*(\bar{x}_t))-\nabla_y L(\mathbf{x}_t, \mathbf{1}{y}^*(\bar{x}_t))\|^2 \nonumber \\
&\quad+ \frac{\gamma_y\bar{u}_{t+1}^{-1}(\mu+L_{l,1})\sqrt{1+\sigma_u^2}}{n\mu L_{l,1}}\|\nabla_y L(\mathbf{x}_t, \mathbf{1}{y}^*(\bar{x}_t))-\nabla_yL(\mathbf{1}\bar{x}_t, \mathbf{1}{y}^*(\bar{x}_t))\|^2 \nonumber \\
&\overset{(b)}{\leq} \left( \frac{1}{n} - \frac{\gamma_y\bar{u}^{-1}_{t+1} \mu L_{l,1}}{n(\mu + L_{l,1})} \right) \|\mathbf{y}_{t} - \mathbf{1}y^*(\bar{x}_t))\|^2\notag\\
&\quad+ \frac{\gamma_y\bar{u}^{-1}_{t+1}}{n} \left( \gamma_y\bar{u}^{-1}_{t+1}(1+\zeta_u^2) - \frac{\sqrt{1+\sigma_u^2}}{\mu + L_{l,1}} \right) \|\nabla_y L(\mathbf{x}_t, \mathbf{y}_t)\|^2\nonumber\\
&\quad+\! \left(\frac{\gamma_y\bar{u}_{t+1}^{-1}(\mu+L_{l,1})\sqrt{1+\sigma_u^2}}{n\mu L_{l,1}}\!\!+\!\!\frac{2\gamma_y\bar{u}_{t+1}^{-1}\sqrt{1+\sigma_u^2}}{n(\mu + L_{l,1})}\right)\|\nabla_y L(\mathbf{x}_t, \mathbf{1}{y}^*(\bar{x}_t))\!-\!\!\nabla_yL(\mathbf{1}\bar{x}_t, \mathbf{1}{y}^*(\bar{x}_t))\|^2 \nonumber \\
&\overset{(c)}{\leq} \left( \frac{1}{n} - \frac{\gamma_y\bar{u}^{-1}_{t+1} \mu L_{l,1}}{n(\mu + L_{l,1})} \right) \|\mathbf{y}_{t} - \mathbf{1}y^*(\bar{x}_t))\|^2\notag\\
&\quad+ \frac{\gamma_y\bar{u}^{-1}_{t+1}}{n} \left( \gamma_y\bar{u}^{-1}_{t+1}(1+\zeta_u^2) - \frac{\sqrt{1+\sigma_u^2}}{\mu + L_{l,1}} \right) \|\nabla_y L(\mathbf{x}_t, \mathbf{y}_t)\|^2 \nonumber\\
&\quad + \left(\frac{\gamma_y\bar{u}_{t+1}^{-1}L_{l,1}(\mu+L_{l,1})\sqrt{1+\sigma_u^2}}{n\mu }+\frac{2\gamma_y\bar{u}_{t+1}^{-1}L_{l,1}^2\sqrt{1+\sigma_u^2}}{n(\mu + L_{l,1})}\right) \|\mathbf{x}_t-\mathbf{1}\bar{x}_t\|^2 \nonumber \\
&\overset{(d)}{\leq} \left( \frac{1}{n} - \frac{\gamma_y\bar{u}^{-1}_{t+1} \mu L_{l,1}}{n(\mu + L_{l,1})} \right) \|\mathbf{y}_{t} - \mathbf{1}y^*(\bar{x}_t))\|^2  
- \frac{{\gamma_y\bar{u}^{-1}_{t+1}\sqrt{1+\sigma_u^2}}}{2n (\mu + L_{l,1})} \|\nabla_y L(\mathbf{x}_t, \mathbf{y}_t)\|^2 \nonumber\\
&\quad + \left(\frac{\gamma_y\bar{u}_{t+1}^{-1}L_{l,1}(\mu+L_{l,1})\sqrt{1+\sigma_u^2}}{n\mu }+\frac{2\gamma_y\bar{u}_{t+1}^{-1}L_{l,1}^2\sqrt{1+\sigma_u^2}}{n(\mu + L_{l,1})}\right) \|\mathbf{x}_t-\mathbf{1}\bar{x}_t\|^2,\label{eq63}
\end{align}
where (a) employs Lemma 3.11 in \citep{bubeck2015convex},
(b) uses Young's inequality, (c) refers to \cref{Assumption2}, and (d) follows from \(\bar{m}^y_{t+1} \geq C_{m^y} \geq \frac{2\gamma_y(1+\zeta_u^2)(\mu + L_{l,1})}{\sqrt{1+\sigma_u^2}}\).
By plugging equation \cref{eq63} into equation \cref{eq62}, we have:
\begin{align}
&\frac{1}{n}\|\mathbf{y}_{t+1} - \mathbf{1}y^*(\bar{x}_{t+1})\|^2 \nonumber\\
&\leq (1 + \bar{\lambda}_{t+1}) \frac{1}{n}\left( 1 - \frac{\gamma_y\bar{u}^{-1}_{t+1} \mu L_{l,1}}{\mu + L_{l,1}} \right) \|\mathbf{y}_{t} - \mathbf{1}y^*(\bar{x}_t))\|^2\notag\\
&\quad- (1 + \bar{\lambda}_{t+1})  \frac{{\gamma_y\bar{u}^{-1}_{t+1}\sqrt{1+\sigma_u^2}}}{2n (\mu + L_{l,1})}  \|\nabla_y L(\mathbf{x}_t, \mathbf{y}_t)\|^2\nonumber \\
&\quad +(1 + \bar{\lambda}_{t+1})  \left(\frac{\gamma_y\bar{u}_{t+1}^{-1}L_{l,1}(\mu+L_{l,1})\sqrt{1+\sigma_u^2}}{n\mu }+\frac{2\gamma_y\bar{u}_{t+1}^{-1}L_{l,1}^2\sqrt{1+\sigma_u^2}}{n(\mu + L_{l,1})}\right) \|\mathbf{x}_t-\mathbf{1}\bar{x}_t\|^2\nonumber\\
&\quad+ \left( 1 + \frac{1}{\bar{\lambda}_{t+1}} \right) \|y^*(\bar{x}_t) - y^*(\bar{x}_{t+1})\|^2. \label{eq64}
\end{align}

By rearranging the terms in \cref{eq64}, we have:
\begin{align}
&(1 + \bar{\lambda}_{t+1})  \frac{{\gamma_y\bar{u}^{-1}_{t+1}\sqrt{1+\sigma_u^2}}}{2n (\mu + L_{l,1})} \|\nabla_y L(\mathbf{x}_t, \mathbf{y}_t)\|^2\nonumber\\
&\leq (1 + \bar{\lambda}_{t+1}) \frac{1}{n}\left( 1 - \frac{\gamma_y\bar{u}^{-1}_{t+1} \mu L_{l,1}}{\mu + L_{l,1}} \right) \|\mathbf{y}_{t} - \mathbf{1}y^*(\bar{x}_t))\|^2
- \frac{1}{n}\|\mathbf{y}_{t+1} - \mathbf{1}y^*(\bar{x}_{t+1})\|^2 \nonumber \\
&\quad +(1 + \bar{\lambda}_{t+1})  \left(\frac{\gamma_y\bar{u}_{t+1}^{-1}L_{l,1}(\mu+L_{l,1})\sqrt{1+\sigma_u^2}}{n\mu }+\frac{2\gamma_y\bar{u}_{t+1}^{-1}L_{l,1}^2\sqrt{1+\sigma_u^2}}{n(\mu + L_{l,1})}\right) \|\mathbf{x}_t-\mathbf{1}\bar{x}_t\|^2\nonumber\\
&\quad+ \left( 1 + \frac{1}{\bar{\lambda}_{t+1}} \right) \|y^*(\bar{x}_t) - y^*(\bar{x}_{t+1})\|^2.\label{eq65}
\end{align}

We take \(\bar{\lambda}_{t+1} := \frac{\gamma_y\bar{u}^{-1}_{t+1} \mu L_{l,1}}{\mu + L_{l,1}}\). Since \(\bar{m}^y_{t+1} > C_{m^y} \geq \frac{\gamma_y \mu L_{l,1}}{\mu + L_{l,1}}\) in \cref{eq53}, we have \(\bar{\lambda}_{t+1} \leq 1\). Then, we have:
\begin{align}
&\frac{\gamma_y\bar{u}^{-1}_{t+1}\sqrt{1+\sigma_u^2}}{2n} \|\nabla_y L(\mathbf{x}_t, \mathbf{y}_t)\|^2\nonumber\\
&\leq (1 + \bar{\lambda}_{t+1}) \frac{\gamma_y\bar{u}^{-1}_{t+1}\sqrt{1+\sigma_u^2}}{2n} \|\nabla_y L(\mathbf{x}_t, \mathbf{y}_t)\|^2  \nonumber \\
&\leq \frac{\mu + L_{l,1}}{n} \left( \|\mathbf{y}_{t} - \mathbf{1}y^*(\bar{x}_t))\|^2
- \|\mathbf{y}_{t+1} - \mathbf{1}y^*(\bar{x}_{t+1})\|^2 \right) + \frac{2 (\mu + L_{l,1}) }{\bar{\lambda}_{t+1}} \|y^*(\bar{x}_t) - y^*(\bar{x}_{t+1})\|^2 \nonumber\\
&\quad +\left(\frac{2\gamma_y\bar{u}_{t+1}^{-1}L_{l,1}(\mu+L_{l,1})^2\sqrt{1+\sigma_u^2}}{n\mu }+\frac{4\gamma_y\bar{u}_{t+1}^{-1}L_{l,1}^2\sqrt{1+\sigma_u^2}}{n}\right) \|\mathbf{x}_t-\mathbf{1}\bar{x}_t\|^2\nonumber \\
&= \frac{\mu + L_{l,1}}{n} \left( \|\mathbf{y}_{t} - \mathbf{1}y^*(\bar{x}_t))\|^2
- \|\mathbf{y}_{t+1} - \mathbf{1}y^*(\bar{x}_{t+1})\|^2 \right)+ \frac{2(\mu + L_{l,1})^2}{\gamma_y\bar{u}^{-1}_{t+1} \mu L_{l,1}} \|y^*(\bar{x}_t) - y^*(\bar{x}_{t+1})\|^2 \nonumber\\
&\quad+\left(\frac{2\gamma_y\bar{u}_{t+1}^{-1}L_{l,1}(\mu+L_{l,1})^2\sqrt{1+\sigma_u^2}}{n\mu }+\frac{4\gamma_y\bar{u}_{t+1}^{-1}L_{l,1}^2\sqrt{1+\sigma_u^2}}{n}\right) \|\mathbf{x}_t-\mathbf{1}\bar{x}_t\|^2\nonumber \\
&\overset{(a)}{\leq} \frac{\mu + L_{l,1}}{n} \left( \|\mathbf{y}_{t} - \mathbf{1}y^*(\bar{x}_t))\|^2
- \|\mathbf{y}_{t+1} - \mathbf{1}y^*(\bar{x}_{t+1})\|^2 \right)+ \frac{2(\mu + L_{l,1})^2L_y^2}{\gamma_y\bar{u}^{-1}_{t+1} \mu L_{l,1}} \|\bar{x}_t -\bar{x}_{t+1}\|^2 \nonumber\\
&\quad+\left(\frac{2(\mu+L_{l,1})^3\sqrt{1+\sigma_u^2}}{n\mu^2 }+\frac{4L_{l,1}(\mu+L_{l,1})\sqrt{1+\sigma_u^2}}{n\mu }\right) \|\mathbf{x}_t-\mathbf{1}\bar{x}_t\|^2,\label{eq66}
\end{align}
where (a) uses the Lipschitzness of \(y^*(\bar{x})\). Summing the above inequality over \(k = k_2, \dots, t\), we have:
\begin{align}
&\sum_{k=k_2}^t \frac{\gamma_y\bar{u}^{-1}_{k+1}\sqrt{1+\sigma_u^2}}{2n} \|\nabla_y L(\mathbf{x}_k, \mathbf{y}_k)\|^2\nonumber\\
&\leq \sum_{k=k_2-1}^t \frac{\gamma_y\bar{u}^{-1}_{k+1}\sqrt{1+\sigma_u^2}}{2n} \|\nabla_y L(\mathbf{x}_k, \mathbf{y}_k)\|^2\nonumber \\
&\leq \frac{\mu + L_{l,1}}{n}\|\mathbf{y}_{k_2-1} - \mathbf{1}y^*(\bar{x}_{k_2-1})\|^2
+ \frac{2(\mu + L_{l,1})^2L_y^2}{\gamma_y \mu L_{l,1}} \sum_{k=k_2-1}^t \bar{u}_{k+1} \|\bar{x}_k - \bar{x}_{k+1}\|^2\nonumber\\
&\quad+\left(\frac{2(\mu+L_{l,1})^3\sqrt{1+\sigma_u^2}}{n\mu^2 }+\frac{4L_{l,1}(\mu+L_{l,1})\sqrt{1+\sigma_u^2}}{n\mu }\right)\sum_{k=k_2-1}^t \|\mathbf{x}_k-\mathbf{1}\bar{x}_k\|^2\nonumber \\
&\overset{(a)}{\leq} \frac{\mu + L_{l,1}}{n\mu^2} \sum_{i=1}^n \left\| \nabla_y l_i(x_{i,k_2-1}, y_{i,k_2-1}) - \nabla_y l_i(x_{i,k_2-1}, y^*(\bar{x}_{k_2-1})) \right\|^2\nonumber\\
&\quad+\left(\frac{2(\mu+L_{l,1})^3\sqrt{1+\sigma_u^2}}{n\mu^2 }+\frac{4L_{l,1}(\mu+L_{l,1})\sqrt{1+\sigma_u^2}}{n\mu }\right)\sum_{k=0}^t \|\mathbf{x}_k-\mathbf{1}\bar{x}_k\|^2 \nonumber\\
&\quad + \frac{2(\mu + L_{l,1})^2 L_y^2}{\gamma_y\mu L_{l,1}} \sum_{k=k_2-1}^t \bar{u}_{k+1} \left\| 
\gamma_x \bar{q}_{k+1}^{-1} \frac{\mathbf{1}^\top}{n} \bar{\nabla} F(\mathbf{x}_k,\mathbf{y}_k,\mathbf{v}_k) 
- \gamma_x \frac{\left(\tilde{\mathbf{q}}_{k+1}^{-1}\right)^\top}{n} \bar{\nabla} F(\mathbf{x}_k,\mathbf{y}_k,\mathbf{v}_k)
\right\|^2\nonumber \\
&\overset{(b)}{\leq} \frac{\mu + L_{l,1}}{n\mu^2} \sum_{i=1}^n \left\| \nabla_y l_i(x_{i,k_2-1}, y_{i,k_2-1}) - \nabla_y l_i(x_{i,k_2-1}, y^*(\bar{x}_{k_2-1})) \right\|^2 \nonumber\\
&\quad + \frac{ 4\gamma_x^2  L_y^2\left( 1 + \zeta_q^2 \right) (\mu + L_{l,1})^2}{n\gamma_y\mu L_{l,1}} \sum_{k=k_2-1}^t \bar{u}_{k+1}
\bar{q}_{k+1}^{-2} \| \bar{\nabla} F(\mathbf{x}_k,\mathbf{y}_k,\mathbf{v}_k)\|^2 + \frac{16 \gamma_x^2 \rho_W (1 + \zeta_q^2)}{(1 - \rho_W)^2}\notag\\
&\quad\cdot\left(\frac{(\mu+L_{l,1})^3\sqrt{1+\sigma_u^2}}{n\mu^2 }+\frac{2L_{l,1}(\mu+L_{l,1})\sqrt{1+\sigma_u^2}}{n\mu }\right)\sum_{k=0}^{t} 
 \bar{q}_{k+1}^{-2} \|\bar{\nabla}F(\mathbf{x}_k, \mathbf{y}_k, \mathbf{v}_k)\|^2\nonumber\\
&\quad + \frac{4\Delta_0^x(\mu+L_{l,1})^3\sqrt{1+\sigma_u^2}}{n\mu^2(1-\rho_W) }+\frac{8\Delta_0^xL_{l,1}(\mu+L_{l,1})\sqrt{1+\sigma_u^2}}{n\mu(1-\rho_W) }\nonumber\\
&\overset{(c)}{\leq} \frac{(\mu + L_{l,1}) C_{m^y}^2}{\mu^2}+\left(\frac{4\Delta_0^x}{1-\rho_W}+\frac{16 \gamma_x^2 \rho_WC_{m^x}^2 (1 + \zeta_q^2)}{\bar{q}_0^2(1 - \rho_W)^2} \right)\notag\\
&\quad\cdot\left(\frac{(\mu+L_{l,1})^3\sqrt{1+\sigma_u^2}}{n\mu^2 }+\frac{2L_{l,1}(\mu+L_{l,1})\sqrt{1+\sigma_u^2}}{n\mu }\right) \nonumber\\
&\quad+ \left[\frac{16 \gamma_x^2 \rho_W(1 + \zeta_q^2)}{(1 - \rho_W)^2} \left(\frac{(\mu+L_{l,1})^3\sqrt{1+\sigma_u^2}}{n\mu^2}+\frac{2L_{l,1}(\mu+L_{l,1})\sqrt{1+\sigma_u^2}}{n\mu }\right)\right.\nonumber\\
&\quad\left.+\frac{4\gamma_x^2  L_y^2\left( 1 + \zeta_q^2 \right) (\mu + L_{l,1})^2}{n\gamma_y\mu L_{l,1}}\right]\!\! \sum_{k=\min\{k_1-1,k_2-1\}}^t \frac{\|\bar{\nabla} F(\mathbf{x}_k,\mathbf{y}_k,\mathbf{v}_k)\|^2}{[\bar{m}^x_{k+1}]^2\max \big\{\bar{m}^v_{k+1}, \bar{m}^y_{k+1}\big\}} \nonumber\\
&\leq \frac{(\mu + L_{l,1}) C_{m^y}^2}{\mu^2} +\left(\frac{4\Delta_0^x}{1-\rho_W}+\frac{16 \gamma_x^2 \rho_WC_{m^x}^2 (1 + \zeta_q^2)}{\bar{q}_0^2(1 - \rho_W)^2} \right)\notag\\
&\quad\cdot\left(\frac{(\mu+L_{l,1})^3\sqrt{1+\sigma_u^2}}{n\mu^2 }+\frac{2L_{l,1}(\mu+L_{l,1})\sqrt{1+\sigma_u^2}}{n\mu }\right)+\frac{4\gamma_x^2  L_y^2\left( 1 + \zeta_q^2 \right) (\mu + L_{l,1})^2}{n\gamma_y\mu L_{l,1}\bar{z}_0} \nonumber\\
&\quad+ \frac{16 \gamma_x^2 \rho_W(1 + \zeta_q^2)}{\bar{z}_0(1 - \rho_W)^2} \left(\frac{(\mu+L_{l,1})^3\sqrt{1+\sigma_u^2}}{n\mu^2}+\frac{2L_{l,1}(\mu+L_{l,1})\sqrt{1+\sigma_u^2}}{n\mu }\right)\nonumber\\
&\quad+ \left[\frac{16 \gamma_x^2 \rho_W(1 + \zeta_q^2)}{(1 - \rho_W)^2} \left(\frac{(\mu+L_{l,1})^3\sqrt{1+\sigma_u^2}}{n\mu^2}+\frac{2L_{l,1}(\mu+L_{l,1})\sqrt{1+\sigma_u^2}}{n\mu}\right)\right.\nonumber\\
&\quad\left.+\frac{4\gamma_x^2  L_y^2\left( 1 + \zeta_q^2 \right) (\mu + L_{l,1})^2}{n\gamma_y\mu L_{l,1}}\right]\!\! \sum_{k=\min\{k_1,k_2\}}^t \frac{\|\bar{\nabla} F(\mathbf{x}_k,\mathbf{y}_k,\mathbf{v}_k)\|^2}{[\bar{m}^x_{k+1}]^2\max \big\{\bar{m}^v_{k+1}, \bar{m}^y_{k+1}\big\}},\label{eq67}
\end{align}
where (a) uses \cref{Assumption1}; (b) refers to \cref{Lemmanewx}; (c) results from \(\|\nabla_y l_i({x}_{i,k_2-1}, {y}_{i,k_2-1})\|^2 \leq [m^y_{i,k_2}]^2 \leq C_{m^y}^2\) and \(\|\bar{\nabla} f_i({x}_{i,k_1-1}, {y}_{i,k_1-1},{v}_{i,k_1-1})\|^2 \leq [m^x_{i,k_1}]^2 \leq C_{m^x}^2\). Then, the proof is complete.
\end{proof}

\subsection{The Upper Bound of $\sum_{k=k_3}^t \frac{\|\nabla_v R(\mathbf{x}_k, \mathbf{y}_k, \mathbf{v}_k)\|^2 }{\bar{z}_{k+1}} $}
\begin{lemma}\label{Lemma8} Under \cref{Assumption1} and \cref{Assumption2}, for \cref{alg1}, suppose the total iteration rounds is \(T\). If \(k_3\) in \cref{Lemma6} exists within \(T\) iterations, for all integers \(t \in [k_3, T)\), we have:
\begin{align}
&\sum_{k=k_3}^t \frac{\|\nabla_v R(\mathbf{x}_k, \mathbf{y}_k, \mathbf{v}_k)\|^2 }{\bar{z}_{k+1}} \nonumber\\
&\leq 
\frac{4nC_{m^y}^2(\mu + C_{l_{yy}})}{\mu^4\gamma_v\sqrt{1+\sigma_z^2}}\left(\frac{nL_{l,2}C_{f_y}}{\mu}+L_{f,1}\right)^2 + \frac{8nC_{m^v}^2(\mu + C_{l_{yy}})}{\mu^2\gamma_v\sqrt{1+\sigma_z^2}}\notag \\
&\quad + \frac{16(\mu + C_{l_{yy}})}{\gamma_v\mu^2}\left(\frac{nL_{l,2}C_{f_y}}{\mu}\!+\!L_{f,1}\right)^2 \left(\frac{\Delta_0^x}{1\!-\!\rho_W}\!\!+\!\!\frac{4 \gamma_x^2 \rho_WC_{m^x}^2 (1 \!\!+\!\! \zeta_q^2)}{\bar{q}_0^2(1 \!-\! \rho_W)^2} \right)\left(\frac{(\mu\!\!+\!\!L_{l,1})^2}{n\mu^2 }\!\!+\!\!\frac{2L_{l,1}}{n\mu }\right)\notag \\
&\quad+\frac{16(\mu + C_{l_{yy}})}{\gamma_v\mu^2\sqrt{1+\sigma_u^2}}\left(\frac{nL_{l,2}C_{f_y}}{\mu}\!+\!L_{f,1}\right)^2\left[\frac{4 \gamma_x^2 \rho_W (1 + \zeta_q^2)}{\bar{z}_0^2(1 - \rho_W)^2} \left(\frac{(\mu+L_{l,1})^2\sqrt{1+\sigma_u^2}}{n\mu^2 }\right.\right.\nonumber\\
&\quad\left.\left.+\frac{2L_{l,1}\sqrt{1+\sigma_u^2}}{n\mu }\right) +\frac{\gamma_x^2L_y^2(1+\zeta_q^2)(\mu + L_{l,1})}{n\bar{z}_0\gamma_y \mu L_{l,1}}\right]\sum_{k=\min\{k_1-1,k_2-1\}}^{k_3-2} 
 \frac{\| \bar{\nabla} F(\mathbf{x}_{k},\mathbf{y}_{k},\mathbf{v}_{k})\|^2}{[\bar{m}^x_{k+1}]^2}\nonumber\\
&\quad+ \left(\frac{64 \gamma_x^2 \rho_W \bar{L}_r^2(\mu + C_{l_{yy}})(1 + \zeta_q^2)}{\mu C_{l_{yy}}\gamma_v(1 - \rho_W)^2} + \frac{2\gamma_x^2L_v^2(1+\zeta_q^2)(\mu + C_{l_{yy}})^2}{\mu C_{l_{yy}}C_{m^v}\gamma_v^2\sqrt{1+\sigma_z^2}}\right)\notag\\
&\quad\cdot\sum_{k=\min\{k_1-1,k_3-1\}}^{t} 
 \frac{\| \bar{\nabla} F(\mathbf{x}_{k},\mathbf{y}_{k},\mathbf{v}_{k})\|^2}{[\bar{m}^x_{k+1}]^2}\nonumber\\
&\quad  + \left(\frac{8}{\mu^2}+\frac{4(\mu + C_{l_{yy}})^2}{\mu^3 C_{l_{yy}}} \right)\left(\frac{nL_{l,2}C_{f_y}}{\mu}+L_{f,1}\right)^2 \sum_{k=k_3-1}^t\frac{\|\nabla_y L(\mathbf{x}_k,\mathbf{y}_k)\|^2}{\bar{m}_{k+1}^y}\nonumber\\
&\quad +\left(\frac{16\Delta_0^x}{1-\rho_W}+\frac{64\gamma_x^2\rho_WC_{m^x}^2(1+\zeta_q^2)}{\bar{q}_0^2(1-\rho_W)^2}\right)\left(\frac{ \bar{L}_r^2(\mu + C_{l_{yy}})}{\mu C_{l_{yy}}\gamma_v}+\frac{\bar{L}_r^2(\mu+C_{l_{yy}})}{\mu^2\gamma_v\sqrt{1+\sigma_z^2}}\right).\label{eq68}
\end{align}

\end{lemma}

\begin{proof}
For \(k_3 \leq t < T\), we have \(\bar{m}^v_{t+1} > C_{m^v}\). For any positive scalar \(\hat{\lambda}_{t+1}\), using Young's inequality, we have:
\begin{align}
\frac{1}{n}\|\mathbf{v}_{t+1} - \mathbf{1}v^*(\bar{x}_{t+1})\|^2 
\leq \frac{(1 + \bar{\lambda}_{t+1})}{n} \|\mathbf{v}_{t+1} - \mathbf{1}v^*(\bar{x}_t)\|^2 
+ \left( 1 + \frac{1}{\bar{\lambda}_{t+1}} \right) \|v^*(\bar{x}_t) - v^*(\bar{x}_{t+1})\|^2. \label{eq69}
\end{align}

For the first term on the RHS of \cref{eq69}, we have:
\begin{align}
&\frac{1}{n}\|\mathbf{v}_{t+1} - \mathbf{1}v^*(\bar{x}_{t+1})\|^2 \nonumber\\
&= \frac{1}{n}\left\| \mathcal{P}_{\mathcal{V}}\left(\mathbf{v}_{t}- \gamma_v Z^{-1}_{t+1} \nabla_v R(\mathbf{x}_t, \mathbf{y}_t, \mathbf{v}_t) \right)- \mathbf{1}v^*(\bar{x}_{t}) \right\|^2 \nonumber \\
&\overset{(a)}{\leq} \frac{1}{n}\left\| \mathbf{v}_{t}- \gamma_v Z^{-1}_{t+1} \nabla_v R(\mathbf{x}_t, \mathbf{y}_t, \mathbf{v}_t) - \mathbf{1}v^*(\bar{x}_{t}) \right\|^2 \nonumber \\
&= \frac{1}{n}\|\mathbf{v}_{t} -\mathbf{1}v^*(\bar{x}_{t})\|^2\! 
+ \frac{\gamma_v^2}{n} \|Z^{-1}_{t+1}\nabla_v R(\mathbf{x}_t, \mathbf{y}_t, \mathbf{v}_t)\|^2 
\notag\\
&\quad- \frac{1}{n}\sum_{i=1}^n 2\bar{z}^{-1}_{t+1}\gamma_v \langle \nabla_v r_i(x_{i,t}, y_{i,t}, v_{i,t}) , v_{i,t} - v^*(\bar{x}_t)\rangle\nonumber\\
&\quad-\frac{1}{n}\sum_{i=1}^n 2\bar{z}_{t+1}^{-1} \gamma_y\Bigg\langle \left(\frac{z_{i,t+1}^{-1}-\bar{z}_{t+1}^{-1}}{\bar{z}_{t+1}^{-1}}\right) \nabla_v r_i(x_{i,t}, y_{i,t}, v_{i,t}) , v_{i,t} - v^*(\bar{x}_t)\Bigg\rangle,\label{eq70}
\end{align}
where (a) uses the non-expansiveness of the projection operator, as established in Lemma 1 of \citep{nedic2010constrained}. Then, for the last two terms on the RHS of \cref{eq70}, we have:
\begin{align}
&- \frac{1}{n}\sum_{i=1}^n 2\bar{z}^{-1}_{t+1}\gamma_v \langle \nabla_v r_i(x_{i,t}, y_{i,t}, v_{i,t}) , v_{i,t} - v^*(\bar{x}_t)\rangle\nonumber\\
&\quad-\frac{1}{n}\sum_{i=1}^n 2\bar{z}_{t+1}^{-1} \gamma_y\Bigg\langle\!\!\! \left(\!\!\frac{z_{i,t+1}^{-1}\!\!-\!\!\bar{z}_{t+1}^{-1}}{\bar{z}_{t+1}^{-1}}\right) \nabla_v r_i(x_{i,t}, y_{i,t}, v_{i,t}) , v_{i,t} \!\!-\!\! v^*(\bar{x}_t)\!\!\!\Bigg\rangle\nonumber\\ 
&= - \frac{2\gamma_v}{n}\sum_{i=1}^n \bar{z}^{-1}_{t+1} \langle \nabla_v r_i(x_{i,t}, y_{i,t}, v_{i,t})- \nabla_v r_i(x_{i,t}, y_{i,t}, v_i^*(\bar{x}_{t})) , v_{i,t} - v_i^*(\bar{x}_t)\rangle \nonumber \\
&\quad - \frac{2\gamma_v}{n}\sum_{i=1}^n \bar{z}^{-1}_{t+1} \langle \nabla_v r_i(x_{i,t}, y_{i,t}, v_i^*(\bar{x}_{t}))- \nabla_v r_i(x_{i,t},y_i^*(\bar{x}_{t}), v^*(\bar{x}_{t})) , v_{i,t} - v_i^*(\bar{x}_t) \rangle \nonumber \\
&\quad - \frac{2\gamma_v}{n}\sum_{i=1}^n \bar{z}^{-1}_{t+1} \Bigg\langle\!\!\! \left(\!\!\frac{z_{i,t+1}^{-1}\!\!-\!\!\bar{z}_{t+1}^{-1}}{\bar{z}_{t+1}^{-1}}\!\!\right)\!\!(\nabla_v r_i(x_{i,t}, y_{i,t}, v_{i,t})\!\!-\!\! \nabla_v r_i(x_{i,t}, y_{i,t}, v_i^*(\bar{x}_{t}))) , v_{i,t} \!\!-\!\! v_i^*(\bar{x}_t)\!\!\!\Bigg\rangle \nonumber \\
&\quad - \frac{2\gamma_v}{n}\sum_{i=1}^n \bar{z}^{-1}_{t+1} \Bigg\langle\!\!\! \left(\!\!\frac{z_{i,t+1}^{-1}\!\!-\!\!\bar{z}_{t+1}^{-1}}{\bar{z}_{t+1}^{-1}}\!\!\right)\!\!(\nabla_v r_i(x_{i,t}, y_{i,t}, v_i^*(\bar{x}_{t}))\!\!-\!\!\nabla_v r_i(x_{i,t},y_i^*(\bar{x}_{t}), v^*(\bar{x}_{t}))) , v_{i,t} \!\!-\!\! v_i^*(\bar{x}_t) \!\!\!\Bigg\rangle \nonumber \\
&\overset{(a)}{\leq} -\frac{2\gamma_v\bar{z}_{t+1}^{-1}\sqrt{1+\sigma_z^2}}{n(\mu + C_{l_{yy}})} \|\nabla_v R(\mathbf{x}_t, \mathbf{y}_t, \mathbf{v}_t) - \nabla_v R(\mathbf{x}_t, \mathbf{y}_t, \mathbf{1}v^*(\bar{x}_t))\|^2 \nonumber \\
&\quad + \frac{\gamma_v\bar{z}_{t+1}^{-1}\sqrt{1+\sigma_z^2}(\mu + C_{l_{yy}})}{n\mu C_{l_{yy}}} \|\nabla_v R(\mathbf{x}_t, \mathbf{y}_t, \mathbf{1}v^*(\bar{x}_t)) - \nabla_v R(\mathbf{x}_t, \mathbf{1}y^*(\bar{x}_t), \mathbf{1}v^*(\bar{x}_t))\|^2\nonumber\\
&\quad+ \frac{2\mu C_{l_{yy}}\gamma_v\bar{z}_{t+1}^{-1}}{n(\mu + C_{l_{yy}})} \|\mathbf{v}_t - \mathbf{1}v^*(\bar{x}_t)\|^2  - \frac{4\mu C_{l_{yy}}\gamma_v\bar{z}_{t+1}^{-1}}{n(\mu + C_{l_{yy}})} \|\mathbf{v}_t - \mathbf{1}v^*(\bar{x}_t)\|^2\nonumber\\
&\overset{(b)}{\leq} -\frac{\gamma_v\bar{z}_{t+1}^{-1}\sqrt{1+\sigma_z^2}}{n(\mu + C_{l_{yy}})} \|\nabla_v R(\mathbf{x}_t, \mathbf{y}_t, \mathbf{v}_t)\|^2 
+ \frac{2\gamma_v\bar{z}_{t+1}^{-1}\sqrt{1+\sigma_z^2}}{n(\mu + C_{l_{yy}})}\| \nabla_v R(\mathbf{x}_t, \mathbf{y}_t, \mathbf{1}v^*(\bar{x}_t))\|^2 \nonumber \\
&\quad + \frac{\gamma_v\bar{z}_{t+1}^{-1}\sqrt{1+\sigma_z^2}(\mu + C_{l_{yy}})}{n\mu C_{l_{yy}}} \|\nabla_v R(\mathbf{x}_t, \mathbf{y}_t, \mathbf{1}v^*(\bar{x}_t)) - \nabla_v R(\mathbf{x}_t, \mathbf{1}y^*(\bar{x}_t), \mathbf{1}v^*(\bar{x}_t))\|^2 \nonumber \\
&\quad - \frac{2\mu C_{l_{yy}}\gamma_v\bar{z}_{t+1}^{-1}}{n(\mu + C_{l_{yy}})}  \|\mathbf{v}_t - \mathbf{1}v^*(\bar{x}_t)\|^2 \nonumber\\
&\overset{(c)}{=} -\frac{\gamma_v\bar{z}_{t+1}^{-1}\sqrt{1+\sigma_z^2}}{n(\mu + C_{l_{yy}})} \|\nabla_v R(\mathbf{x}_t, \mathbf{y}_t, \mathbf{v}_t)\|^2 - \frac{2\mu C_{l_{yy}}\gamma_v\bar{z}_{t+1}^{-1}}{n(\mu + C_{l_{yy}})}  \|\mathbf{v}_t - \mathbf{1}v^*(\bar{x}_t)\|^2\nonumber\\
&\quad+\frac{4\gamma_v\bar{z}_{t+1}^{-1}\bar{L}_r^2\sqrt{1+\sigma_z^2}}{n(\mu + C_{l_{yy}})}\|\mathbf{x}_t-\mathbf{1}\bar{x}_t\|^2+\left(\frac{4\gamma_v\bar{z}_{t+1}^{-1}\sqrt{1+\sigma_z^2}}{n(\mu + C_{l_{yy}})}+\frac{\gamma_v\bar{z}_{t+1}^{-1}\sqrt{1+\sigma_z^2}(\mu + C_{l_{yy}})}{n\mu C_{l_{yy}}} \right)\nonumber \\
&\quad \cdot\|\nabla_v R(\mathbf{x}_t, \mathbf{y}_t, \mathbf{1}v^*(\bar{x}_t)) \!-\! \nabla_v R(\mathbf{x}_t, \mathbf{1}y^*(\bar{x}_t), \mathbf{1}v^*(\bar{x}_t))\|^2  \nonumber\\
&\overset{(d)}{\leq} -\frac{\gamma_v\bar{z}_{t+1}^{-1}\sqrt{1+\sigma_z^2}}{n(\mu + C_{l_{yy}})} \|\nabla_v R(\mathbf{x}_t, \mathbf{y}_t, \mathbf{v}_t)\|^2 - \frac{2\mu C_{l_{yy}}\gamma_v\bar{z}_{t+1}^{-1}}{n(\mu + C_{l_{yy}})}  \|\mathbf{v}_t - \mathbf{1}v^*(\bar{x}_t)\|^2\nonumber\\
&\quad+\frac{4\gamma_v\bar{z}_{t+1}^{-1}\bar{L}_r^2\sqrt{1+\sigma_z^2}}{n(\mu + C_{l_{yy}})}\|\mathbf{x}_t-\mathbf{1}\bar{x}_t\|^2 +\left(\frac{2\gamma_v\bar{z}_{t+1}^{-1}\sqrt{1+\sigma_z^2}}{n(\mu + C_{l_{yy}})}+\frac{\gamma_v\bar{z}_{t+1}^{-1}\sqrt{1+\sigma_z^2}(\mu + C_{l_{yy}})}{n\mu C_{l_{yy}}} \right)\nonumber \\
&\quad \cdot\left(L_{l,2}\|\mathbf{1}v^*(\bar{x}_t)\|+L_{f,1}\right)^2\|\mathbf{y}_t - \mathbf{1}y^*(\bar{x}_t)\|^2\nonumber\\
&\overset{(e)}{\leq} -\frac{\gamma_v\bar{z}_{t+1}^{-1}\sqrt{1+\sigma_z^2}}{n(\mu + C_{l_{yy}})} \|\nabla_v R(\mathbf{x}_t, \mathbf{y}_t, \mathbf{v}_t)\|^2 - \frac{2\mu C_{l_{yy}}\gamma_v\bar{z}_{t+1}^{-1}}{n(\mu + C_{l_{yy}})}  \|\mathbf{v}_t - \mathbf{1}v^*(\bar{x}_t)\|^2\notag\\
&\quad+\frac{4\gamma_v\bar{z}_{t+1}^{-1}\bar{L}_r^2\sqrt{1+\sigma_z^2}}{n(\mu + C_{l_{yy}})}\|\mathbf{x}_t-\mathbf{1}\bar{x}_t\|^2+ \left(\frac{2\gamma_v\bar{z}_{t+1}^{-1}\sqrt{1+\sigma_z^2}}{n(\mu + C_{l_{yy}})}+\frac{\gamma_v\bar{z}_{t+1}^{-1}\sqrt{1+\sigma_z^2}(\mu + C_{l_{yy}})}{n\mu C_{l_{yy}}} \right) \nonumber \\
&\quad \cdot\left(\frac{nL_{l,2}C_{f_y}}{\mu}+L_{f,1}\right)^2\|\mathbf{y}_t - \mathbf{1}y^*(\bar{x}_t)\|^2,\label{eq71}
\end{align}
where (a) follows from Lemma 3.11 in \citep{bubeck2015convex}; (b) uses $-\|a - b\|^2 \leq -\frac{1}{2}\|a\|^2 + \|b\|^2$ since $\|a - b + b\|^2 \leq 2\|a - b\|^2 + 2\|b\|^2$; (c) uses $\nabla_v R(\mathbf{1}\bar{x}_t, \mathbf{1}y^*(\bar{x}_t)), \mathbf{1}v^*(\bar{x}_t))= 0$; (d) and (e) use \cref{Lemma3}. Plugging \cref{eq71} into \cref{eq70}, we have:
\begin{align}
&\frac{1}{n}\|\mathbf{v}_{t+1} - \mathbf{1}v^*(\bar{x}_{t})\|^2 \nonumber\\
&\leq \frac{1}{n}\left( 1 -  \frac{2\mu C_{l_{yy}}\gamma_v\bar{z}_{t+1}^{-1}}{\mu + C_{l_{yy}}} \right) \|\mathbf{v}_t - \mathbf{1}v^*(\bar{x}_t)\|^2 +\frac{4\gamma_v\bar{z}_{t+1}^{-1}\bar{L}_r^2\sqrt{1+\sigma_z^2}}{n(\mu + C_{l_{yy}})}\|\mathbf{x}_t-\mathbf{1}\bar{x}_t\|^2\nonumber \\
&\quad + \frac{\gamma_v\bar{z}_{t+1}^{-1}}{n}\left(\gamma_v\bar{z}_{t+1}^{-1}(1+\zeta_z^2)-\frac{\sqrt{1+\sigma_z^2}}{\mu+C_{l_{yy}}}\right)\|\nabla_v R(\mathbf{x}_t, \mathbf{y}_t, \mathbf{v}_t)\|^2
 \nonumber \\
&\quad + \frac{1}{n}\left(\frac{2\gamma_v\sqrt{1+\sigma_z^2}}{\mu + C_{l_{yy}}}+\frac{\gamma_v\sqrt{1+\sigma_z^2}(\mu + C_{l_{yy}})}{\mu C_{l_{yy}}} \right)\left(\frac{nL_{l,2}C_{f_y}}{\mu}+L_{f,1}\right)^2\bar{z}_{t+1}^{-1}\|\mathbf{y}_t - \mathbf{1}y^*(\bar{x}_t)\|^2\nonumber \\
&\overset{(a)}{\leq} \frac{1}{n}\left( 1 -  \frac{2\mu C_{l_{yy}}\gamma_v\bar{z}_{t+1}^{-1}}{\mu + C_{l_{yy}}} \right) \|\mathbf{v}_t - \mathbf{1}v^*(\bar{x}_t)\|^2 \notag\\
&\quad- \frac{\gamma_v\bar{z}_{t+1}^{-1}\sqrt{1+\sigma_z^2}}{2n(\mu + C_{l_{yy}}) } \|\nabla_v R(\mathbf{x}_t, \mathbf{y}_t, \mathbf{v}_t)\|^2+\frac{4\gamma_v\bar{z}_{t+1}^{-1}\bar{L}_r^2\sqrt{1+\sigma_z^2}}{n(\mu + C_{l_{yy}})}\|\mathbf{x}_t-\mathbf{1}\bar{x}_t\|^2 \nonumber \\
&\quad + \frac{1}{n}\left(\frac{2\gamma_v\sqrt{1+\sigma_z^2}}{\mu + C_{l_{yy}}}+\frac{\gamma_v\sqrt{1+\sigma_z^2}(\mu + C_{l_{yy}})}{\mu C_{l_{yy}}} \right)\left(\frac{nL_{l,2}C_{f_y}}{\mu}+L_{f,1}\right)^2\bar{z}_{t+1}^{-1}\|\mathbf{y}_t - \mathbf{1}y^*(\bar{x}_t)\|^2,\label{eq72}
\end{align}
where (a) follows from $\bar{z}_{t+1} \geq \bar{m}^v_{t+1} \geq C_{m^v} \geq \frac{2\gamma_v(1+\zeta_z^2)(\mu + C_{l_{yy}})}{\sqrt{1+\sigma_z^2}}$. Combining \cref{eq72} with \cref{eq69}, we have:
\begin{align}
  &\frac{1}{n}\|\mathbf{v}_{t+1} - \mathbf{1}v^*(\bar{x}_{t+1})\|^2 \nonumber\\
&\leq (1 + \hat{\lambda}_{t+1}) \frac{1}{n}\left( 1 -  \frac{2\mu C_{l_{yy}}\gamma_v\bar{z}_{t+1}^{-1}}{\mu + C_{l_{yy}}} \right) \|\mathbf{v}_t - \mathbf{1}v^*(\bar{x}_t)\|^2 \nonumber\\
&\quad- (1 + \hat{\lambda}_{t+1})\frac{\gamma_v\bar{z}_{t+1}^{-1}\sqrt{1+\sigma_z^2}}{2n(\mu + C_{l_{yy}}) } \|\nabla_v R(\mathbf{x}_t, \mathbf{y}_t, \mathbf{v}_t)\|^2 \nonumber \\
&\quad + (1 + \hat{\lambda}_{t+1})\frac{1}{n}\left(\frac{2\gamma_v\sqrt{1+\sigma_z^2}}{\mu + C_{l_{yy}}}+\frac{\gamma_v\sqrt{1+\sigma_z^2}(\mu + C_{l_{yy}})}{\mu C_{l_{yy}}} \right)\left(\frac{nL_{l,2}C_{f_y}}{\mu}+L_{f,1}\right)^2\nonumber\\
&\quad\cdot\bar{z}_{t+1}^{-1}\|\mathbf{y}_t - \mathbf{1}y^*(\bar{x}_t)\|^2 +(1 + \hat{\lambda}_{t+1})\frac{4\gamma_v\bar{z}_{t+1}^{-1}\bar{L}_r^2\sqrt{1+\sigma_z^2}}{n(\mu + C_{l_{yy}})}\|\mathbf{x}_t-\mathbf{1}\bar{x}_t\|^2\nonumber\\
&\quad+\left( 1 + \frac{1}{\hat{\lambda}_{t+1}} \right)  \|v^*(\bar{x}_t) - v^*(\bar{x}_{t+1})\|^2.\label{eq73}
\end{align}
By rearranging the terms in \cref{eq73}, we have:
\begin{align}
&(1 + \hat{\lambda}_{t+1})\frac{\gamma_v\bar{z}_{t+1}^{-1}\sqrt{1+\sigma_z^2}}{2n(\mu + C_{l_{yy}}) } \|\nabla_v R(\mathbf{x}_t, \mathbf{y}_t, \mathbf{v}_t)\|^2  \nonumber \\
&\leq (1 + \hat{\lambda}_{t+1}) \frac{1}{n}\left( 1 -  \frac{2\mu C_{l_{yy}}\gamma_v\bar{z}_{t+1}^{-1}}{\mu + C_{l_{yy}}} \right) \|\mathbf{v}_t - \mathbf{1}v^*(\bar{x}_t)\|^2 - \frac{1}{n}\|\mathbf{v}_{t+1} - \mathbf{1}v^*(\bar{x}_{t+1})\|^2 \nonumber \\
&\quad + (1 + \hat{\lambda}_{t+1})\frac{1}{n}\left(\frac{2\gamma_v\sqrt{1+\sigma_z^2}}{\mu + C_{l_{yy}}}+\frac{\gamma_v\sqrt{1+\sigma_z^2}(\mu + C_{l_{yy}})}{\mu C_{l_{yy}}} \right)\left(\frac{nL_{l,2}C_{f_y}}{\mu}+L_{f,1}\right)^2\nonumber\\
&\quad\cdot\bar{z}_{t+1}^{-1}\|\mathbf{y}_t - \mathbf{1}y^*(\bar{x}_t)\|^2+(1 + \hat{\lambda}_{t+1})\frac{4\gamma_v\bar{z}_{t+1}^{-1}\bar{L}_r^2\sqrt{1+\sigma_z^2}}{n(\mu + C_{l_{yy}})}\|\mathbf{x}_t-\mathbf{1}\bar{x}_t\|^2\nonumber \\
&\quad + \left( 1 + \frac{1}{\hat{\lambda}_{t+1}} \right)  \|v^*(\bar{x}_t) - v^*(\bar{x}_{t+1})\|^2.\label{eq74}
\end{align}
We now take $\hat{\lambda}_{t+1} :=  \frac{2\mu C_{l_{yy}}\gamma_v\bar{z}_{t+1}^{-1}}{\mu + C_{l_{yy}}}$. Since 
$\bar{z}_{t+1} \geq \bar{m}^v_{t+1} \geq C_{m^v} \geq \frac{2\mu C_{l_{yy}}\gamma_v}{\mu + C_{l_{yy}}}$ in \cref{eq54}, 
we have $\hat{\lambda}_{t+1} \leq 1$. Then we get:                                           
\begin{align}
&\frac{\gamma_v\bar{z}_{t+1}^{-1}\sqrt{1+\sigma_z^2}}{n} \|\nabla_v R(\mathbf{x}_t, \mathbf{y}_t, \mathbf{v}_t)\|^2 \nonumber\\
&\leq (1 + \hat{\lambda}_{t+1}) 
\frac{\gamma_v\bar{z}_{t+1}^{-1}\sqrt{1+\sigma_z^2}}{n} \|\nabla_v R(\mathbf{x}_t, \mathbf{y}_t, \mathbf{v}_t)\|^2 \nonumber \\
&\overset{(a)}{\leq} \frac{2(\mu + C_{l_{yy}})}{n}
\left( \|\mathbf{v}_t - \mathbf{1}v^*(\bar{x}_t)\|^2 - \|\mathbf{v}_{t+1} - \mathbf{1}v^*(\bar{x}_{t+1})\|^2 \right) \nonumber \\
&\quad + \frac{4(\mu + C_{l_{yy}})}{n}\left(\frac{2\gamma_v\sqrt{1+\sigma_z^2}}{\mu + C_{l_{yy}}}+\frac{\gamma_v\sqrt{1+\sigma_z^2}(\mu + C_{l_{yy}})}{\mu C_{l_{yy}}} \right)\left(\frac{nL_{l,2}C_{f_y}}{\mu}+L_{f,1}\right)^2\nonumber\\
&\quad\cdot\bar{z}_{t+1}^{-1}\|\mathbf{y}_t - \mathbf{1}y^*(\bar{x}_t)\|^2+\frac{16\gamma_v\bar{z}_{t+1}^{-1}\bar{L}_r^2\sqrt{1+\sigma_z^2}}{n}\|\mathbf{x}_t-\mathbf{1}\bar{x}_t\|^2 \nonumber \\
&\quad + 2(\mu + C_{l_{yy}}) 
\left(1 +  \frac{\mu + C_{l_{yy}}}{2\mu C_{l_{yy}}\gamma_v\bar{z}_{t+1}^{-1}} \right) 
L_v^2 \|\bar{x}_t - \bar{x}_{t+1}\|^2 \nonumber \\
&\overset{(b)}{\leq}  \frac{2(\mu + C_{l_{yy}})}{n}
\left( \|\mathbf{v}_t - \mathbf{1}v^*(\bar{x}_t)\|^2 - \|\mathbf{v}_{t+1} - \mathbf{1}v^*(\bar{x}_{t+1})\|^2 \right) \nonumber \\
&\quad + \frac{1}{n}\left(8\gamma_v\sqrt{1+\sigma_z^2}+\frac{4\gamma_v\sqrt{1+\sigma_z^2}(\mu + C_{l_{yy}})^2}{\mu C_{l_{yy}}} \right)\left(\frac{nL_{l,2}C_{f_y}}{\mu}+L_{f,1}\right)^2\bar{z}_{t+1}^{-1}\|\mathbf{y}_t - \mathbf{1}y^*(\bar{x}_t)\|^2 \nonumber \\
&\quad +\frac{8\bar{L}_r^2(\mu + C_{l_{yy}})\sqrt{1+\sigma_z^2}}{n\mu C_{l_{yy}}}\|\mathbf{x}_t-\mathbf{1}\bar{x}_t\|^2+  \frac{2(\mu + C_{l_{yy}})^2L_v^2}{\mu C_{l_{yy}}\gamma_v\bar{z}_{t+1}^{-1}} \|\bar{x}_t - \bar{x}_{t+1}\|^2,\label{eq75}
\end{align}
where (a) multiplies both sides of \cref{eq74} by $2(\mu + C_{l_{yy}})$ and uses $\hat{\lambda}_{t+1} \leq 1$; (b) uses $\bar{z}_{t+1} \geq \bar{m}^v_{t+1} \geq C_{m^v} \geq \frac{2\mu C_{l_{yy}}\gamma_v}{\mu + C_{l_{yy}}}$. Take summation of \cref{eq75}, then we have:
\begin{align}
&\sum_{k=k_3}^t \frac{\gamma_v\bar{z}_{k+1}^{-1}\sqrt{1+\sigma_z^2}}{n} \|\nabla_v R(\mathbf{x}_k, \mathbf{y}_k, \mathbf{v}_k)\|^2 \nonumber\\
&\leq \sum_{k=k_3-1}^t \frac{\gamma_v\bar{z}_{k+1}^{-1}\sqrt{1+\sigma_z^2}}{n} \|\nabla_v R(\mathbf{x}_k, \mathbf{y}_k, \mathbf{v}_k)\|^2\nonumber \\
&\leq \frac{2(\mu + C_{l_{yy}})}{n}\|\mathbf{v}_{k_3-1} - \mathbf{1}v^*(\bar{x}_{k_3-1})\|^2 + \frac{2(\mu + C_{l_{yy}})^2L_v^2}{\mu C_{l_{yy}}\gamma_v}  \sum_{k=k_3-1}^t \bar{z}_{k+1}\|\bar{x}_k - \bar{x}_{k+1}\|^2\nonumber\\
&\quad + \frac{1}{n}\left(8\gamma_v\sqrt{1+\sigma_z^2}+\frac{4\gamma_v\sqrt{1+\sigma_z^2}(\mu + C_{l_{yy}})^2}{\mu C_{l_{yy}}} \right)\left(\frac{nL_{l,2}C_{f_y}}{\mu}+L_{f,1}\right)^2 \nonumber\\
&\quad\cdot\sum_{k=k_3-1}^t\bar{z}_{k+1}^{-1}\|\mathbf{y}_k - \mathbf{1}y^*(\bar{x}_k)\|^2 +\frac{8\bar{L}_r^2(\mu + C_{l_{yy}})\sqrt{1+\sigma_z^2}}{n\mu C_{l_{yy}}} \sum_{k=k_3-1}^t\|\mathbf{x}_k-\mathbf{1}\bar{x}_k\|^2 \nonumber\\
&\leq  \frac{2(\mu + C_{l_{yy}})}{n}\|\mathbf{v}_{k_3-1} - \mathbf{1}v^*(\bar{x}_{k_3-1})\|^2+\frac{8\bar{L}_r^2(\mu + C_{l_{yy}})\sqrt{1+\sigma_z^2}}{n\mu C_{l_{yy}}} \sum_{k=0}^t\|\mathbf{x}_k-\mathbf{1}\bar{x}_k\|^2 \nonumber\\
&\quad+ \frac{2(\mu + C_{l_{yy}})^2L_v^2}{\mu C_{l_{yy}}\gamma_v}  \sum_{k=k_3-1}^t \bar{z}_{k+1} \left\| 
\gamma_x \bar{q}_{k+1}^{-1} \frac{\mathbf{1}^\top}{n} \bar{\nabla} F(\mathbf{x}_k,\mathbf{y}_k,\mathbf{v}_k) 
- \gamma_x \frac{\left(\tilde{\mathbf{q}}_{k+1}^{-1}\right)^\top}{n} \bar{\nabla} F(\mathbf{x}_k,\mathbf{y}_k,\mathbf{v}_k)
\right\|^2 \nonumber \\
&\quad +\frac{1}{n}\left(8\gamma_v\sqrt{1+\sigma_z^2}+\frac{4\gamma_v\sqrt{1+\sigma_z^2}(\mu + C_{l_{yy}})^2}{\mu C_{l_{yy}}} \right)\left(\frac{nL_{l,2}C_{f_y}}{\mu}+L_{f,1}\right)^2\nonumber\\
&\quad\cdot\sum_{k=k_3-1}^t\bar{z}_{k+1}^{-1}\|\mathbf{y}_k - \mathbf{1}y^*(\bar{x}_k)\|^2 \nonumber \\
&\overset{(a)}{\leq} \frac{2(\mu + C_{l_{yy}})}{n}\|\mathbf{v}_{k_3-1} - \mathbf{1}v^*(\bar{x}_{k_3-1})\|^2 \nonumber\\
&\quad+ \frac{2\gamma_x^2L_v^2(1+\zeta_q^2)(\mu + C_{l_{yy}})^2}{n\mu C_{l_{yy}}\gamma_v}  \sum_{k=k_3-1}^t \bar{z}_{k+1}\bar{q}_{k+1}^{-2} \left\| \bar{\nabla} F(\mathbf{x}_k,\mathbf{y}_k,\mathbf{v}_k)
\right\|^2 \nonumber \\
&\quad + \frac{1}{n}\left(8\gamma_v\sqrt{1+\sigma_z^2}+\frac{4\gamma_v\sqrt{1+\sigma_z^2}(\mu + C_{l_{yy}})^2}{\mu C_{l_{yy}}} \right)\left(\frac{nL_{l,2}C_{f_y}}{\mu}+L_{f,1}\right)^2\nonumber\\
&\quad\cdot\sum_{k=k_3-1}^t[\bar{m}^y_{k+1}]^{-1}\|\mathbf{y}_k - \mathbf{1}y^*(\bar{x}_k)\|^2+ \frac{16\Delta_0^x  \bar{L}_r^2(\mu + C_{l_{yy}})\sqrt{1+\sigma_z^2}}{n\mu C_{l_{yy}}(1 - \rho_W)} \nonumber \\
&\quad+ \frac{64 \gamma_x^2 \rho_W \bar{L}_r^2(\mu + C_{l_{yy}})(1 + \zeta_q^2)\sqrt{1+\sigma_z^2}}{n\mu C_{l_{yy}}(1 - \rho_W)^2} \sum_{k=0}^{t} 
 \bar{q}_{k+1}^{-2} \|\bar{\nabla}F(\mathbf{x}_k, \mathbf{y}_k, \mathbf{v}_k)\|^2 \nonumber\\
&\overset{(b)}{\leq}  \frac{2(\mu + C_{l_{yy}})}{n}\|\mathbf{v}_{k_3-1} - \mathbf{1}v^*(\bar{x}_{k_3-1})\|^2 \nonumber\\
&\quad+\frac{16 \bar{L}_r^2(\mu + C_{l_{yy}})\sqrt{1+\sigma_z^2}}{n\mu C_{l_{yy}}}\left(\frac{\Delta_0^x}{1-\rho_W}+\frac{4\gamma_x^2\rho_WC_{m^x}^2(1+\zeta_q^2)}{\bar{q}_0^2(1-\rho_W)^2}\right) \nonumber\\
&\quad+ \left(\frac{64 \gamma_x^2 \rho_W \bar{L}_r^2(\mu + C_{l_{yy}})(1 + \zeta_q^2)\sqrt{1+\sigma_z^2}}{n\mu C_{l_{yy}}(1 - \rho_W)^2} + \frac{2\gamma_x^2L_v^2(1+\zeta_q^2)(\mu + C_{l_{yy}})^2}{n\mu C_{l_{yy}}C_{m^v}\gamma_v}\right)\nonumber\\
&\quad\cdot\sum_{k=\min\{k_1-1,k_3-1\}}^{t} 
 \frac{\| \bar{\nabla} F(\mathbf{x}_{k},\mathbf{y}_{k},\mathbf{v}_{k})\|^2}{[\bar{m}^x_{k+1}]^2}+ \frac{1}{n}\left(8\gamma_v\sqrt{1+\sigma_z^2}+\frac{4\gamma_v\sqrt{1+\sigma_z^2}(\mu + C_{l_{yy}})^2}{\mu C_{l_{yy}}} \right)\nonumber\\
&\quad \cdot\left(\frac{nL_{l,2}C_{f_y}}{\mu}+L_{f,1}\right)^2 \frac{1}{\mu^2}\sum_{k=k_3-1}^t\frac{\|\nabla_y L(\mathbf{x}_k,\mathbf{y}_k)\|^2}{\bar{m}_{k+1}^y},\label{eq76}
\end{align}

where (a) uses \cref{Lemmanewx} and (b) refers to \cref{Assumption1} and \(\|\bar{\nabla} f_i({x}_{i,k_1-1}, {y}_{i,k_1-1},{v}_{i,k_1-1})\|^2 \leq [m^x_{i,k_1}]^2 \leq C_{m^x}^2\). Further, the approximation term $\|\mathbf{v}_{k_3-1} - \mathbf{1}v^*(\bar{x}_{k_3-1})\|^2$ satisfies:
\begin{align}
    &\|\mathbf{v}_{k_3-1} - \mathbf{1}v^*(\bar{x}_{k_3-1})\|^2\nonumber\\
    &\overset{(a)}{\leq}\frac{1}{\mu^2}\sum_{i=1}^n \|\nabla_v r_i(x_{i,k_3-1}, y_{i,k_3-1}, v_{i,k_3-1}) - \nabla_v r_i(x_{i,k_3-1}, y_{i,k_3-1}, v^*(\bar{x}_{k_3-1}))\|^2\nonumber\\
     &{\leq}\frac{2}{\mu^2}\sum_{i=1}^n \|\nabla_v r_i(x_{i,k_3-1}, y^*(\bar{x}_{k_3-1}), v^*(\bar{x}_{k_3-1})) - \nabla_v r_i(x_{i,k_3-1}, y_{i,k_3-1}, v^*(\bar{x}_{k_3-1}))\|^2 \nonumber\\
&\quad + \frac{2}{\mu^2}\sum_{i=1}^n \|\nabla_v r_i(x_{i,k_3-1}, y_{i,k_3-1}, v_{i,k_3-1}) - \nabla_v r_i(x_{i,k_3-1}, y^*(\bar{x}_{k_3-1}), v^*(\bar{x}_{k_3-1}))\|^2 \nonumber\\
&\overset{(b)}{\leq}  \frac{2}{\mu^2}\left(\frac{nL_{l,2}C_{f_y}}{\mu}+L_{f,1}\right)^2\!\!\|\mathbf{y}_{k_3-1}\!-\mathbf{1}y^*(\bar{x}_{k_3-1})\|^2 + \frac{4}{\mu^2}\sum_{i=1}^n \|\nabla_v r_i(x_{i,k_3-1}, y_{i,k_3-1}, v_{i,k_3-1})\|^2\nonumber\\
&\quad+ \frac{4\bar{L}_r^2}{\mu^2}\|\mathbf{x}_{k_3-1}-\mathbf{1}\bar{x}_{k_3-1}\|^2\nonumber\\
&\overset{(c)}{\leq} \frac{2}{\mu^2}\left(\frac{nL_{l,2}C_{f_y}}{\mu}+L_{f,1}\right)^2\!\!\|\mathbf{y}_{k_3-1}\!-\mathbf{1}y^*(\bar{x}_{k_3-1})\|^2 + \frac{4}{\mu^2}\sum_{i=1}^n \|\nabla_v r_i(x_{i,k_3-1}, y_{i,k_3-1}, v_{i,k_3-1})\|^2\nonumber\\
&\quad+\frac{8\bar{L}_r^2}{\mu^2}\left(\frac{\Delta_0^x}{1-\rho_W}+  \frac{4 \gamma_x^2 \rho_WC_{m^x}^2 (1 + \zeta_q^2)}{\bar{q}_0^2(1 - \rho_W)^2} \right)
,\label{eqnew3}
\end{align}
where (a) uses the strong convexity; (b) follows from \cref{Lemma3} and $\nabla_v r_i(\bar{x}_{k_3-1}, y^*(\bar{x}_{k_3-1}), v^*(\bar{x}_{k_3-1})) = 0$; (c) refers to \cref{Lemmanewx}. By plugging \cref{eqnew3} into \cref{eq76}, we have:
\begin{align}
&\sum_{k=k_3}^t \frac{\gamma_v\bar{z}_{k+1}^{-1}\sqrt{1+\sigma_z^2}}{n} \|\nabla_v R(\mathbf{x}_k, \mathbf{y}_k, \mathbf{v}_k)\|^2 \nonumber\\
&\leq  \frac{4(\mu + C_{l_{yy}})}{n\mu^2}\left(\frac{nL_{l,2}C_{f_y}}{\mu}+L_{f,1}\right)^2\!\!\|\mathbf{y}_{k_3-1}\!-\mathbf{1}y^*(\bar{x}_{k_3-1})\|^2 \nonumber\\
&\quad+ \frac{8(\mu + C_{l_{yy}})}{n\mu^2}\sum_{i=1}^n \|\nabla_v r_i(x_{i,k_3-1}, y_{i,k_3-1}, v_{i,k_3-1})\|^2 \nonumber\\
&\quad+ \left(\frac{64 \gamma_x^2 \rho_W \bar{L}_r^2(\mu + C_{l_{yy}})(1 + \zeta_q^2)\sqrt{1+\sigma_z^2}}{n\mu C_{l_{yy}}(1 - \rho_W)^2} + \frac{2\gamma_x^2L_v^2(1+\zeta_q^2)(\mu + C_{l_{yy}})^2}{n\mu C_{l_{yy}}C_{m^v}\gamma_v}\right)\nonumber\\
&\quad\cdot\sum_{k=\min\{k_1-1,k_3-1\}}^{t}  
 \frac{\| \bar{\nabla} F(\mathbf{x}_{k},\mathbf{y}_{k},\mathbf{v}_{k})\|^2}{[\bar{m}^x_{k+1}]^2}+ \frac{1}{n}\left(8\gamma_v\sqrt{1+\sigma_z^2}+\frac{4\gamma_v\sqrt{1+\sigma_z^2}(\mu + C_{l_{yy}})^2}{\mu C_{l_{yy}}} \right)\nonumber\\
&\quad \cdot\left(\frac{nL_{l,2}C_{f_y}}{\mu}+L_{f,1}\right)^2 \frac{1}{\mu^2}\sum_{k=k_3-1}^t\frac{\|\nabla_y L(\mathbf{x}_k,\mathbf{y}_k)\|^2}{\bar{m}_{k+1}^y}\nonumber\\
&\quad +\left(\frac{16\Delta_0^x}{1-\rho_W}+\frac{64\gamma_x^2\rho_WC_{m^x}^2(1+\zeta_q^2)}{\bar{q}_0^2(1-\rho_W)^2}\right)\left(\frac{ \bar{L}_r^2(\mu + C_{l_{yy}})\sqrt{1+\sigma_z^2}}{n\mu C_{l_{yy}}}+\frac{\bar{L}_r^2(\mu+C_{l_{yy}})}{n\mu^2}\right).\label{eqnew4}
\end{align}

Our next step is bounding $\|\mathbf{y}_{k_3-1}\!-\mathbf{1}y^*(\bar{x}_{k_3-1})\|^2$ on the RHS of \cref{eqnew4} in two cases. The first case is $\bar{m}^y_{k_3} \leq C_{m^y}$. In this case, by using strong convexity of $l$ and the definition of $\bar{m}^y_{k_3}$, we can easily have:
\begin{align}
\|\mathbf{y}_{k_3-1}\!-\mathbf{1}y^*(\bar{x}_{k_3-1})\|^2 &\leq
\frac{1}{\mu^2} \|\nabla_y l(\mathbf{x}_{k_3-1}, \mathbf{y}_{k_3-1})\|^2 \leq \frac{nC_{m^y}^2}{\mu^2}.\label{eq77}
\end{align}

For the second case, when $\bar{m}^y_{k_3} > C_{m^y}$, note that $k_2$ exists and $k_3 > k_2$ by \cref{Lemma6}. By plugging $\bar{\lambda}_{k_3-1} := \frac{\gamma_y\bar{u}^{-1}_{k_3-1} \mu L_{l,1}}{\mu + L_{l,1}}$ into \cref{eq64} and noting $\bar{\lambda}_{k_3-1} \leq 1$, we have:
\begin{align}
&\|\mathbf{y}_{k_3-1}\!-\mathbf{1}y^*(\bar{x}_{k_3-1})\|^2 \notag\\
&\leq \|\mathbf{y}_{k_3-2}\!-\mathbf{1}y^*(\bar{x}_{k_3-2})\|^2 + \frac{2(\mu + L_{l,1})}{
\gamma_y\bar{u}^{-1}_{k_3-1} \mu L_{l,1}} \|y^*(\bar{x}_{k_3-2}) -y^*(\bar{x}_{k_3-1})\|^2 \notag \\
&\quad +  \left(\frac{2\gamma_y\bar{u}_{t+1}^{-1}L_{l,1}(\mu+L_{l,1})\sqrt{1+\sigma_u^2}}{n\mu }+\frac{4\gamma_y\bar{u}_{t+1}^{-1}L_{l,1}^2\sqrt{1+\sigma_u^2}}{n(\mu + L_{l,1})}\right) \|\mathbf{x}_{k_3-2}-\mathbf{1}\bar{x}_{k_3-2}\|^2\nonumber\\
&\overset{(a)}{\leq} \|\mathbf{y}_{k_3-2}\!-\mathbf{1}y^*(\bar{x}_{k_3-2})\|^2 +  \frac{2(\mu + L_{l,1})L_y^2}{\gamma_y\bar{u}^{-1}_{k_3-1} \mu L_{l,1}} \|\bar{x}_{k_3-2} - \bar{x}_{k_3-1}\|^2 \notag \\
&\quad +   \left(\frac{2(\mu+L_{l,1})^2\sqrt{1+\sigma_u^2}}{n\mu^2 }+\frac{4L_{l,1}\sqrt{1+\sigma_u^2}}{n\mu}\right)\sum_{k=0}^{k_3-2}\|\mathbf{x}_{k}-\mathbf{1}\bar{x}_{k}\|^2\nonumber\\
&\overset{(b)}{\leq}  \|\mathbf{y}_{k_2-1}\!-\!\mathbf{1}y^*(\bar{x}_{k_2-1})\|^2 \!\!+\! \frac{4\gamma_x^2L_y^2(1\!+\!\zeta_q^2)(\mu \!+\! L_{l,1})}{n\gamma_y \mu L_{l,1}}  \bar{u}_{k_3-1}
\bar{q}_{k_3-1}^{-2} \| \bar{\nabla} F(\mathbf{x}_{k_3-2},\mathbf{y}_{k_3-2},\mathbf{v}_{k_3-2})\|^2 \notag \\
&\quad + \frac{16 \gamma_x^2 \rho_W (1 + \zeta_q^2)}{(1 - \rho_W)^2} \left(\frac{(\mu+L_{l,1})^2\sqrt{1+\sigma_u^2}}{n\mu^2 }\!+\!\frac{2L_{l,1}\sqrt{1+\sigma_u^2}}{n\mu }\right)\sum_{k=0}^{k_3-2} 
 \bar{q}_{k+1}^{-2} \|\bar{\nabla}F(\mathbf{x}_k, \mathbf{y}_k, \mathbf{v}_k)\|^2\nonumber\\
&\quad + \frac{4\Delta_0^x(\mu+L_{l,1})^2\sqrt{1+\sigma_u^2}}{n\mu^2(1-\rho_W) }+\frac{8\Delta_0^xL_{l,1}\sqrt{1+\sigma_u^2}}{n\mu(1-\rho_W) }\nonumber\\
&\overset{(c)}{\leq}\frac{nC_{m^y}^2}{\mu^2}+\left(\frac{4\Delta_0^x}{1-\rho_W}+\frac{16 \gamma_x^2 \rho_WC_{m^x}^2 (1 + \zeta_q^2)}{\bar{q}_0^2(1 - \rho_W)^2} \right)\left(\frac{(\mu+L_{l,1})^2\sqrt{1+\sigma_u^2}}{n\mu^2 }+\frac{2L_{l,1}\sqrt{1+\sigma_u^2}}{n\mu }\right)\notag \\
&\quad + \left[\frac{16 \gamma_x^2 \rho_W (1 + \zeta_q^2)}{\bar{z}_0^2(1 - \rho_W)^2} \left(\frac{(\mu+L_{l,1})^2\sqrt{1+\sigma_u^2}}{n\mu^2 }+\frac{2L_{l,1}\sqrt{1+\sigma_u^2}}{n\mu }\right)\right.\nonumber\\
&\quad +\left.\frac{4\gamma_x^2L_y^2(1+\zeta_q^2)(\mu + L_{l,1})}{n\bar{z}_0\gamma_y \mu L_{l,1}}\right]\sum_{k=\min\{k_1-1,k_2-1\}}^{k_3-2} 
 \frac{\| \bar{\nabla} F(\mathbf{x}_{k},\mathbf{y}_{k},\mathbf{v}_{k})\|^2}{[\bar{m}^x_{k+1}]^2},\label{eq78}
\end{align}
where (a) uses $L_y$-Lipschitz continuous and \(\bar{m}^y_{k_3+1} > C_{m^y} \geq \frac{\gamma_y \mu L_{l,1}}{\mu + L_{l,1}}\); (b) refers to \cref{Lemmanewx}; (c) follows from \cref{eq77} by replacing $k_3$ with $k_2$ since $\bar{m}_{k_2}^y \leq C_{m^y}$ and \(\|\bar{\nabla} f_i({x}_{i,k_1-1}, {y}_{i,k_1-1},{v}_{i,k_1-1})\|^2 \leq [m^x_{i,k_1}]^2 \leq C_{m^x}^2\).
By combining \cref{eq77} and \cref{eq78}, we obtain a general upper bound for $\|\mathbf{y}_{k_3-1}\!-\mathbf{1}y^*(\bar{x}_{k_3-1})\|^2$ as:
\begin{align}
&\|\mathbf{y}_{k_3-1}\!-\mathbf{1}y^*(\bar{x}_{k_3-1})\|^2\notag\\ &\leq
\frac{nC_{m^y}^2}{\mu^2}+\left(\frac{4\Delta_0^x}{1-\rho_W}+\frac{16 \gamma_x^2 \rho_WC_{m^x}^2 (1 + \zeta_q^2)}{\bar{q}_0^2(1 - \rho_W)^2} \right)\left(\frac{(\mu+L_{l,1})^2\sqrt{1+\sigma_u^2}}{n\mu^2 }+\frac{2L_{l,1}\sqrt{1+\sigma_u^2}}{n\mu }\right)\notag \\
&\quad + \left[\frac{16 \gamma_x^2 \rho_W (1 + \zeta_q^2)}{\bar{z}_0^2(1 - \rho_W)^2} \left(\frac{(\mu+L_{l,1})^2\sqrt{1+\sigma_u^2}}{n\mu^2 }+\frac{2L_{l,1}\sqrt{1+\sigma_u^2}}{n\mu }\right)\right.\nonumber\\
&\quad +\left.\frac{4\gamma_x^2L_y^2(1+\zeta_q^2)(\mu + L_{l,1})}{n\bar{z}_0\gamma_y \mu L_{l,1}}\right]\sum_{k=\min\{k_1-1,k_2-1\}}^{k_3-2} 
 \frac{\| \bar{\nabla} F(\mathbf{x}_{k},\mathbf{y}_{k},\mathbf{v}_{k})\|^2}{[\bar{m}^x_{k+1}]^2},\label{eq79}
\end{align}
where we define $\sum_{t=m}^n p_t = 0$ for any $m > n$ and non-negative sequence $\{p_t\}$. By plugging \cref{eq79} into \cref{eq76} and using $\|\nabla_v r_i(x_{i,k_3-1}, y_{i,k_3-1}, v_{i,k_3-1})\|^2 \leq [\bar{m}^v_{k_3}]^2 \leq C_{m^v}^2$, we have:
\begin{align}
&\sum_{k=k_3}^t \frac{\gamma_v\bar{z}_{k+1}^{-1}\sqrt{1+\sigma_z^2}}{n} \|\nabla_v R(\mathbf{x}_k, \mathbf{y}_k, \mathbf{v}_k)\|^2 \nonumber\\
&\leq 
\frac{4C_{m^y}^2(\mu + C_{l_{yy}})}{\mu^4}\left(\frac{nL_{l,2}C_{f_y}}{\mu}+L_{f,1}\right)^2\nonumber\\
&\quad+ \frac{8(\mu + C_{l_{yy}})C_{m^v}^2}{\mu^2}+ \frac{16(\mu + C_{l_{yy}})}{n\mu^2}\left(\frac{nL_{l,2}C_{f_y}}{\mu}\!+\!L_{f,1}\right)^2\nonumber\\
&\quad \cdot \left(\frac{\Delta_0^x}{1-\rho_W}\!+\!\frac{4 \gamma_x^2 \rho_WC_{m^x}^2 (1 + \zeta_q^2)}{\bar{q}_0^2(1 - \rho_W)^2} \right)\left(\frac{(\mu+L_{l,1})^2\sqrt{1+\sigma_u^2}}{n\mu^2 }\!+\!\frac{2L_{l,1}\sqrt{1+\sigma_u^2}}{n\mu }\right)\notag \\
&\quad+\frac{16(\mu\! +\! C_{l_{yy}})}{n\mu^2}\!\!\left(\frac{nL_{l,2}C_{f_y}}{\mu}\!\!+\!\!L_{f,1}\right)^2\!\!\left[\frac{4 \gamma_x^2 \rho_W (1 \!+\! \zeta_q^2)}{\bar{z}_0^2(1 \!-\! \rho_W)^2} \left(\frac{(\mu\!+\!L_{l,1})^2\sqrt{1\!+\!\sigma_u^2}}{n\mu^2 }\!+\!\frac{2L_{l,1}\sqrt{1\!+\!\sigma_u^2}}{n\mu }\right)\right.\nonumber\\
&\quad +\left.\frac{\gamma_x^2L_y^2(1+\zeta_q^2)(\mu + L_{l,1})}{n\bar{z}_0\gamma_y \mu L_{l,1}}\right]\sum_{k=\min\{k_1-1,k_2-1\}}^{k_3-2} 
 \frac{\| \bar{\nabla} F(\mathbf{x}_{k},\mathbf{y}_{k},\mathbf{v}_{k})\|^2}{[\bar{m}^x_{k+1}]^2}\nonumber\\
&\quad+ \left(\frac{64 \gamma_x^2 \rho_W \bar{L}_r^2(\mu + C_{l_{yy}})(1 + \zeta_q^2)\sqrt{1+\sigma_z^2}}{n\mu C_{l_{yy}}(1 - \rho_W)^2} + \frac{2\gamma_x^2L_v^2(1+\zeta_q^2)(\mu + C_{l_{yy}})^2}{n\mu C_{l_{yy}}C_{m^v}\gamma_v}\right)\nonumber\\
&\quad\cdot\sum_{k=\min\{k_1-1,k_3-1\}}^{t} 
 \frac{\| \bar{\nabla} F(\mathbf{x}_{k},\mathbf{y}_{k},\mathbf{v}_{k})\|^2}{[\bar{m}^x_{k+1}]^2}+ \frac{1}{n}\left(8\gamma_v\sqrt{1+\sigma_z^2}+\frac{4\gamma_v\sqrt{1+\sigma_z^2}(\mu + C_{l_{yy}})^2}{\mu C_{l_{yy}}} \right)\nonumber\\
&\quad \cdot\left(\frac{nL_{l,2}C_{f_y}}{\mu}+L_{f,1}\right)^2 \frac{1}{\mu^2}\sum_{k=k_3-1}^t\frac{\|\nabla_y L(\mathbf{x}_k,\mathbf{y}_k)\|^2}{\bar{m}_{k+1}^y}\nonumber\\
&\quad +\left(\frac{16\Delta_0^x}{1-\rho_W}+\frac{64\gamma_x^2\rho_WC_{m^x}^2(1+\zeta_q^2)}{\bar{q}_0^2(1-\rho_W)^2}\right)\left(\frac{ \bar{L}_r^2(\mu + C_{l_{yy}})\sqrt{1+\sigma_z^2}}{n\mu C_{l_{yy}}}+\frac{\bar{L}_r^2(\mu+C_{l_{yy}})}{n\mu^2}\right).\label{eq80}
\end{align}

Then, the proof is complete.
\end{proof}

\subsection{The Upper Bounds of $\bar{m}^y_t$ and $\bar{z}_t$}
Supported by \cref{Lemma7} and \cref{Lemma8}, we derive upper bounds of $\bar{m}^y_t$ and $\bar{z}_t$.

\begin{lemma}\label{Lemma9} Suppose the total iteration rounds of \cref{alg1} is $T$. Under \cref{Assumption1} and \cref{Assumption2}, if $k_2$ in \cref{Lemma6} exists within $T$ iterations, we have:
\begin{align}
\bar{m}^y_{t+1} \leq 
\begin{cases} 
C_{m^y}, & t < k_2, \\
C_{m^y}+c_0+d_0 \sum_{k=\min\{k_1,k_2\}}^t \frac{\|\bar{\nabla} F(\mathbf{x}_k,\mathbf{y}_k,\mathbf{v}_k)\|^2}{[\bar{m}^x_{k+1}]^2\max \big\{\bar{m}^v_{k+1}, \bar{m}^y_{k+1}\big\}}, 
& t \geq k_2.
\end{cases}\label{eqnewc0d081}
\end{align}
where $c_0, d_0$ are defined as:
\begin{align}
&c_0:= \frac{ 2C_{m^y}^2(\mu \!+\! L_{l,1})}{\mu^2\gamma_y\sqrt{1\!+\!\sigma_u^2}}\! \!+\!\!\left(\frac{8\Delta_0^x}{n(1-\rho_W)}\!\!+\!\!\frac{32 \gamma_x^2 \rho_WC_{m^x}^2 (1 + \zeta_q^2)}{n\bar{q}_0^2(1 - \rho_W)^2} \right)\!\!\left(\frac{(\mu+L_{l,1})^3}{\gamma_y\mu^2}\!\!+\!\!\frac{2L_{l,1}(\mu+L_{l,1})}{\gamma_y\mu}\right)\notag\\
&\quad+ \frac{32 \gamma_x^2 \rho_W(1 + \zeta_q^2)}{n\bar{z}_0(1 - \rho_W)^2} \left(\frac{(\mu+L_{l,1})^3}{\gamma_y\mu^2}+\frac{2L_{l,1}(\mu+L_{l,1})}{\gamma_y\mu}\right)+\frac{8\gamma_x^2  L_y^2\left( 1 + \zeta_q^2 \right) (\mu + L_{l,1})^2}{n\gamma_y^2\mu L_{l,1}\bar{z}_0\sqrt{1+\sigma_u^2}},\notag\\
&d_0:= \left[\frac{32 \gamma_x^2 \rho_W(1 + \zeta_q^2)}{n(1 - \rho_W)^2} \left(\frac{(\mu+L_{l,1})^3}{\gamma_y\mu^2}+\frac{2L_{l,1}(\mu+L_{l,1})}{\gamma_y\mu}\right)+\frac{8\gamma_x^2  L_y^2\left( 1 + \zeta_q^2 \right) (\mu + L_{l,1})^2}{n\gamma_y^2\mu L_{l,1}\sqrt{1+\sigma_u^2}}\right],\label{eq84}
\end{align}

When such $k_2$ does not exist, $\bar{m}^y_{t+1} \leq C_{m^y}$ holds for any $t < T$.
\end{lemma}

\begin{proof}
According to \cref{Lemma6}, the proof can be split into the following three cases:

\textbf{Case 1:} $k_2$ does not exist. In this case, based on \cref{Lemma6}, we have $\bar{m}^y_T \leq C_{m^y}$, and hence $\bar{m}^y_{t+1} \leq C_{m^y}$ for any $t < T$ because $\bar{m}^y_t$ is non-decreasing with $t$.

\textbf{Case 2:} $k_2$ exists and $t < k_2$: In this case, based on \cref{Lemma6}, we have $\bar{m}^y_{t+1} \leq C_{m^y}$.

\textbf{Case 3:} $k_2$ exists and $t \geq k_2$: Using telescoping, we have:
\begin{align}
\bar{m}^y_{t+1} &\overset{(a)}{=} \bar{m}^y_t + \frac{\|\nabla_y L(\mathbf{x}_t, \mathbf{y}_t)\|^2}{n(\bar{m}^y_{t+1} + \bar{m}^y_t)}\nonumber \\
&\leq \bar{m}^y_t + \frac{\|\nabla_y L(\mathbf{x}_t, \mathbf{y}_t)\|^2}{n\bar{m}^y_{t+1}} \nonumber\\
&= \bar{m}^y_{k_2} + \sum_{k=k_2}^{t-1}  \frac{\|\nabla_y L(\mathbf{x}_k, \mathbf{y}_k)\|^2}{n(\bar{m}^y_{t+1} + \bar{m}^y_t)}+\frac{\|\nabla_y L(\mathbf{x}_t, \mathbf{y}_t)\|^2}{n\bar{m}^y_{t+1}} \nonumber\\
&\leq \bar{m}^y_{k_2} + \sum_{k=k_2}^t  \frac{\|\nabla_y L(\mathbf{x}_k, \mathbf{y}_k)\|^2}{n\bar{m}^y_{k+1}} \nonumber \\
&\overset{(b)}{\leq} C_{m^y}+\frac{ 2C_{m^y}^2(\mu + L_{l,1})}{\mu^2\gamma_y\sqrt{1+\sigma_u^2}} \nonumber\\
&\quad+\left(\frac{8\Delta_0^x}{n(1-\rho_W)}+\frac{32 \gamma_x^2 \rho_WC_{m^x}^2 (1 + \zeta_q^2)}{n\bar{q}_0^2(1 - \rho_W)^2} \right)\left(\frac{(\mu+L_{l,1})^3}{\gamma_y\mu^2}+\frac{2L_{l,1}(\mu+L_{l,1})}{\gamma_y\mu}\right)\nonumber\\
&\quad+ \frac{32 \gamma_x^2 \rho_W(1 + \zeta_q^2)}{n\bar{z}_0(1 - \rho_W)^2} \left(\frac{(\mu+L_{l,1})^3}{\gamma_y\mu^2}+\frac{2L_{l,1}(\mu+L_{l,1})}{\gamma_y\mu}\right)+\frac{8\gamma_x^2  L_y^2\left( 1 + \zeta_q^2 \right) (\mu + L_{l,1})^2}{n\gamma_y^2\mu L_{l,1}\bar{z}_0\sqrt{1+\sigma_u^2}} \nonumber\\
&\quad+ \left[\frac{32 \gamma_x^2 \rho_W(1 + \zeta_q^2)}{n(1 - \rho_W)^2} \left(\frac{(\mu+L_{l,1})^3}{\gamma_y\mu^2}+\frac{2L_{l,1}(\mu+L_{l,1})}{\gamma_y\mu}\right)\right.\nonumber\\
&\quad\left.+\frac{8\gamma_x^2  L_y^2\left( 1 + \zeta_q^2 \right) (\mu + L_{l,1})^2}{n\gamma_y^2\mu L_{l,1}\sqrt{1+\sigma_u^2}}\right]\!\! \sum_{k=\min\{k_1,k_2\}}^t \frac{\|\bar{\nabla} F(\mathbf{x}_k,\mathbf{y}_k,\mathbf{v}_k)\|^2}{[\bar{m}^x_{k+1}]^2\max \big\{\bar{m}^v_{k+1}, \bar{m}^y_{k+1}\big\}},\label{eq82}
\end{align}
where (a) employs $(\bar{m}^y_{t+1}+\bar{m}^y_{t})(\bar{m}^y_{t+1}-\bar{m}^y_{t})=[\bar{m}^y_{t+1}]^2-[\bar{m}^y_{t}]^2$ and (b) uses \cref{Lemma7}. Thus, the proof is complete. 
\end{proof}

\begin{lemma}\label{Lemma10} Under \cref{Assumption1} and \cref{Assumption2}, suppose the total iteration rounds of \cref{alg1} is $T$. If at least one of $k_2$ and $k_3$ in \cref{Lemma6} exists, we denote $k_{\text{min}} := \min\{k_2, k_3\}$. Then we have the upper bound of $\bar{z}_t$ as:
\begin{align}
\bar{z}_t \leq
\begin{cases}
C_z, & t \leq k_{\text{min}}, \\
a_1 \log(t) + b_1, & t > k_{\text{min}},
\end{cases}\label{eq83}
\end{align}
where $a_1, b_1$ are defined as:
\begin{align}
a_1:= 6a_0,\quad b_1:= 4a_0 \log \left( 1 + \frac{nC_{l_{xy}} \bar{b} + nC_{f_x} + \bar{m}^x_0}{nC_{l_{xy}} \bar{a}} \right) + 4a_0\log \left( nC_{l_{xy}} \bar{a}\right) + 4a_0 + 2b_0,\label{eq84}
\end{align}
in which we define constants
\begin{align}
\bar{a} &:= \frac{\sqrt{2n}}{\mu}, \quad
\bar{b} := \frac{\sqrt{2n}C_{f_y}}{\mu}, \notag\\
a_0 &:= \left( \frac{1}{\mu^2}\left(8+\frac{4(\mu + C_{l_{yy}})^2}{\mu C_{l_{yy}}} \right)\left(\frac{nL_{l,2}C_{f_y}}{\mu}+L_{f,1}\right)^2+1\right) \notag \\
&\quad \cdot\left(\frac{32 \gamma_x^2 \rho_W(1 + \zeta_q^2)}{n(1 - \rho_W)^2} \left(\frac{(\mu+L_{l,1})^3}{\gamma_y\mu^2}+\frac{2L_{l,1}(\mu+L_{l,1})}{\gamma_y\mu}\right)+\frac{8\gamma_x^2  L_y^2\left( 1 + \zeta_q^2 \right) (\mu + L_{l,1})^2}{n\gamma_y^2\mu L_{l,1}\sqrt{1+\sigma_u^2}}\right)\nonumber\\
&\quad +\frac{16(\mu + C_{l_{yy}})}{n\gamma_v\mu^2\sqrt{1+\sigma_u^2}}\left(\frac{nL_{l,2}C_{f_y}}{\mu}\!+\!L_{f,1}\right)^2\left(\frac{4 \gamma_x^2 \rho_W (1 + \zeta_q^2)}{\bar{z}_0^2(1 - \rho_W)^2} \right.\nonumber\\
&\quad \left.\cdot\left(\frac{(\mu+L_{l,1})^2\sqrt{1+\sigma_u^2}}{n\mu^2 }+\frac{2L_{l,1}\sqrt{1+\sigma_u^2}}{n\mu }\right)+\frac{\gamma_x^2L_y^2(1+\zeta_q^2)(\mu + L_{l,1})}{n\bar{z}_0\gamma_y \mu L_{l,1}}\right)\nonumber\\
&\quad+ \frac{64 \gamma_x^2 \rho_W \bar{L}_r^2(\mu + C_{l_{yy}})(1 + \zeta_q^2)}{n\mu C_{l_{yy}}\gamma_v(1 - \rho_W)^2} + \frac{2\gamma_x^2L_v^2(1+\zeta_q^2)(\mu + C_{l_{yy}})^2}{n\mu C_{l_{yy}}\bar{m}_0^v\gamma_v^2\sqrt{1+\sigma_z^2}}, \notag\\
b_0 &:=C_{m^y} + C_{m^v} +\frac{4C_{m^y}^2(\mu + C_{l_{yy}})}{\mu^4\gamma_v\sqrt{1+\sigma_z^2}}\left(\frac{nL_{l,2}C_{f_y}}{\mu}+L_{f,1}\right)^2 \nonumber\\
&\quad+ \frac{8(\mu + C_{l_{yy}})C_{m^v}^2}{\mu^2\gamma_v\sqrt{1+\sigma_z^2}} + \frac{16(\mu + C_{l_{yy}})}{n\gamma_v\mu^2}\left(\frac{nL_{l,2}C_{f_y}}{\mu}\!+\!L_{f,1}\right)^2 \notag \\
&\quad \cdot \left(\frac{\Delta_0^x}{1-\rho_W}\!+\!\frac{4 \gamma_x^2 \rho_WC_{m^x}^2 (1 + \zeta_q^2)}{\bar{q}_0^2(1 - \rho_W)^2} \right)\left(\frac{(\mu+L_{l,1})^2}{n\mu^2 }\!+\!\frac{2L_{l,1}}{n\mu }\right)\notag \\
&\quad+\left(\frac{16\Delta_0^x}{n(1-\rho_W)}+\frac{64\gamma_x^2\rho_WC_{m^x}^2(1+\zeta_q^2)}{n\bar{q}_0^2(1-\rho_W)^2}\right)\left(\frac{ \bar{L}_r^2(\mu + C_{l_{yy}})}{\mu C_{l_{yy}}\gamma_v}+\frac{\bar{L}_r^2(\mu+C_{l_{yy}})}{\mu^2\gamma_v\sqrt{1+\sigma_z^2}}\right) \notag \\
&\quad  +\left[ \frac{1}{\mu^2}\left(8+\frac{4(\mu + C_{l_{yy}})^2}{\mu C_{l_{yy}}} \right)\left(\frac{nL_{l,2}C_{f_y}}{\mu}+L_{f,1}\right)^2+1\right]\left(\frac{C_{m^y}^2}{\bar{m}^y_0} - {\bar{m}^y_0}\right)  \notag \\
&\quad  +\left[ \frac{1}{\mu^2}\left(8+\frac{4(\mu + C_{l_{yy}})^2}{\mu C_{l_{yy}}} \right)\left(\frac{nL_{l,2}C_{f_y}}{\mu}+L_{f,1}\right)^2+1\right]\left(\frac{ 2C_{m^y}^2(\mu + L_{l,1})}{\mu^2\gamma_y\sqrt{1+\sigma_u^2}} \right.\notag\\
&\quad+\left(\frac{8\Delta_0^x}{n(1-\rho_W)}+\frac{32 \gamma_x^2 \rho_WC_{m^x}^2 (1 + \zeta_q^2)}{n\bar{q}_0^2(1 - \rho_W)^2} \right)\left(\frac{(\mu+L_{l,1})^3}{\gamma_y\mu^2}+\frac{2L_{l,1}(\mu+L_{l,1})}{\gamma_y\mu}\right)\nonumber\\
&\quad\left.+ \frac{32 \gamma_x^2 \rho_W(1 + \zeta_q^2)}{n\bar{z}_0(1 - \rho_W)^2} \left(\frac{(\mu+L_{l,1})^3}{\gamma_y\mu^2}+\frac{2L_{l,1}(\mu+L_{l,1})}{\gamma_y\mu}\right)+\frac{8\gamma_x^2  L_y^2\left( 1 + \zeta_q^2 \right) (\mu + L_{l,1})^2}{n\gamma_y^2\mu L_{l,1}\bar{z}_0\sqrt{1+\sigma_u^2}} \right). \label{eq85}
\end{align}

When neither $k_2$ nor $k_3$ exists, we have $\bar{z}_t \leq C_z$ for all $t \leq T$.
\end{lemma}

\begin{proof}
To begin with, we first show the following result as the first two lines of \cref{eq82}: since $\bar{m}^y_t$ and $\bar{m}^v_t$ are positive and increasing monotonically with $t$, we can easily have:
\begin{align}
0& \leq \min\{[\bar{m}^y_{t+1}]^2, [\bar{m}^v_{t+1}]^2\} - \min\{[\bar{m}^y_{t}]^2, [\bar{m}^v_{t}]^2\} \nonumber\\
&= \left([\bar{m}^y_{t+1}]^2 + [\bar{m}^v_{t+1}]^2 - \max\{[\bar{m}^y_{t+1}]^2, [\bar{m}^v_{t+1}]^2\}\right) - \left([\bar{m}^y_{t}]^2 +[\bar{m}^v_{t}]^2 - \max\{[\bar{m}^y_{t}]^2, [\bar{m}^v_{t}]^2\}\right) \nonumber\\
&\overset{(a)}{=} ([\bar{m}^y_{t+1}]^2 + [\bar{m}^v_{t+1}]^2) - ([\bar{m}^y_{t}]^2 +[\bar{m}^v_{t}]^2) - (\bar{z}_{t+1}^2 - \bar{z}_t^2),\label{eq86}
\end{align}
where (a) uses the definition $\bar{z}_t := \max\{\bar{m}^v_t, \bar{m}^y_t\}$. Similar to \cref{eq82}, we have:
\begin{align}
\bar{z}_{t+1}^2 - \bar{z}_t^2
\leq ([\bar{m}^y_{t+1}]^2-[\bar{m}^y_{t}]^2) + ([\bar{m}^v_{t+1}]^2-[\bar{m}^v_{t}]^2)= \frac{\|\nabla_y L(\mathbf{x}_t, \mathbf{y}_t)\|^2}{n} + \frac{\|\nabla_v R(\mathbf{x}_t, \mathbf{y}_t, \mathbf{v}_t)\|^2}{n},\label{eq87}
\end{align}
which indicates that
\begin{align}
\bar{z}_{t+1} 
&\leq \bar{z}_{t} + \frac{\|\nabla_y L(\mathbf{x}_t, \mathbf{y}_t)\|^2}{n(\bar{z}_{t+1}+\bar{z}_{t})} + \frac{\|\nabla_v R(\mathbf{x}_t, \mathbf{y}_t, \mathbf{v}_t)\|^2}{n(\bar{z}_{t+1}+\bar{z}_{t})} \notag \\
&\leq \bar{z}_{t} + \frac{\|\nabla_y L(\mathbf{x}_t, \mathbf{y}_t)\|^2}{n(\bar{m}^y_{t+1}+\bar{m}^y_{t})} + \frac{\|\nabla_v R(\mathbf{x}_t, \mathbf{y}_t, \mathbf{v}_t)\|^2}{n\bar{z}_{t+1}}\notag \\
&\leq \bar{z}_{t} + \frac{\|\nabla_y L(\mathbf{x}_t, \mathbf{y}_t)\|^2}{n\bar{m}^y_{t+1}} + \frac{\|\nabla_v R(\mathbf{x}_t, \mathbf{y}_t, \mathbf{v}_t)\|^2}{n\bar{z}_{t+1}}.\label{eq88}
\end{align}

Note that, to simplify the proof, we define $\sum_{t=m}^n p_t = 0$ for any $m > n$ and non-negative sequence $\{p_t\}$.
According to the definitions of $k_2$ and $k_3$ in \cref{Lemma6}, the proof can be split into the following four cases.

\textbf{Case 1: Neither $k_2$ nor $k_3$ exists.}  
For any $t \in (0, T)$, we can easily have $\bar{z}_t= \max\{\bar{m}^y_{t}, \bar{m}^v_{t}\} \leq \max\{C_{m^y}, C_{m^v}\} \leq C_z$.

\textbf{Case 2: $k_2$ exists but $k_3$ does not.}  
By using the fourth line of \cref{eq82}, for any $t \in (0, T)$, we have:
\begin{align}
\bar{z}_{t+1}  \leq \bar{m}^y_{t+1}+\bar{m}^v_{t+1} \leq C_{m^y} + \sum_{k=k_2}^t  \frac{\|\nabla_y L(\mathbf{x}_k, \mathbf{y}_k)\|^2}{n\bar{m}^y_{k+1}}  + C_{m^v},\label{eq89}
\end{align}
where we take $\sum_{k=k_2}^t \frac{\|\nabla_y L(\mathbf{x}_k, \mathbf{y}_k)\|^2}{\bar{m}^y_{k+1}} = 0$ for any $t < k_2$.

\textbf{Case 3: $k_3$ exists but $k_2$ does not.}  
From the second line of \cref{eq88}, for any $t \in (0, T)$, we have:
\begin{align}
\bar{z}_{t+1} 
&\leq \bar{z}_{t} + \frac{\|\nabla_y L(\mathbf{x}_t, \mathbf{y}_t)\|^2}{n(\bar{m}^y_{t+1}+\bar{m}^y_{t})} + \frac{\|\nabla_v R(\mathbf{x}_t, \mathbf{y}_t, \mathbf{v}_t)\|^2}{n\bar{z}_{t+1}} \notag \\
&\leq \bar{z}_{k_3} + \sum_{k=k_3}^t  \frac{\|\nabla_y L(\mathbf{x}_k, \mathbf{y}_k)\|^2}{n(\bar{m}^y_{k+1}+\bar{m}^y_{k})} + \sum_{k=k_3}^t \frac{\|\nabla_v R(\mathbf{x}_k, \mathbf{y}_k, \mathbf{v}_k)\|^2}{n\bar{z}_{k+1}}  \notag \\
&\leq \bar{m}^y_{k_3} + \bar{m}^v_{k_3} + \sum_{k=k_3}^t  \frac{\|\nabla_y L(\mathbf{x}_k, \mathbf{y}_k)\|^2}{n(\bar{m}^y_{k+1}+\bar{m}^y_{k})} + \sum_{k=k_3}^t \frac{\|\nabla_v R(\mathbf{x}_k, \mathbf{y}_k, \mathbf{v}_k)\|^2}{n\bar{z}_{k+1}} \notag \\
&\overset{(a)}{=} \bar{m}^y_{t+1} + \bar{m}^v_{k_3} + \sum_{k=k_3}^t \frac{\|\nabla_v R(\mathbf{x}_k, \mathbf{y}_k, \mathbf{v}_k)\|^2}{n\bar{z}_{k+1}} \notag \\
&\leq C_{m^y} + C_{m^v} + \sum_{k=k_3}^t \frac{\|\nabla_v R(\mathbf{x}_k, \mathbf{y}_k, \mathbf{v}_k)\|^2}{n\bar{z}_{k+1}},\label{eq90}
 \end{align}
where we take $\sum_{k=k_3}^t \frac{\|\nabla_v  R(\mathbf{x}_k, \mathbf{y}_k, \mathbf{v}_k)\|^2}{\bar{z}_{k+1}} = 0$ for any $t < k_3$; (a) uses the first line of \cref{eq82}.

\textbf{Case 4: Both $k_2$ and $k_3$ exist.}  
From the third line of \cref{eq90}, for any $t \in (0, T)$, we have:
\begin{align}
\bar{z}_{t+1} 
&\leq \bar{m}^y_{k_3} + \bar{m}^v_{k_3} + \sum_{k=k_3}^t  \frac{\|\nabla_y L(\mathbf{x}_k, \mathbf{y}_k)\|^2}{n\bar{m}^y_{k+1}} + \sum_{k=k_3}^t \frac{\|\nabla_v R(\mathbf{x}_k, \mathbf{y}_k, \mathbf{v}_k)\|^2}{n\bar{z}_{k+1}} \notag \\
&\overset{(a)}{\leq} \bar{m}^y_{k_2}  + \sum_{k=k_2}^{k_3-1} \frac{\|\nabla_y L(\mathbf{x}_k, \mathbf{y}_k)\|^2}{n\bar{m}^y_{k+1}}\nonumber\\
&\quad+ C_{m^v} + \sum_{k=k_3}^t \frac{\|\nabla_y L(\mathbf{x}_k, \mathbf{y}_k)\|^2}{n\bar{m}^y_{k+1}} + \sum_{k=k_3}^t\frac{\|\nabla_v R(\mathbf{x}_k, \mathbf{y}_k, \mathbf{v}_k)\|^2}{n\bar{z}_{k+1}} \notag \\
&= C_{m^y} + C_{m^v} + \sum_{k=k_2}^t \frac{\|\nabla_y L(\mathbf{x}_k, \mathbf{y}_k)\|^2}{n\bar{m}^y_{k+1}} + \sum_{k=k_3}^t \frac{\|\nabla_v R(\mathbf{x}_k, \mathbf{y}_k, \mathbf{v}_k)\|^2}{n\bar{z}_{k+1}},\label{eq91}
\end{align}
where (a) uses the fourth line of \cref{eq82}; we take $\sum_{k=k_2}^{k_3-1} \frac{\|\nabla_y L(\mathbf{x}_k, \mathbf{y}_k)\|^2}{\bar{m}^y_{k+1}} = 0$ when $k_2 \geq k_3$, $\sum_{k=k_2}^t \frac{\|\nabla_y L(\mathbf{x}_k, \mathbf{y}_k)\|^2}{\bar{m}^y_{k+1}} = 0$ for any $t < k_2$, and $\sum_{k=k_3}^t \frac{\|\nabla_v R(\mathbf{x}_k, \mathbf{y}_k, \mathbf{v}_k)\|^2}{\bar{z}_{k+1}} = 0$ for any $t < k_3$. 
It is easy to see that the upper bound of $\bar{z}_{t+1}$ in \cref{eq91} is the largest among all cases. Thus, in the remaining proof, we only explore the upper bound of $\bar{z}_t$ in Case 4.

To further explore the bound of $\bar{z}_t$, we need to use some auxiliary results and bounds. So we split them into three parts as follows:

\textbf{Part I: An Auxiliary Bound of $\sum \frac{\|\bar{\nabla} F(\mathbf{x}_t,\mathbf{y}_t,\mathbf{v}_t)\|^2}{[\bar{m}^x_{t+1}]^2}$.}  
To further explore Case 4, we begin with a common term $\sum_{k=k_0}^t \frac{\|\bar{\nabla} F(\mathbf{x}_k,\mathbf{y}_k,\mathbf{v}_k)\|^2}{[\bar{m}^x_{k+1}]^2}$ for any $k_0 \leq t$. By the strong convexity of $l$ in \cref{Assumption1}, we have:
\begin{align}
\sum_{k=1}^t \frac{\mu^2}{n} \|\mathbf{v}_k\|^2 
&\leq \sum_{k=1}^t \frac{\|\nabla_y \nabla_y L(\mathbf{x}_k, \mathbf{y}_k) \mathbf{v}_k\|^2}{n} \notag \\
&\leq \sum_{k=1}^t \frac{2 \|\nabla_y \nabla_y L(\mathbf{x}_k, \mathbf{y}_k) \mathbf{v}_k - \nabla_y F(\mathbf{x}_k, \mathbf{y}_k)\|^2}{n} 
+ \sum_{k=1}^t \frac{2 \|\nabla_y F(\mathbf{x}_k, \mathbf{y}_k)\|^2}{n} \notag \\
&= \sum_{k=1}^t \frac{2 \|\nabla_v R(\mathbf{x}_k, \mathbf{y}_k, \mathbf{v}_k)\|^2}{n} + \sum_{k=1}^t \frac{2 \|\nabla_y F(\mathbf{x}_k, \mathbf{y}_k)\|^2}{n} \notag \\
&\leq 2 [\bar{m}^v_{t+1}]^2 + 2t C_{f_y}^2,\label{eq92}
\end{align}
which indicates that for any $t \geq 0$, $\|\mathbf{v}_t\|$ can be bounded as:
\begin{align}
\|\mathbf{v}_t\| \leq \frac{\sqrt{2n [\bar{m}^v_{t+1}]^2 + 2nt C_{f_y}^2}}{\mu}\leq \frac{\sqrt{2n [\bar{z}_{t+1}]^2 + 2nt C_{f_y}^2}}{\mu} \leq \frac{\sqrt{2n}\left( \bar{z}_{t+1} + \sqrt{t} C_{f_y} \right)}{\mu}.\label{eq93}
\end{align}
Then we have:
\begin{equation}
    \|\mathbf{v}_t\| \leq \frac{\sqrt{2n}}{\mu} \bar{z}_{t+1} + \frac{\sqrt{2n}C_{f_y}}{\mu} \sqrt{t} =: \bar{a} \bar{z}_{t+1} + \bar{b} \sqrt{t},\label{eq94}
\end{equation}
where $\bar{a}$ and $\bar{b}$ refer to \cref{eq85}. According to \cref{Lemma1}, since $\bar{m}^x_0 \geq 1$, for any integer $t > 0$, we have:
\begin{align}
&\sum_{k=k_0}^t \frac{\| \bar{\nabla} F(\mathbf{x}_{k},\mathbf{y}_{k},\mathbf{v}_{k})\|^2}{[\bar{m}^x_{k+1}]^2}\notag\\
&\leq \sum_{k=0}^t \frac{\| \bar{\nabla} F(\mathbf{x}_{k},\mathbf{y}_{k},\mathbf{v}_{k})\|^2}{[\bar{m}^x_{k+1}]^2} \notag \\
&\leq \log \left( \sum_{k=0}^t \|\bar{\nabla} F(\mathbf{x}_{k},\mathbf{y}_{k},\mathbf{v}_{k})\|^2 + [\bar{m}^x_0]^2 \right) + 1 \notag \\
&\leq \log \left( \sum_{k=0}^t \left(nC_{l_{xy}} \bar{a} \bar{z}_{k+1} + nC_{l_{xy}} \bar{b} \sqrt{k} + nC_{f_x} \right)^2 + [\bar{m}^x_0]^2 \right) + 1 \notag \\
&\leq \log \left( \left(\sum_{k=0}^t nC_{l_{xy}} \bar{a} \bar{z}_{k+1} +n C_{l_{xy}} \bar{b} \sqrt{k} +nC_{f_x} + \bar{m}^x_0 \right)^2\right) + 1 \notag \\
&=2\log \left( \sum_{k=0}^t nC_{l_{xy}} \bar{a} \bar{z}_{k+1} + nC_{l_{xy}} \bar{b} \sqrt{k} + nC_{f_x} + \bar{m}^x_0 \right) + 1 \notag \\
&\leq 2 \log \left((t+1)\left( nC_{l_{xy}} \bar{a} \bar{z}_{t+1} + nC_{l_{xy}} \bar{b} \sqrt{t} + nC_{f_x} + \bar{m}^x_0 \right)\right) + 1 \notag \\
&\leq 2 \log(t + 1) + 2 \log \left(\left( nC_{l_{xy}} \bar{a} \bar{z}_{t+1} + nC_{l_{xy}} \bar{b} + nC_{f_x} + \bar{m}^x_0 \right)\sqrt{t}\right) + 1 \notag \\
&\leq 3 \log(t + 1) + 2 \log \left( nC_{l_{xy}} \bar{a} \bar{z}_{t+1} + nC_{l_{xy}} \bar{b} + nC_{f_x} +\bar{m}^x_0 \right) + 1,\label{eq95}
\end{align}
here we obtain the upper bound of $\sum_{k=k_0}^t \frac{\| \bar{\nabla} F(\mathbf{x}_{k},\mathbf{y}_{k},\mathbf{v}_{k})\|^2}{[\bar{m}^x_{k+1}]^2}$ for any $k_0 \leq t$ in \cref{eq95}. Part I is completed.

\textbf{Part II: A More General Bound of} $\sum\frac{\|\nabla_y L(\mathbf{x}_t, \mathbf{y}_t)\|^2}{\bar{m}^y_{t+1}}$.

In \cref{Lemma7}, we show the bound of $\sum_{k=k_2}^t \frac{\|\nabla_y L(\mathbf{x}_k, \mathbf{y}_k)\|^2}{\bar{m}^y_{k+1}}$
when $k_2$ exists. In Part II, we further provide a rough bound of $\sum_{k=\tilde{k}}^t \frac{\|\nabla_y L(\mathbf{x}_k, \mathbf{y}_k)\|^2}{\bar{m}^y_{k+1}}$
for any potential $\tilde{k} \leq T$. Firstly, if $\tilde{k} \geq k_2$, it is easy to have:
\begin{align}
\sum_{k=\tilde{k}}^t \frac{\|\nabla_y L(\mathbf{x}_k, \mathbf{y}_k)\|^2}{\bar{m}^y_{k+1}}
\leq \sum_{k=k_2}^t \frac{\|\nabla_y L(\mathbf{x}_k, \mathbf{y}_k)\|^2}{\bar{m}^y_{k+1}}.\label{eq96}
\end{align}

Secondly, if $\tilde{k} < k_2$, we have:
\begin{align}
\sum_{k=\tilde{k}}^t \frac{\|\nabla_y L(\mathbf{x}_k, \mathbf{y}_k)\|^2}{\bar{m}^y_{k+1}}
&\leq \sum_{k=\tilde{k}}^{k_2-1}\frac{\|\nabla_y L(\mathbf{x}_k, \mathbf{y}_k)\|^2}{\bar{m}^y_{k+1}}
+ \sum_{k=k_2}^t \frac{\|\nabla_y L(\mathbf{x}_k, \mathbf{y}_k)\|^2}{\bar{m}^y_{k+1}}\nonumber \\
&\leq \frac{\sum_{k=\tilde{k}}^{k_2-1}\|\nabla_y L(\mathbf{x}_k, \mathbf{y}_k)\|^2}{\bar{m}^y_{0}}
+ \sum_{k=k_2}^t \frac{\|\nabla_y L(\mathbf{x}_k, \mathbf{y}_k)\|^2}{\bar{m}^y_{k+1}}\nonumber\\
&\leq \frac{n([\bar{m}^y_{k_2}]^2 - [\bar{m}^y_{\tilde{k}}]^2)}{\bar{m}^y_{0}} 
+ \sum_{k=k_2}^t \frac{\|\nabla_y L(\mathbf{x}_k, \mathbf{y}_k)\|^2}{\bar{m}^y_{k+1}}\nonumber \\
&\leq \frac{n(C_{m^y}^2- [\bar{m}^y_0]^2)}{\bar{m}^y_0}+\sum_{k=k_2}^t \frac{\|\nabla_y L(\mathbf{x}_k, \mathbf{y}_k)\|^2}{\bar{m}^y_{k+1}}\nonumber \\
&= \frac{nC_{m^y}^2}{\bar{m}^y_0} - n\bar{m}^y_0
+\sum_{k=k_2}^t \frac{\|\nabla_y L(\mathbf{x}_k, \mathbf{y}_k)\|^2}{\bar{m}^y_{k+1}}.\label{eq97}
\end{align}

Combining these two situations, since $C_{m^y} \geq \bar{m}^y_0$, for any $\tilde{k} \leq t$, we have:
\begin{align}
&\sum_{k=\tilde{k}}^t \frac{\|\nabla_y L(\mathbf{x}_k, \mathbf{y}_k)\|^2}{\bar{m}^y_{k+1}}\nonumber\\
&\leq \frac{nC_{m^y}^2}{\bar{m}^y_0} - n\bar{m}^y_0
+\sum_{k=k_2}^t \frac{\|\nabla_y L(\mathbf{x}_k, \mathbf{y}_k)\|^2}{\bar{m}^y_{k+1}}\nonumber \\
&\leq \frac{nC_{m^y}^2}{\bar{m}^y_0} - n\bar{m}^y_0 + \frac{ 2nC_{m^y}^2(\mu + L_{l,1})}{\mu^2\gamma_y\sqrt{1+\sigma_u^2}}\nonumber\\
&\quad+\left(\frac{8\Delta_0^x}{1-\rho_W}+\frac{32 \gamma_x^2 \rho_WC_{m^x}^2 (1 + \zeta_q^2)}{\bar{q}_0^2(1 - \rho_W)^2} \right)\left(\frac{(\mu+L_{l,1})^3}{\gamma_y\mu^2}+\frac{2L_{l,1}(\mu+L_{l,1})}{\gamma_y\mu}\right)\nonumber\\
&\quad+ \frac{32 \gamma_x^2 \rho_W(1 + \zeta_q^2)}{\bar{z}_0(1 - \rho_W)^2} \left(\frac{(\mu+L_{l,1})^3}{\gamma_y\mu^2}+\frac{2L_{l,1}(\mu+L_{l,1})}{\gamma_y\mu}\right)+\frac{8\gamma_x^2  L_y^2\left( 1 + \zeta_q^2 \right) (\mu + L_{l,1})^2}{\gamma_y^2\mu L_{l,1}\bar{z}_0\sqrt{1+\sigma_u^2}} \nonumber\\
&\quad+ \left[\frac{32 \gamma_x^2 \rho_W(1 + \zeta_q^2)}{(1 - \rho_W)^2} \left(\frac{(\mu+L_{l,1})^3}{\gamma_y\mu^2}+\frac{2L_{l,1}(\mu+L_{l,1})}{\gamma_y\mu}\right)\right.\nonumber\\
&\quad\left.+\frac{8\gamma_x^2  L_y^2\left( 1 + \zeta_q^2 \right) (\mu + L_{l,1})^2}{\gamma_y^2\mu L_{l,1}\sqrt{1+\sigma_u^2}}\right]\!\! \sum_{k=\min\{k_1,k_2\}}^t \frac{\|\bar{\nabla} F(\mathbf{x}_k,\mathbf{y}_k,\mathbf{v}_k)\|^2}{[\bar{m}^x_{k+1}]^2\max \big\{\bar{m}^v_{k+1}, \bar{m}^y_{k+1}\big\}},\label{eq98}
\end{align}
where the second inequality uses \cref{Lemma7}. Thus, Part II is completed.

\textbf{Part III: The Bound of } $\bar{z}_t$ \textbf{ in Case 4.}

Here, we explore the upper bound of $\bar{z}_t$ in Case 4. Recalling \cref{eq91}, we have:
\begin{align}
\bar{z}_{t+1} &\leq C_{m^y} + C_{m^v} + \sum_{k=k_2}^t \frac{\|\nabla_y L(\mathbf{x}_k, \mathbf{y}_k)\|^2}{n\bar{m}^y_{k+1}} + \sum_{k=k_3}^t \frac{\|\nabla_v R(\mathbf{x}_k, \mathbf{y}_k, \mathbf{v}_k)\|^2}{n\bar{z}_{k+1}} = C_{m^y} + C_{m^v} = C_z,\label{eq99}
\end{align}
for $t \leq k_{\min} := \min\{k_2, k_3\}$. For $t > k_{\min}$, we have:
\begin{align}
\bar{z}_{t+1} &\leq C_{m^y} + C_{m^v} + \sum_{k=k_2}^t \frac{\|\nabla_y L(\mathbf{x}_k, \mathbf{y}_k)\|^2}{n\bar{m}^y_{k+1}} + \sum_{k=k_3}^t \frac{\|\nabla_v R(\mathbf{x}_k, \mathbf{y}_k, \mathbf{v}_k)\|^2}{n\bar{z}_{k+1}}  \notag \\
&\overset{(a)}{\leq} C_{m^y} + C_{m^v}  +\frac{4C_{m^y}^2(\mu + C_{l_{yy}})}{\mu^4\gamma_v\sqrt{1+\sigma_z^2}}\left(\frac{nL_{l,2}C_{f_y}}{\mu}+L_{f,1}\right)^2 \nonumber\\
&\quad + \frac{8(\mu + C_{l_{yy}})C_{m^v}^2}{\mu^2\gamma_v\sqrt{1+\sigma_z^2}}+ \frac{16(\mu + C_{l_{yy}})}{n\gamma_v\mu^2}\left(\frac{nL_{l,2}C_{f_y}}{\mu}\!+\!L_{f,1}\right)^2\nonumber\\
&\quad\cdot\left(\frac{\Delta_0^x}{1-\rho_W}\!+\!\frac{4 \gamma_x^2 \rho_WC_{m^x}^2 (1 + \zeta_q^2)}{\bar{q}_0^2(1 - \rho_W)^2} \right)\left(\frac{(\mu+L_{l,1})^2}{n\mu^2 }\!+\!\frac{2L_{l,1}}{n\mu }\right)\notag \\
&\quad+\left(\frac{16\Delta_0^x}{n(1-\rho_W)}+\frac{64\gamma_x^2\rho_WC_{m^x}^2(1+\zeta_q^2)}{n\bar{q}_0^2(1-\rho_W)^2}\right)\left(\frac{ \bar{L}_r^2(\mu + C_{l_{yy}})}{\mu C_{l_{yy}}\gamma_v}+\frac{\bar{L}_r^2(\mu+C_{l_{yy}})}{\mu^2\gamma_v\sqrt{1+\sigma_z^2}}\right) \notag \\
&\quad+\frac{16(\mu + C_{l_{yy}})}{n\gamma_v\mu^2\sqrt{1+\sigma_u^2}}\left(\frac{nL_{l,2}C_{f_y}}{\mu}\!+\!L_{f,1}\right)^2\nonumber\\
&\quad\cdot\left[\frac{4 \gamma_x^2 \rho_W (1 + \zeta_q^2)}{\bar{z}_0^2(1 - \rho_W)^2} \left(\frac{(\mu+L_{l,1})^2\sqrt{1+\sigma_u^2}}{n\mu^2 }+\frac{2L_{l,1}\sqrt{1+\sigma_u^2}}{n\mu }\right)\right.\nonumber\\
&\quad +\left.\frac{\gamma_x^2L_y^2(1+\zeta_q^2)(\mu + L_{l,1})}{n\bar{z}_0\gamma_y \mu L_{l,1}}\right]\sum_{k=\min\{k_1-1,k_2-1\}}^{t} 
 \frac{\| \bar{\nabla} F(\mathbf{x}_{k},\mathbf{y}_{k},\mathbf{v}_{k})\|^2}{[\bar{m}^x_{k+1}]^2}\nonumber\\
&\quad + \left(\frac{64 \gamma_x^2 \rho_W \bar{L}_r^2(\mu + C_{l_{yy}})(1 + \zeta_q^2)}{n\mu C_{l_{yy}}\gamma_v(1 - \rho_W)^2} + \frac{2\gamma_x^2L_v^2(1+\zeta_q^2)(\mu + C_{l_{yy}})^2}{n\mu C_{l_{yy}}C_{m^v}\gamma_v^2\sqrt{1+\sigma_z^2}}\right)\notag\\
&\quad\cdot\sum_{k=\min\{k_1-1,k_3-1\}}^{t} 
 \frac{\| \bar{\nabla} F(\mathbf{x}_{k},\mathbf{y}_{k},\mathbf{v}_{k})\|^2}{[\bar{m}^x_{k+1}]^2}\notag \\
&\quad  +\left[ \frac{1}{\mu^2}\left(8+\frac{4(\mu + C_{l_{yy}})^2}{\mu C_{l_{yy}}} \right)\left(\frac{nL_{l,2}C_{f_y}}{\mu}+L_{f,1}\right)^2+1\right]\left(\frac{C_{m^y}^2}{\bar{m}^y_0} - {\bar{m}^y_0}\right)  \notag \\
&\quad  +\frac{1}{n}\left[ \frac{1}{\mu^2}\left(8+\frac{4(\mu + C_{l_{yy}})^2}{\mu C_{l_{yy}}} \right)\left(\frac{nL_{l,2}C_{f_y}}{\mu}+L_{f,1}\right)^2+1\right]\sum_{k=k_2}^t\frac{\|\nabla_yL(\mathbf{x}_k,\mathbf{y}_k)\|^2}{\bar{m}_{k+1}^y}  \notag \\
&\overset{(b)}{\leq} C_{m^y} + C_{m^v} +\frac{4C_{m^y}^2(\mu + C_{l_{yy}})}{\mu^4\gamma_v\sqrt{1+\sigma_z^2}}\left(\frac{nL_{l,2}C_{f_y}}{\mu}+L_{f,1}\right)^2   \notag \\
&\quad + \frac{8(\mu + C_{l_{yy}})C_{m^v}^2}{\mu^2\gamma_v\sqrt{1+\sigma_z^2}}+ \frac{16(\mu + C_{l_{yy}})}{n\gamma_v\mu^2}\left(\frac{nL_{l,2}C_{f_y}}{\mu}\!+\!L_{f,1}\right)^2 \nonumber\\
&\quad\cdot\left(\frac{\Delta_0^x}{1-\rho_W}\!+\!\frac{4 \gamma_x^2 \rho_WC_{m^x}^2 (1 + \zeta_q^2)}{\bar{q}_0^2(1 - \rho_W)^2} \right)\left(\frac{(\mu+L_{l,1})^2}{n\mu^2 }\!+\!\frac{2L_{l,1}}{n\mu }\right)\notag \\
&\quad+\left(\frac{16\Delta_0^x}{n(1-\rho_W)}+\frac{64\gamma_x^2\rho_WC_{m^x}^2(1+\zeta_q^2)}{n\bar{q}_0^2(1-\rho_W)^2}\right)\left(\frac{ \bar{L}_r^2(\mu + C_{l_{yy}})}{\mu C_{l_{yy}}\gamma_v}+\frac{\bar{L}_r^2(\mu+C_{l_{yy}})}{\mu^2\gamma_v\sqrt{1+\sigma_z^2}}\right) \notag \\
&\quad + \left[\left( \frac{1}{\mu^2}\left(8+\frac{4(\mu + C_{l_{yy}})^2}{\mu C_{l_{yy}}} \right)\left(\frac{nL_{l,2}C_{f_y}}{\mu}+L_{f,1}\right)^2+1\right)  \right.\notag \\
&\quad \cdot\left(\frac{32 \gamma_x^2 \rho_W(1 + \zeta_q^2)}{n(1 - \rho_W)^2} \left(\frac{(\mu+L_{l,1})^3}{\gamma_y\mu^2}+\frac{2L_{l,1}(\mu+L_{l,1})}{\gamma_y\mu}\right)+\frac{8\gamma_x^2  L_y^2\left( 1 + \zeta_q^2 \right) (\mu + L_{l,1})^2}{n\gamma_y^2\mu L_{l,1}\sqrt{1+\sigma_u^2}}\right)\nonumber\\
&\quad +\frac{16(\mu + C_{l_{yy}})}{n\gamma_v\mu^2\sqrt{1+\sigma_u^2}}\left(\frac{nL_{l,2}C_{f_y}}{\mu}\!+\!L_{f,1}\right)^2\nonumber\\
&\quad\cdot\left(\frac{4 \gamma_x^2 \rho_W (1 + \zeta_q^2)}{\bar{z}_0^2(1 - \rho_W)^2} \left(\frac{(\mu+L_{l,1})^2\sqrt{1+\sigma_u^2}}{n\mu^2 }+\frac{2L_{l,1}\sqrt{1+\sigma_u^2}}{n\mu }\right)\right.\nonumber\\
&\quad +\left.\left.\frac{\gamma_x^2L_y^2(1+\zeta_q^2)(\mu + L_{l,1})}{n\bar{z}_0\gamma_y \mu L_{l,1}}\right)\right]
\sum_{k=\{k_1-1,k_2-1\}}^{t}  \frac{\| \bar{\nabla} F(\mathbf{x}_{k},\mathbf{y}_{k},\mathbf{v}_{k})\|^2}{[\bar{m}^x_{k+1}]^2}\notag\\
&\quad + \left(\frac{64 \gamma_x^2 \rho_W \bar{L}_r^2(\mu + C_{l_{yy}})(1 + \zeta_q^2)}{n\mu C_{l_{yy}}\gamma_v(1 - \rho_W)^2} + \frac{2\gamma_x^2L_v^2(1+\zeta_q^2)(\mu + C_{l_{yy}})^2}{n\mu C_{l_{yy}}\bar{m}_0^v\gamma_v^2\sqrt{1+\sigma_z^2}}\right)\nonumber\\
&\quad\cdot\sum_{k=\min\{k_1-1,k_3-1\}}^{t} 
 \frac{\| \bar{\nabla} F(\mathbf{x}_{k},\mathbf{y}_{k},\mathbf{v}_{k})\|^2}{[\bar{m}^x_{k+1}]^2} \notag \\
&\quad  +\left[ \frac{1}{\mu^2}\left(8+\frac{4(\mu + C_{l_{yy}})^2}{\mu C_{l_{yy}}} \right)\left(\frac{nL_{l,2}C_{f_y}}{\mu}+L_{f,1}\right)^2+1\right]\left(\frac{C_{m^y}^2}{\bar{m}^y_0} - {\bar{m}^y_0}\right)  \notag \\
&\quad  +\left[ \frac{1}{\mu^2}\left(8+\frac{4(\mu + C_{l_{yy}})^2}{\mu C_{l_{yy}}} \right)\left(\frac{nL_{l,2}C_{f_y}}{\mu}+L_{f,1}\right)^2+1\right]\left(\frac{ 2C_{m^y}^2(\mu + L_{l,1})}{\mu^2\gamma_y\sqrt{1+\sigma_u^2}} \right.\notag\\
&\quad+\left(\frac{8\Delta_0^x}{n(1-\rho_W)}+\frac{32 \gamma_x^2 \rho_WC_{m^x}^2 (1 + \zeta_q^2)}{n\bar{q}_0^2(1 - \rho_W)^2} \right)\left(\frac{(\mu+L_{l,1})^3}{\gamma_y\mu^2}+\frac{2L_{l,1}(\mu+L_{l,1})}{\gamma_y\mu}\right)\nonumber\\
&\quad\left.+ \frac{32 \gamma_x^2 \rho_W(1 + \zeta_q^2)}{n\bar{z}_0(1 - \rho_W)^2} \left(\frac{(\mu+L_{l,1})^3}{\gamma_y\mu^2}+\frac{2L_{l,1}(\mu+L_{l,1})}{\gamma_y\mu}\right)+\frac{8\gamma_x^2  L_y^2\left( 1 + \zeta_q^2 \right) (\mu + L_{l,1})^2}{n\gamma_y^2\mu L_{l,1}\bar{z}_0\sqrt{1+\sigma_u^2}} \right)\nonumber\\
&\overset{(c)}{=}: a_0 \sum_{k=\min\{k_1-1,k_2-1, k_3-1\}}^t \frac{\| \bar{\nabla} F(\mathbf{x}_{k},\mathbf{y}_{k},\mathbf{v}_{k})\|^2}{[\bar{m}^x_{k+1}]^2} + b_0, \notag \\
&\leq a_0 \sum_{k=\min\{k_1,k_2, k_3\}}^t \frac{\| \bar{\nabla} F(\mathbf{x}_{k},\mathbf{y}_{k},\mathbf{v}_{k})\|^2}{[\bar{m}^x_{k+1}]^2} + a_0 + b_0\notag \\
&\overset{(d)}{\leq} a_0 \left(3 \log(t + 1) + 2 \log \left(  \bar{z}_{t+1} + \frac{nC_{l_{xy}} \bar{b} + nC_{f_x} + \bar{m}^x_0}{nC_{l_{xy}} \bar{a}} \right)+2 \log \left( nC_{l_{xy}} \bar{a}\right) + 1\right)\nonumber\\
&\quad+ a_0 + b_0,\label{eq100}
\end{align}
where (a) uses \cref{Lemma8} and the first line in \cref{eq98} by replacing $\tilde{k}$ with $k_3 - 1$; (b) results from \cref{Lemma7}; (c) refers to \cref{eq85}; (d) uses \cref{eq95}. Since $\min\{k_2, k_3\} \leq T$, we have $\bar{z}_{t+1} \geq \min\{C_{m^y}, C_{m^v}\} \geq \max\{64a_0^2, 1\}$, which indicates that

(i) if $8a_0 \leq 1$, we have:
\begin{equation}
    4a_0 \log(\bar{z}_{t+1}) \leq \frac{\log(\bar{z}_{t+1})}{2} \leq \frac{\bar{z}_{t+1}}{2} \leq \bar{z}_{t+1};\label{eq101}
\end{equation}

(ii) if $8a_0 > 1$, we have:
\begin{equation}
\bar{z}_{t+1} - 4a_0 \log(\bar{z}_{t+1}) = \bar{z}_{t+1} - 8a_0 \log(\sqrt{\bar{z}_{t+1}}) \geq 8a_0 \left(\sqrt{\bar{z}_{t+1}} - \log(\sqrt{\bar{z}_{t+1}})\right) \geq 0.\label{eq102}
\end{equation}

Combining (i) and (ii), we have $4a_0 \log(\bar{z}_{t+1}) \leq \bar{z}_{t+1}$. Then we obtain:
\begin{align}
  \bar{z}_{t+1} & \leq a_0 \left(3 \log(t + 1) + 2 \log \left(  \bar{z}_{t+1} + \frac{nC_{l_{xy}} \bar{b} + nC_{f_x} + \bar{m}^x_0}{nC_{l_{xy}} \bar{a}} \right)+2 \log \left(n C_{l_{xy}} \bar{a}\right) + 1\right)\nonumber\\
  &\quad+ a_0 + b_0\nonumber\\ 
  &\leq  a_0 \left(3 \log(t + 1) \!+\! 2 \log \left(  \bar{z}_{t+1}\right)\!+\!2 \log \left(\! 1\! +\! \frac{nC_{l_{xy}} \bar{b} \!+\! nC_{f_x} \!+\! \bar{m}^x_0}{nC_{l_{xy}} \bar{a}} \right)\!+\!2 \log \left( nC_{l_{xy}} \bar{a}\right) \!+\! 1\right)\nonumber\\
  &\quad + a_0 + b_0\nonumber\\
  &\leq \frac{1}{2} \bar{z}_{t+1} + a_0 \left(3 \log(t + 1) +2 \log \left( 1 + \frac{nC_{l_{xy}} \bar{b} + nC_{f_x} + \bar{m}^x_0}{nC_{l_{xy}} \bar{a}} \right)+2 \log \left( nC_{l_{xy}} \bar{a}\right) + 1\right)\nonumber\\
  &\quad + a_0 + b_0,\label{eq103}
\end{align}
which indicates that
\begin{align}
    \bar{z}_{t+1}&\leq 6a_0 \log(t + 1) + 4a_0 \log \left( 1 + \frac{nC_{l_{xy}} \bar{b} + nC_{f_x} + \bar{m}^x_0}{nC_{l_{xy}} \bar{a}} \right) + 4a_0\log \left( nC_{l_{xy}} \bar{a}\right) + 4a_0 + 2b_0\nonumber\\
    &\overset{(a)}{=}a_1 \log(t + 1) + b_1,\label{eq104}
\end{align}
where (a) refers to \cref{eq84}. Therefore, we complete the proof of this lemma.
\end{proof}

\subsection{The Upper Bounds of $\sum \frac{\| \bar{\nabla} F(\mathbf{x}_{t},\mathbf{y}_{t},\mathbf{v}_{t})\|^2}{[\bar{m}^x_{t+1}]^2}$, $\sum \frac{\|\nabla_y L(\mathbf{x}_t, \mathbf{y}_t)\|^2}{\bar{m}^y_{t+1}}$, and $\sum \frac{\|\nabla_v R(\mathbf{x}_t, \mathbf{y}_t, \mathbf{v}_t)\|^2}{\bar{z}_{t+1}}$}\label{secc.6}
\begin{lemma}\label{Lemma11} Under \cref{Assumption1} and \cref{Assumption2}, for any integer $k_0 \in [0, t)$, we have the upper bounds in terms of logarithmic functions as:
\begin{align}
\sum_{k=k_0}^t \frac{\|\bar{\nabla} F(\mathbf{x}_k,\mathbf{y}_k,\mathbf{v}_k)\|^2}{[\bar{m}^x_{k+1}]^2} &\leq 5 \log(t + 1) + c_2,\nonumber\\
\sum_{k=k_0}^t \frac{\|\nabla_y L(\mathbf{x}_k, \mathbf{y}_k)\|^2}{\bar{m}^y_{k+1}} &\leq a_2 \log(t + 1) + b_2, \nonumber\\
\sum_{k=k_0}^t \frac{\|\nabla_v R(\mathbf{x}_k, \mathbf{y}_k, \mathbf{v}_k)\|^2}{\bar{z}_{k+1}} &\leq a_3 \log(t + 1) + b_3,\label{eq105}
\end{align}
where, referring to \cref{eq84} and \cref{eq85}, $c_2$, $a_2$, $b_2$, $a_3$, and $b_3$ are defined as:
\begin{align}
c_2 &:= 2 \log \left( nC_{l_{xy}} \bar{a} a_1+nC_{l_{xy}} \bar{a}b_1 + nC_{l_{xy}} \bar{b} + nC_{f_x} +\bar{m}^x_0\right) + 1, \nonumber\\
a_2 &:= \frac{160 \gamma_x^2 \rho_W(1 + \zeta_q^2)}{(1 - \rho_W)^2} \left(\frac{(\mu+L_{l,1})^3}{\gamma_y\mu^2}+\frac{2L_{l,1}(\mu+L_{l,1})}{\gamma_y\mu}\right)+\frac{40\gamma_x^2  L_y^2\left( 1 + \zeta_q^2 \right) (\mu + L_{l,1})^2}{\gamma_y^2\mu L_{l,1}\sqrt{1+\sigma_u^2}}, \nonumber\\
b_2 &:=\frac{nC_{m^y}^2}{\bar{m}^y_0} - n\bar{m}^y_0 + \frac{ 2nC_{m^y}^2(\mu + L_{l,1})}{\mu^2\gamma_y\sqrt{1+\sigma_u^2}}\nonumber\\
&\quad+\left(\frac{8\Delta_0^x}{1-\rho_W}+\frac{32 \gamma_x^2 \rho_WC_{m^x}^2 (1 + \zeta_q^2)}{\bar{q}_0^2(1 - \rho_W)^2} \right)\left(\frac{(\mu+L_{l,1})^3}{\gamma_y\mu^2}+\frac{2L_{l,1}(\mu+L_{l,1})}{\gamma_y\mu}\right)\nonumber\\
&\quad +\!\left[\!\frac{32 \gamma_x^2 \rho_W(1 \!+\! \zeta_q^2)}{(1 \!-\! \rho_W)^2} \!\!\left(\frac{(\mu\!+\!L_{l,1})^3}{\gamma_y\mu^2}\!+\!\frac{2L_{l,1}(\mu\!+\!L_{l,1})}{\gamma_y\mu}\right)\!+\!\frac{8\gamma_x^2  L_y^2\left( 1 \!+\! \zeta_q^2 \right) (\mu \!+\! L_{l,1})^2}{\gamma_y^2\mu L_{l,1}\sqrt{1\!+\!\sigma_u^2}}\!\right]\!\!\left(\!\frac{1}{\bar{z}_0}\!+\!c_2\!\right), \nonumber\\
a_3 &:=\frac{80(\mu + C_{l_{yy}})}{\gamma_v\mu^2\sqrt{1+\sigma_u^2}}\left(\frac{nL_{l,2}C_{f_y}}{\mu}\!+\!L_{f,1}\right)^2\nonumber\\
&\quad\cdot\left(\frac{4 \gamma_x^2 \rho_W (1 + \zeta_q^2)}{\bar{z}_0^2(1 - \rho_W)^2} \left(\frac{(\mu+L_{l,1})^2\sqrt{1+\sigma_u^2}}{n\mu^2 }+\frac{2L_{l,1}\sqrt{1+\sigma_u^2}}{n\mu }\right)\right.\nonumber\\
&\quad +\!\left.\frac{\gamma_x^2L_y^2(1+\zeta_q^2)(\mu + L_{l,1})}{n\bar{z}_0\gamma_y \mu L_{l,1}}\right) \!+\! \frac{320\gamma_x^2 \rho_W \bar{L}_r^2(\mu + C_{l_{yy}})(1 + \zeta_q^2)}{\mu C_{l_{yy}}\gamma_v(1 - \rho_W)^2} \notag \\
&\quad  + \frac{10\gamma_x^2L_v^2(1+\zeta_q^2)(\mu + C_{l_{yy}})^2}{\mu C_{l_{yy}}C_{m^v}\gamma_v^2\sqrt{1+\sigma_z^2}}+ \left(\frac{8a_2}{\mu^2}+\frac{4a_2(\mu + C_{l_{yy}})^2}{\mu^3 C_{l_{yy}}} \right)\left(\frac{nL_{l,2}C_{f_y}}{\mu}+L_{f,1}\right)^2,  \notag\\
b_3 &:=\frac{nC_{m^v}^2}{\bar{m}^v_{0}}\! -\! n\bar{m}^v_{0} \!+\! \frac{4nC_{m^y}^2(\mu + C_{l_{yy}})}{\mu^4\gamma_v\sqrt{1+\sigma_z^2}}\left(\frac{nL_{l,2}C_{f_y}}{\mu}+L_{f,1}\right)^2 + \frac{8nC_{m^v}^2(\mu + C_{l_{yy}})}{\mu^2\gamma_v\sqrt{1+\sigma_z^2}}\notag \\ 
&\quad +\left(\frac{16\Delta_0^x}{1-\rho_W}+\frac{64\gamma_x^2\rho_WC_{m^x}^2(1+\zeta_q^2)}{\bar{q}_0^2(1-\rho_W)^2}\right)\left(\frac{ \bar{L}_r^2(\mu + C_{l_{yy}})}{\mu C_{l_{yy}}\gamma_v}+\frac{\bar{L}_r^2(\mu+C_{l_{yy}})}{\mu^2\gamma_v\sqrt{1+\sigma_z^2}}\right)\nonumber\\
&\quad + \frac{16(\mu + C_{l_{yy}})}{\gamma_v\mu^2}\left(\frac{nL_{l,2}C_{f_y}}{\mu}\!+\!L_{f,1}\right)^2 \left(\frac{\Delta_0^x}{1-\rho_W}\!+\!\frac{4 \gamma_x^2 \rho_WC_{m^x}^2 (1 + \zeta_q^2)}{\bar{q}_0^2(1 - \rho_W)^2} \right)\notag\\
&\quad\cdot\left(\frac{(\mu+L_{l,1})^2}{n\mu^2 }\!+\!\frac{2L_{l,1}}{n\mu }\right)+\frac{16c_2(\mu + C_{l_{yy}})}{\gamma_v\mu^2\sqrt{1+\sigma_u^2}}\left(\frac{nL_{l,2}C_{f_y}}{\mu}\!+\!L_{f,1}\right)^2\notag \\
&\quad \cdot\left(\!\frac{4 \gamma_x^2 \rho_W (1 + \zeta_q^2)}{\bar{z}_0^2(1 - \rho_W)^2} \left(\frac{(\mu+L_{l,1})^2\sqrt{1+\sigma_u^2}}{n\mu^2 }\!+\!\frac{2L_{l,1}\sqrt{1+\sigma_u^2}}{n\mu }\right)\!+\!\frac{\gamma_x^2L_y^2(1+\zeta_q^2)(\mu + L_{l,1})}{n\bar{z}_0\gamma_y \mu L_{l,1}}\!\right) \nonumber\\
&\quad \!+\! \frac{64 c_2\gamma_x^2 \rho_W \bar{L}_r^2(\mu + C_{l_{yy}})(1 + \zeta_q^2)}{\mu C_{l_{yy}}\gamma_v(1 - \rho_W)^2} \!+\! \frac{2c_2\gamma_x^2L_v^2(1+\zeta_q^2)(\mu + C_{l_{yy}})^2}{\mu C_{l_{yy}}C_{m^v}\gamma_v^2\sqrt{1+\sigma_z^2}} \notag \\
&\quad  + \left(\frac{8b_2}{\mu^2}+\frac{4b_2(\mu + C_{l_{yy}})^2}{\mu^3 C_{l_{yy}}} \right)\left(\frac{nL_{l,2}C_{f_y}}{\mu}+L_{f,1}\right)^2  .\label{eq106}
\end{align}
\end{lemma}

\begin{proof}
Based on the results in \cref{Lemma10}, we have the following bounds.

\textbf{Part I: Bounding} $\sum \frac{\| \bar{\nabla} F(\mathbf{x}_{t},\mathbf{y}_{t},\mathbf{v}_{t})\|^2}{[\bar{m}^x_{t+1}]^2}$.

Firstly, we bound $\sum_{k=k_0}^t \frac{\| \bar{\nabla} F(\mathbf{x}_{k},\mathbf{y}_{k},\mathbf{v}_{k})\|^2}{[\bar{m}^x_{k+1}]^2}$ for arbitrary $k_0 < t$. Back to \cref{eq95}, by plugging in \cref{eq104}, we have:
\begin{align}
&\sum_{k=k_0}^t \frac{\| \bar{\nabla} F(\mathbf{x}_{k},\mathbf{y}_{k},\mathbf{v}_{k})\|^2}{[\bar{m}^x_{k+1}]^2}\notag\\
&\leq 3 \log(t + 1) + 2 \log \left( nC_{l_{xy}} \bar{a} \bar{z}_{t+1} + nC_{l_{xy}} \bar{b} + nC_{f_x} +\bar{m}^x_0 \right) + 1 \notag \\
&\overset{(a)}{\leq} 3 \log(t + 1) + 2 \log \left( nC_{l_{xy}} \bar{a} a_1 \log(t+1)+nC_{l_{xy}} \bar{a}b_1 + nC_{l_{xy}} \bar{b} + nC_{f_x} +\bar{m}^x_0 \right) + 1 \notag \\
&\leq 3 \log(t + 1) + 2 \log \left( nC_{l_{xy}} \bar{a} a_1 (t+1)+nC_{l_{xy}} \bar{a}b_1 + nC_{l_{xy}} \bar{b} + nC_{f_x} +\bar{m}^x_0 \right) + 1 \notag \\
&\leq 3 \log(t + 1) + 2 \log \left( \left(nC_{l_{xy}} \bar{a} a_1+nC_{l_{xy}} \bar{a}b_1 + nC_{l_{xy}} \bar{b} + nC_{f_x} +\bar{m}^x_0 \right)(t + 1) \right) + 1 \notag \\
&\leq 5 \log(t + 1) + 2 \log \left( nC_{l_{xy}} \bar{a} a_1+nC_{l_{xy}} \bar{a}b_1 + nC_{l_{xy}} \bar{b} + nC_{f_x} +\bar{m}^x_0\right) + 1 \notag \\
&\overset{(b)}{=}: 5 \log(t + 1) + c_2, \label{eq107}
\end{align}
where (a) results from \cref{eq104}; (b) refers to \cref{eq106}.

\textbf{Part II: Bounding} $\sum\frac{\|\nabla_y L(\mathbf{x}_t, \mathbf{y}_t)\|^2}{\bar{m}^y_{t+1}}$.

Secondly, we bound $\sum_{k=k_0}^t\frac{\|\nabla_y L(\mathbf{x}_k, \mathbf{y}_k)\|^2}{\bar{m}^y_{k+1}}$. We split this part into two cases using \cref{Lemma6}.

\textbf{Case 1:} If $\bar{m}^y_{t+1} \leq C_{m^y}$, we have:
\begin{align}
&\sum_{k=k_0}^t \frac{\|\nabla_y L(\mathbf{x}_k, \mathbf{y}_k)\|^2}{\bar{m}^y_{k+1}} \nonumber\\
&\leq \sum_{k=k_0}^t \frac{\|\nabla_y L(\mathbf{x}_k, \mathbf{y}_k)\|^2}{\bar{m}^y_{0}}\leq \frac{n([\bar{m}^y_{t+1}]^2 - [\bar{m}^y_{k_0}]^2)}{\bar{m}^y_0} \leq \frac{n(C_{m^y}^2 - [\bar{m}^y_{0}]^2)}{\bar{m}^y_{0}} = \frac{nC_{m^y}^2}{\bar{m}^y_{0}} - n\bar{m}^y_{0}\! \leq \!b_2.\label{eq108}
\end{align}

\textbf{Case 2:} If $\bar{m}^y_{t+1}> C_{m^y}$, we have $k_2 \leq t$, where $k_2$ refers to \cref{Lemma6}. Then based on \cref{eq98}, we have:
\begin{align}
&\sum_{k=k_0}^t \frac{\|\nabla_y L(\mathbf{x}_k, \mathbf{y}_k)\|^2}{\bar{m}^y_{k+1}}\nonumber\\
&\leq\frac{nC_{m^y}^2}{\bar{m}^y_0} - n\bar{m}^y_0 + \frac{ 2nC_{m^y}^2(\mu + L_{l,1})}{\mu^2\gamma_y\sqrt{1+\sigma_u^2}}\notag\\
&\quad+\left(\frac{8\Delta_0^x}{1-\rho_W}+\frac{32 \gamma_x^2 \rho_WC_{m^x}^2 (1 + \zeta_q^2)}{\bar{q}_0^2(1 - \rho_W)^2} \right)\left(\frac{(\mu+L_{l,1})^3}{\gamma_y\mu^2}+\frac{2L_{l,1}(\mu+L_{l,1})}{\gamma_y\mu}\right)\nonumber\\
&\quad+ \frac{32 \gamma_x^2 \rho_W(1 + \zeta_q^2)}{\bar{z}_0(1 - \rho_W)^2} \left(\frac{(\mu+L_{l,1})^3}{\gamma_y\mu^2}+\frac{2L_{l,1}(\mu+L_{l,1})}{\gamma_y\mu}\right)+\frac{8\gamma_x^2  L_y^2\left( 1 + \zeta_q^2 \right) (\mu + L_{l,1})^2}{\gamma_y^2\mu L_{l,1}\bar{z}_0\sqrt{1+\sigma_u^2}} \nonumber\\
&\quad+ \left[\frac{32 \gamma_x^2 \rho_W(1 + \zeta_q^2)}{(1 - \rho_W)^2} \left(\frac{(\mu+L_{l,1})^3}{\gamma_y\mu^2}+\frac{2L_{l,1}(\mu+L_{l,1})}{\gamma_y\mu}\right)+\frac{8\gamma_x^2  L_y^2\left( 1 + \zeta_q^2 \right) (\mu + L_{l,1})^2}{\gamma_y^2\mu L_{l,1}\sqrt{1+\sigma_u^2}}\right]\nonumber\\
&\quad\cdot\sum_{k=\min\{k_1,k_2\}}^t \frac{\|\bar{\nabla} F(\mathbf{x}_k,\mathbf{y}_k,\mathbf{v}_k)\|^2}{[\bar{m}^x_{k+1}]^2\max \big\{\bar{m}^v_{k+1}, \bar{m}^y_{k+1}\big\}}\nonumber\\
&\overset{(a)}{\leq}\frac{nC_{m^y}^2}{\bar{m}^y_0} - n\bar{m}^y_0 + \frac{ 2nC_{m^y}^2(\mu + L_{l,1})}{\mu^2\gamma_y\sqrt{1+\sigma_u^2}} \notag\\
&\quad+\left(\frac{8\Delta_0^x}{1-\rho_W}+\frac{32 \gamma_x^2 \rho_WC_{m^x}^2 (1 + \zeta_q^2)}{\bar{q}_0^2(1 - \rho_W)^2} \right)\left(\frac{(\mu+L_{l,1})^3}{\gamma_y\mu^2}+\frac{2L_{l,1}(\mu+L_{l,1})}{\gamma_y\mu}\right)\nonumber\\
&\quad+ \frac{32 \gamma_x^2 \rho_W(1 + \zeta_q^2)}{\bar{z}_0(1 - \rho_W)^2} \left(\frac{(\mu+L_{l,1})^3}{\gamma_y\mu^2}+\frac{2L_{l,1}(\mu+L_{l,1})}{\gamma_y\mu}\right)+\frac{8\gamma_x^2  L_y^2\left( 1 + \zeta_q^2 \right) (\mu + L_{l,1})^2}{\gamma_y^2\mu L_{l,1}\bar{z}_0\sqrt{1+\sigma_u^2}} \nonumber\\
&\quad+\left[\frac{160 \gamma_x^2 \rho_W(1 + \zeta_q^2)}{(1 - \rho_W)^2} \left(\frac{(\mu+L_{l,1})^3}{\gamma_y\mu^2}+\frac{2L_{l,1}(\mu+L_{l,1})}{\gamma_y\mu}\right)\right.\nonumber\\
&\quad\left.+\frac{40\gamma_x^2  L_y^2\left( 1 + \zeta_q^2 \right) (\mu + L_{l,1})^2}{\gamma_y^2\mu L_{l,1}\sqrt{1+\sigma_u^2}}\right]\log(t+1)\nonumber\\
&\quad +\left[\frac{32 \gamma_x^2 \rho_W(1 + \zeta_q^2)}{(1 - \rho_W)^2} \left(\frac{(\mu+L_{l,1})^3}{\gamma_y\mu^2}+\frac{2L_{l,1}(\mu+L_{l,1})}{\gamma_y\mu}\right)+\frac{8\gamma_x^2  L_y^2\left( 1 + \zeta_q^2 \right) (\mu + L_{l,1})^2}{\gamma_y^2\mu L_{l,1}\sqrt{1+\sigma_u^2}}\right]c_2\notag\\
&\overset{(b)}{=}:a_2\log(t+1)+b_2,\label{eq109}
\end{align}
where (a) uses \cref{eq107}, and (b) refers to \cref{eq109}. Since the upper bound of Case 2 is larger, we take \cref{eq106} as our final result.

\textbf{Part III: Bounding} $\sum\frac{\|\nabla_v  R(\mathbf{x}_t, \mathbf{y}_t, \mathbf{v}_t)\|^2}{\bar{z}_{t+1}}$.

Last, we bound $\sum_{k=k_0}^t \frac{\|\nabla_v  R(\mathbf{x}_k, \mathbf{y}_k, \mathbf{v}_k)\|^2}{\bar{z}_{k+1}}$. We split this part into two cases using \cref{Lemma6}.

\textbf{Case 1:} If $\bar{m}^v_{t+1} \leq C_{m^v}$, we have:
\begin{align}
&\sum_{k=k_0}^t \frac{\|\nabla_v  R(\mathbf{x}_k, \mathbf{y}_k, \mathbf{v}_k)\|^2}{\bar{z}_{k+1}}\notag\\
&\leq \sum_{k=k_0}^t \frac{\|\nabla_v  R(\mathbf{x}_k, \mathbf{y}_k, \mathbf{v}_k)\|^2}{\bar{z}_{0}} \leq \frac{n([\bar{m}^v_{t+1}]^2 - [\bar{m}^v_{0}]^2)}{\bar{m}^v_{0}}\leq \frac{nC_{m^v}^2}{\bar{m}^v_{0}} - n\bar{m}^v_{0} \leq b_3.\label{eq110}
\end{align}

\textbf{Case 2:} If $\bar{m}^v_{t+1} > C_{m^v}$, we have $k_3 \leq t$, where $k_3$ refers to \cref{Lemma6}.
\begin{align}
&\sum_{k=k_0}^t \frac{\|\nabla_v  R(\mathbf{x}_k, \mathbf{y}_k, \mathbf{v}_k)\|^2}{\bar{z}_{k+1}} \nonumber\\
&\overset{(a)}{\leq} \sum_{k=k_0}^{k_3-1} \frac{\|\nabla_v  R(\mathbf{x}_k, \mathbf{y}_k, \mathbf{v}_k)\|^2}{\bar{z}_{k+1}}+ \sum_{k=k_3}^t \frac{\|\nabla_v  R(\mathbf{x}_k, \mathbf{y}_k, \mathbf{v}_k)\|^2}{\bar{z}_{k+1}} \notag \\
&\overset{(b)}{\leq}\frac{nC_{m^v}^2}{\bar{m}^v_{0}}\! -\! n\bar{m}^v_{0} \!+\! \frac{4nC_{m^y}^2(\mu + C_{l_{yy}})}{\mu^4\gamma_v\sqrt{1+\sigma_z^2}}\left(\frac{nL_{l,2}C_{f_y}}{\mu}+L_{f,1}\right)^2 + \frac{8nC_{m^v}^2(\mu + C_{l_{yy}})}{\mu^2\gamma_v\sqrt{1+\sigma_z^2}}\notag \\
&\quad +\left(\frac{16\Delta_0^x}{1-\rho_W}+\frac{64\gamma_x^2\rho_WC_{m^x}^2(1+\zeta_q^2)}{\bar{q}_0^2(1-\rho_W)^2}\right)\left(\frac{ \bar{L}_r^2(\mu + C_{l_{yy}})}{\mu C_{l_{yy}}\gamma_v}+\frac{\bar{L}_r^2(\mu+C_{l_{yy}})}{\mu^2\gamma_v\sqrt{1+\sigma_z^2}}\right)\nonumber\\
&\quad + \frac{16(\mu + C_{l_{yy}})}{\gamma_v\mu^2}\left(\frac{nL_{l,2}C_{f_y}}{\mu}\!+\!L_{f,1}\right)^2 \left(\frac{\Delta_0^x}{1-\rho_W}\!+\!\frac{4 \gamma_x^2 \rho_WC_{m^x}^2 (1 + \zeta_q^2)}{\bar{q}_0^2(1 - \rho_W)^2} \right)\notag\\
&\quad\cdot\left(\frac{(\mu+L_{l,1})^2}{n\mu^2 }\!+\!\frac{2L_{l,1}}{n\mu }\right)+\frac{16(\mu + C_{l_{yy}})}{\gamma_v\mu^2\sqrt{1+\sigma_u^2}}\left(\frac{nL_{l,2}C_{f_y}}{\mu}\!+\!L_{f,1}\right)^2\notag \\
&\quad\cdot\left[\frac{4 \gamma_x^2 \rho_W (1 + \zeta_q^2)}{\bar{z}_0^2(1 - \rho_W)^2} \left(\frac{(\mu+L_{l,1})^2\sqrt{1+\sigma_u^2}}{n\mu^2 }+\frac{2L_{l,1}\sqrt{1+\sigma_u^2}}{n\mu }\right)\right.\nonumber\\
&\quad +\left.\frac{\gamma_x^2L_y^2(1+\zeta_q^2)(\mu + L_{l,1})}{n\bar{z}_0\gamma_y \mu L_{l,1}}\right]\sum_{k=\min\{k_1-1,k_2-1\}}^{k_3-2} 
 \frac{\| \bar{\nabla} F(\mathbf{x}_{k},\mathbf{y}_{k},\mathbf{v}_{k})\|^2}{[\bar{m}^x_{k+1}]^2}\nonumber\\
&\quad+ \left(\frac{64 \gamma_x^2 \rho_W \bar{L}_r^2(\mu + C_{l_{yy}})(1 + \zeta_q^2)}{\mu C_{l_{yy}}\gamma_v(1 - \rho_W)^2} + \frac{2\gamma_x^2L_v^2(1+\zeta_q^2)(\mu + C_{l_{yy}})^2}{\mu C_{l_{yy}}C_{m^v}\gamma_v^2\sqrt{1+\sigma_z^2}}\right)\notag\\
&\quad\cdot\sum_{k=\min\{k_1-1,k_3-1\}}^{t} 
 \frac{\| \bar{\nabla} F(\mathbf{x}_{k},\mathbf{y}_{k},\mathbf{v}_{k})\|^2}{[\bar{m}^x_{k+1}]^2}\nonumber\\
&\quad  + \left(\frac{8}{\mu^2}+\frac{4(\mu + C_{l_{yy}})^2}{\mu^3 C_{l_{yy}}} \right)\left(\frac{nL_{l,2}C_{f_y}}{\mu}+L_{f,1}\right)^2 \sum_{k=k_3-1}^t\frac{\|\nabla_y L(\mathbf{x}_k,\mathbf{y}_k)\|^2}{\bar{m}_{k+1}^y}\notag \\
&\overset{(c)}{\leq} \frac{nC_{m^v}^2}{\bar{m}^v_{0}}\! -\! n\bar{m}^v_{0} \!+\! \frac{4nC_{m^y}^2(\mu + C_{l_{yy}})}{\mu^4\gamma_v\sqrt{1+\sigma_z^2}}\left(\frac{nL_{l,2}C_{f_y}}{\mu}+L_{f,1}\right)^2 + \frac{8nC_{m^v}^2(\mu + C_{l_{yy}})}{\mu^2\gamma_v\sqrt{1+\sigma_z^2}}\notag \\ 
&\quad +\left(\frac{16\Delta_0^x}{1-\rho_W}+\frac{64\gamma_x^2\rho_WC_{m^x}^2(1+\zeta_q^2)}{\bar{q}_0^2(1-\rho_W)^2}\right)\left(\frac{ \bar{L}_r^2(\mu + C_{l_{yy}})}{\mu C_{l_{yy}}\gamma_v}+\frac{\bar{L}_r^2(\mu+C_{l_{yy}})}{\mu^2\gamma_v\sqrt{1+\sigma_z^2}}\right)\nonumber\\
&\quad + \frac{16(\mu + C_{l_{yy}})}{\gamma_v\mu^2}\left(\frac{nL_{l,2}C_{f_y}}{\mu}\!+\!L_{f,1}\right)^2 \left(\frac{\Delta_0^x}{1-\rho_W}\!+\!\frac{4 \gamma_x^2 \rho_WC_{m^x}^2 (1 + \zeta_q^2)}{\bar{q}_0^2(1 - \rho_W)^2} \right)\notag\\
&\quad\cdot\left(\frac{(\mu+L_{l,1})^2}{n\mu^2 }\!+\!\frac{2L_{l,1}}{n\mu }\right)+\left[\frac{16(\mu + C_{l_{yy}})}{\gamma_v\mu^2\sqrt{1+\sigma_u^2}}\left(\frac{nL_{l,2}C_{f_y}}{\mu}\!+\!L_{f,1}\right)^2\right.\notag \\
&\quad \cdot\left(\frac{4 \gamma_x^2 \rho_W (1 + \zeta_q^2)}{\bar{z}_0^2(1 - \rho_W)^2} \left(\frac{(\mu+L_{l,1})^2\sqrt{1+\sigma_u^2}}{n\mu^2 }+\frac{2L_{l,1}\sqrt{1+\sigma_u^2}}{n\mu }\right) +\frac{\gamma_x^2L_y^2(1+\zeta_q^2)(\mu + L_{l,1})}{n\bar{z}_0\gamma_y \mu L_{l,1}}\right)\nonumber\\
&\quad + \left.\left(\frac{64 \gamma_x^2 \rho_W \bar{L}_r^2(\mu + C_{l_{yy}})(1 + \zeta_q^2)}{\mu C_{l_{yy}}\gamma_v(1 - \rho_W)^2} \!+\! \frac{2\gamma_x^2L_v^2(1+\zeta_q^2)(\mu + C_{l_{yy}})^2}{\mu C_{l_{yy}}C_{m^v}\gamma_v^2\sqrt{1+\sigma_z^2}}\right)\right]\! (5 \log(t + 1) \!+\! c_2) \notag \\
&\quad  + \left(\frac{8}{\mu^2}+\frac{4(\mu + C_{l_{yy}})^2}{\mu^3 C_{l_{yy}}} \right)\left(\frac{nL_{l,2}C_{f_y}}{\mu}+L_{f,1}\right)^2 (a_2 \log(t + 1) + b_2)\notag \\
&\overset{(d)}{=}:a_3 \log(t + 1) + b_3, \label{eq111}
\end{align}
where (a) allows $\sum_{k=k_0}^{k_3-1} \frac{\|\nabla_v  R(\mathbf{x}_k, \mathbf{y}_k, \mathbf{v}_k)\|^2}{\bar{z}_{k+1}} = 0$ when $k_0 \geq k_3$; (b) uses $C_{m^v} \geq \bar{m}^v_0$ and \cref{Lemma8}; (c) follows from \cref{eq107} and \cref{eq109}; (d) refers to \cref{eq106}. Since the upper bound of Case 2 is larger, we take \cref{eq111} as our final result.

Thus, the proof is complete.
\end{proof}

\subsection{The Upper Bound of Consensus Errors}
\begin{lemma}\label{Lemma5} Suppose \cref{Assumption3}, \cref{Assumption1}, and \cref{Assumption2} hold. Then, the consensus error $\Delta$ satisfies:
\begin{align}
 \sum_{k=0}^{t} \Delta_k 
&\leq \frac{2 \Delta_0}{1 - \rho_W} 
+ \frac{8 \gamma_x^2 \rho_W (1 + \zeta_q^2)( 5 \log(t) + c_2)}{(1 - \rho_W)^2} \nonumber \\
&\quad + \frac{8 \gamma_y^2 \rho_W (1 + \zeta_u^2)(a_2 \log(t) + b_2)}{(1 - \rho_W)^2}+\frac{8 \gamma_v^2 \rho_W (1 + \zeta_z^2)(a_3 \log(t) + b_3)}{(1 - \rho_W)^2}, \label{eq39}
\end{align}
where \(\Delta_0 \) is the initial consensus error, which can be set to 0 with proper initialization.
\end{lemma}

\begin{proof}
According to \cref{Lemmanewx}, we have:
\begin{align}
\sum_{k=0}^{t-1} \|\mathbf{x}_{t+1} - \mathbf{1} \bar{x}_{t+1} \|^2 \leq \frac{2}{1 - \rho_W} \|\mathbf{x}_0 - \mathbf{1} \bar{x}_0\|^2 
+ \frac{8 \gamma_x^2 \rho_W (1 + \zeta_q^2)}{(1 - \rho_W)^2} \sum_{k=0}^{t-1} 
 \bar{q}_{k+1}^{-2} \|\bar{\nabla}F(\mathbf{x}_k, \mathbf{y}_k, \mathbf{v}_k)\|^2.\label{eq42}
\end{align}

With the help of \cref{Lemma11}, we have:
\begin{align}
\sum_{k=0}^{t-1}  \bar{q}_{k+1}^{-2} \|\bar{\nabla} F(\mathbf{x}_k, \mathbf{y}_k, \mathbf{v}_k)\|^2 \leq\sum_{k=0}^{t-1}\frac{\|\bar{\nabla} F(\mathbf{x}_k,\mathbf{y}_k,\mathbf{v}_k)\|^2}{[\bar{m}^x_{k+1}]^2}
 \leq  5 \log(t) + c_2.\label{eq43}
\end{align}

Similarly, we can get the following inequality for the dual variable:
\begin{align}
    &\sum_{k=0}^{t-1}\|\mathbf{y}_{k+1} - \mathbf{1} \bar{y}_{k+1} \|^2\leq  \frac{2}{1 - \rho_W} \|\mathbf{y}_0 - \mathbf{1} \bar{y}_0\|^2 
+ \frac{8 \gamma_y^2 \rho_W (1 + \zeta_u^2)(a_2 \log(t) + b_2)}{(1 - \rho_W)^2}.\label{eq44}
\end{align}

For the auxiliary variable, we have:
\begin{align}
    v_{t+1} = \mathcal{P}_{\mathcal{V}} \left( W \left( \mathbf{v}_t - \gamma_v Z_{t+1}^{-1} \nabla_v R(\mathbf{x}_t, \mathbf{y}_t, \mathbf{v}_t) \right) \right) = W \mathbf{v}_t - \gamma_v \nabla_v \hat{G},\label{eq45}
\end{align}
where
\begin{equation}
    \nabla_v \hat{G} = \frac{1}{\gamma_v} \left( \mathcal{P}_{\mathcal{V}}  \left( W \left( \mathbf{v}_t - \gamma_v Z_{t+1}^{-1} \nabla_v R(\mathbf{x}_t, \mathbf{y}_t, \mathbf{v}_t)\right) \right) - W \mathbf{v}_t\right).\label{eq46}
\end{equation}

Using Young’s inequality with parameter \(\lambda\), we have:
\begin{align}
&\|\mathbf{v}_{t+1} - \mathbf{1} \bar{v}_{t+1} \|^2\nonumber\\
&=  \left\| W \mathbf{v}_{t} - \gamma_v \nabla_v \hat{G} - \mathbf{J} \left( W\mathbf{v}_{t}- \gamma_v \nabla_v \hat{G} \right) \right\|^2 \nonumber \\
&\leq (1 + \lambda) \rho_W  \|\mathbf{v}_{t} - \mathbf{J} \mathbf{v}_{t}\|^2 
+ \left( 1 + \frac{1}{\lambda} \right)  \left\|\mathcal{P}_{\mathcal{V}}  \left( W \left( \mathbf{v}_t - \gamma_v Z_{t+1}^{-1} \nabla_v R(\mathbf{x}_t, \mathbf{y}_t, \mathbf{v}_t)\right) \right) - W \mathbf{v}_t \right\|^2 \nonumber \\
&\leq \frac{1 + \rho_W}{2}  \|\mathbf{v}_{t} - \mathbf{J} \mathbf{v}_{t}\|^2
+ \frac{1 + \rho_W}{1 - \rho_W} \left\|\mathcal{P}_{\mathcal{V}}  \left( W \left( \mathbf{v}_t - \gamma_v Z_{t+1}^{-1} \nabla_v R(\mathbf{x}_t, \mathbf{y}_t, \mathbf{v}_t)\right) \right) - W \mathbf{v}_t \right\|^2.\label{eq47}
\end{align}

Noticing that \(W \mathbf{v}_t = \mathcal{P}_{\mathcal{V}}(W \mathbf{v}_t)\) holds for the convex set \(\mathcal{V}\), we get:
\begin{align}
&\|\mathbf{v}_{t+1} - \mathbf{1} \bar{v}_{t+1} \|^2\nonumber\\ 
&\leq \frac{1 + \rho_W}{2}  \|\mathbf{v}_{t} - \mathbf{J} \mathbf{v}_{t}\|^2
+ \frac{1 + \rho_W}{1 - \rho_W} \left\|\mathcal{P}_{\mathcal{V}}  \left( W \left( \mathbf{v}_t - \gamma_v Z_{t+1}^{-1} \nabla_v R(\mathbf{x}_t, \mathbf{y}_t, \mathbf{v}_t)\right) \right) - \mathcal{P}_{\mathcal{V}}(W \mathbf{v}_t) \right\|^2\nonumber\\
&\overset{(a)}{\leq}\frac{1 + \rho_W}{2}  \|\mathbf{v}_{t} - \mathbf{J} \mathbf{v}_{t}\|^2
+ \frac{1 + \rho_W}{1 - \rho_W} \rho_W\| \gamma_v Z_{t+1}^{-1} \nabla_v R(\mathbf{x}_t, \mathbf{y}_t, \mathbf{v}_t) \|^2, \label{eq48}
\end{align}
where (a) uses the non-expansiveness of the projection operator, as shown in Lemma 1 of \citep{nedic2010constrained}. Then, we have:
\begin{align}
\sum_{k=0}^{t-1} \|\mathbf{v}_{k} - \mathbf{1} \bar{v}_{k} \|^2
&\leq \frac{2}{1 - \rho_W} \|\mathbf{v}_{0} - \mathbf{1} \mathbf{v}_{0} \|^2 
+ \frac{8 \gamma_v^2\rho_W (1 + \zeta_z^2)}{(1 - \rho_W)^2} \sum_{k=0}^{t-1} \bar{z}_{k+1}^{-2} \|\nabla_v R(\mathbf{x}_k, \mathbf{y}_k, \mathbf{v}_k)\|^2.\label{eq49}
\end{align}
Similar to the primal and the dual variable, we can bound the last term above, which completes the proof.
\end{proof}

\subsection{The Upper Bounds of $\sum \frac{\| \bar{\nabla} f(\bar{x}_{t},\bar{y}_{t},\bar{v}_{t})\|^2}{[\bar{m}^x_{t+1}]^2}$, $\sum \frac{\|\nabla_y l(\bar{x}_t, \bar{y}_t)\|^2}{\bar{m}^y_{t+1}}$, and $\sum \frac{\|\nabla_v r(\bar{x}_t, \bar{y}_t, \bar{v}_t)\|^2}{\bar{z}_{t+1}}$}\label{secc80}
Through \cref{Lemma3}, \cref{Lemma11}, and \cref{Lemma5}, we can derive the upper bound for $\sum_{k=k_0}^t \frac{\|\bar{\nabla} f(\bar{x}_k,\bar{y}_k,\bar{v}_k)\|^2}{[\bar{m}^x_{k+1}]^2} $, $\sum_{k=k_0}^t \frac{\|\nabla_y l(\bar{x}_k, \bar{y}_k)\|^2}{\bar{m}^y_{k+1}}$, and $\sum_{k=k_0}^t \frac{\|\nabla_v r(\bar{x}_k, \bar{y}_k, \bar{v}_k)\|^2}{\bar{z}_{k+1}}$ in the following lemma. 
\begin{lemma}\label{lemmanew12}
Under \cref{Assumption3}, \cref{Assumption1} and \cref{Assumption2}, for any integer $k_0 \in [0, t)$, we have the upper bounds in terms of logarithmic functions as:
\begin{align}
\sum_{k=k_0}^t \frac{\|\bar{\nabla} f(\bar{x}_k,\bar{y}_k,\bar{v}_k)\|^2}{[\bar{m}^x_{k+1}]^2} &\leq{a_4\log(t + 1)+b_4},\nonumber\\
\sum_{k=k_0}^t \frac{\|\nabla_y l(\bar{x}_k, \bar{y}_k)\|^2}{\bar{m}^y_{k+1}} &\leq {a_5 \log(t + 1)+b_5}, \nonumber\\
\sum_{k=k_0}^t \frac{\|\nabla_v R(\bar{x}_k, \bar{y}_k, \bar{v}_k)\|^2}{\bar{z}_{k+1}} &\leq {a_6 \log(t + 1) + b_6},\label{eqnewlemma}
\end{align}
where
\begin{align}
a_4&:=\frac{5}{2n^2}+ \frac{8\rho_W\bar{L}_{f}^2(5 \gamma_x^2 (1 + \zeta_v^2)+a_2 \gamma_y^2 (1 + \zeta_u^2)+a_3 \gamma_v^2  (1 + \zeta_z^2))}{[\bar{m}_0^x]^2(1 - \rho_W)^2 }, \nonumber\\
b_4&:=\frac{c_2}{2n^2}+\frac{2\bar{L}_{f}^2 \Delta_0}{[\bar{m}^x_0]^2(1 - \rho_W)}+\frac{8\rho_W\bar{L}_{f}^2(c_2 \gamma_x^2  (1 + \zeta_v^2)+b_2\gamma_y^2 (1 + \zeta_u^2)+b_3\gamma_v^2 (1 + \zeta_z^2))}{[\bar{m}^x_0]^2(1 - \rho_W)^2 }, \nonumber\\
a_5&:=\frac{a_2}{2n^2}+ \frac{8\rho_WL_{l,1}^2(5 \gamma_x^2 (1 + \zeta_v^2)+a_2 \gamma_y^2 (1 + \zeta_u^2))}{\bar{m}_0^y(1 - \rho_W)^2 }, \nonumber\\
b_5&:=\frac{b_2}{2n^2}+\frac{2L_{l,1}^2  (\|\mathbf{x}_{0} - \mathbf{1} \bar{x}_{0} \|^2+\|\mathbf{y}_{0} - \mathbf{1} \bar{y}_{0} \|^2)}{\bar{m}^y_0(1 - \rho_W)}+\frac{8\rho_WL_{l,1}(c_2 \gamma_x^2  (1 + \zeta_v^2)+b_2\gamma_y^2 (1 + \zeta_u^2))}{\bar{m}^y_0(1 - \rho_W)^2 }, \nonumber\\
a_6&:=\frac{a_3}{2n^2}+ \frac{8\rho_W\bar{L}_{R}^2(5 \gamma_x^2 (1 + \zeta_v^2)+a_2 \gamma_y^2 (1 + \zeta_u^2)+a_3 \gamma_v^2  (1 + \zeta_z^2))}{\bar{m}_0^v(1 - \rho_W)^2 }, \nonumber\\
b_6&:=\frac{b_3}{2n^2}+\frac{2\bar{L}_{R}^2 \Delta_0}{\bar{m}^v_0(1 - \rho_W)}+\frac{8\rho_W\bar{L}_{R}^2(c_2 \gamma_x^2  (1 + \zeta_v^2)+b_2\gamma_y^2 (1 + \zeta_u^2)+b_3\gamma_v^2 (1 + \zeta_z^2))}{\bar{m}^v_0(1 - \rho_W)^2 }.
\label{eqnewlemma13}
\end{align}
\end{lemma}

\begin{proof}
According to \cref{Lemma3}, we have:
\begin{equation}
   \|\bar{\nabla} F(\mathbf{x}_t, \mathbf{y}_t, \mathbf{v}_t)\|^2 \leq 2\|\bar{\nabla} F(\mathbf{1}\bar{x}_t, \mathbf{1}\bar{y}_t, \mathbf{1}\bar{v}_t)\|^2 + 2\bar{L}_{f}^2 (\|\mathbf{x}_t - \mathbf{1}\bar{x}_t\|^2 + \|\mathbf{y}_t - \mathbf{1}\bar{y}_t\|^2+ \|\mathbf{v}_t - \mathbf{1}\bar{v}_t\|^2),\label{eq121}
\end{equation}
\begin{equation}
   \|\nabla_y L(\mathbf{x}_t, \mathbf{y}_t)\|^2 \leq 2\|\nabla_yL(\mathbf{1}\bar{x}_t, \mathbf{1}\bar{y}_t)\|^2 + 2L_{l,1}^2 (\|\mathbf{x}_t - \mathbf{1}\bar{x}_t\|^2 + \|\mathbf{y}_t - \mathbf{1}\bar{y}_t\|^2),\label{eq122}
\end{equation}
\begin{equation}
   \|\nabla_v R(\mathbf{x}_t, \mathbf{y}_t, \mathbf{v}_t)\|^2 \leq 2\|\nabla_v R(\mathbf{1}\bar{x}_t, \mathbf{1}\bar{y}_t, \mathbf{1}\bar{v}_t)\|^2 + 2\bar{L}_r^2( \|\mathbf{x}_t - \mathbf{1}\bar{x}_t\|^2 + \|\mathbf{y}_t - \mathbf{1}\bar{y}_t\|^2+ \|\mathbf{v}_t - \mathbf{1}\bar{v}_t\|^2).\label{eq123}
\end{equation}

Based on \cref{eqnewlemmaproof1}, we have $\|\bar{\nabla} F(\mathbf{1}\bar{x}_t, \mathbf{1}\bar{y}_t, \mathbf{1}\bar{v}_t)\|^2 \leq\|\bar{\nabla} F(\mathbf{1}\bar{x}_t, \mathbf{1}\bar{y}_t, \mathbf{1}\bar{v}_t)\|_{\rm F}^2=\|n\bar{\nabla} f(\bar{x}_t,\bar{y}_t,\bar{v}_t)\|^2$, $\|\nabla_y L(\mathbf{1}\bar{x}_t, \mathbf{1}\bar{y}_t)\|^2\leq\|\nabla_y L(\mathbf{1}\bar{x}_t, \mathbf{1}\bar{y}_t)\|_{\rm F}^2=\|n\nabla_y l(\bar{x}_t, \bar{y}_t)\|^2$ and $\|\nabla_v R(\mathbf{1}\bar{x}_t, \mathbf{1}\bar{y}_t, \mathbf{1}\bar{v}_t)\|^2\leq\|\nabla_v R(\mathbf{1}\bar{x}_t, \mathbf{1}\bar{y}_t, \mathbf{1}\bar{v}_t)\|_{\rm F}^2=\|n\nabla_v r(\bar{x}_t, \bar{y}_t, \bar{v}_t)\|^2$, then according to \cref{Lemma11} and \cref{Lemma5}, we have:
\begin{align}
\sum_{k=k_0}^t \frac{\|\bar{\nabla} f(\bar{x}_k,\bar{y}_k,\bar{v}_k)\|^2}{[\bar{m}^x_{k+1}]^2} &\leq \frac{5\log(t + 1)+c_2}{2n^2} + \frac{\bar{L}_{f}^2d_1}{[\bar{m}^x_0]^2}\leq{a_4\log(t + 1)+b_4},\nonumber\\
\sum_{k=k_0}^t \frac{\|\nabla_y l(\bar{x}_k, \bar{y}_k)\|^2}{\bar{m}^y_{k+1}} &\leq \frac{a_2 \log(t + 1)+b_2}{2n^2} + \frac{L_{l,1}^2d_2}{\bar{m}^y_0}\leq {a_5 \log(t + 1)+b_5}, \nonumber\\
\sum_{k=k_0}^t \frac{\|\nabla_v r(\bar{x}_k, \bar{y}_k, \bar{v}_k)\|^2}{\bar{z}_{k+1}} &\leq \frac{a_3 \log(t + 1) + b_3}{2n^2}+\frac{\bar{L}_{R}^2d_1}{\bar{m}^y_0}\leq {a_6 \log(t + 1) + b_6},\label{eq124}
\end{align}
where $a_4, b_4, a_5, b_5, a_6, b_6$ can refer to \cref{eqnewlemma13} and 
\begin{align}
d_1&:=\frac{2 (\|\mathbf{x}_{0} - \mathbf{1} \bar{x}_{0} \|^2+\|\mathbf{y}_{0} - \mathbf{1} \bar{y}_{0} \|^2+\|\mathbf{v}_{0} - \mathbf{1} \bar{v}_{0} \|^2)}{1 - \rho_W}  
+ \frac{8 \gamma_x^2 \rho_W (1 + \zeta_q^2)( 5 \log(t + 1) + c_2)}{(1 - \rho_W)^2 }, \nonumber \\
&\quad + \frac{8 \gamma_y^2 \rho_W (1 + \zeta_u^2)(a_2 \log(t + 1) + b_2)}{(1 - \rho_W)^2}+\frac{8 \gamma_v^2 \rho_W (1 + \zeta_z^2)(a_3 \log(t + 1) + b_3)}{(1 - \rho_W)^2}\nonumber\\
d_2&:=\frac{2 (\|\mathbf{x}_{0} - \mathbf{1} \bar{x}_{0} \|^2+\|\mathbf{y}_{0} - \mathbf{1} \bar{y}_{0} \|^2)}{1 - \rho_W} \notag\\
&\quad+ \frac{8 \gamma_x^2 \rho_W (1 + \zeta_q^2)( 5 \log(t + 1) + c_2)}{(1 - \rho_W)^2} + \frac{8 \gamma_y^2 \rho_W (1 + \zeta_u^2)(a_2 \log(t + 1) + b_2)}{(1 - \rho_W)^2}.
\label{eq125}
\end{align}
Thus, the proof is completed.
\end{proof}

\subsection{The Upper Bound of Stepsize Inconsistencies}
\begin{lemma}\label{Lemma12}Suppose \cref{Assumption3}, \cref{Assumption1}, and \cref{Assumption2} hold. For the proposed \cref{alg1}, we have:
\begin{align}
&\sum_{k=0}^{t-1} \left\| \frac{\big(\tilde{\mathbf{q}}_{k+1}^{-1}\big)^\top}{n\bar{q}_{k+1}^{-1}} \bar{\nabla} F(\mathbf{x}_k, \mathbf{y}_k, \mathbf{v}_k) \right\|^2\notag\\
&\leq \frac{(1+\zeta_q^2)(5 \log(t) + c_2)(2(1 + \rho_W) \rho_W)}{n\bar{z}_0^2(1 - \rho_W)^2}\left(\frac{C_{l_{xy}}C_{f_y}}{\mu}+C_{f_x}\right)^2\left[C_{l_y}^2+\left(\frac{C_{l_{yy}}C_{f_y}}{\mu}+C_{f_y}\right)^2\right],\notag\\
&\sum_{k=0}^{t-1} \left\| \frac{\big(\tilde{\mathbf{u}}_{k+1}^{-1}\big)^\top}{n\bar{u}_{k+1}^{-1}} {\nabla}_y L(\mathbf{x}_k, \mathbf{y}_k) \right\|^2\leq \frac{(1+\zeta_u^2)(a_2 \log(t) + b_2)(2(1 + \rho_W) \rho_W)C_{l_y}^2}{n\bar{u}_0(1 - \rho_W)^2},\notag\\
&\sum_{k=0}^{t-1} \left\| \frac{\big(\tilde{\mathbf{z}}_{k+1}^{-1}\big)^\top}{n\bar{z}_{k+1}^{-1}} \nabla_v R(\mathbf{x}_k, \mathbf{y}_k, \mathbf{v}_k) \right\|^2\notag\\
&\leq \frac{(1+\zeta_z^2)(a_3 \log(t) + b_3)(2(1 + \rho_W) \rho_W)}{n\bar{z}_0(1 - \rho_W)^2}\left[C_{l_y}^2+\left(\frac{C_{l_{yy}}C_{f_y}}{\mu}+C_{f_y}\right)^2\right].\label{eq112}
\end{align}
\end{lemma}

\begin{proof}
By the definition of $q_{i,k}$ in \cref{eqquz}, we have:
\begin{align} 
&\left\|\frac{\big(\tilde{\mathbf{q}}_{t+1}^{-1}\big)^\top}{n\bar{q}_{t+1}^{-1}} \bar{\nabla} F(\mathbf{x}_t, \mathbf{y}_t, \mathbf{v}_t)
\right\|^2 \notag\\
&\leq 
\frac{1}{n^2} \sum_{i=1}^n 
\left(
\bar{q}_{t+1} - q_{i,t+1} 
\right)^2 
\frac{\| \bar{\nabla} F(\mathbf{x}_{t}, \mathbf{y}_{t}, \mathbf{v}_{t}) \|^2}{q_{i,t+1}^{2}}\notag \\
&\leq  \sum_{i=1}^n 
\left\| 
\bar{q}_{t+1} - q_{i,t+1}\right\|^2 \frac{1+\zeta_q^2}{n^2}\frac{\| \bar{\nabla} F(\mathbf{x}_{t}, \mathbf{y}_{t}, \mathbf{v}_{t}) \|^2}{\bar{q}_{t+1}^2}\notag\\
&\leq \sum_{i=1}^n 
\left\| 
\bar{q}_{t+1} - q_{i,t+1}\right\|^2 \frac{1+\zeta_q^2}{n^2\bar{z}_0^2}\frac{\| \bar{\nabla} F(\mathbf{x}_{t}, \mathbf{y}_{t}, \mathbf{v}_{t}) \|^2}{[\bar{m}^x_{t+1}]^2}.\label{eq113}
\end{align}

According to \cref{eq105}, we have:
\begin{align}
&\sum_{k=0}^{t-1} 
\left\|\frac{\big(\tilde{\mathbf{q}}_{k+1}^{-1}\big)^\top}{n\bar{q}_{k+1}^{-1}} \bar{\nabla} F(\mathbf{x}_k, \mathbf{y}_k, \mathbf{v}_k)
\right\|^2 \notag\\
&\leq\frac{1+\zeta_q^2}{n^2\bar{z}_0^2}\left\| \mathbf{q}_{k+1}-\mathbf{1}\bar{q}_{k+1}\right\|^2 \sum_{k=0}^{t-1}\frac{\| \bar{\nabla} F(\mathbf{x}_{k}, \mathbf{y}_{k}, \mathbf{v}_{k}) \|^2}{[\bar{m}^x_{k+1}]^2}\notag \\
&\overset{(a)}{\leq}\frac{(1+\zeta_q^2)(5 \log(t + 1) + c_2)}{n^2\bar{z}_0^2}\left\| \mathbf{q}_{k+1}-\mathbf{1}\bar{q}_{k+1}\right\|^2,\label{eq114}
\end{align}
where (a) uses \cref{Lemma11}.
Next, for the term of inconsistency of the stepsize $\left\| \mathbf{q}_{k}-\mathbf{1}\bar{q}_{k}\right\|^2$, we consider two cases due to the max operator used ({i.e.,} $\mathbf{m}_k^y \geq \mathbf{m}_k^v$ and $\mathbf{m}_k^y < \mathbf{m}_k^v$). First of all, we derive the bound for $\left\| \mathbf{m}_{k+1}^x-\mathbf{1}\bar{m}_{k+1}^x\right\|^2$, $\left\| \mathbf{m}_{k+1}^y-\mathbf{1}\bar{m}_{k+1}^y\right\|^2$, and $\left\| \mathbf{m}_{k+1}^v-\mathbf{1}\bar{m}_{k+1}^v\right\|^2$. For $\left\| \mathbf{m}_{k+1}^x-\mathbf{1}\bar{m}_{k+1}^x\right\|^2$, we have:
\begin{align}
&\left\| \mathbf{m}_{k+1}^x-\mathbf{1}\bar{m}_{k+1}^x\right\|^2\notag\\
&\leq\| (W - \mathbf{J}) ( \mathbf{m}_{k}^x-\mathbf{1}\bar{m}_{k}^x)\|^2 + \|(W - \mathbf{J}) \mathbf{g}_k^x \|^2  \notag \\
&\leq \frac{1 + \rho_W}{2} 
\| \mathbf{m}_k^x - \mathbf{1}\bar{m}_k^x \|^2 
+ \frac{(1 + \rho_W) \rho_W}{1 - \rho_W} \| \mathbf{g}_k^x \|^2 \notag \\
&\overset{(a)}{\leq} 
\left( \frac{1 + \rho_W}{2} \right)^k 
\| \mathbf{m}_0^x - \mathbf{1}\bar{m}_0^x \|^2 
+ \frac{n (1 + \rho_W) \rho_W}{(1 - \rho_W)^2}\left(\frac{C_{l_{xy}}C_{f_y}}{\mu}+C_{f_x}\right)^2 
\sum_{t=0}^k 
\left( \frac{1 + \rho_W}{2} \right)^{k-t} \notag \\
&\leq \frac{2n (1 + \rho_W) \rho_W}{(1 - \rho_W)^2}\left(\frac{C_{l_{xy}}C_{f_y}}{\mu}+C_{f_x}\right)^2,\label{eq115}
\end{align}
where (a) uses \cref{Lemma2}. For $\left\| \mathbf{m}_{k+1}^y-\mathbf{1}\bar{m}_{k+1}^y\right\|^2$, we have:
\begin{align}
&\left\| \mathbf{m}_{k+1}^y-\mathbf{1}\bar{m}_{k+1}^y\right\|^2\notag\\
&\leq\| (W - \mathbf{J}) ( \mathbf{m}_{k}^y-\mathbf{1}\bar{m}_{k}^y)\|^2 +  \|(W - \mathbf{J}) \mathbf{g}_k^y \|^2  \notag \\
&\leq \frac{1 + \rho_W}{2} 
\| \mathbf{m}_k^y - \mathbf{1}\bar{m}_k^y \|^2 
+ \frac{(1 + \rho_W) \rho_W}{1 - \rho_W} \| \mathbf{g}_k^y \|^2 \notag \\
&\overset{(a)}{\leq}
\left( \frac{1 + \rho_W}{2} \right)^k 
\| \mathbf{m}_0^y - \mathbf{1}\bar{m}_0^y \|^2 
+ \frac{nC_{l_y}^2 (1 + \rho_W) \rho_W}{(1 - \rho_W)^2} 
\sum_{t=0}^k 
\left( \frac{1 + \rho_W}{2} \right)^{k-t} \notag \\
&\leq \frac{2nC_{l_y}^2 (1 + \rho_W) \rho_W}{(1 - \rho_W)^2},\label{eq116}
\end{align}
where $C_{l_y} = n^2(C_{m^y} + c_0 + d_0(5 \log 2 + c_2))$. The inequality in (a) follows from \cref{Lemma9} and \cref{Lemma11}. Specifically, since $\|\mathbf{h}_0^y\|^2 \leq \|\mathbf{m}_1^y\|^2 \leq \|n \bar{m}_1^y\|^2 \leq n^2(C_{m^y} + c_0 + d_0(5 \log 2 + c_2))$, it follows that the magnitude of the gradient $\|\mathbf{h}_0^y\|$ is upper bounded by $C_{l_y}$.
For $\left\| \mathbf{m}_{k+1}^v-\mathbf{1}\bar{m}_{k+1}^v\right\|^2$, we have:
\begin{align}
&\left\| \mathbf{m}_{k+1}^v-\mathbf{1}\bar{m}_{k+1}^v\right\|^2\notag\\
&\leq\| (W - \mathbf{J}) ( \mathbf{m}_{k}^v-\mathbf{1}\bar{m}_{k}^v)\|^2  +  \|(W - \mathbf{J}) \mathbf{g}_k^v \|^2  \notag \\
&\leq \frac{1 + \rho_W}{2} 
\| \mathbf{m}_k^v - \mathbf{1}\bar{m}_k^v \|^2 
+ \frac{(1 + \rho_W) \rho_W}{1 - \rho_W} \| \mathbf{g}_k^v \|^2 \notag \\
&\overset{(a)}{\leq} 
\left( \frac{1 + \rho_W}{2} \right)^k 
\| \mathbf{m}_0^v - \mathbf{1}\bar{m}_0^v \|^2 
+ \frac{n (1 + \rho_W) \rho_W}{(1 - \rho_W)^2} \left(\frac{C_{l_{yy}}C_{f_y}}{\mu}+C_{f_y}\right)^2
\sum_{t=0}^k 
\left( \frac{1 + \rho_W}{2} \right)^{k-t} \notag \\
&\leq \frac{2n (1 + \rho_W) \rho_W}{(1 - \rho_W)^2}\left(\frac{C_{l_{yy}}C_{f_y}}{\mu}+C_{f_y}\right)^2,\label{eq117}
\end{align}
where (a) uses \cref{Lemma3}.

At iteration $k$, for the case $\mathbf{m}_k^y \geq \mathbf{m}_k^v$ with $\| \mathbf{m}_0^x\mathbf{m}_0^y - \mathbf{1}\bar{m}_0^x\bar{m}_0^y \|^2 = 0$, we have:
\begin{align}
&\left\| \mathbf{q}_{k+1}-\mathbf{1}\bar{q}_{k+1}\right\|^2=\left\| \mathbf{m}_{k+1}^x\mathbf{m}_{k+1}^y-\mathbf{1}\bar{m}_{k+1}^x\bar{m}_{k+1}^y\right\|^2\notag\\
&\leq\| (W - \mathbf{J}) ( \mathbf{m}_{k}^x\mathbf{m}_{k}^y-\mathbf{1}\bar{m}_{k}^x\bar{m}_{k}^y)\|^2 +  \|(W - \mathbf{J}) \mathbf{g}_k^x\mathbf{g}_k^y \|^2  \notag \\
&\leq \frac{1 + \rho_W}{2} 
\| \mathbf{m}_k^x\mathbf{m}_k^y - \mathbf{1}\bar{m}_k^x\bar{m}_k^y \|^2 
+ \frac{(1 + \rho_W) \rho_W}{1 - \rho_W} \| \mathbf{g}_k^x\mathbf{g}_k^y \|^2 \notag \\
&{\leq} 
\left( \frac{1 + \rho_W}{2} \right)^k 
\| \mathbf{m}_0^x\mathbf{m}_0^y - \mathbf{1}\bar{m}_0^x\bar{m}_0^y \|^2 
\nonumber\\
&\quad+ \frac{nC_{l_y}^2 (1 + \rho_W) \rho_W}{(1 - \rho_W)^2} 
\left(\frac{C_{l_{xy}}C_{f_y}}{\mu}+C_{f_x}\right)^2\sum_{t=0}^k 
\left( \frac{1 + \rho_W}{2} \right)^{k-t} \notag \\
&\leq \frac{2nC_{l_y}^2 (1 + \rho_W) \rho_W}{(1 - \rho_W)^2}\left(\frac{C_{l_{xy}}C_{f_y}}{\mu}+C_{f_x}\right)^2.\label{eq118}
\end{align}

For the case $\mathbf{m}_k^y < \mathbf{m}_k^v$ with $\| \mathbf{m}_0^x\mathbf{m}_0^v - \mathbf{1}\bar{m}_0^x\bar{m}_0^v\|^2 = 0$, we have:
\begin{align}
&\left\| \mathbf{q}_{k+1}-\mathbf{1}\bar{q}_{k+1}\right\|^2\notag\\
&=\left\| \mathbf{m}_{k+1}^x\mathbf{m}_{k+1}^v-\mathbf{1}\bar{m}_{k+1}^x\bar{m}_{k+1}^v\right\|^2\leq \frac{2n (1 + \rho_W) \rho_W}{(1 - \rho_W)^2}\left(\frac{C_{l_{xy}}C_{f_y}}{\mu}+C_{f_x}\right)^2\left(\frac{C_{l_{yy}}C_{f_y}}{\mu}+C_{f_y}\right)^2\!\!.\label{eq119}
\end{align}

By summing \cref{eq118} and \cref{eq119}, we obtain the following inequality:
\begin{align}
\left\| \mathbf{q}_{k+1}-\mathbf{1}\bar{q}_{k+1}\right\|^2\leq \frac{2n (1 + \rho_W) \rho_W}{(1 - \rho_W)^2}\left(\frac{C_{l_{xy}}C_{f_y}}{\mu}+C_{f_x}\right)^2\left[C_{l_y}^2+\left(\frac{C_{l_{yy}}C_{f_y}}{\mu}+C_{f_y}\right)^2\right].\label{eq120}
\end{align}

Combining \cref{eq120} and \cref{eq114}, we can get the upper bound for $\sum_{k=0}^{t-1} \left\| {\big(\tilde{\mathbf{u}}_{k+1}^{-1}\big)^\top} {\nabla}_y L(\mathbf{x}_k, \mathbf{y}_k) /{n\bar{u}_{k+1}^{-1}}\right\|^2$ in \cref{eq112}. 

Similarly, we have:
\begin{align}
\left\| \mathbf{u}_{k+1}-\mathbf{1}\bar{u}_{k+1}\right\|^2\leq \frac{2nC_{l_y}^2 (1 + \rho_W) \rho_W}{(1 - \rho_W)^2},\label{eq11119}
\end{align}
and
\begin{align}
\left\| \mathbf{z}_{k+1}-\mathbf{1}\bar{z}_{k+1}\right\|^2\leq \frac{2n (1 + \rho_W) \rho_W}{(1 - \rho_W)^2}\left[C_{l_y}^2+\left(\frac{C_{l_{yy}}C_{f_y}}{\mu}+C_{f_y}\right)^2\right].\label{eq111191}
\end{align}
Then combine the results in \cref{Lemma11}, we can get the upper bound for $\sum_{k=0}^{t-1} \left\| {\big(\tilde{\mathbf{z}}_{k+1}^{-1}\big)^\top} \nabla_v R(\mathbf{x}_k, \mathbf{y}_k, \mathbf{v}_k) /{n\bar{z}_{k+1}^{-1}}\right\|^2$ in \cref{eq112}.
Thus, the \cref{Lemma12} has been proved.
\end{proof}

\subsection{The Upper Bound of $\bar{m}^x_t$}
\begin{lemma}\label{Lemma13} Under \cref{Assumption1} and \cref{Assumption2}, suppose the number of total iteration rounds in \cref{alg1} is $T$. If there exists $k_1 \leq T$ as described in \cref{Lemma6}, then we have:
\begin{equation}
\bar{m}^x_t \leq
\begin{cases}
C_{m^x}, & t \leq k_1, \\
C_{m^x} + {\left(4\left(\frac{\Phi(\bar{x}_{0})-\Phi^*}{\gamma_x}\right)+a_7\log(t+1)+b_7\right)\bar{z}_{t+1}}, & t \geq k_1,
\end{cases}\label{eq126}
\end{equation}
where $a_7$ and $b_7$ are defined as:
\begin{align}
  a_7 &:=  \frac{4\bar{L}^2a_6 }{\mu^2\bar{m}^x_{0}}+\frac{2\bar{L}^2a_5}{\mu^2\bar{m}^x_{0}} \left( 1 + \frac{2}{\mu^2} \left( \frac{L_{l,2} C_{f_y}}{\mu} + L_{f,1} \right)^2 \right) \nonumber\\
  &\quad+ \frac{80(1+\zeta_q^2)((1 + \rho_W) \rho_W)}{n\bar{z}_0^2\bar{q}_0(1 - \rho_W)^2}\left(\frac{C_{l_{xy}}C_{f_y}}{\mu}+C_{f_x}\right)^2\left[C_{l_y}^2+\left(\frac{C_{l_{yy}}C_{f_y}}{\mu}+C_{f_y}\right)^2\right] \nonumber\\
&\quad+\frac{4\bar{q}_{0}^{-1}(4\gamma_x \bar{L}_{f}^2L_\Phi \bar{q}_0^{-1} \left(1 +\zeta_q^2\right)+\bar{L}_f^2)}{n}\left(\frac{2 \Delta_0}{1 - \rho_W} 
+ \frac{40 \gamma_x^2 \rho_W (1 + \zeta_q^2)}{(1 - \rho_W)^2}\right. \nonumber \\
&\quad \left.+ \frac{8a_2 \gamma_y^2 \rho_W (1 + \zeta_u^2)}{(1 - \rho_W)^2}+\frac{8a_3 \gamma_v^2 \rho_W (1 + \zeta_z^2)}{(1 - \rho_W)^2}\right),\nonumber
\end{align}
\begin{align}
b_7 &:=  \frac{4\bar{L}^2b_6}{\mu^2\bar{m}^x_{0}}+\frac{2\bar{L}^2b_5}{\mu^2\bar{m}^x_{0}} \left( 1 + \frac{2}{\mu^2} \left( \frac{L_{l,2} C_{f_y}}{\mu} + L_{f,1} \right)^2 \right)\nonumber\\
&\quad+ \frac{16c_2(1+\zeta_q^2)((1 + \rho_W) \rho_W)}{n\bar{z}_0^2\bar{q}_0(1 - \rho_W)^2}\left(\frac{C_{l_{xy}}C_{f_y}}{\mu}+C_{f_x}\right)^2\left[C_{l_y}^2+\left(\frac{C_{l_{yy}}C_{f_y}}{\mu}+C_{f_y}\right)^2\right] \nonumber\\
&\quad+\frac{4\bar{q}_{0}^{-1}(4\gamma_x \bar{L}_{f}^2L_\Phi \bar{q}_0^{-1} \left(1 +\zeta_q^2\right)+\bar{L}_f^2)}{n}\left(\frac{2 \Delta_0}{1 - \rho_W} 
+ \frac{8c_2 \gamma_x^2 \rho_W (1 + \zeta_q^2)}{(1 - \rho_W)^2}\right. \nonumber \\
&\quad \left.+ \frac{8b_2 \gamma_y^2 \rho_W (1 + \zeta_u^2)}{(1 - \rho_W)^2}+\frac{8 b_3\gamma_v^2 \rho_W (1 + \zeta_z^2)}{(1 - \rho_W)^2}\right)+ \frac{8n\gamma_x C_{m^x}^2 L_\Phi \left(1 +\zeta_q^2\right)}{\bar{q}_0^2},\label{eq128}
\end{align}
and the upper bound of $\bar{z}_t := \max\{\bar{m}^y_t, \bar{m}^v_t\}$ refers to \cref{Lemma10}. When such $k_1$ does not exist, we have $\bar{m}^x_t \leq C_{m^x}$ for any $t \leq T$.
\end{lemma}

\begin{proof}
According to \cref{Lemma6}, the proof can be split into the following three cases:

\textbf{Case 1:} If $\bar{m}^x_T \leq C_{m^x}$, for any $t < T$, we have the upper bound of $\bar{m}^x_{t+1}$ as $\bar{m}^x_{t+1} \leq C_{m^x}$.

\textbf{Case 2:} If $\bar{m}^x_T > C_{m^x}$, there exists $k_1 \leq T$ described in \cref{Lemma6}. Then we have the upper bound of $\bar{m}^x_{t+1}$ as $\bar{m}^x_{t+1} \leq C_{m^x}$ for any $t < k_1$.

\textbf{Case 3:} In the remaining proof, we only consider and explore the case $k_1 \leq t \leq T$ when $\bar{m}^x_T > C_{m^x}$.

From \cref{Lemma4}, for $k \geq k_1$, we have:
\begin{align}
&\Phi(\bar{x}_{t+1})\nonumber\\
&\leq\Phi(\bar{x}_t)-\frac{\gamma_x \bar{q}_{t+1}^{-1}}{8}\|\nabla \Phi(\bar{x}_t)\|^2+\frac{\gamma_x\bar{q}_{t+1}^{-1}(2\gamma_x \bar{L}_{f}^2L_\Phi \bar{q}_{t+1}^{-1} \left(1 +\zeta_q^2\right)+\bar{L}_{f}^2)}{n} \Delta_t\nonumber\\
&\quad -\left(\frac{\gamma_x}{2}-{2n\gamma_x^2 L_\Phi \bar{q}_{t+1}^{-1} \left(1 +\zeta_q^2\right)}\right)\frac{\|\nabla_x f(\bar{x}_t, \bar{y}_t, \bar{v}_t)\|^2}{\bar{q}_{t+1}}
+ \frac{\bar{L}^2\gamma_x }{\mu^2} \frac{\|\nabla_v r(\bar{x}_t, \bar{y}_t, \bar{v}_t)\|^2}{\bar{q}_{t+1}} \nonumber \\
&\quad +\!\left(\! \frac{\gamma_x\bar{L}^2}{2\mu^2} \!+\! \frac{\gamma_x\bar{L}^2}{\mu^4}\!\! \left(\! \frac{L_{l,2} C_{f_y}}{\mu} \!+\! L_{f,1} \!\right)^2 \right) \frac{\|\nabla_y l(\bar{x}_t, \bar{y}_t)\|^2}{\bar{q}_{t+1}}\!+\! 2\gamma_x \bar{q}_{t+1}^{-1}  \left\| \frac{\big(\tilde{\mathbf{q}}_{t+1}^{-1}\big)^\top}{n\bar{q}_{t+1}^{-1}} \bar{\nabla} F(\mathbf{x}_t, \mathbf{y}_t, \mathbf{v}_t) \right\|^2\!\!.\label{eq129}
\end{align}
In addition, if $k_1$ in \cref{Lemma6} exists, then for $t \geq k_1$, we have $\bar{m}^x_{t+1} > C_{m^x} \geq \frac{8n\gamma_x L_\Phi \left(1 +\zeta_q^2\right)}{\bar{z}_0}$ and 
\begin{align}
&\Phi(\bar{x}_{t+1})\nonumber\\
&\leq\Phi(\bar{x}_t)-\frac{\gamma_x \bar{q}_{t+1}^{-1}}{8}\|\nabla \Phi(\bar{x}_t)\|^2+\frac{\gamma_x\bar{q}_{t+1}^{-1}(2\gamma_x \bar{L}_{f}^2L_\Phi \bar{q}_{t+1}^{-1} \left(1 +\zeta_q^2\right)+\bar{L}_{f}^2)}{n} \Delta_t\nonumber\\
&\quad -\frac{\gamma_x }{4}\frac{\|\nabla_x f(\bar{x}_t, \bar{y}_t, \bar{v}_t)\|^2}{\bar{q}_{t+1}}
+ \frac{\bar{L}^2\gamma_x }{\mu^2} \frac{\|\nabla_v r(\bar{x}_t, \bar{y}_t, \bar{v}_t)\|^2}{\bar{q}_{t+1}} \nonumber\\
&\quad+\!\left(\! \frac{\gamma_x\bar{L}^2}{2\mu^2} \!+\! \frac{\gamma_x\bar{L}^2}{\mu^4} \!\!\left(\! \frac{L_{l,2} C_{f_y}}{\mu} + L_{f,1} \!\right)^2 \right)\!\! \frac{\|\nabla_y l(\bar{x}_t, \bar{y}_t)\|^2}{\bar{q}_{t+1}}\!+\! 2\gamma_x \bar{q}_{t+1}^{-1}  \left\| \frac{\big(\tilde{\mathbf{q}}_{t+1}^{-1}\big)^\top}{n\bar{q}_{t+1}^{-1}} \bar{\nabla}F(\mathbf{x}_t, \mathbf{y}_t, \mathbf{v}_t) \right\|^2\!\!,\label{eq131}
\end{align}
which indicates that
\begin{align}
&\bar{q}_{t+1}^{-1}\|\nabla_x f(\bar{x}_t, \bar{y}_t, \bar{v}_t)\|^2\nonumber\\
&\leq4\left(\frac{\Phi(\bar{x}_t)-\Phi(\bar{x}_{t+1})}{\gamma_x }\right)+\frac{4\bar{q}_{t+1}^{-1}(2\gamma_x \bar{L}_{f}^2L_\Phi \bar{q}_{t+1}^{-1} \left(1 +\zeta_q^2\right)+\bar{L}_f^2)}{n} \Delta_t+ \frac{4\bar{L}^2 }{\mu^2} \frac{\|\nabla_v r(\bar{x}_t, \bar{y}_t, \bar{v}_t)\|^2}{\bar{q}_{t+1}}\nonumber\\
&\quad + 8 \bar{q}_{t+1}^{-1}  \left\| \frac{\big(\tilde{\mathbf{q}}_{t+1}^{-1}\big)^\top}{n\bar{q}_{t+1}^{-1}} \bar{\nabla} F(\mathbf{x}_t, \mathbf{y}_t, \mathbf{v}_t) \right\|^2 +\left(\! \frac{2\bar{L}^2}{\mu^2} + \frac{4\bar{L}^2}{\mu^4} \!\!\left(\! \frac{L_{l,2} C_{f_y}}{\mu} + L_{f,1} \!\right)^2 \right)\!\! \frac{\|\nabla_y l(\bar{x}_t, \bar{y}_t)\|^2}{\bar{q}_{t+1}}.\label{eq132}
\end{align}
By taking summation, we have:
\begin{align}
&\sum_{k=k_1}^t  \bar{q}_{k+1}^{-1}\|\nabla_x f(\bar{x}_k, \bar{y}_k, \bar{v}_k)\|^2\nonumber\\
&\leq4\left(\frac{\Phi(\bar{x}_{k_1})-\Phi^*}{\gamma_x}\right)+\frac{4\bar{q}_{0}^{-1}(2\gamma_x \bar{L}_{f}^2L_\Phi \bar{q}_0^{-1} \left(1 +\zeta_q^2\right)+\bar{L}_f^2)}{n}\sum_{k=k_1}^t \Delta_k\nonumber\\
&\quad + \frac{4\bar{L}^2 }{\mu^2\bar{m}^x_{0}}\sum_{k=k_1}^t \frac{\|\nabla_v r(\bar{x}_k, \bar{y}_k, \bar{v}_k)\|^2}{\max\{\bar{m}^y_{k+1},\bar{m}^v_{k+1}\}}+ \frac{8}{\bar{q}_{0}}\sum_{k=k_1}^t \left\| \frac{\big(\tilde{\mathbf{q}}_{k+1}^{-1}\big)^\top}{n\bar{q}_{k+1}^{-1}} \bar{\nabla} F(\mathbf{x}_k, \mathbf{y}_k, \mathbf{v}_k) \right\|^2 \nonumber\\
&\quad+ \left( \frac{2\bar{L}^2}{\mu^2\bar{m}^x_{0}} + \frac{4\bar{L}^2}{\mu^4\bar{m}^x_{0}} \left( \frac{L_{l,2} C_{f_y}}{\mu} + L_{f,1} \right)^2 \right)\sum_{k=k_1}^t \frac{\|\nabla_y l(\bar{x}_k, \bar{y}_k)\|^2}{\max\{\bar{m}^y_{k+1},\bar{m}^v_{k+1}\}}.\label{eq133}
\end{align}
For $\Phi(\bar{x}_{k_1})$, by telescoping \cref{eq129}, we get:
\begin{align}
&\Phi(\bar{x}_{k_1})\leq\Phi(\bar{x}_0)+\frac{\gamma_x\bar{q}_{0}^{-1}(2\gamma_x \bar{L}_{f}^2L_\Phi \bar{q}_{0}^{-1} \left(1 +\zeta_q^2\right)+\bar{L}_f^2)}{n} \sum_{k=0}^{k_1-1}\Delta_k\nonumber\\
&\quad + \frac{\bar{L}^2\gamma_x }{\mu^2} \sum_{k=0}^{k_1-1} \frac{\|\nabla_v r(\bar{x}_k, \bar{y}_k, \bar{v}_k)\|^2}{\bar{m}^x_{k+1}\max\{\bar{m}^y_{k+1},\bar{m}^v_{k+1}\}}+ \frac{2\gamma_x}{\bar{q}_{0}}\sum_{k=0}^{k_1-1} \left\| \frac{\big(\tilde{\mathbf{q}}_{k+1}^{-1}\big)^\top}{n\bar{q}_{k+1}^{-1}} \bar{\nabla} F(\mathbf{x}_k, \mathbf{y}_k, \mathbf{v}_k) \right\|^2 
\nonumber\\
&\quad+{2n\gamma_x^2 L_\Phi \left(1 +\zeta_q^2\right)}\sum_{k=0}^{k_1-1}\frac{\|\nabla_x f(\bar{x}_k, \bar{y}_k, \bar{v}_k)\|^2}{[\bar{m}^x_{k+1}]^2\max\{[\bar{m}^y_{k+1}]^2,[\bar{m}^v_{k+1}]^2\}}\nonumber \\
&\quad + \left( \frac{\gamma_x\bar{L}^2}{2\mu^2} + \frac{\gamma_x\bar{L}^2}{\mu^4} \left( \frac{L_{l,2} C_{f_y}}{\mu} + L_{f,1} \right)^2 \right) \sum_{k=0}^{k_1-1}\frac{\|\nabla_y l(\bar{x}_k, \bar{y}_k)\|^2}{\bar{m}^x_{k+1}\max\{\bar{m}^y_{k+1},\bar{m}^v_{k+1}\}}.\label{eq134}
\end{align}
By plugging \cref{eq134} into \cref{eq133}, we have:
\begin{align}
&\sum_{k=k_1}^t  \bar{q}_{t+1}^{-1}\|\nabla_x f(\bar{x}_k, \bar{y}_k, \bar{v}_k)\|^2\nonumber\\
&\leq4\left(\frac{\Phi(\bar{x}_{0})-\Phi^*}{\gamma_x}\right)+\frac{4\bar{q}_{0}^{-1}(2\gamma_x \bar{L}_{f}^2L_\Phi \bar{q}_0^{-1} \left(1 +\zeta_q^2\right)+\bar{L}_f^2)}{n}\sum_{k=0}^t \Delta_k\nonumber\\
&\quad+ \frac{4\bar{L}^2}{\mu^2\bar{m}^x_{0}}\sum_{k=0}^t \frac{\|\nabla_v r(\bar{x}_k, \bar{y}_k, \bar{v}_k)\|^2}{\max\{\bar{m}^y_{k+1},\bar{m}^v_{k+1}\}} + \frac{8}{\bar{q}_{0}}\sum_{k=0}^t \left\| \frac{\big(\tilde{\mathbf{q}}_{k+1}^{-1}\big)^\top}{n\bar{q}_{k+1}^{-1}} \bar{\nabla} F(\mathbf{x}_k, \mathbf{y}_k, \mathbf{v}_k) \right\|^2 \nonumber\\
&\quad+\left( \frac{2\bar{L}^2}{\mu^2\bar{m}^x_{0}} + \frac{4\bar{L}^2}{\mu^4\bar{m}^x_{0}} \left( \frac{L_{l,2} C_{f_y}}{\mu} + L_{f,1} \right)^2 \right)\sum_{k=0}^t \frac{\|\nabla_y l(\bar{x}_k, \bar{y}_k)\|^2}{\max\{\bar{m}^y_{k+1},\bar{m}^v_{k+1}\}}+ \frac{8n\gamma_x C_{m^x}^2 L_\Phi \left(1 +\zeta_q^2\right)}{\bar{q}_0^2}\nonumber\\
&\leq4\left(\frac{\Phi(\bar{x}_{0})-\Phi^*}{\gamma_x}\right)+\frac{4\bar{q}_{0}^{-1}(2\gamma_x \bar{L}_{f}^2L_\Phi \bar{q}_0^{-1} \left(1 +\zeta_q^2\right)+\bar{L}_f^2)}{n}\sum_{k=0}^t \Delta_k\nonumber\\
&\quad + \frac{4\bar{L}^2}{\mu^2\bar{m}^x_{0}}\sum_{k=0}^t \frac{\|\nabla_v r(\bar{x}_k, \bar{y}_k, \bar{v}_k)\|^2}{\max\{\bar{m}^y_{k+1},\bar{m}^v_{k+1}\}}+ \frac{8}{\bar{q}_{0}}\sum_{k=0}^t \left\| \frac{\big(\tilde{\mathbf{q}}_{k+1}^{-1}\big)^\top}{n\bar{q}_{k+1}^{-1}} \bar{\nabla} F(\mathbf{x}_k, \mathbf{y}_k, \mathbf{v}_k) \right\|^2 \nonumber\\
&\quad+ \left( \frac{2\bar{L}^2}{\mu^2\bar{m}^x_{0}} + \frac{4\bar{L}^2}{\mu^4\bar{m}^x_{0}} \left( \frac{L_{l,2} C_{f_y}}{\mu} + L_{f,1} \right)^2 \right)\sum_{k=0}^t \frac{\|\nabla_y l(\bar{x}_k, \bar{y}_k)\|^2}{\bar{m}^y_{k+1}}+\frac{8n\gamma_x C_{m^x}^2 L_\Phi \left(1 +\zeta_q^2\right)}{\bar{q}_0^2}\nonumber\\
&\overset{(a)}{\leq}4\left(\frac{\Phi(\bar{x}_{0})-\Phi^*}{\gamma_x}\right)+ \frac{4\bar{L}^2(a_6\log(t + 1)+b_6) }{\mu^2\bar{m}^x_{0}}\nonumber\\
&\quad+\frac{2\bar{L}^2(a_5\log(t + 1)+b_5)}{\mu^2\bar{m}^x_{0}} \left( 1 + \frac{2}{\mu^2} \left( \frac{L_{l,2} C_{f_y}}{\mu} + L_{f,1} \right)^2 \right)\nonumber\\
&\quad + \frac{8(1+\zeta_q^2)(5 \log(t) + c_2)(2(1 + \rho_W) \rho_W)}{n\bar{z}_0^2\bar{q}_0(1 - \rho_W)^2}\left(\frac{C_{l_{xy}}C_{f_y}}{\mu}\!+\!C_{f_x}\right)^2\!\left[\!C_{l_y}^2\!+\!\left(\frac{C_{l_{yy}}C_{f_y}}{\mu}+C_{f_y}\right)^2\!\right] \nonumber\\
&\quad+\frac{4\bar{q}_{0}^{-1}(2\gamma_x \bar{L}_{f}^2L_\Phi \bar{q}_0^{-1} \left(1 +\zeta_q^2\right)+\bar{L}_f^2)}{n}\left(\frac{2 \Delta_0}{1 - \rho_W} 
+ \frac{8 \gamma_x^2 \rho_W (1 + \zeta_q^2)( 5 \log(t+1) + c_2)}{(1 - \rho_W)^2}\right. \nonumber \\
&\quad \left.+ \frac{8 \gamma_y^2 \rho_W (1 + \zeta_u^2)(a_2 \log(t+1) + b_2)}{(1 - \rho_W)^2}+\frac{8 \gamma_v^2 \rho_W (1 + \zeta_z^2)(a_3 \log(t+1) + b_3)}{(1 - \rho_W)^2}\right)\nonumber\\
&\quad+ \frac{8n\gamma_x C_{m^x}^2 L_\Phi \left(1 +\zeta_q^2\right)}{\bar{q}_0^2}\nonumber\\
&=:4\left(\frac{\Phi(\bar{x}_{0})-\Phi^*}{\gamma_x}\right)+a_7\log(t+1)+b_7,\label{eq135}
\end{align}
where (a) uses \cref{lemmanew12} and \cref{Lemma12}. This immediately implies that
\begin{align}
\sum_{k=k_1}^t  \frac{\|\nabla_x f(\bar{x}_k, \bar{y}_k, \bar{v}_k)\|^2}{\bar{m}^x_{k+1}}
\leq \left(4\left(\frac{\Phi(\bar{x}_{0})-\Phi^*}{\gamma_x}\right)+a_7\log(t+1)+b_7\right)\bar{z}_{t+1}.\label{eq136}
\end{align}
Similarly, we can have the upper bound of $\bar{m}^x_{t+1}$ as:
\begin{align}
   \bar{m}^x_{t+1} &\leq \bar{m}^x_{k_1} + \sum_{k=k_1}^t \frac{\|\nabla_x f(\bar{x}_k, \bar{y}_k, \bar{v}_k)\|^2}{\bar{m}^x_{k+1}}\nonumber\\
&\leq C_{m^x} + {\left(4\left(\frac{\Phi(\bar{x}_{0})-\Phi^*}{\gamma_x}\right)+a_7\log(t+1)+b_7\right)\bar{z}_{t+1}}. \label{eq137}
\end{align}
Then the upper bound of $\bar{m}^x_{t+1}$ is proved.
\end{proof}

\subsection{Proof of \cref{theorem 1}}
Here we still assume the total iteration rounds of \cref{alg1} is $T$. According to \cref{Lemma4}, the proof can be split into the following two cases.

\textbf{Case 1:} If $\bar{m}^x_{T} \leq C_{m^x}$, then we have:
\begin{align}
&{ \bar{q}_{t+1}^{-1}}\|\nabla \Phi(\bar{x}_t)\|^2\nonumber\\
&\leq8\left(\frac{\Phi(\bar{x}_t)-\Phi(\bar{x}_{t+1})}{\gamma_x}\right)+\frac{8\bar{q}_{t+1}^{-1}(2\gamma_x \bar{L}_{f}^2L_\Phi \bar{q}_{t+1}^{-1} \left(1 +\zeta_q^2\right)+\bar{L}_{f}^2)}{n} \Delta_t+ \frac{8\bar{L}^2}{\mu^2} \frac{\|\nabla_v r(\bar{x}_t, \bar{y}_t, \bar{v}_t)\|^2}{\bar{q}_{t+1}} \nonumber\\
&\quad +{16n\gamma_x L_\Phi \bar{q}_{t+1}^{-2} \left(1 +\zeta_q^2\right)}{\|\nabla_x f(\bar{x}_t, \bar{y}_t, \bar{v}_t)\|^2}
+ 16 \bar{q}_{t+1}^{-1}  \left\| \frac{\big(\tilde{\mathbf{q}}_{t+1}^{-1}\big)^\top}{n\bar{q}_{t+1}^{-1}} \bar{\nabla} F(\mathbf{x}_t, \mathbf{y}_t, \mathbf{v}_t) \right\|^2 \nonumber \\
&\quad +\left( \frac{4\bar{L}^2}{\mu^2} + \frac{8\bar{L}^2}{\mu^4} \left( \frac{L_{l,2} C_{f_y}}{\mu} + L_{f,1} \right)^2 \right) \frac{\|\nabla_y l(\bar{x}_t, \bar{y}_t)\|^2}{\bar{q}_{t+1}}.\label{eq138}
\end{align}

By taking the average, we have:
\begin{align}
&\frac{1}{T}\sum_{t=0}^{T-1}{\bar{q}_{t+1}^{-1}}\|\nabla \Phi(\bar{x}_t)\|^2\nonumber\\
&\leq\frac{8}{T}\left(\frac{\Phi(\bar{x}_0)-\Phi(\bar{x}_{T})}{\gamma_x}\right)+\frac{8\bar{q}_{t+1}^{-1}(2\gamma_x \bar{L}_{f}^2L_\Phi \bar{q}_{t+1}^{-1} \left(1 +\zeta_q^2\right)+\bar{L}_{f}^2)}{nT} \sum_{t=0}^{T-1}\Delta_t\nonumber\\
&\quad +\frac{16n\gamma_x L_\Phi \bar{q}_{0}^{-2} \left(1 +\zeta_q^2\right)}{T}\sum_{t=0}^{T-1}\|\nabla_x f(\bar{x}_t, \bar{y}_t, \bar{v}_t)\|^2
+ \frac{16\bar{q}_{0}^{-1}}{T}\sum_{t=0}^{T-1}  \left\| \frac{\big(\tilde{\mathbf{q}}_{t+1}^{-1}\big)^\top}{n\bar{q}_{t+1}^{-1}} \bar{\nabla} F(\mathbf{x}_t, \mathbf{y}_t, \mathbf{v}_t) \right\|^2 \nonumber \\
&\quad +\left( \frac{4\bar{L}^2}{\mu^2} + \frac{8\bar{L}^2}{\mu^4} \left( \frac{L_{l,2} C_{f_y}}{\mu} + L_{f,1} \right)^2 \right) \frac{1}{T}\sum_{t=0}^{T-1}\frac{\|\nabla_y l(\bar{x}_t, \bar{y}_t)\|^2}{\bar{m}^x_{t+1}\max\{\bar{m}^y_{t+1},\bar{m}^v_{t+1}\}}\nonumber\\
&\quad+ \frac{8\bar{L}^2 }{\mu^2T}\sum_{t=0}^{T-1}\frac{\|\nabla_v r(\bar{x}_t, \bar{y}_t, \bar{v}_t)\|^2}{\bar{m}^x_{t+1}\max\{\bar{m}^y_{t+1},\bar{m}^v_{t+1}\}}\nonumber\\
&\leq\frac{8}{T}\left(\frac{\Phi(\bar{x}_0)-\Phi(\bar{x}_{T})}{\gamma_x}\right)+\frac{8\bar{q}_{t+1}^{-1}(2\gamma_x \bar{L}_{f}^2L_\Phi \bar{q}_{t+1}^{-1} \left(1 +\zeta_q^2\right)+\bar{L}_{f}^2)}{nT}  \sum_{t=0}^{T-1}\Delta_t\nonumber\\
&\quad +\frac{16n\gamma_x L_\Phi \bar{q}_{0}^{-2} C_{m^x}^2\left(1 +\zeta_q^2\right)}{T}
+ \frac{16 \bar{q}_{0}^{-1}}{T}\sum_{t=0}^{T-1}  \left\| \frac{\big(\tilde{\mathbf{q}}_{t+1}^{-1}\big)^\top}{n\bar{q}_{t+1}^{-1}} \bar{\nabla} F(\mathbf{x}_t, \mathbf{y}_t, \mathbf{v}_t) \right\|^2 \nonumber \\
&\quad +\left( \frac{4\bar{L}^2}{\bar{m}^x_{0}\mu^2} + \frac{8\bar{L}^2}{\bar{m}^x_{0}\mu^4} \left( \frac{L_{l,2} C_{f_y}}{\mu} + L_{f,1} \right)^2 \right)\frac{1}{T} \sum_{t=0}^{T-1}\frac{\|\nabla_y l(\bar{x}_t, \bar{y}_t)\|^2}{\bar{m}^y_{t+1}}\nonumber\\
&\quad+ \frac{8\bar{L}^2 }{\mu^2T\bar{m}^x_{0}}\sum_{t=0}^{T-1}\frac{\|\nabla_v r(\bar{x}_t, \bar{y}_t, \bar{v}_t)\|^2}{\max\{\bar{m}^y_{t+1},\bar{m}^v_{t+1}\}}\nonumber\\
&\overset{(a)}{=}\frac{2}{T}\left(4\left(\frac{\Phi(\bar{x}_{0})-\Phi^*}{\gamma_x}\right)+a_7\log(t+1)+b_7\right),\label{eq139}
\end{align}
where (a) uses \cref{lemmanew12} with $k_0=0$.

\textbf{Case 2:} If $\bar{m}^x_{T} > C_{m^x}$, by \cref{Lemma6}, there exists $k_1 \leq T_0$ such that $\bar{m}^x_{k_1} \leq C_{m^x}$, $\bar{m}^x_{k_1+1} > C_{m^x}$.

Then for $t < k_1$ when $\bar{m}^x_{T} > C_{m^x}$, from \cref{eqnew1}, we have:
\begin{align}
&{\bar{q}_{t+1}^{-1}}\|\nabla \Phi(\bar{x}_t)\|^2\nonumber\\
&\leq8\left(\frac{\Phi(\bar{x}_t)-\Phi(\bar{x}_{t+1})}{\gamma_x}\right)+\frac{8\bar{q}_{t+1}^{-1}(2\gamma_x \bar{L}_{f}^2L_\Phi \bar{q}_{t+1}^{-1} \left(1 +\zeta_q^2\right)+\bar{L}_{f}^2)}{n} \Delta_t+ \frac{8\bar{L}^2 }{\mu^2} \frac{\|\nabla_v r(\bar{x}_t, \bar{y}_t, \bar{v}_t)\|^2}{\bar{q}_{t+1}} \nonumber\\
&\quad +{16n\gamma_x L_\Phi \bar{q}_{t+1}^{-2} \left(1 +\zeta_q^2\right)}{\|\nabla_x f(\bar{x}_t, \bar{y}_t, \bar{v}_t)\|^2}
+ 16 \bar{q}_{t+1}^{-1}  \left\| \frac{\big(\tilde{\mathbf{q}}_{t+1}^{-1}\big)^\top}{n\bar{q}_{t+1}^{-1}} \bar{\nabla} F(\mathbf{x}_t, \mathbf{y}_t, \mathbf{v}_t) \right\|^2 \nonumber \\
&\quad +\left( \frac{4\bar{L}^2}{\mu^2} + \frac{8\bar{L}^2}{\mu^4} \left( \frac{L_{l,2} C_{f_y}}{\mu} + L_{f,1} \right)^2 \right) \frac{\|\nabla_y l(\bar{x}_t, \bar{y}_t)\|^2}{\bar{q}_{t+1}}.\label{eq140}
\end{align}

For $t \geq k_1$ when $\bar{m}^x_{T} > C_{m^x}$, from \cref{eqnew2}, we have:
\begin{align}
&{\bar{q}_{t+1}^{-1}}\|\nabla \Phi(\bar{x}_t)\|^2\nonumber\\
&\leq8\left(\frac{\Phi(\bar{x}_t)-\Phi(\bar{x}_{t+1})}{\gamma_x}\right)+\frac{8\bar{q}_{t+1}^{-1}(2\gamma_x \bar{L}_{f}^2L_\Phi \bar{q}_{t+1}^{-1} \left(1 +\zeta_q^2\right)+\bar{L}_{f}^2)}{n} \Delta_t+ \frac{8\bar{L}^2 }{\mu^2} \frac{\|\nabla_v r(\bar{x}_t, \bar{y}_t, \bar{v}_t)\|^2}{\bar{q}_{t+1}} \nonumber\\
&\quad + 16 \bar{q}_{t+1}^{-1}  \left\| \frac{\big(\tilde{\mathbf{q}}_{t+1}^{-1}\big)^\top}{n\bar{q}_{t+1}^{-1}} \bar{\nabla} F(\mathbf{x}_t, \mathbf{y}_t, \mathbf{v}_t) \right\|^2  +\left( \frac{4\bar{L}^2}{\mu^2} + \frac{8\bar{L}^2}{\mu^4} \left( \frac{L_{l,2} C_{f_y}}{\mu} + L_{f,1} \right)^2 \right) \frac{\|\nabla_y l(\bar{x}_t, \bar{y}_t)\|^2}{\bar{q}_{t+1}}.\label{eq141}
\end{align}

By taking the average, we can merge $t < k_1$ and $t \geq k_1$ as:
\begin{align}
&\frac{1}{T} \sum_{t=0}^{T-1} { \bar{q}_{t+1}^{-1}}\|\nabla \Phi(\bar{x}_t)\|^2\nonumber\\
&= \frac{1}{T} \sum_{t=0}^{k_1-1} {\bar{q}_{t+1}^{-1}}\|\nabla \Phi(\bar{x}_t)\|^2
+ \frac{1}{T} \sum_{t=k_1}^{T-1} {\bar{q}_{t+1}^{-1}}\|\nabla \Phi(\bar{x}_t)\|^2\nonumber \\
&\leq \frac{8}{T}\left(\frac{\Phi(\bar{x}_0)-\Phi(\bar{x}_{k_1})}{\gamma_x}\right)+\frac{8\bar{q}_{t+1}^{-1}(2\gamma_x \bar{L}_{f}^2L_\Phi \bar{q}_{t+1}^{-1} \left(1 +\zeta_q^2\right)+\bar{L}_{f}^2)}{nT} \sum_{t=0}^{k_1-1}\Delta_t\nonumber\\
&\quad +\frac{16n\gamma_x L_\Phi \bar{q}_{0}^{-2} \left(1 +\zeta_q^2\right)}{T}\sum_{t=0}^{k_1-1}\|\nabla_x f(\bar{x}_t, \bar{y}_t, \bar{v}_t)\|^2
+ \frac{16\bar{q}_{0}^{-1}}{T}\sum_{t=0}^{k_1-1} \left\| \frac{\big(\tilde{\mathbf{q}}_{t+1}^{-1}\big)^\top}{n\bar{q}_{t+1}^{-1}} \bar{\nabla} F(\mathbf{x}_t, \mathbf{y}_t, \mathbf{v}_t) \right\|^2 \nonumber \\
&\quad +\left( \frac{4\bar{L}^2}{\mu^2} + \frac{8\bar{L}^2}{\mu^4} \left( \frac{L_{l,2} C_{f_y}}{\mu} + L_{f,1} \right)^2 \right) \frac{1}{T}\sum_{t=0}^{k_1-1}\frac{\|\nabla_y l(\bar{x}_t, \bar{y}_t)\|^2}{\bar{m}^x_{t+1}\max\{\bar{m}^y_{t+1},\bar{m}^v_{t+1}\}}\nonumber\\
&\quad+ \frac{8\bar{L}^2 }{\mu^2T}\sum_{t=0}^{k_1-1}\frac{\|\nabla_v r(\bar{x}_t, \bar{y}_t, \bar{v}_t)\|^2}{\bar{m}^x_{t+1}\max\{\bar{m}^y_{t+1},\bar{m}^v_{t+1}\}}\nonumber\\
&\quad+ \frac{8}{T}\left(\frac{\Phi(\bar{x}_{k_1})-\Phi(\bar{x}_{T})}{\gamma_x}\right)+\frac{8\bar{q}_{t+1}^{-1}(2\gamma_x \bar{L}_{f}^2L_\Phi \bar{q}_{t+1}^{-1} \left(1 +\zeta_q^2\right)+\bar{L}_{f}^2)}{nT} \sum_{t=k_1}^{T-1} \Delta_t\nonumber\\
&\quad + \frac{16 \bar{q}_{0}^{-1}}{T}\sum_{t=k_1}^{T-1}  \left\| \frac{\big(\tilde{\mathbf{q}}_{t+1}^{-1}\big)^\top}{n\bar{q}_{t+1}^{-1}} \bar{\nabla} F(\mathbf{x}_t, \mathbf{y}_t, \mathbf{v}_t) \right\|^2 \nonumber \\
&\quad +\left( \frac{4\bar{L}^2}{\mu^2} + \frac{8\bar{L}^2}{\mu^4} \left( \frac{L_{l,2} C_{f_y}}{\mu} + L_{f,1} \right)^2 \right) \frac{1}{T}\sum_{t=k_1}^{T-1} \frac{\|\nabla_y l(\bar{x}_t, \bar{y}_t)\|^2}{\bar{m}^x_{t+1}\max\{\bar{m}^y_{t+1},\bar{m}^v_{t+1}\}}\nonumber\\
&\quad+ \frac{8\bar{L}^2 }{\mu^2T}\sum_{t=k_1}^{T-1} \frac{\|\nabla_v r(\bar{x}_t, \bar{y}_t, \bar{v}_t)\|^2}{\bar{m}^x_{t+1}\max\{\bar{m}^y_{t+1},\bar{m}^v_{t+1}\}}\nonumber\\
&\leq \frac{8}{T}\left(\frac{\Phi(\bar{x}_0)-\Phi^*}{\gamma_x}\right)+\frac{8\bar{q}_{t+1}^{-1}(2\gamma_x \bar{L}_{f}^2L_\Phi \bar{q}_{t+1}^{-1} \left(1 +\zeta_q^2\right)+\bar{L}_{f}^2)}{nT} \sum_{t=0}^{T-1}\Delta_t\nonumber\\
&\quad +\frac{16n\gamma_x L_\Phi \bar{q}_{0}^{-2} \left(1 +\zeta_q^2\right)}{T}\sum_{t=0}^{k_1-1}\|\nabla_x f(\bar{x}_t, \bar{y}_t, \bar{v}_t)\|^2
+ \frac{16\bar{q}_{0}^{-1}}{T}\sum_{t=0}^{T-1} \left\| \frac{\big(\tilde{\mathbf{q}}_{t+1}^{-1}\big)^\top}{n\bar{q}_{t+1}^{-1}} \bar{\nabla} F(\mathbf{x}_t, \mathbf{y}_t, \mathbf{v}_t) \right\|^2 \nonumber \\
&\quad +\left( \frac{4\bar{L}^2}{\bar{m}^x_{0}\mu^2} + \frac{8\bar{L}^2}{\bar{m}^x_{0}\mu^4} \left( \frac{L_{l,2} C_{f_y}}{\mu} + L_{f,1} \right)^2 \right)\frac{1}{T}\sum_{t=0}^{T-1}\frac{\|\nabla_y l(\bar{x}_t, \bar{y}_t)\|^2}{\bar{m}^y_{t+1}}\nonumber\\
&\quad+ \frac{8\bar{L}^2 }{\mu^2T\bar{m}^x_{0}}\sum_{t=0}^{T-1}\frac{\|\nabla_v r(\bar{x}_t, \bar{y}_t, \bar{v}_t)\|^2}{\max\{\bar{m}^y_{t+1},\bar{m}^v_{t+1}\}}\nonumber\\
&\overset{(a)}{=}\frac{2}{T}\left(4\left(\frac{\Phi(\bar{x}_0)-\Phi^*}{\gamma_x}\right)+a_7\log(t+1)+b_7\right),\label{eq142}
\end{align}
where (a) uses \cref{lemmanew12} by plugging in $k_0=0$.

Note that Case 1 and Case 2 indicate the same result. Thus, we have:
\begin{align}
&\frac{1}{T} \sum_{t=0}^{T-1} \|\nabla \Phi(\bar{x}_t)\|^2\nonumber\\
&\leq \frac{2}{T}\left(4\left(\frac{\Phi(\bar{x}_0)-\Phi^*}{\gamma_x}\right)+a_7\log(T)+b_7\right)\bar{m}^x_T \bar{z}_T\nonumber \\
&\overset{(a)}{\leq} \frac{2}{T}\left[\left(4\left(\frac{\Phi(\bar{x}_0)-\Phi^*}{\gamma_x}\right)+a_7\log(T)+b_7\right)^2{\bar{z}_T^2} \right. \notag\\
&\quad\left.+ C_{m^x} \left(4\left(\frac{\Phi(\bar{x}_0)-\Phi^*}{\gamma_x}\right)+a_7\log(T)+b_7\right)\bar{z}_T\right] \nonumber \\
&\overset{(b)}{\leq} \frac{2}{T}\left[\left(4\left(\frac{\Phi(\bar{x}_0)-\Phi^*}{\gamma_x}\right)+a_7\log(T)+b_7\right)^2{(a_1 \log(T) + b_1)^2}  \right.\nonumber\\
&\quad \left.+ C_{m^x} \left(4\left(\frac{\Phi(\bar{x}_0)-\Phi^*}{\gamma_x}\right)+a_7\log(T)+b_7\right)(a_1 \log(T) + b_1)\right]\nonumber \\
&= \mathcal{O}\left(\frac{\log^4(T)}{T}\right),\label{eq143}
\end{align}
where (a) uses \cref{Lemma13} and (b) uses \cref{Lemma10}. Thus, the proof is finished.

\subsection{Proof of \cref{collary1}}
Recall from \cref{theorem 1} that there exists a constant \( M \) such that:
\begin{align}
    \frac{1}{T} \sum_{t=0}^{T-1} \|\nabla \Phi(x_t)\|^2 \leq \frac{M \log^4(T)}{T}.
\end{align}

By setting the total number of iterations \( T \) as $T = \frac{ML}{\epsilon} \log^4\left(\frac{M}{\epsilon}\right)$ and assuming the constant \( L = 12^4 \), we have:
\begin{align}
    \frac{M \log^4(T)}{T} 
    &= \frac{M \log^4\left(\frac{MN}{\epsilon} \log^4\left(\frac{M}{\epsilon}\right)\right)}{\frac{MN}{\epsilon} \log^4\left(\frac{M}{\epsilon}\right)} \notag \\
    &\leq \frac{\left[\log(N) + \log\left(\frac{M}{\epsilon}\right) + 4\log\left(\log\left(\frac{M}{\epsilon}\right)\right)\right]^4}{N \log^4\left(\frac{M}{\epsilon}\right)}  \epsilon \notag \\
    &{\leq} \frac{\left(\log(N) + 2\log\left(\frac{M}{\epsilon}\right)\right)^4}{N^{1 + \frac{1}{4}} \log\left(\frac{M}{\epsilon}\right)}  \epsilon 
    {\leq} \epsilon. \notag
\end{align}

Here we have used two key inequalities:
\begin{enumerate}
    \item \(\log\left(\log\left(\frac{M}{\epsilon}\right)\right) \leq \frac{1}{4} \log\left(\frac{M}{\epsilon}\right)\) for sufficiently small \(\epsilon\),
    \item \(\log(L) + 2 \log\left(\frac{M}{\epsilon}\right) \leq L^{1 + \frac{1}{4}} \log\left(\frac{M}{\epsilon}\right)\) when \(L = 12^4\) and \(\epsilon\) is sufficiently small.
\end{enumerate}
Then we can ensure that $\frac{M \log^4(T)}{T} \leq \epsilon$.
Thus, to achieve an \(\epsilon\)-accurate stationary point, the required number of iterations is:
\begin{align}
    T &= \frac{ML}{\epsilon} \log^4\left(\frac{M}{\epsilon}\right)= \mathcal{O}\left(\frac{1}{\epsilon} \log^4\left(\frac{1}{\epsilon}\right)\right).
\end{align}

Finally, the gradient complexity is given by:
\begin{align}
    \mathrm{Gc}(\epsilon)={\Omega}(T) = \mathcal{O}\left(\frac{1}{\epsilon} \log^4\left(\frac{1}{\epsilon}\right)\right).
\end{align}
Thus, the proof is finished.

\section{Additional Experiments}\label{experimentsec}
\subsection{Hyperparameter Optimization Problem}\label{hypersec}
Our experiments are conducted on the following hyperparameter optimization problem:
\begin{equation*}
\begin{aligned}
&\!\!\min_{\lambda \in \mathbb{R}^p} \frac{1}{n} \sum_{i=1}^n f_i(\lambda, \omega^*(\lambda)), \\
&\text{s.t.}\quad \omega^*(\lambda) = \arg \min_{\omega \in \mathbb{R}^q} \frac{1}{n} \sum_{i=1}^n l_i(\lambda, \omega),
\end{aligned}
\end{equation*}
where the goal is to find the optimal hyperparameter \(\lambda\), subject to the constraint that \(\omega^*(\lambda)\) represents the optimal model given \(\lambda\).

\subsection{Synthetic Data Experiments}\label{synthesec}
For the synthetic data experiments, we follow the experimental setups of prior works~\citep{pedregosa2016hyperparameter, grazzi2020iteration, chen2024decentralized}. For any agent $i$, the private objective functions \(f_i\) and \(l_i\) are defined as:
\begin{equation*}
\begin{aligned}
    f_i(\lambda, \omega) &= \sum_{(x_e, y_e) \in D'_i} \psi(y_e x_e^\top \omega),\\
    l_i(\lambda, \omega) &= \sum_{(x_e, y_e) \in D_i} \psi(y_e x_e^\top \omega) + \frac{1}{2} \sum_{j=1}^p e^{\lambda_j} \omega_j^2,
\end{aligned}
\end{equation*}
where \(\psi(x) = \log(1+e^{-x})\) and \(p\) represents the dimensionality of the data. A ground truth vector \(\omega^*\) is generated and each \(x_e \in \mathbb{R}^p\) is sampled from a normal distribution. The data distribution for \(x_e\) at node \(i\) follows \(\mathcal{N}(0, i^2\cdot r^2)\), where $r$ quantifies the degree of heterogeneity across agents. The corresponding labels \(y_e\) are defined as \(y_e = x_e^\top \omega^* + 0.1 z\), where \(z\) is sampled from a standard normal distribution.

\begin{figure*}[t]
\centering
\begin{minipage}[b]{0.41\textwidth}
    \includegraphics[width=\textwidth]{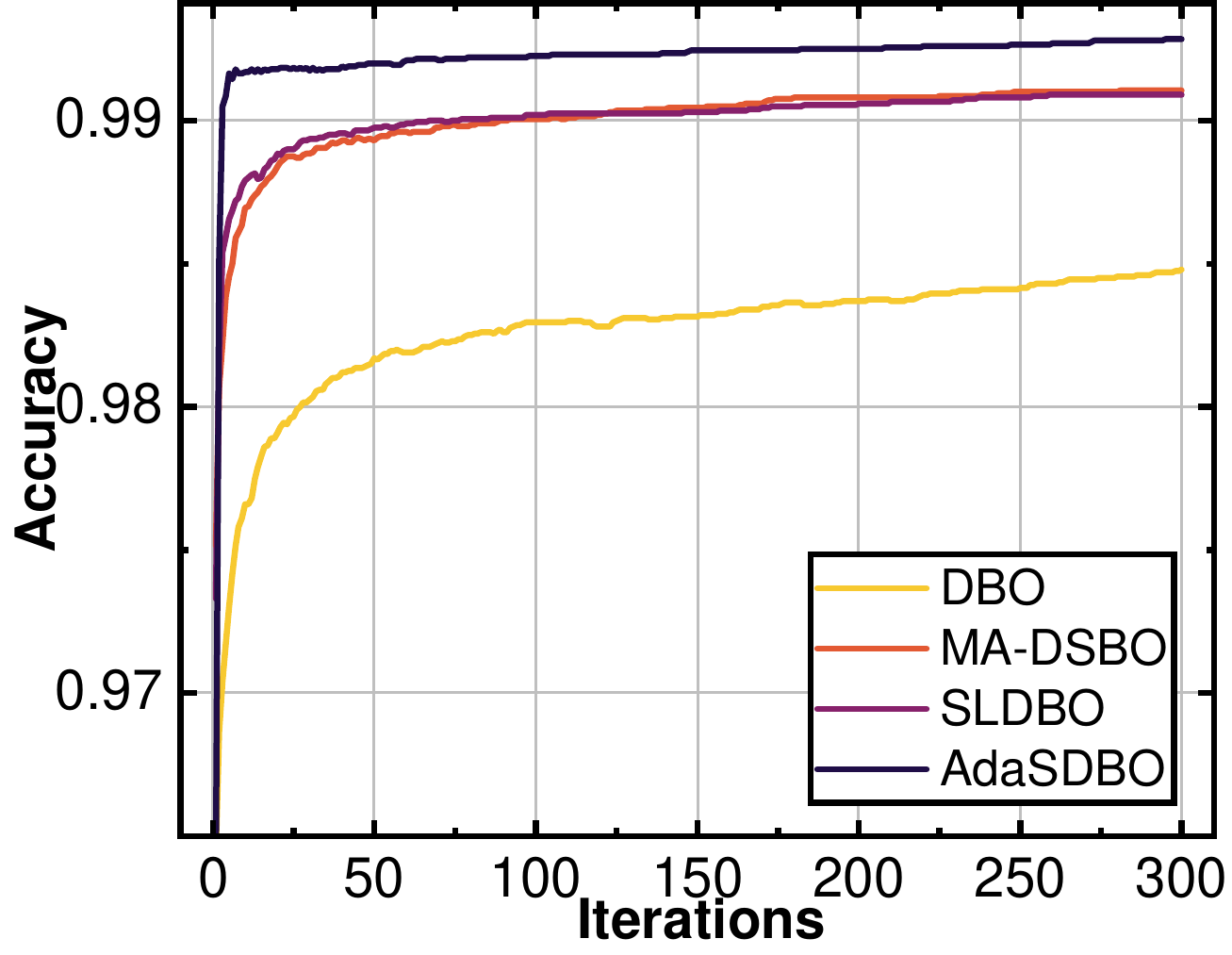}
    \par\vspace{-0.1cm}
    \makebox[\textwidth]{{\hspace{0.66cm}(a) Synthetic with $p=50$}}
\end{minipage}
\begin{minipage}[b]{0.41\textwidth}
    \includegraphics[width=\textwidth]{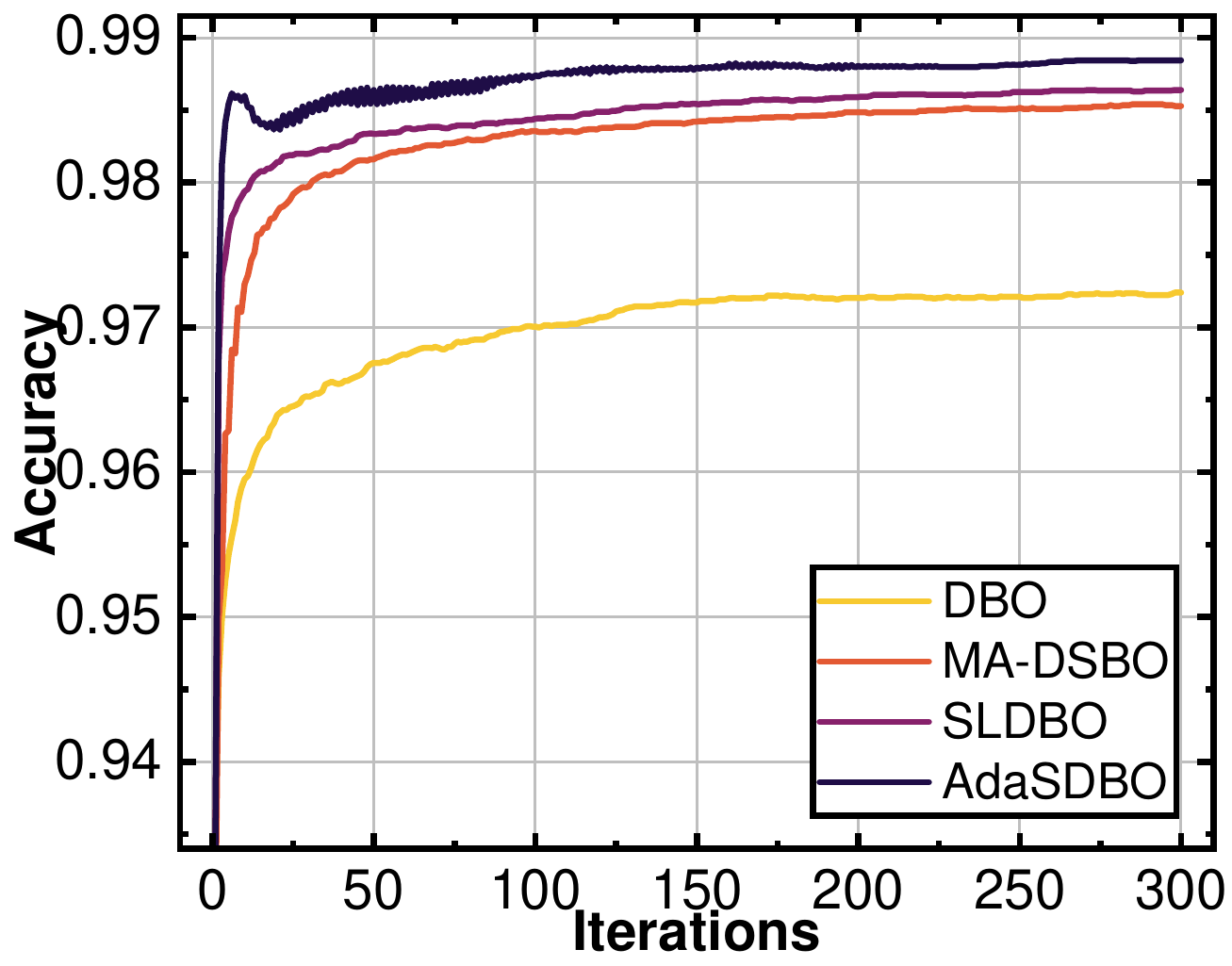}
    \par\vspace{-0.1cm}
    \makebox[\textwidth]{{\hspace{0.66cm}(b) Synthetic with $p=200$}}
\end{minipage}
\caption{Test accuracy on synthetic dataset with $r=5$.}
\label{fig:syntheticr5}
\end{figure*}

\begin{figure*}[t]
\centering
\begin{minipage}[b]{0.41\textwidth}
    \includegraphics[width=\textwidth]{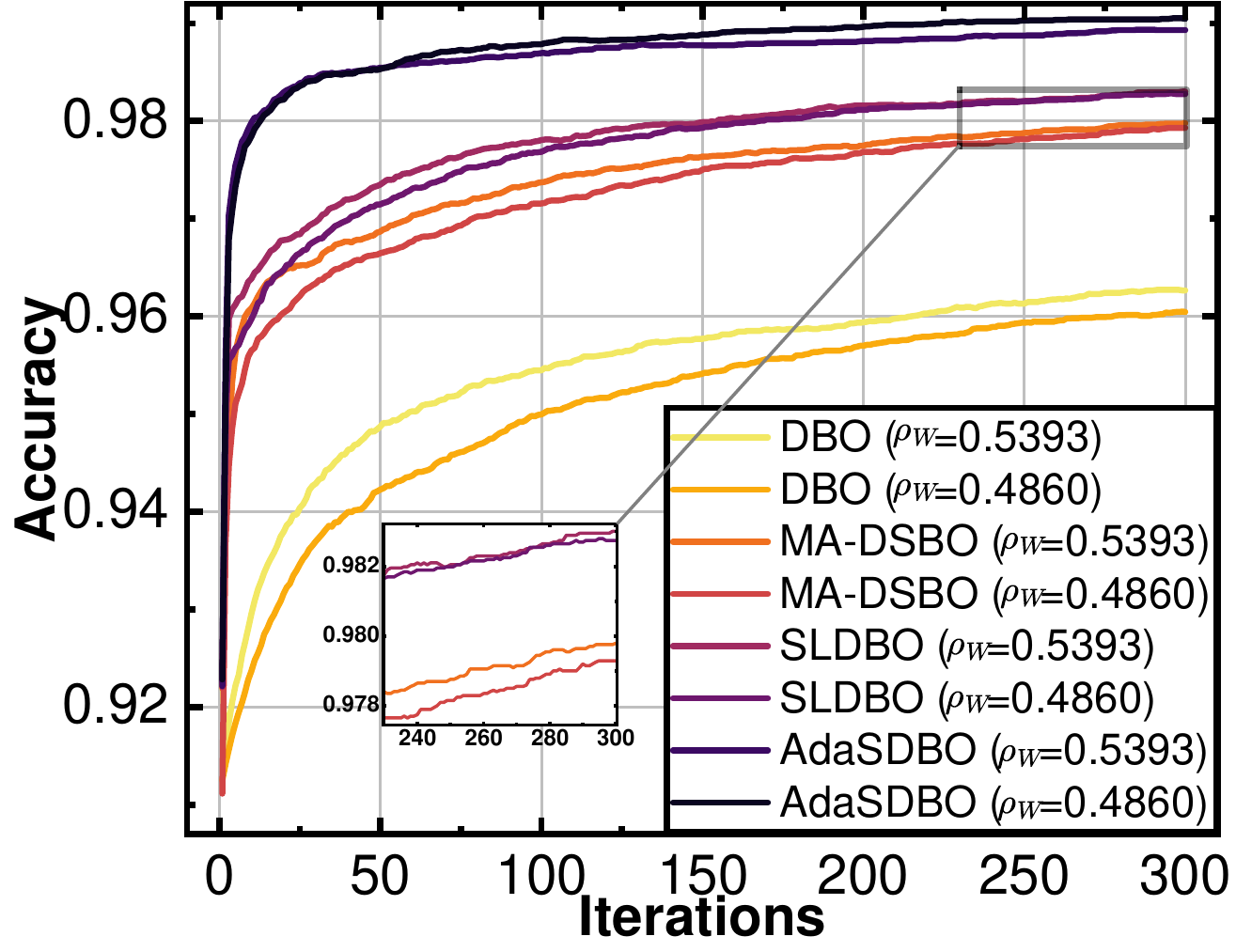}
    \par\vspace{-0.1cm}
    \makebox[\textwidth]{{\hspace{0.66cm}(a) Test accuracy}}
\end{minipage}
\begin{minipage}[b]{0.41\textwidth}
    \includegraphics[width=\textwidth]{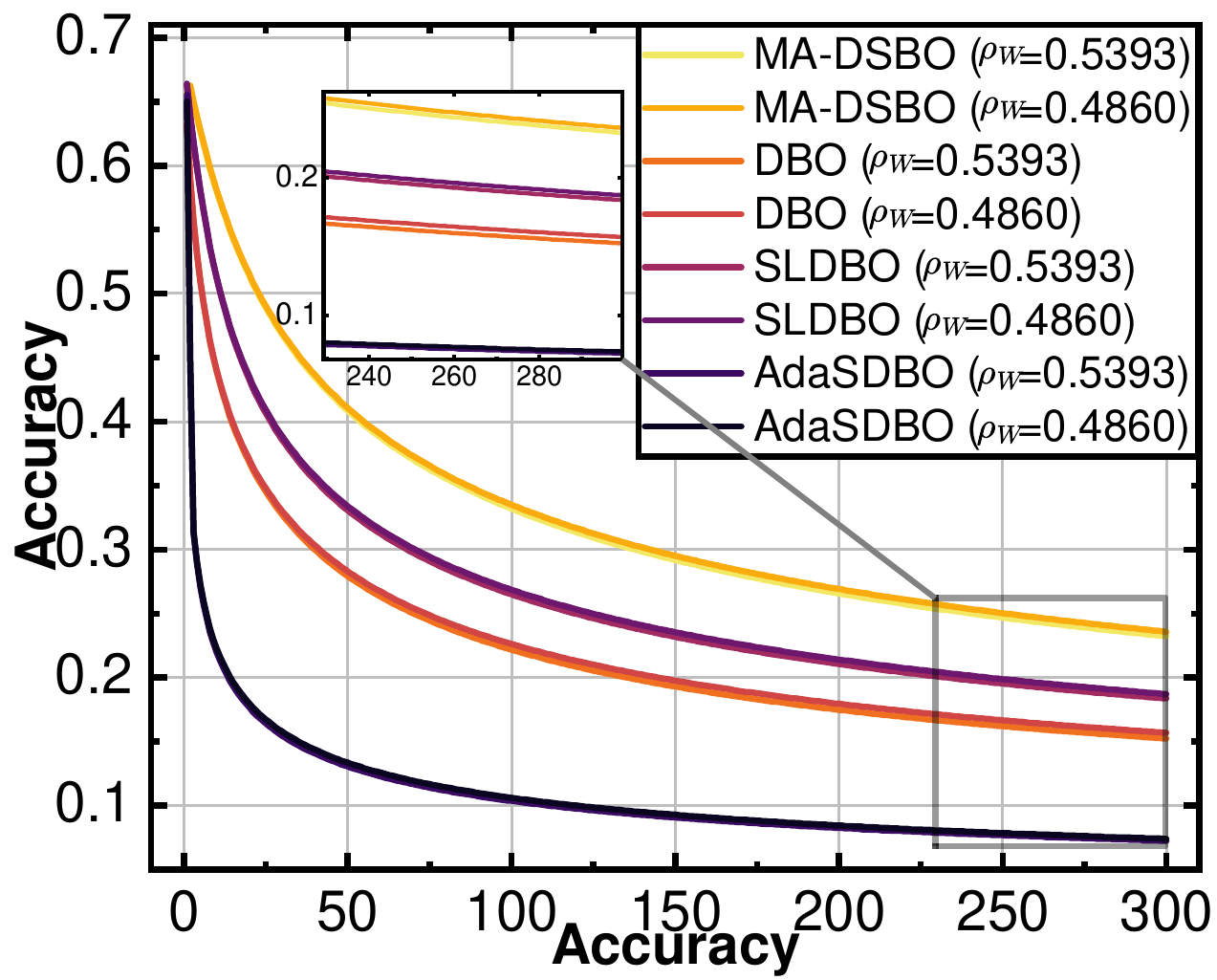}
    \par\vspace{-0.1cm}
    \makebox[\textwidth]{{\hspace{0.66cm}(b) Upper-level loss}}
\end{minipage}
\caption{Test accuracy and upper-level loss for synthetic dataset with different $\rho_W$.}
\label{fig:rhosynthetic}
\end{figure*}

\begin{figure*}[t]
\centering
\begin{minipage}[b]{0.41\textwidth}
    \includegraphics[width=\textwidth]{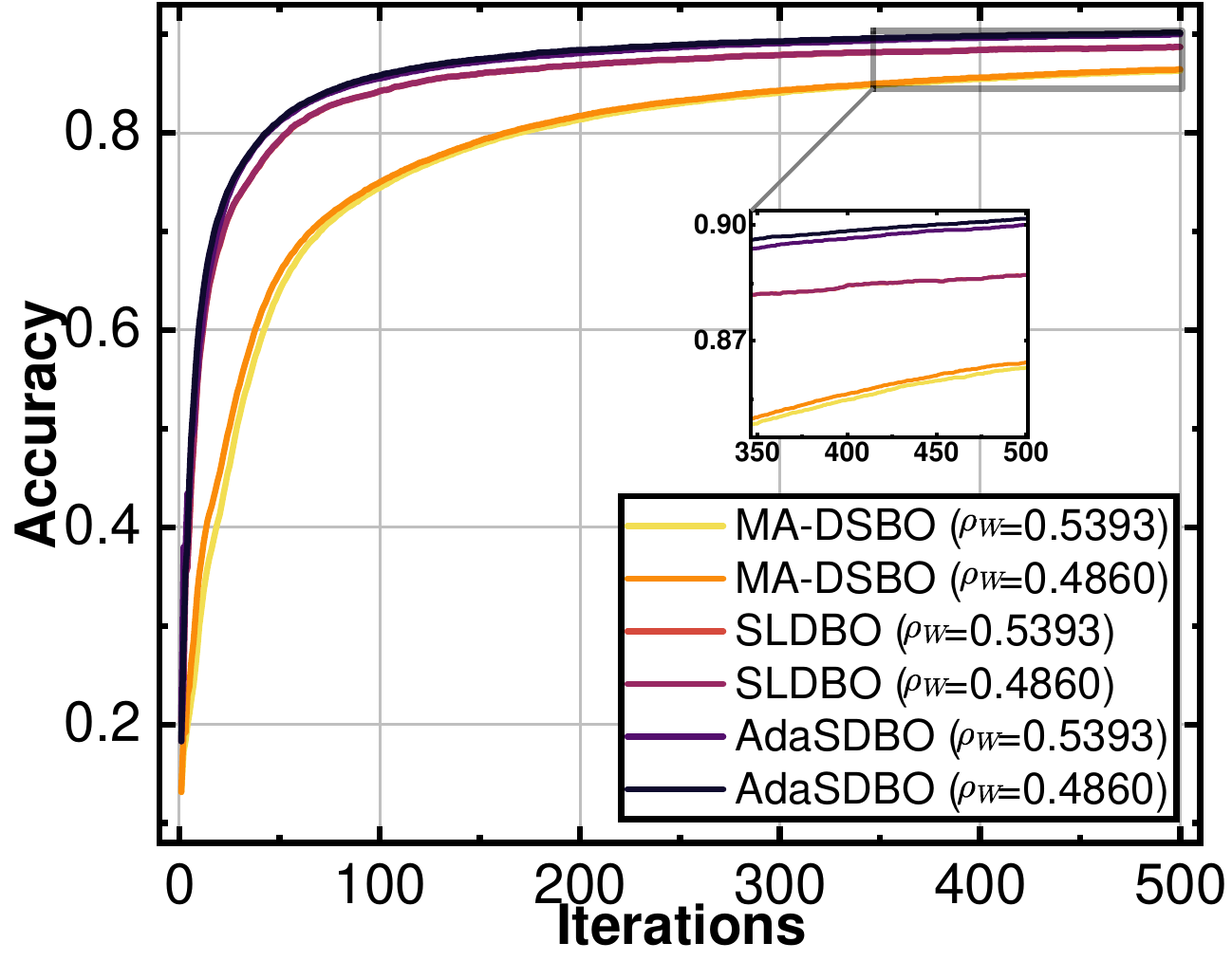}
    \par\vspace{-0.1cm}
    \makebox[\textwidth]{{\hspace{0.66cm}(a) Test accuracy}}
\end{minipage}
\begin{minipage}[b]{0.41\textwidth}
    \includegraphics[width=\textwidth]{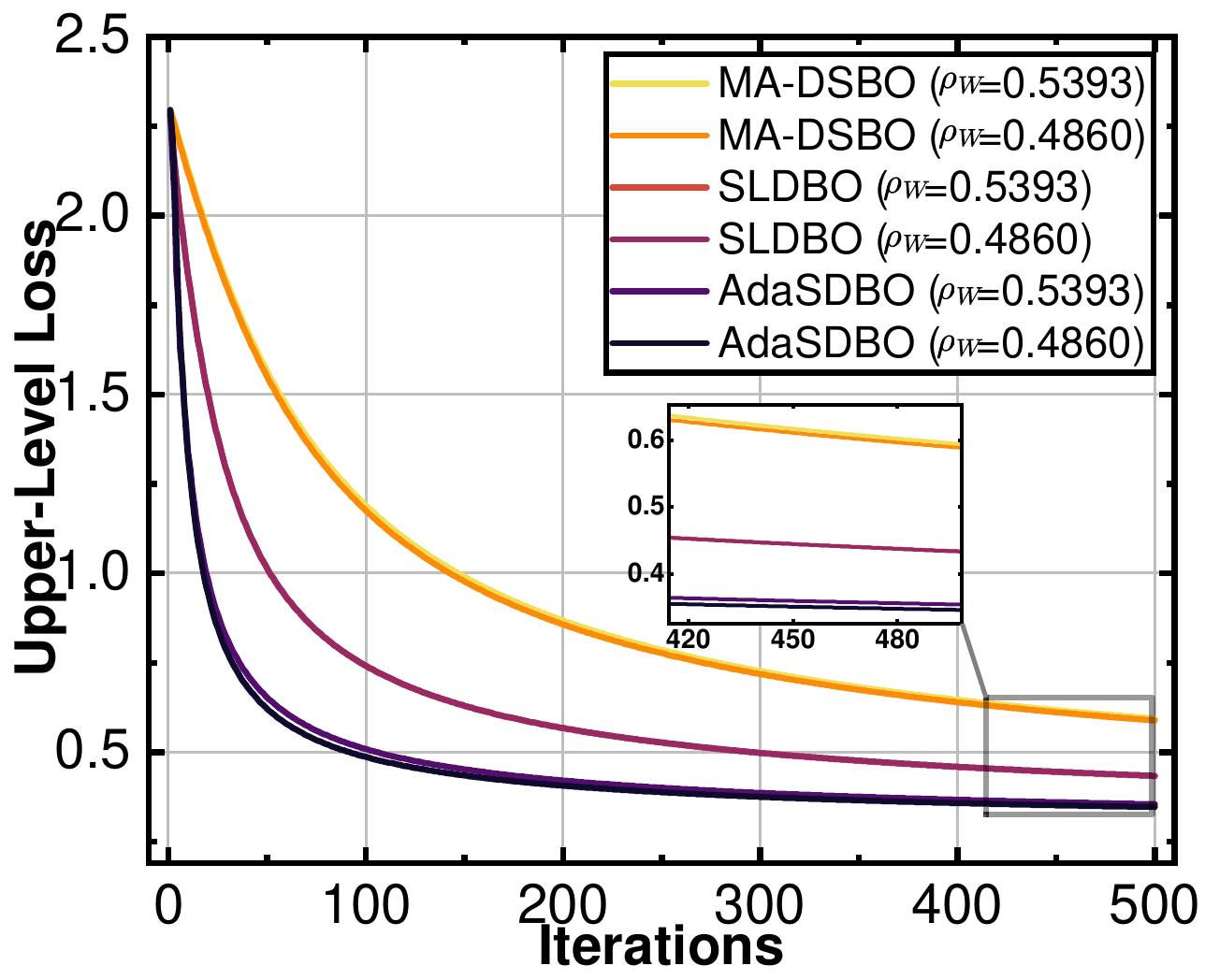}
    \par\vspace{-0.1cm} 
    \makebox[\textwidth]{{\hspace{0.66cm}(b) Upper-level loss}}
\end{minipage}
\caption{Test accuracy and upper-level loss for MNIST dataset with different $\rho_W$.}
\label{fig:rhomnist}
\end{figure*}

\begin{figure*}[t]
\centering
\begin{minipage}[b]{0.41\textwidth}
    \includegraphics[width=\textwidth]{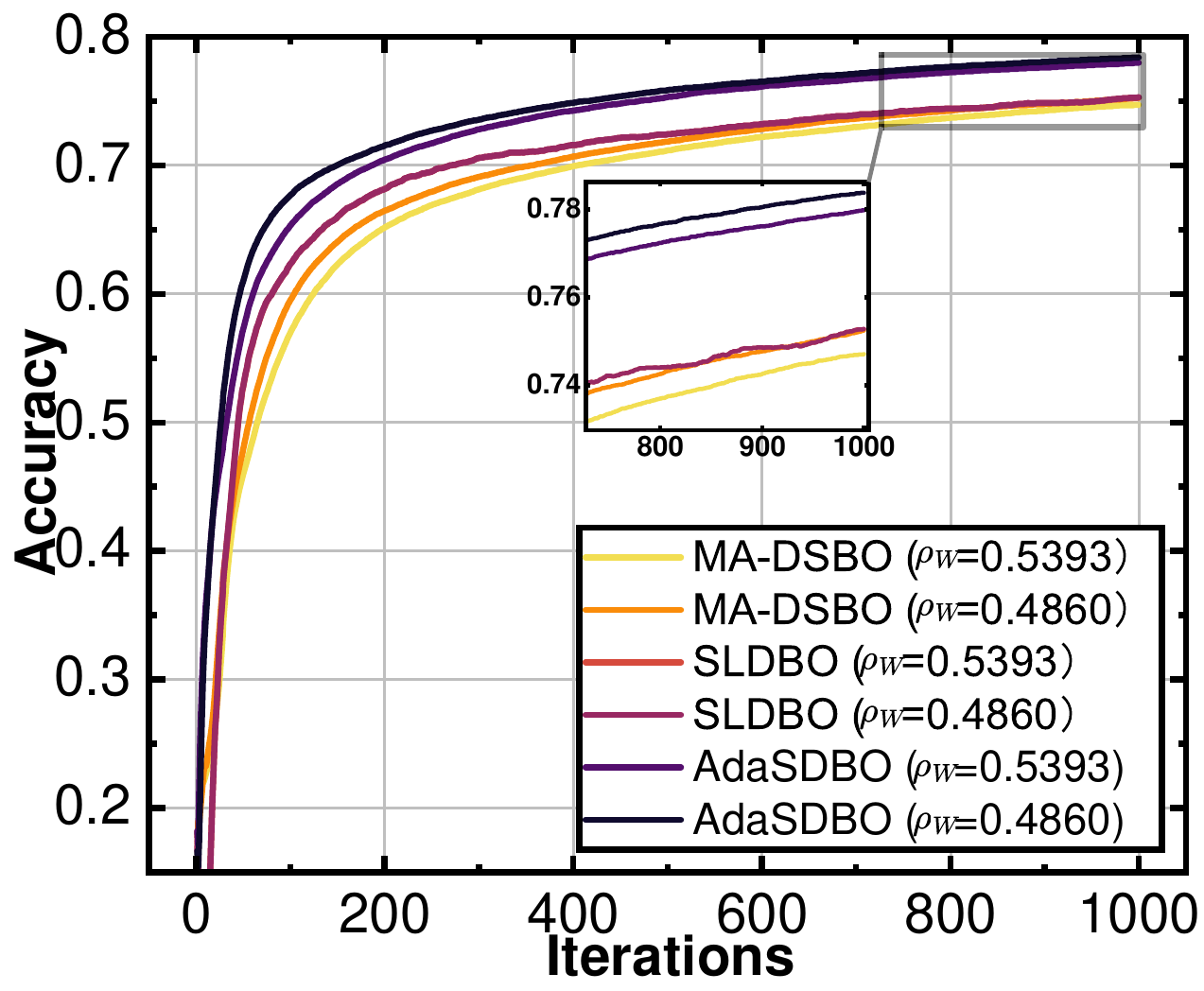}
    \par\vspace{-0.1cm}
    \makebox[\textwidth]{{\hspace{0.66cm}(a) Test accuracy}}
\end{minipage}
\begin{minipage}[b]{0.41\textwidth}
    \includegraphics[width=\textwidth]{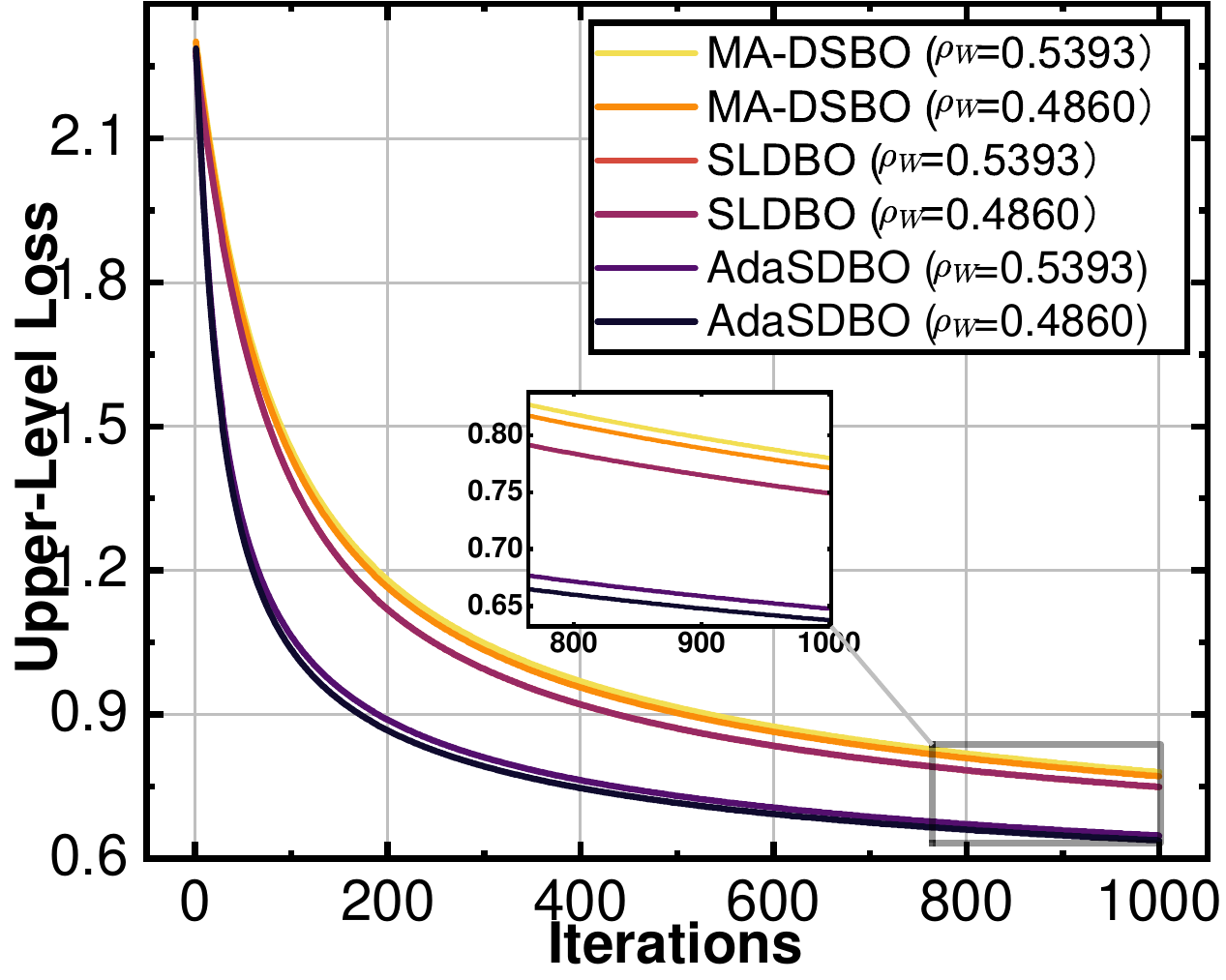}
    \par\vspace{-0.1cm}
    \makebox[\textwidth]{{\hspace{0.66cm}(b) Upper-level loss}}
\end{minipage}
\caption{Test accuracy and upper-level loss for FMNIST dataset with different $\rho_W$.}
\label{fig:rhofmnist}
\end{figure*}

\begin{figure*}[t]
  \centering
  \begin{minipage}[b]{0.32\textwidth}
    \includegraphics[width=\textwidth]{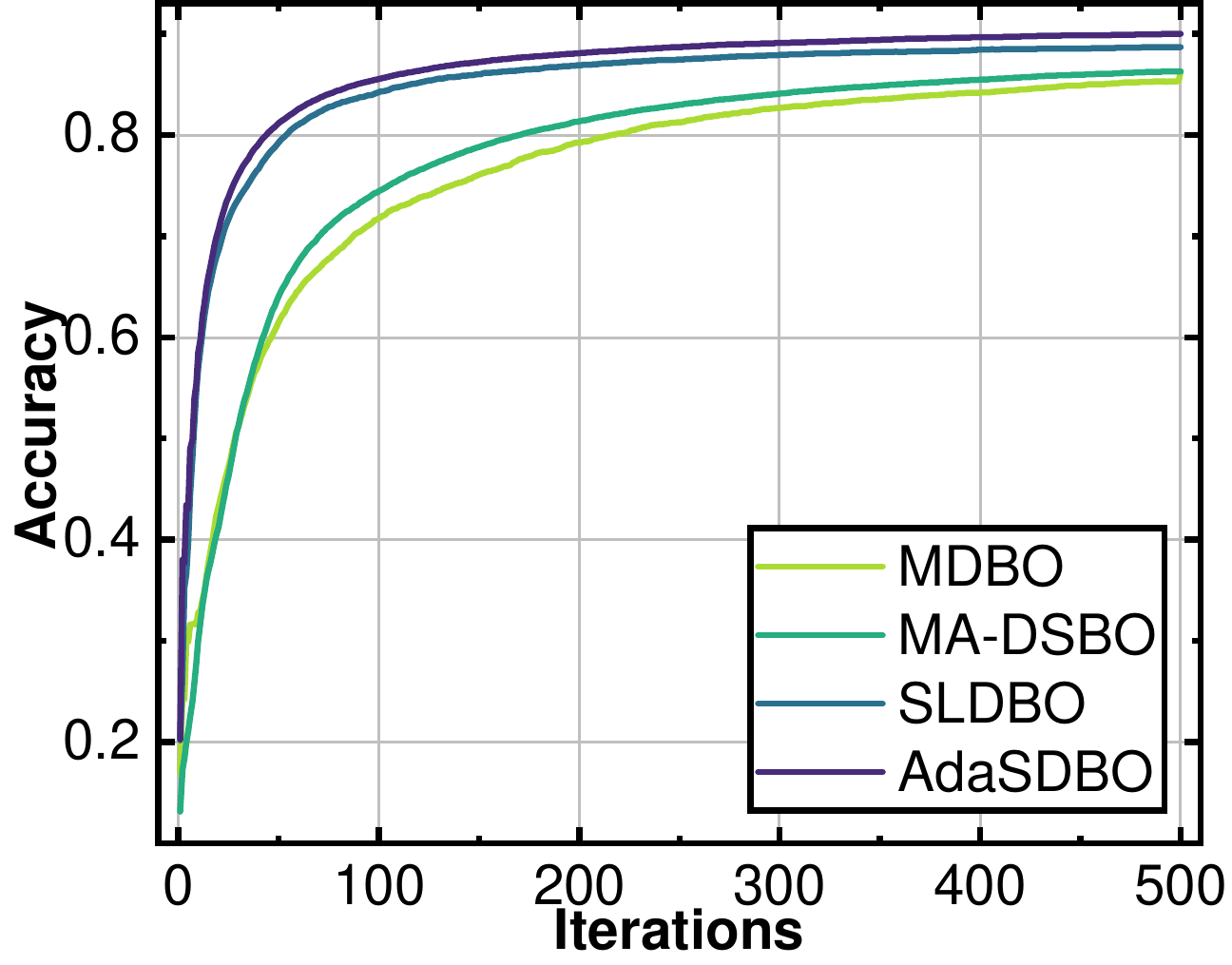}
    \par\vspace{-0.1cm} 
    \makebox[\textwidth]{{\hspace{0.66cm}(a) MNIST with $n=5$}}
  \end{minipage}
  \begin{minipage}[b]{0.32\textwidth}
    \includegraphics[width=\textwidth]{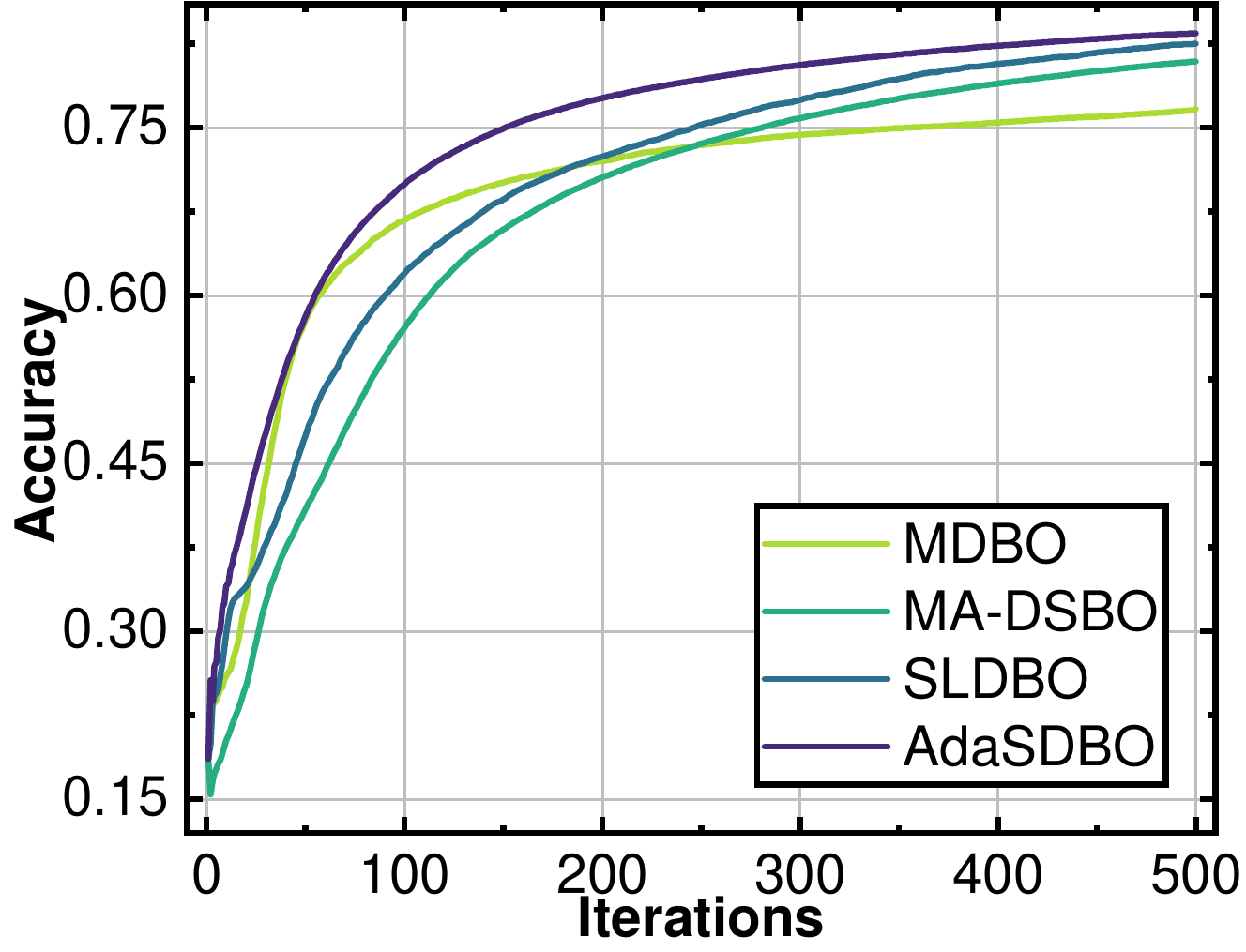}
    \par\vspace{-0.1cm} 
    \makebox[\textwidth]{{\hspace{0.66cm}(b) MNIST with $n=8$}}
  \end{minipage}
  \begin{minipage}[b]{0.32\textwidth}
    \includegraphics[width=\textwidth]{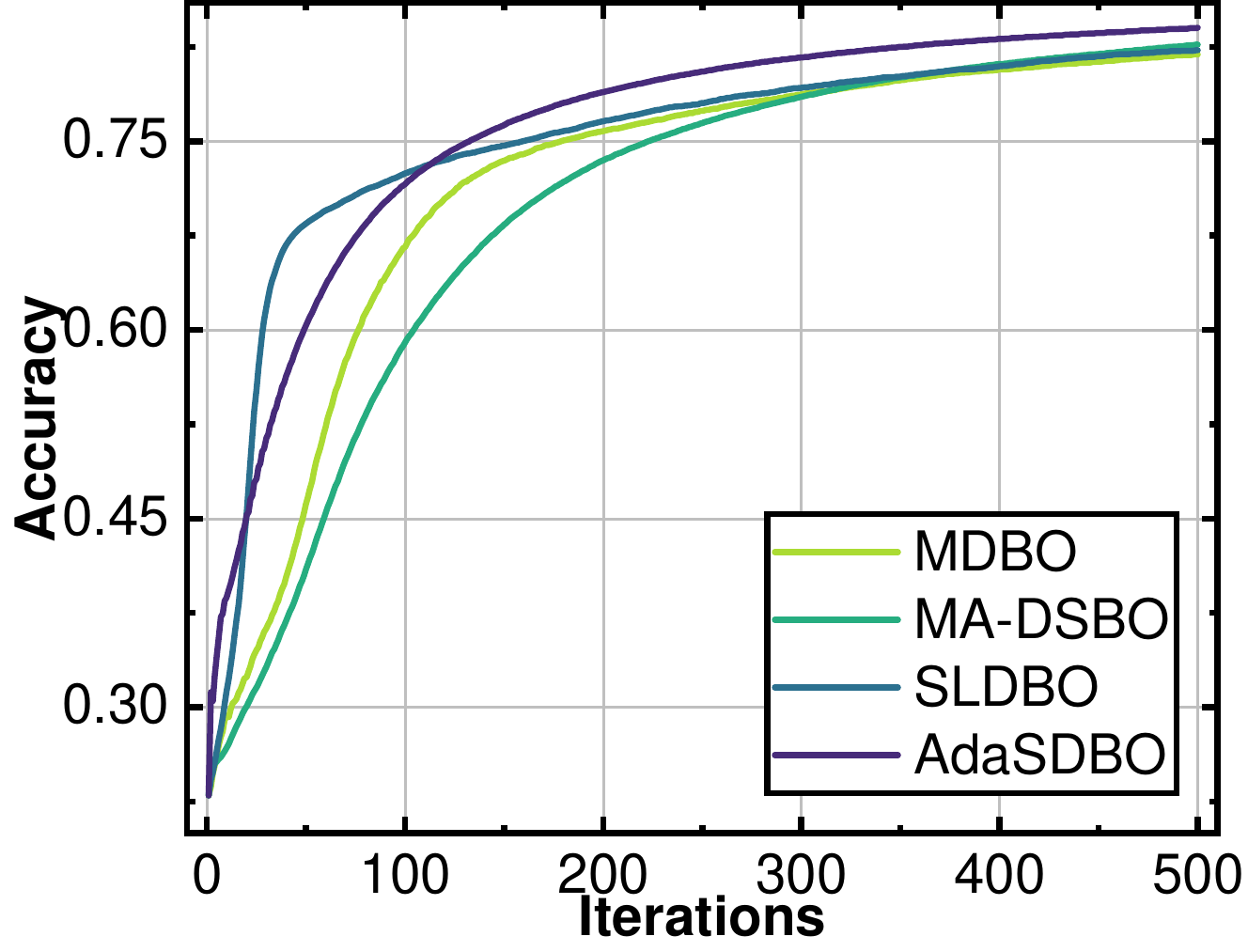}
    \par\vspace{-0.1cm} 
    \makebox[\textwidth]{{\hspace{0.66cm}(c) MNIST with $n=12$}}
  \end{minipage}
  \caption{Scalability analysis on the MNIST dataset under varying network sizes ($n = 5, 8, 12$).}
  \label{figsca1}
\end{figure*}

\subsection{Real-World Data Experiments}\label{realesec}
For the real-world data experiment, we apply our method to hyperparameter optimization on the MNIST dataset \citep{lecun1998gradient} and Fashion-MNIST (FMNIST) \citep{xiao2017fashion} dataset. Following \citep{grazzi2020iteration}, the functions \(f_i\) and \(l_i\) are defined as:
\begin{align}
    f_i(\lambda, \omega) &= \frac{1}{|D'_i|} \sum_{(x_e, y_e) \in D'_i} \ell(x_e^\top \omega, y_e), \nonumber\\
    l_i(\lambda, \omega) &= \frac{1}{|D_i|} \sum_{(x_e, y_e) \in D_i} \ell(x_e^\top \omega, y_e) + \frac{1}{cp} \sum_{j=1}^c \sum_{k=1}^p e^{\lambda_k} \omega_{jk}^2,\nonumber
\end{align}
where \(c=10\) and \(p=784\) denote the number of classes and features, respectively, \(\omega \in \mathbb{R}^{c \times p}\) is the model parameter, and \(\ell\) denotes the cross-entropy loss. \(D_i\) and \(D'_i\) represent the training and validation sets, respectively. The batch size for each computing agent is set to 1,000.

\subsection{Decentralized Meta-Learning Experiments}\label{dmetasec}
In our decentralized meta-learning experiment, following the MAML framework~\citep{finn2017model}, we consider a setting involving \(M\) distinct tasks, denoted by \(\{T_q\}_{q=1}^M\). Each task \(T_q\) is associated with a loss function \(L(x, y_q)\), where \(x\) denotes a shared embedding parameter across tasks, and \(y_q\) represents a task-specific parameter. The objective of meta-learning is to identify a universal parameter \(x^*\) that facilitates fast adaptation to new tasks by enabling efficient fine-tuning of \(y_q\) using a limited number of data points and update steps.

This problem naturally fits within a bilevel optimization framework. At the lower level, given a fixed \(x\), each task seeks the corresponding optimal adaptation parameter \(y_q^*\) by minimizing the loss over its training data \(D_q^{\text{tr}}\). The upper-level optimization then aims to select a shared parameter \(x\) such that the adapted models \(y_q^*\) perform well on the corresponding validation data \(D_q^{\text{val}}\). Let \(y^* = \text{col}\{y_1^*, \dots, y_M^*\}\) denote the collection of all task-specific solutions.

Unlike traditional centralized meta-learning, where all data is accessible at a single location, we consider a decentralized setup in which training and validation data for each task \(T_q\) are partitioned across \(n\) agents. Specifically, each agent \(i \in [n]\) maintains its own local training dataset \(D_{i,q}^{\text{tr}}\) and validation dataset \(D_{i,q}^{\text{val}}\) for task \(T_q\). Given a shared parameter \(x\), the local base-learners collaboratively solve for \(y_q^*(x)\) using decentralized lower-level optimization. The upper-level meta-update of \(x\) is then performed through cooperation among agents based on their local validation losses.

The decentralized bilevel optimization problem is formally expressed as:
\begin{align}
\min_x \quad & F(x) := \frac{1}{n} \sum_{i=1}^n \frac{1}{M} \sum_{q=1}^M f_{i,q}(x, y_q^*(x)), \notag\\
\text{s.t.} \quad & y_q^*(x) := \arg\min_y \frac{1}{n} \sum_{i=1}^n \frac{1}{M} \sum_{q=1}^M l_{i,q}(x, y),\notag
\end{align}
where $f_{i,q}(x, y_q^*) = \frac{1}{|D_{i,q}^{\text{val}}|} \sum_{(x, y_q^*) \in D_{i,q}^{\text{val}}} L(x, y_q^*(x))$ and $l_{i,q}(x, y) = \frac{1}{|D_{i,q}^{\text{tr}}|} \sum_{(x, y_q) \in D_{i,q}^{\text{tr}}} L(x, y_q) + R_{i,x}(y_q)$,
with \(R_{i,x}(y)\) denoting a strongly-convex regularizer with respect to \(y\). The experiment was conducted over 32 batches of tasks across 1,000 iterations. Each task included a training dataset and a validation dataset, both configured for 5-way classification with 50 shots per class. 
Specifically, the training and validation data were distributed among different agents to enable cooperative learning. 
For each task, 30\% of the data from the $i$-th class was assigned to agent $i$, while the remaining 70\% was evenly distributed among the other agents.

\subsection{Configurations}
All experiments were performed with $n=5$ using PyTorch \citep{paszke2019pytorch}. The network topology was configured as a ring topology, where the weight matrix \(W = (w_{ij})\) is defined as: 
\begin{equation*}
    w_{ii} = w, \quad w_{i,i+1} = w_{i,i-1} = \frac{1-w}{2}, 
\end{equation*}
where $w \in (0,1)$, \(w_{1,0} = w_{1,n}\), and \(w_{n,n+1} = w_{n,1}\). In this setup, each agent \(i\) is only connected to its immediate neighbors \(i-1\) and \(i+1\) for $i=1,\cdots,n$, with the indices \(0\) and \(n+1\) representing \(n\) and \(1\), respectively.

\begin{figure*}[t]
  \centering
  \begin{minipage}[b]{0.32\textwidth}
    \includegraphics[width=\textwidth]{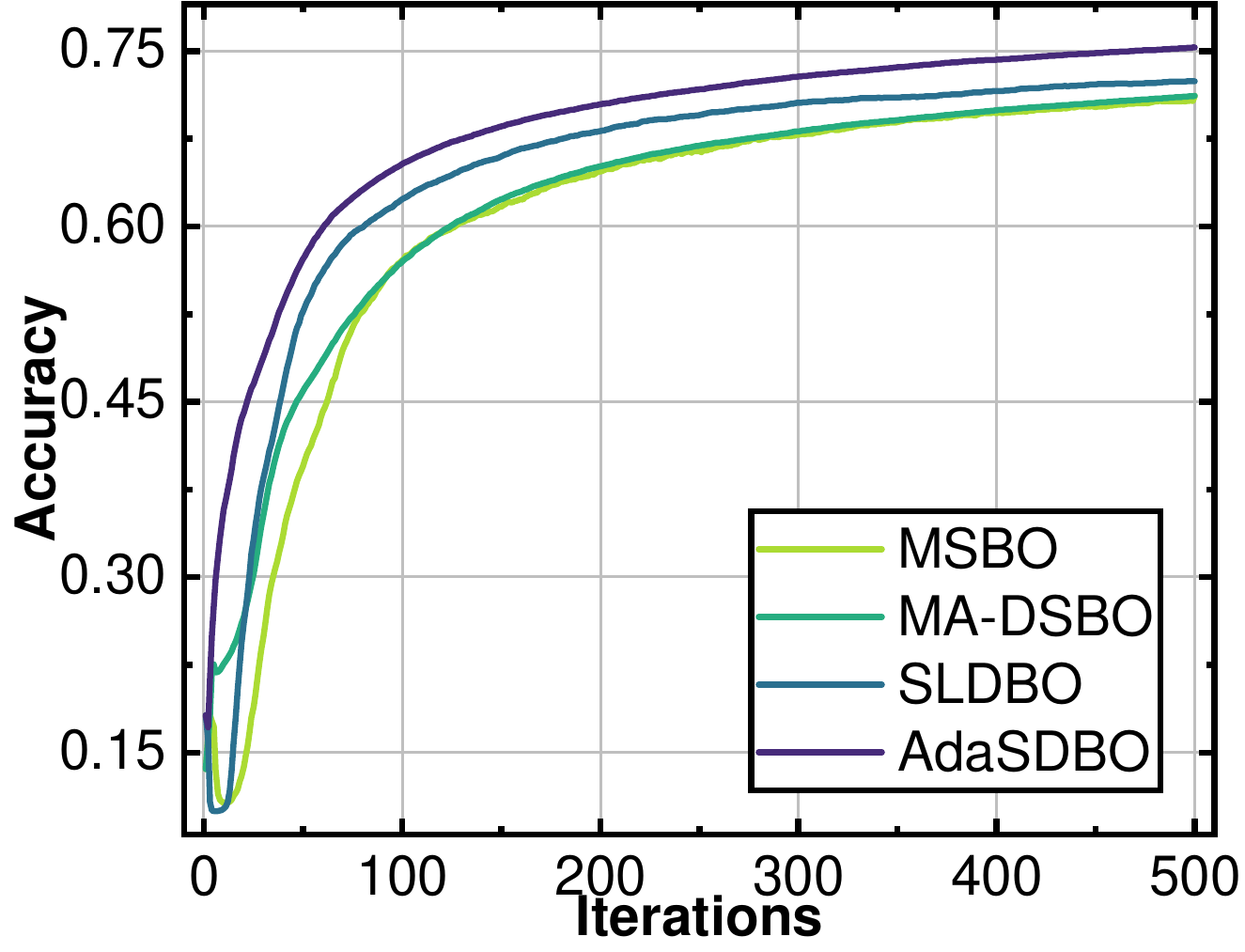}
    \par\vspace{-0.1cm} 
    \makebox[\textwidth]{{\hspace{0.66cm}(a) FMNIST with $n=5$}}
  \end{minipage}
  \begin{minipage}[b]{0.32\textwidth}
    \includegraphics[width=\textwidth]{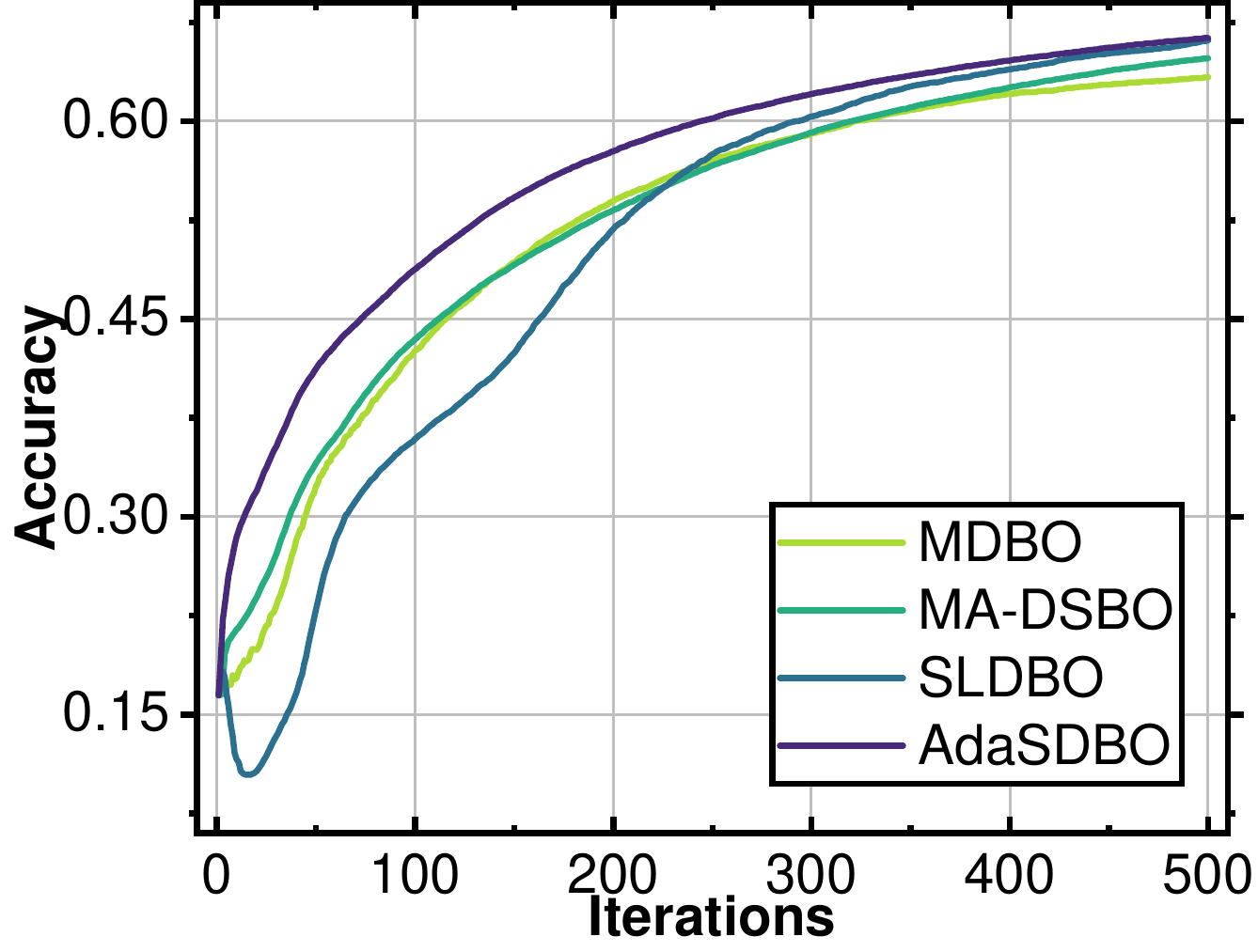}
    \par\vspace{-0.1cm} 
    \makebox[\textwidth]{{\hspace{0.66cm}(b) FMNIST with $n=8$}}
  \end{minipage}
  \begin{minipage}[b]{0.32\textwidth}
    \includegraphics[width=\textwidth]{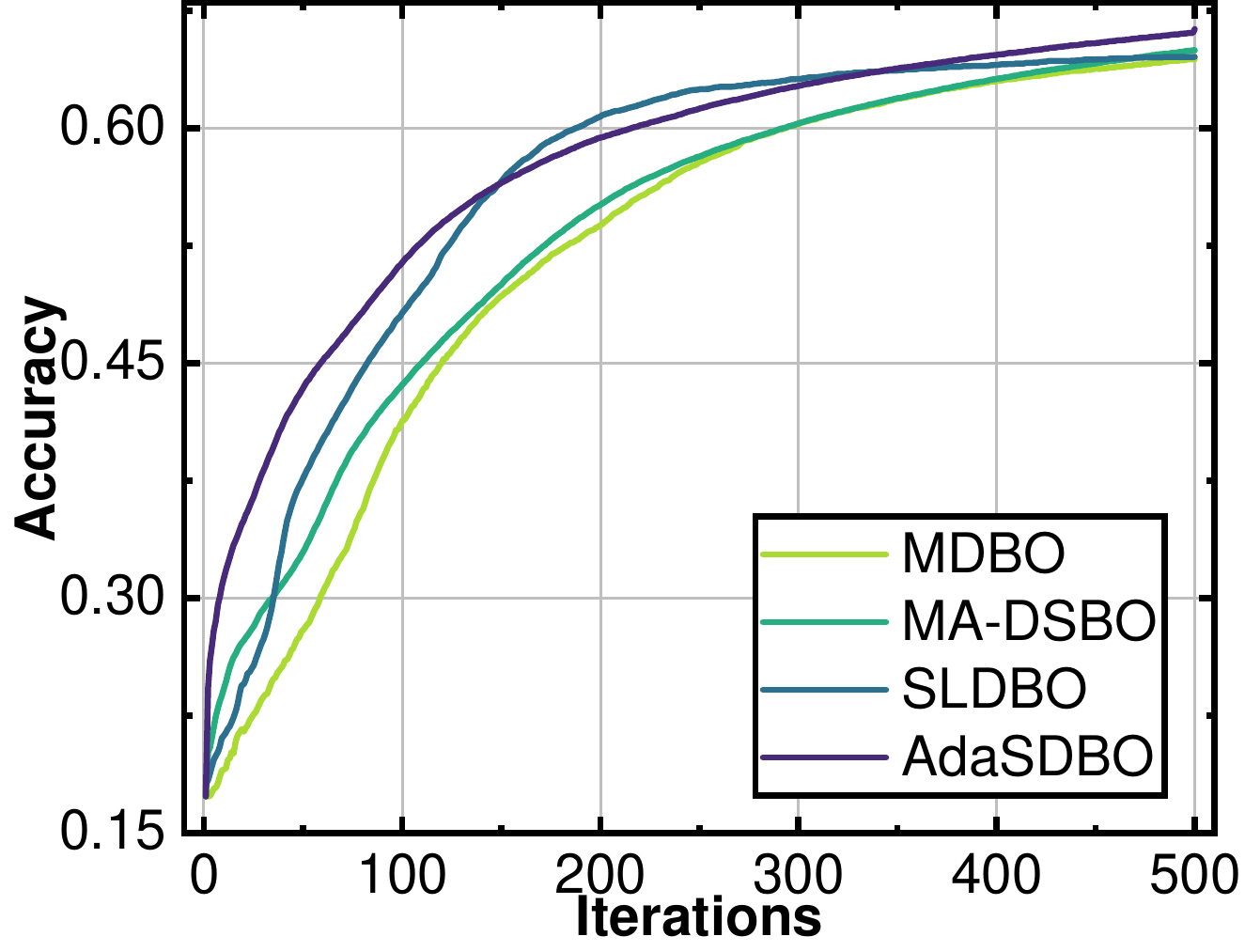}
    \par\vspace{-0.1cm} 
    \makebox[\textwidth]{{\hspace{0.66cm}(c) FMNIST with $n=12$}}
  \end{minipage}
  \caption{Scalability analysis on the FMNIST dataset under varying network sizes ($n = 5, 8, 12$).}
  \label{figsca2}
\end{figure*}

\begin{table}[H]
\centering
\small
\setlength{\tabcolsep}{3pt}
\caption{Test accuracy on the MNIST and FMNIST datasets under different communication topologies ($n = 8$).}
\begin{tabular}{lcccccc}
\toprule
\multirow{2}{*}{Algorithms} & \multicolumn{3}{c}{MNIST} & \multicolumn{3}{c}{FMNIST} \\
\cmidrule(lr){2-4} \cmidrule(lr){5-7}
 & Ring & Ladder & Random & Ring & Ladder & Random \\
\midrule
AdaSDBO & $0.908 \pm 0.001$ & $0.911 \pm 0.001$ & $0.913 \pm 0.001$ & $0.774 \pm 0.003$ & $0.788 \pm 0.003$ & $0.790 \pm 0.003$ \\
SLDBO & $0.871 \pm 0.002$ & $0.871 \pm 0.001$ & $0.871 \pm 0.001$ & $0.758 \pm 0.002$ & $0.758 \pm 0.002$ & $0.758 \pm 0.001$ \\
MA-DSBO & $0.850 \pm 0.001$ & $0.850 \pm 0.001$ & $0.850 \pm 0.002$ & $0.709 \pm 0.002$ & $0.716 \pm 0.001$ & $0.719 \pm 0.002$ \\
MDBO & $0.753 \pm 0.002$ & $0.754 \pm 0.001$ & $0.754 \pm 0.002$ & $0.650 \pm 0.001$ & $0.650 \pm 0.002$ & $0.653 \pm 0.001$ \\
\bottomrule
\end{tabular}
\label{topologydifferent}
\end{table}

For all experiments, except for the test accuracy versus stepsize comparison, we use the following parameter settings. For the baseline methods SLDBO and MA-DSBO, the stepsizes for updating \(x\) and \(v\) are set to 0.01, while the stepsize for updating \(y\) is set to 0.02, following the optimal stepsize order described in \citep{dong2023single,chen2023decentralized}. For the baseline methods DBO and MDBO, the stepsizes for updating both \(x\) and \(y\) are set to 0.01. For AdaSDBO, we set \(\gamma_x = \gamma_y = \gamma_v = 1\) and initialize \(m^x_{i,0} = m^y_{i,0} = m^v_{i,0} = 10\), \(\forall i \in [n]\). All experiments were conducted on a host machine equipped with an Intel(R) Xeon(R) W9-3475X CPU running at 2.20\,GHz (maximum turbo frequency: 4.80\,GHz), featuring 36 physical cores and 72 threads. The system was configured with 256\,GB of DDR5 ECC RAM and a single NVIDIA(R) RTX(TM) A6000 GPU with 48\,GB of memory.

\subsection{Additional Results}\label{addiresult}
For the synthetic dataset, we increased the data heterogeneity parameter \(r\) to 5 and analyzed the convergence performance of different methods under two data dimensions (\(p = 50\) and \(p = 200\)). It can be observed in \cref{fig:syntheticr5} that our proposed algorithm consistently outperforms the baseline methods in both convergence and test accuracy, even as the level of data heterogeneity increases.
This superior performance can be attributed to the adaptive stepsizes design, which enables our algorithm to dynamically adjust stepsizes to accommodate varying data distributions. Consequently, our proposed method demonstrates robust performance across different data heterogeneity settings, effectively adapting to changes in the data environment.

\begin{figure*}[t]
\centering
\begin{minipage}[b]{0.43\textwidth}
    \includegraphics[width=\textwidth]{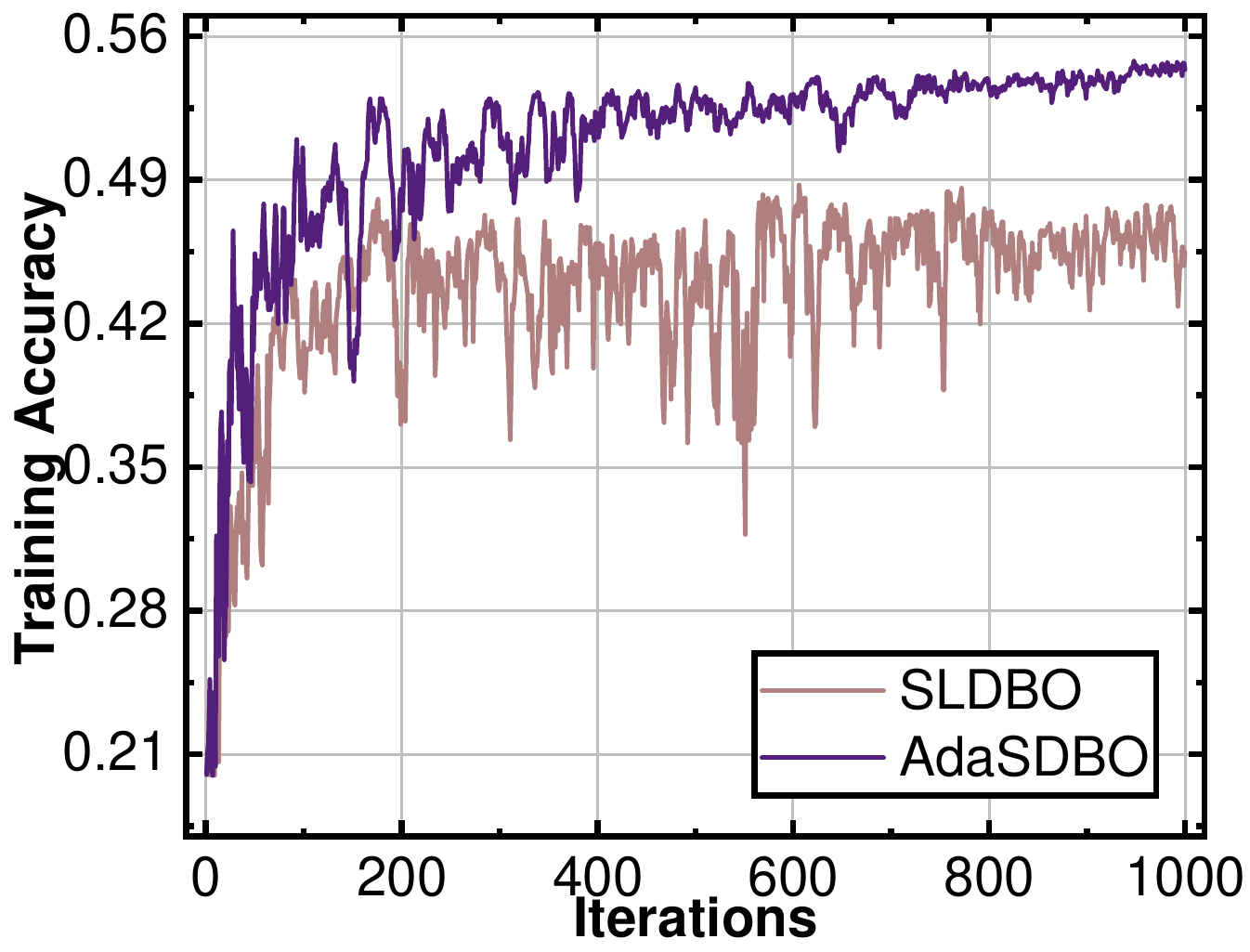}
    \par\vspace{-0.1cm} 
    \makebox[\textwidth]{{\hspace{0.66cm}(a) Training Accuracy}}
\end{minipage}
\begin{minipage}[b]{0.43\textwidth}
    \includegraphics[width=\textwidth]{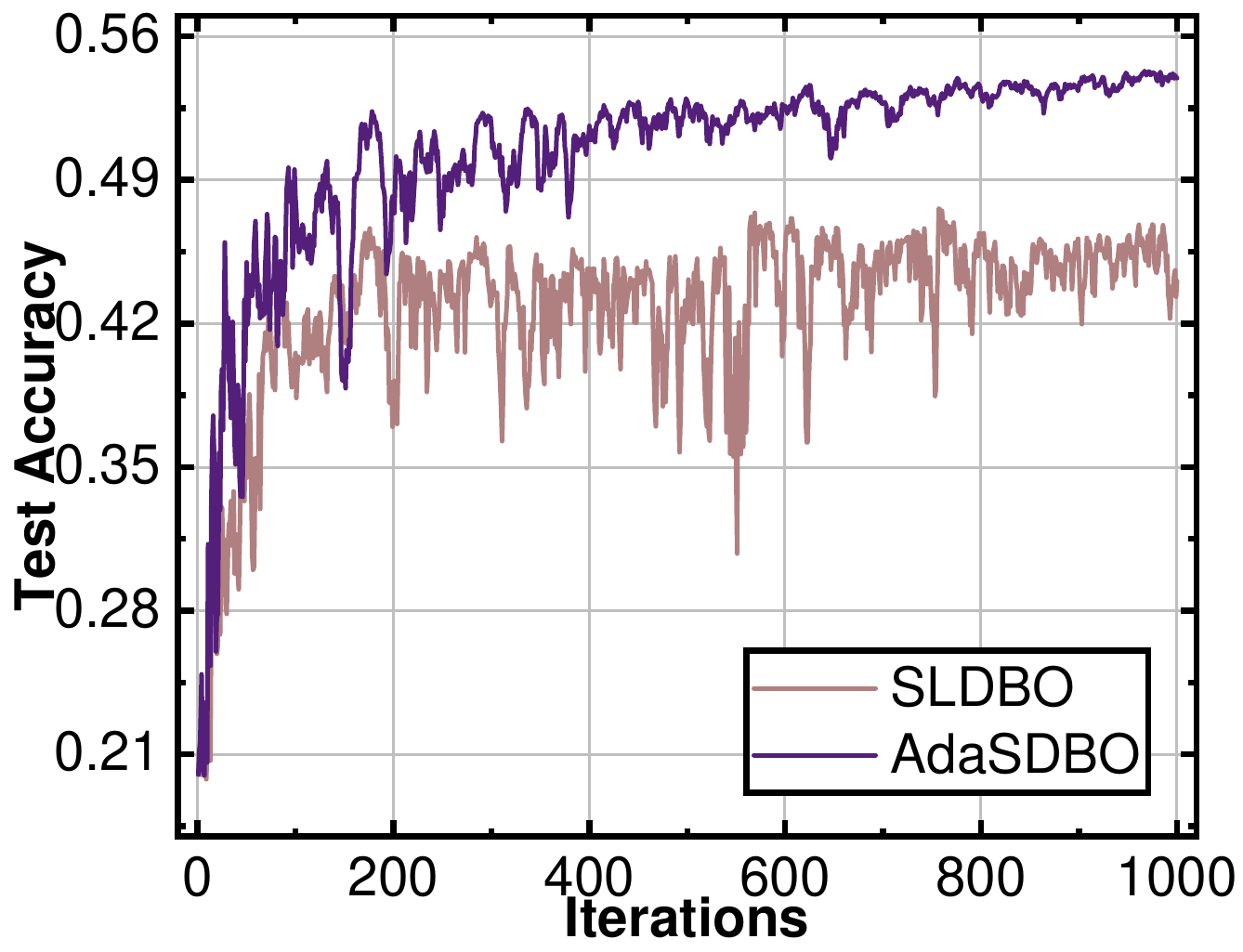}
    \par\vspace{-0.1cm} 
    \makebox[\textwidth]{{\hspace{0.66cm}(b) Test Accuracy}}
\end{minipage}
\caption{Performance comparison between AdaSDBO and SLDBO under the MAML framework, where each node employs an identical CNN for 5-way, 50-shot classification on the CIFAR-10 dataset.}
\label{fig:meta}
\end{figure*}

\begin{table}[t]
\centering
\small 
\setlength{\tabcolsep}{3pt} 
\caption{Performance comparison between AdaSDBO and SLDBO under varying network sizes in the decentralized meta-learning task.}
\label{tab:convergence_agents}
\begin{tabular}{ccccccc}
\toprule
\multirow{2}{*}{Algorithms} & \multicolumn{2}{c}{$n = 10$} & \multicolumn{2}{c}{$n = 20$} & \multicolumn{2}{c}{$n = 30$} \\
\cmidrule(lr){2-3} \cmidrule(lr){4-5} \cmidrule(lr){6-7}
 & Train Acc & Test Acc & Train Acc & Test Acc & Train Acc & Test Acc \\
\midrule
AdaSDBO & $0.543 \pm 0.002$ & $0.534 \pm 0.002$ & $0.542 \pm 0.002$ & $0.534 \pm 0.001$ & $0.538 \pm 0.001$ & $0.533 \pm 0.001$ \\
SLDBO & $0.503 \pm 0.004$ & $0.488 \pm 0.005$ & $0.472 \pm 0.006$ & $0.461 \pm 0.006$ & $0.486 \pm 0.006$ & $0.475 \pm 0.007$ \\
\bottomrule
\end{tabular}
\label{numbermetatestn102030}
\end{table}

Furthermore, we assessed the performance of different methods across varying network connectivity levels (\(\rho_W\)) on the synthetic, MNIST, and FMNIST datasets. As depicted in \cref{fig:rhosynthetic}, \cref{fig:rhomnist}, and \cref{fig:rhofmnist}, increasing the network connectivity ({i.e.,} a decrease in \(\rho_W\) from 0.5393 to 0.4860), leads to improved accuracy for all methods across the different datasets. Notably, our proposed algorithm maintains superior convergence performance compared to all baseline methods on each dataset, validating its effectiveness and reliability under different levels of network connectivity.

In \cref{figsca1} and \cref{figsca2}, we evaluate the broader scalability of our proposed method by varying the number of nodes \(n = 5, 8, 12\) on the MNIST and FMNIST datasets. We compare AdaSDBO against several baseline algorithms, including SLDBO \citep{dong2023single}, MA-DSBO \citep{chen2023decentralized}, and MDBO \citep{gao2023convergence}. The results show that AdaSDBO achieves convergence performance competitive with these state-of-the-art methods. Furthermore, AdaSDBO maintains stable performance across different system configurations, underscoring its ease of deployment and practical applicability.
This stability can be attributed to the problem-parameter-free nature of AdaSDBO, which eliminates the need for manually tuned stepsizes. In contrast, other methods in decentralized settings often suffer from sensitivity to stepsize selection due to their reliance on problem-specific parameters, which are typically unknown or difficult to estimate in practice. As a result, the parameter-free property of AdaSDBO makes it particularly well-suited for real-world decentralized applications.

To assess sensitivity to network structure, we conducted experiments under three commonly used communication topologies—ring, ladder, and random—and compared the performance of AdaSDBO with baseline methods. Structural details for these topologies are provided in the Appendix of \citep{li2024problem}. The results summarized in \cref{topologydifferent} indicate that stronger connectivity leads to faster and more stable convergence: in particular, the random topology, which has the highest connectivity, yields the fastest and most stable convergence for all methods. Across the three topologies, AdaSDBO consistently outperforms the baselines, maintaining strong performance under variations in the communication topology. This robustness is facilitated by the adaptive stepsize mechanism of AdaSDBO, which accommodates topology-induced heterogeneity.

For the decentralized meta-learning experiment, we compared our proposed algorithm with the single-loop decentralized bilevel optimization method SLDBO~\citep{dong2023single}. \Cref{fig:meta} shows that AdaSDBO consistently outperforms SLDBO in both average training accuracy across all nodes and test accuracy, demonstrating the effectiveness of our approach in decentralized meta-learning tasks. We further increase the number of agents to 10, 20, and 30; the results in \cref{numbermetatestn102030} indicate that AdaSDBO remains stable and competitive as the network scales, outperforming the baselines even in more complex settings.
The superior performance of AdaSDBO stems from its problem-parameter-free design, which enables stepsize selection without requiring knowledge of problem-specific parameters. Furthermore, the adaptive nature of our method allows AdaSDBO to dynamically adjust its learning dynamics and consistently achieve optimal convergence rates—even in decentralized settings where hyperparameter tuning is particularly difficult for other methods. This advantage becomes even more evident in complex decentralized meta-learning tasks, further underscoring the robustness and scalability of our proposed approach.

\end{document}